\documentclass[hidelinks,onefignum,onetabnum]{siamart220329}

\usepackage{lipsum}
\usepackage{amsfonts}
\usepackage{graphicx}
\usepackage{epstopdf}
\usepackage{algorithmic}
\usepackage{tikz}
\usetikzlibrary{decorations.pathmorphing,calc}
\ifpdf
  \DeclareGraphicsExtensions{.eps,.pdf,.png,.jpg}
\else
  \DeclareGraphicsExtensions{.eps}
\fi

\def \hueco{\noalign{\medskip}}
\def \pato{\forall\,}
\def \beq{\begin{equation}}
	\def \eeq{\end{equation}}
\def \ba{\begin{array}}
	\def \ea{\end{array}}
\def \dis{\displaystyle}

\def\x  {\boldsymbol x}
\def\n        {{\boldsymbol n}}
\def\tz     {\widetilde{z}}

\def\tpsi       {\widetilde{\psi}}
\def\tw      {\widetilde{w}}
\def\tchi   {\widetilde{\chi}}
\newsiamremark{remark}{Remark}
\newsiamremark{hypothesis}{Hypothesis}
\crefname{hypothesis}{Hypothesis}{Hypotheses}
\newsiamthm{claim}{Claim}

\headers{Uncondtionally Energy Stable Second Order Scheme for ME Model}{N. S. Sharma, and G. Tierra}

\title{Unconditionally Energy Stable Second Order Numerical Scheme for a Microemulsion model%\thanks{Submitted to the editors DATE.}
}

\author{Natasha S. Sharma\thanks{Department of Mathematical Sciences, The University of Texas at El Paso , El Paso TX 
  (\email{nssharma@utep.edu}).}
\and Giordano Tierra\thanks{Department of Mathematics, , University of North Texas, Denton TX 
  (\email{gtierra@unt.edu}).}
}

\usepackage{amsopn}

\ifpdf
\hypersetup{
	pdftitle={Unconditionally Energy Stable Second Order Numerical Scheme for a Microemulsion model},
	pdfauthor={N. S. Sharma, and G. Tierra}
}
\fi

\begin{document}
	
	\maketitle
	
	% REQUIRED
	\begin{abstract}
We present a numerical scheme for solving a sixth-order Cahn-Hilliard type equation that captures the dynamics of phase transitions in a ternary mixture consisting of two immiscible fluids and a surface active molecule that is amphiphilic. We show that by considering a suitable midpoint approximation for the nonlinear terms in the differential equation, we obtain an unconditionally energy-stable numerical scheme that is second-order in time. 
We demonstrate that our proposed numerical scheme satisfies these key properties for a wide range of physical parameters in two and three dimensions. 
Moreover, we present the results of a numerical study to report on the impact of each physical parameter on the behavior of the dynamics of the phase transitions, which are in agreement with the experimental observations.

	\end{abstract}
	
	% REQUIRED
	
	% REQUIRED
	
	\section{Introduction}
	This paper is concerned with a second-order numerical scheme for the sixth-order Cahn-Hilliard type equations that model micro-emulsions. Here, `sixth-order' refers to the order of the spatial derivatives involved in this Cahn-Hilliard system and these high-order derivatives capture the curvature effects.
	On a macroscopic level, microemulsions are a thermodynamically stable mixture of two immiscible fluids and a surface active molecule called the surfactant introduced in the system to reduce the surface tension at the interface of these immiscible fluids. The microemulsion phase can be thought of as a single-phase structured fluid consisting of homogeneous regions of oil-rich, water-rich phases separated by surfactant monolayers that exist at the interface of these homogeneous regions. It is this remarkable property of the coexistence of water-oil-surfactant coupled with its ability to capture many essential static properties of the ternary oil-water-surfactant systems that makes microemulsion models suitable for capturing different physical processes such as enhanced oil recovery \cite{bera2015microemulsions} particularly recovery methods involving microemulsion flooding~\cite{enhancedOilME2009}, development of environmentally-friendly solvents \cite{paul2015ionic}, consumer and commercial cleaning product formulations \cite{klier2000properties}, and drug delivery systems \cite{callender2017microemulsion}.
	
	Despite the aforementioned applications {especially the extensive volume of literature in the pharmaceutical community on the use of microemulsions %as an effective drug delivery system 
		(see Figure 17 in the survey article~\cite{callender2017microemulsion} for a brief overview), as well as experimental studies on the mathematical model~\cite{gompper1994fluctuating, gompper1993ginzburg1, gompper1993ginzburg2,GompperSchick90,GompperSchick94,
			GompperZschocke92}, the advances in the mathematical and numerical analysis and the ensuing numerical simulations have been limited. One of the key challenges impeding this advancement is the presence of sixth-order spatial derivatives lurking in this gradient flow system. Additionally, there is a need to design a suitable temporal scheme that yields an unconditionally stable, solvable scheme with minimal numerical dissipation. We note in response to the challenges of developing a suitable numerical scheme, there are only two available publications~\cite{diegelsharma-me23, hoppelinsenmann-me18} to the best of the authors' knowledge. This mismatch between the advances in the experimental, physical and the mathematical/numerical analysis has motivated the authors in this work to develop a new numerical scheme to approximate solutions to this type of system.}
		%there is still a paucity on the analysis and numerical methods developed for solving the underlying class of sixth-order Cahn Hilliard-type equations that model microemulsions. 

	% is the choice of a suitable spatial discretization given the presence of the sixth-order spatial derivatives in the equation. The other challenge involves designing a temporal scheme that yields an unconditionally stable, solvable scheme with minimal numerical dissipation. 
%	\textcolor{magenta}{We note in response to the challenges in developing a suitable numerical scheme, to the best of the authors' knowledge, there are only two works~\cite{hoppelinsenmann-me18, diegelsharma-me23}.} 
{
To provide a context for the numerical scheme proposed in this work, we first discuss the existing numerical methods available in the literature. Hoppe and Linsenmann} in~\cite{hoppelinsenmann-me18} presented a pioneering numerical scheme where 
	this sixth-order Cahn-Hilliard type equation was rewritten as a parabolic equation and a nonlinear fourth-order partial differential equation. The spatial discretization involved a continuous Galerkin finite element method for this Ciarlet Raviart-type mixed formulation. We note that in using a continuous Galerkin finite element method for the nonlinear fourth order partial differential equation (PDE), the authors {in~\cite{hoppelinsenmann-me18}} end up with a continuous interior penalty (C0-IP) method~\cite{brennerfrontiers2010, brennergudi2012, engel2002} where the C$^1$-continuity requirement (for traditional $H^2$-conforming finite element methods discretizing the fourth order equation) is replaced with interior penalty techniques. The C0-IP method is characterized by a stabilization term that penalizes jumps in the gradient and is scaled by a penalty parameter. As is standard with other discontinuous Galerkin methods, this penalty parameter has to be suitably chosen to guarantee stability and other important properties. The temporal discretization in this work relied on the backward Euler method, and while sub-optimal error estimates were derived, the question of energy dissipation remained unaddressed. In a more recent work~\cite{diegelsharma-me23}, the authors also considered a C0-IP method for the spatial discretization, whereas for the temporal discretization, a variant of the Eyre's convex splitting scheme~\cite{eyre1998} was considered. Although this proposed scheme was shown to be unconditionally stable and uniquely solvable, 
	it is known that such a splitting scheme introduces a large amount of numerical dissipation, which prevents it from obtaining the right dynamics unless the time step is sufficiently small \cite{guillentierra2013, Xuetal2019}. Moreover, some of the proofs rely on considering a very specific range of model parameters, which might limit the applicability of the proposed scheme to all types of realistic situations.
	
{
In this paper, we propose a new numerical scheme where the spatial discretization is based on using continuous Lagrange finite elements. In order to do that we recast the sixth-order Cahn-Hilliard equation as a coupled system of three second-order elliptic equations. The idea of recasting a sixth-order Cahn-Hilliard type equation as a system of second-order equations dates back to the works of Backofen et al.~\cite{backofen2007},  Vignal and co-authors~\cite{vignal2015} and Pei et al.~\cite{pei2019}, on the phase-field-crystal equation, which mathematically can be viewed as a particular case of the system studied in this work. For the temporal discretization, we design a new approximation to guarantee unconditional stability with no numerical dissipation. One of the main advantages of our proposed scheme is that we can prove the key properties of unconditional stability and solvability for a wider range of model parameters than previous works on this topic. 
}
		
% equation due to the additional complexity introduced in our equation thanks to the additional nonlinearity appearing in the first-order derivatives of the phase field and the presence of the phase field-dependent coefficients \textcolor{magenta}{(see equations~\eqref{eq:deffunctiong} and~\eqref{eq:chem-pot} for the coefficient choices)}. Regarding temporal discretization, our approach is motivated by the midpoint approximations considered in the phase field context due to its second-order character and the small numerical dissipation that they introduce in the system 

	Our paper is organized as follows. Section~\ref{sec:model} describes the model. In section~\ref{sec:num-scheme}, we present our numerical scheme and prove the key properties satisfied by it. Section~\ref{sec:simulations} presents the numerical performance of the scheme, including the results of a study of parameters influencing the behaviors. Finally, concluding remarks are provided in section~\ref{sec:conclusions}. 
	\section{Model}\label{sec:model}
	Let $\Omega\subset\mathbb{R}^d$ (with $d=1,2,3$) be a bounded, convex polygonal spatial domain with Lipschitz boundary $\partial\Omega$ and let $[0,T]$ denote a finite time interval. Let $\phi:=\phi(\x,t)$ denote the scalar order parameter, which is proportional to the local differences between the oil and water concentrations.
	
	In this work, we consider the free energy proposed by Gompper et al., which has been physically justified by scattering experiments \cite{GompperSchick90, GompperSchick94, GompperZschocke92}. In particular, the total free energy of the microemulsion system is represented by
	\beq\label{eq:energy-orig}
	E(\phi)
	\,:=\,
	\int_\Omega
	\left(
	\beta f_0(\phi)
	+\frac12g(\phi)|\nabla\phi|^2
	+\frac{\lambda}2(\Delta\phi)^2
	\right)d\x\,,
	\eeq
where $\lambda>0$, $\beta>0$, and $g(\phi)$ are introduced to represent the balance between the different amphiphilic effects. {The coefficient $\lambda$ influences the interacting tendencies between the oil and water phases by incorporating the curvature effects of the interface \cite{gompper1994self}. 
In fact, for binary mixtures it is known that the second-order term in the expansion of the free energy in the region of nonuniform composition
corresponds to the term $(\Delta\phi)^2$ \cite{cahn1958}.} The function $f_0(\phi)$ represents the volumetric free energy whose minima are associated with the oil-rich phase ($\phi=-1$), the water-rich phase ($\phi=1$), and the microemulsion phase ($\phi=0$) respectively and is defined as:
	$$
	f_0(\phi)
	\,=\,
	(\phi^2 - 1)^2(\phi^2 + h_0)
	\,=\,
	\phi^6 
	+ (h_0 - 2)\phi^4
	+ (1 -2h_0)\phi^2
	+ h_0\,,
	$$
	with $h_0>0$ such that
	$$
	f_0'(\phi)
	\,=\,
	2\phi(\phi^2 - 1)\Big(2(\phi^2 + h_0)
	+(\phi^2 - 1)\Big)
	\,=\,
	6\phi^5
	+ 4(h_0 - 2)\phi^3
	+ 2(1 -2h_0)\phi\,.
	$$
	Moreover, the {scalar order parameter}-dependent coefficient $g(\cdot)$ is given by: 
	\beq\label{eq:deffunctiong}
	g(\phi)
	\,=\,
	g_2\phi^2
	+ g_0\,,
	\eeq
	with $g_2>0$ and $g_0\in\mathbb{R}$. {The properties of the amphiphile and its concentration determine the form of $f_0(\cdot), \ g(\cdot)$ and the magnitude of $\lambda$ \cite{gompper1993ginzburg1}. Furthermore, their mathematical structure is deduced from scattering experiments~\cite{GompperZschocke92}. } {A more detailed physical interpretation of the  various physical parameters involved is provided in~\cite{gompper1994self,PawlowZajaczkowski11}.}	
Then, the dynamics of the microemulsion are given by a sixth-order PDE of the Cahn-Hilliard \cite{cahn1958}
	form 
	\beq\label{eq:6thorderPDE}
	\phi_t
	-\nabla\cdot\left(M(\phi)\nabla\left(\frac{\delta E}{\delta \phi}\right)\right)
	=0\,,
	\eeq
	where { $	\phi_t\equiv\frac{\partial \phi}{\partial t}$,} $M(\phi)\geq0$ denotes a mobility {function that we will consider constant for simplicity ($M(\phi)=M>0$)} and
	\beq\label{eq:chem-pot}
	\frac{\delta E}{\delta \phi}
	\,=\,
	\beta f_0'(\phi) 
	+ \frac12g'(\phi)|\nabla\phi|^2
	-\nabla\cdot\big(g(\phi)\nabla\phi\big)
	+ \lambda\Delta \big(\Delta\phi\big)\,.
	\eeq
	In particular, equation \eqref{eq:6thorderPDE}-\eqref{eq:chem-pot} are supplemented with {the following Neumann type boundary conditions:
		\beq\label{eq:bound-cond}
		\nabla\Big(\frac{\delta E}{\delta \phi}\Big)\cdot\n\Big|_{\partial\Omega}
		\,=\,
		\Big(g(\phi)\nabla\phi - \lambda\nabla\Delta \phi\Big)\cdot\n\Big|_{\partial\Omega}
		\,=\,
		\nabla\phi\cdot\n\Big|_{\partial\Omega}
		\,=\,
		0\,
		\eeq
		and the initial condition 
		\beq\label{eq:initial-cond}
		\phi(\x,0)
		\,=\,
		\phi_0(\x)
		\quad
		\mbox{ in }
		\Omega\,.
		\eeq
	}
	Additionally, it satisfies a dissipative energy law (derived by testing with $\frac{\delta E}{\delta \phi}$):
	$$
	\frac{d}{dt}E(\phi(t)) 
	+ M\left\|\nabla\left(\frac{\delta E}{\delta \phi}\right)\right\|^2_{L^2}
	\,=\,0
	\quad
	\pato t>0\,.
	$$
	Furthermore, considering suitable boundary conditions, the equation 
	\eqref{eq:6thorderPDE} satisfies mass conservation, that is:
	$$
	\int_\Omega \phi_t\,d\x
	\,=\,
	\frac{d}{dt}\left(\int_\Omega \phi\, d\x\right)
	\,=\,
	0\,.
	$$
	The well-posedness of the initial boundary value equation system~\eqref{eq:6thorderPDE}-\eqref{eq:initial-cond} has been studied by Pawłow and Zaj\c{a}czkowski in~\cite{PawlowZajaczkowski11}.
	\subsection{Formulation as a system of second order PDEs}
	Introducing two auxiliary unknowns, $\mu(\x,t)$ and $\sigma(\x,t)$, we can rewrite~\eqref{eq:6thorderPDE} as a coupled system of three second-order PDEs:
	\beq\label{eq:system}
	\left\{\ba{rcl}
	\phi_t 
	- M\Delta\mu
	&=&0\,,
	\\ \hueco\dis
	\beta f_0'(\phi) 
	+ \frac12g'(\phi)|\nabla\phi|^2
	-\nabla\cdot\big(g(\phi)\nabla\phi\big)
	- \lambda\Delta \sigma
	&=&\mu\,,
	\\ \hueco\dis
	\sigma
	&=&-\Delta\phi\,,
	\ea\right.
	\eeq
	supplemented by the initial condition
	$$
	\phi\big|_{t=0}
	\,=\,
	\phi_0
	\quad
	\mbox{ in }
	\Omega,
	$$
	and the boundary conditions
	$$
	\nabla\mu\cdot\n\Big|_{\partial\Omega}
	\,=\,
	\Big(g(\phi)\nabla\phi + \lambda\nabla\sigma\Big)\cdot\n\Big|_{\partial\Omega}
	\,=\,
	\nabla\phi\cdot\n\Big|_{\partial\Omega}
	\,=\,
	0\,.
	$$
	The following lemma establishes the key properties of energy stability and mass conservation that hold true for the above system.
	\begin{lemma}
		System \eqref{eq:system} satisfies the following energy law:
		\beq\label{eq:energylaw}
		\frac{d}{dt}E(\phi,\sigma) 
		+ M\|\nabla\mu\|^2_{L^2}
		\,=\,0
		\,,
		\eeq
		with
		\beq\label{eq:energy}
		E(\phi,\sigma)
		\,:=\,
		\int_\Omega
		\left(
		\beta f_0(\phi)
		+\frac12g(\phi)|\nabla\phi|^2
		+\frac{\lambda}2\sigma^2
		\right)d\x\,.
		\eeq
		Additionally, the system \eqref{eq:system} also satisfies the following mass conservation property:
		\beq\label{eq:conservation}
		\frac{d}{dt}\left(\int_\Omega \phi\, d\x\right)
		\,=\,
		0\,.
		\eeq			
	\end{lemma}
	\begin{proof}
		Testing \eqref{eq:system}$_1$ by $\mu$, testing \eqref{eq:system}$_2$ by $\phi_t$ and adding both relations we obtain:
		\beq\label{eq:proofeneaux1}
		\frac{d}{dt}\left(\int_\Omega
		\left(
		\beta f_0(\phi)
		+\frac12g(\phi)|\nabla\phi|^2
		\right)d\x\right)
		+ M\|\nabla\mu\|^2_{L^2}
		-\lambda\int_\Omega\Delta\sigma \phi_t d\x
		\,=\,
		0\,.
		\eeq
		Taking the time derivative \eqref{eq:system}$_3$ we obtain
		$$
		\sigma_t 
		\,=\,
		- \Delta\phi_t
		$$
		and testing this expression by $\lambda\sigma$ we obtain
		\beq\label{eq:proofeneaux2}
		\frac{d}{dt}\left(\int_\Omega\frac\lambda2\sigma^2 \,d\x\right)
		\,=\,
		- \lambda\int_\Omega \Delta\phi_t\,\sigma \,d\x
		\,.
		\eeq
		Hence, integrating by parts and adding \eqref{eq:proofeneaux1} and \eqref{eq:proofeneaux2} we arrive at 
		relation \eqref{eq:energylaw}.
		\\
		Lastly, the mass conservation property~\eqref{eq:conservation} is obtained by testing \eqref{eq:system}$_1$ by $1$.
	\end{proof}

	\section{Numerical Scheme}\label{sec:num-scheme}
	To present our fully discrete scheme, we first describe the space-discrete formulation of the system, followed by its temporal discretization. 
	We then prove its key properties, such as solvability, unconditional stability, and discrete mass conservation.
	{Standard notation will be adopted for Lebesgue spaces $L^p(\Omega)$ and Sobolev spaces $H^s(\Omega)$ with norm $\|\cdot\|_{H^s}$. In the particular case $s = 0$, we write $L^2(\Omega)$ instead of $H^0(\Omega)$. For the sake of simplicity, in what follows we denote the $L^2$-inner product over $\Omega$ as
		$$
		(u,v)
		\,=\,
		\int_\Omega u\,v\,d\x
		\quad\mbox{ with }\quad
		\|u\|^2_{L^2}\,=\,(u,v)\,.
		$$
%		We will first introduce...
%		\\
%		Then the iterative algorithm that will be used to approximate the nonlinear scheme \eqref{eq:scheme} as well as a lemma about its well-posedness. 
	}
	
	%In the remainder of the paper, we let $H^s(\Omega)$ denote the $L^2$-based Sobolev space of differentiation order $s\ge0$ %$C(\bar{\Omega})$ denote the space of continuous functions on $\bar{\Omega}$
	%and let $(\cdot,\cdot)$ denote the $L^2$-inner product over $\Omega$.

	\subsection{Finite Element space-discrete scheme}
	Let $\{\mathcal{T}_h\}_h$ denote a family of geometrically conforming, shape regular triangulations, where we denote the mesh size by $h>0$.
	The unknowns of our problem $(\phi,\mu,\sigma)$ are approximated using the following conforming finite element spaces 
	$$
	\Phi_h\times M_h\times \Sigma_h
	\subset
	H^1(\Omega)\times H^1(\Omega)\times H^1(\Omega)\,.
	$$
	%	\textcolor{red}{G: I do not think that at this point it is important with choice of FEM spaces we are taking, so I commented the definition of the spaces.}
	%which are defined as follows:
	%\begin{subequations}\label{eq:approx-spaces}
	%	\begin{align}
		%		\Phi_h&:=\{\bar{\phi} \in C(\bar{\Omega}) \ : \ \bar{\phi}\big|_T \in P_k(T), \ T \in \mathcal{T}_h \},\\
		%		M_h&:=\{\bar{\mu} \in C(\bar{\Omega}) \ : \ \bar{\mu}\big|_T \in P_{\ell}(T), \ T \in \mathcal{T}_h\},\\
		%		\Sigma_h&:=\{\bar{\sigma} \in C(\bar{\Omega}) \ : \ \bar{\sigma}\big|_T \in P_m(T), \ T \in \mathcal{T}_h\},
		%	\end{align}
	%	\end{subequations}
%
%where $P_{n}(S)$ denotes the space of polynomials of degree at most $n\ge0$ defined on $S\subset\mathbb{R}^d$. 

Lastly, we let $\Pi_{X_h}$ denote the $L^2$-projection operator into the space $X_h\subset L^2(\Omega)$. In order to simplify the notation we omit the use of subscript $h$ to denote the space-discrete functions. Then, the problem reads: 
Find  the continuous-in-time, space-discrete approximation $(\phi, \mu, \sigma)\,\in\,\Phi_h\times M_h\times \Sigma_h$ %$(\phi(t),\mu(t),\sigma(t))\in \Phi_h\times M_h\times \Sigma_h$
such that
$$
\phi\big|_{t=0}
\,=\,
\Pi_{\Phi_h}\phi_0
\quad\mbox{ in }\Omega,
$$
and
\beq\label{eq:weaksystem}
\left\{\ba{rcl}
(\phi_t,\bar\mu) 
+ M(\nabla\mu,\nabla\bar\mu)
&=&0\,,
\\ \hueco\dis
\beta (f_0'(\phi) ,\bar\phi)
+ \frac12(g'(\phi)|\nabla\phi|^2,\bar\phi)
+(g(\phi)\nabla\phi,\nabla\bar\phi)
+ \lambda(\nabla \sigma,\nabla\bar\phi)
&=&
(\mu,\bar\phi)\,,
\\ \hueco\dis
(\sigma,\bar\sigma)
&=&
(\nabla\phi,\nabla\bar\sigma)\,,
\ea\right.
\eeq
holds a.e. $0 \le t\le T$ and for any $(\bar\phi,\bar\mu,\bar\sigma)\in \Phi_h\times M_h\times \Sigma_h$.

Below, we show that the above scheme satisfies the space-discrete analog of the energy law~\eqref{eq:energylaw}.
\begin{lemma}
	Any solution $(\phi(t),\mu(t),\sigma(t))$ of the space-discrete scheme \eqref{eq:weaksystem} satisfies the following space-discrete version of the energy law \eqref{eq:energylaw}:
	\beq\label{eq:energylawdiscrtespace}
	\frac{d}{dt}E(\phi(t),\sigma(t)) 
	+ M\|\nabla\mu(t)\|^2_{L^2}
	\,=\,0
	\,.
	\eeq
\end{lemma}
\begin{proof}
	The proof readily follows by taking $(\bar\phi,\bar\mu,\bar\sigma)=(\phi_t(t),\mu(t),\lambda\sigma(t))$
	as test functions in the space-discrete scheme \eqref{eq:weaksystem} and considering the time derivative of \eqref{eq:weaksystem}$_3$.
\end{proof}

\subsection{Fully discrete numerical scheme}
For simplicity, we consider a uniform partition of the time interval $[0,T]$ {into $N$ subintervals} with time step $\Delta t>0$ such that $t^n=n\Delta t${,  $n=0,\cdots, N$}. Moreover we denote the midpoint approximation of a variable $b$ as $b^{n+\frac12}=(1/2)(b^{n+1} + b^n)$. 
Accordingly, the proposed numerical scheme reads: 
\\
\textbf{Initialization.}
\\
Let $\phi^0=\Pi_{\Phi_h}\phi(0)$, $\mu^0=\Pi_{M_h}\mu(0)$, $\sigma^0=\Pi_{\Sigma_h}\sigma(0)$ with
$$
\sigma(0)
\,=\,
\sigma_0
\,=\,
-\Delta\phi_0
$$
and
$$
\mu(0)
\,=\,
\mu_0
\,=\,
\beta f_0'(\phi_0) 
+ \frac12g'(\phi_0)|\nabla\phi_0|^2
-\nabla\cdot\big(g(\phi_0)\nabla\phi_0\big)
- \lambda\Delta\sigma_0
\,.
$$
\\
\textbf{Step $n+1$.}
\\Given $(\phi^n,\mu^n,\sigma^n)\in \Phi_h\times M_h\times \Sigma_h$, find $(\phi^{n+1},\mu^{n+1},\sigma^{n+1})\in \Phi_h\times M_h\times \Sigma_h$ such that
\beq\label{eq:scheme}
\left\{\ba{rcl}\dis
\left(\frac{\phi^{n+1} - \phi^n}{\Delta t},\bar\mu\right)
+ M(\nabla\mu^{n+\frac12},\nabla\bar\mu)
&=&0\,,
\\ \hueco\dis
\beta (\widehat{f_0}'(\phi^{n+1},\phi^n) ,\bar\phi)
+ \frac12\left(g'(\phi^{n+\frac12})\frac{|\nabla\phi^{n+1}|^2 + |\nabla\phi^n|^2}{2},\bar\phi\right)&&
\\ \hueco\dis
+\left(\frac{g(\phi^{n+1}) + g(\phi^{n})}{2}\nabla\phi^{n+\frac12},\nabla\bar\phi\right)
+ \lambda(\nabla \sigma^{n+\frac12},\nabla\bar\phi)
&=&(\mu^{n+\frac12},\bar\phi)\,,
\\ \hueco\dis
(\sigma^{n+1},\bar\sigma)
&=&(\nabla\phi^{n+1},\nabla\bar\sigma)\,,
\ea\right.
\eeq
for any $(\bar\phi,\bar\mu,\bar\sigma)\in \Phi_h\times M_h\times \Sigma_h$
with
$$
\ba{rcl}
\widehat{f_0}'(\phi^{n+1},\phi^n)
&:=&
\dis\frac{f_0(\phi^{n+1}) - f_0(\phi^{n})}{\phi^{n+1} - \phi^n}
\\ \hueco
&=&\dis
\frac12\big((\phi^{n+1} + \phi^n)^2 + (\phi^{n+1})^2 + (\phi^n)^2\big)\big((\phi^{n+1})^3 + (\phi^n)^3\big)
\\ \hueco
&&
+\,
(h_0 - 2)(\phi^{n+1} + \phi^n)\big((\phi^{n+1})^2 + (\phi^n)^2\big)\\ \hueco
&&
+\,
(1 -2h_0)(\phi^{n+1} + \phi^n)
\\ \hueco
&=&\dis
(\phi^{n+1})^5 + (\phi^{n+1})^4\phi^n + (\phi^{n+1})^3(\phi^n)^2 + (\phi^{n+1})^2(\phi^n)^3 
\\ \hueco
&&
+\, (\phi^{n})^5 
+ \phi^{n+1}(\phi^n)^4 
\\ \hueco
&&
+\, (h_0 - 2)\Big((\phi^{n+1})^3 + (\phi^{n+1})^2\phi^n + \phi^{n+1}(\phi^n)^2 + (\phi^n)^3\Big)
\\ \hueco
&&
+\,(1 - 2h_0)(\phi^{n+1}  + \phi^n)\,.
\ea
$$
%$$
%\ba{rcl}
%\widehat{f_0}'(\phi^{n+1},\phi^n)
%&:=&
%\dis\frac{f_0(\phi^{n+1}) - f_0(\phi^{n})}{\phi^{n+1} - \phi^n}
%\\ \hueco
%&=&\dis
%\frac12\big((\phi^{n+1} + \phi^n)^2 + (\phi^{n+1})^2 + (\phi^n)^2\big)\big((\phi^{n+1})^3 + (\phi^n)^3\big)
%\\ \hueco
%&+&
%(h_0 - 2)(\phi^{n+1} + \phi^n)\big((\phi^{n+1})^2 + (\phi^n)^2\big)
%\\ \hueco
%&+&
%(1 -2h_0)(\phi^{n+1} + \phi^n)
%\\ \hueco
%&=&\dis
%(\phi^{n+1})^5 + (\phi^{n+1})^4\phi^n + (\phi^{n+1})^3(\phi^n)^2 + (\phi^{n+1})^2(\phi^n)^3 + \phi^{n+1}(\phi^n)^4 + (\phi^{n})^5 \,.
%\\ \hueco
%&+&
%(h_0 - 2)\Big((\phi^{n+1})^3 + (\phi^{n+1})^2\phi^n + \phi^{n+1}(\phi^n)^2 + (\phi^n)^3\Big)
%\\ \hueco
%&+&
%(1 - 2h_0)(\phi^{n+1}  + \phi^n)\,.
%\ea
%$$
{
	\begin{remark}
	The midpoint/secant approximation $\widehat{f_0}'(\cdot,\cdot)$ of $f_0'(\cdot)$ has been successfully applied to gradient flows with other polynomial functionals (see~\cite{TierraGuillen15} and the references therein) yielding second-order in time numerical schemes.
	\end{remark}
}
\subsubsection{Energy stability of the numerical scheme}
We proceed to establish the key properties satisfied by our fully discrete scheme, starting with proving the energy law and the discrete mass conservation.
\begin{lemma}\label{lem:stabilityandconservation}
	%Assume that the approximation spaces $\Phi_h,$ $M_h$ and $\Sigma_h$ introduced in~\cref{eq:approx-spaces} satisfies:
	%	\begin{align*}
		%		\Phi_h \subset M_h \cap \Sigma_h.
		%	\end{align*}
	%	\textcolor{red}{G:I do not think there is a restriction of spaces in this lemma}\\
	{If $1\in M_h$,} then scheme \eqref{eq:scheme} satisfies the following discrete version of the energy law \eqref{eq:energylaw}:
	\beq\label{eq:discreteEL}
	\frac1{\Delta t}\Big(E(\phi^{n+1},\sigma^{n+1}) - E(\phi^{n},\sigma^{n})\Big)
	+ M\|\nabla\mu^{n+\frac12}\|^2_{L^2}
	\,=\,0
	\,.
	\eeq
	Moreover, scheme is conservative:
	\beq\label{eq:discreteconservation}
	\int_\Omega \phi^{n+1} \,d\x
	\,=\,
	\int_\Omega \phi^{n} \,d\x
	\,=\,
	\dots
	\,=\,
	\int_\Omega \phi^0 \,d\x\,.
	\eeq
\end{lemma}

\begin{proof}
	Taking $\bar\mu=\mu^{n+\frac12}$ in \eqref{eq:scheme}$_1$ we obtain
	\beq\label{eq:proof1}
	\left(\frac{\phi^{n+1} - \phi^n}{\Delta t},\mu^{n+\frac12}\right)
	+ M\|\nabla\mu^{n+\frac12}\|^2_{L^2}
	\,=\,0\,.
	\eeq
	Taking $\bar\phi=\frac1{\Delta t}(\phi^{n+1} - \phi^n)$ in \eqref{eq:scheme}$_2$ and using that
	$$
	\ba{rcl}\dis
	&&\dis\frac12\left(g'(\phi^{n+\frac12})\frac{|\nabla\phi^{n+1}|^2 + |\nabla\phi^n|^2}{2},\frac{\phi^{n+1} - \phi^n}{\Delta t}\right)
	\\ \hueco
	&+&\dis
	\left(\frac{g(\phi^{n+1}) + g(\phi^{n})}{2}\nabla\phi^{n+\frac12}, \frac{\nabla(\phi^{n+1} - \phi^n)}{\Delta t}\right)
	\\ \hueco\dis
	&=&\dis \frac1{4\Delta t}\int_\Omega\Big(g(\phi^{n+1}) - g(\phi^{n})\Big)\Big(|\nabla\phi^{n+1}|^2 + |\nabla\phi^n|^2\Big) d\x
	\\ \hueco
	&+&\dis \frac1{4\Delta t}\int_\Omega
	\Big(g(\phi^{n+1}) + g(\phi^{n})\Big)\Big(|\nabla\phi^{n+1}|^2 - |\nabla\phi^n|^2\Big)d\x
	\\ \hueco\dis
	&=&\dis \frac1{2\Delta t}\int_\Omega\Big(g(\phi^{n+1})|\nabla\phi^{n+1}|^2 - g(\phi^{n})|\nabla\phi^n|^2\Big) d\x\,,
	\ea
	$$
	we obtain
	\beq\label{eq:proof2}
	\ba{c}\dis
	\frac1{\Delta t}\int_\Omega \Big(f_0(\phi^{n+1}) - f_0(\phi^{n})\Big)d\x
	+\dis \frac1{2\Delta t}\int_\Omega\Big(g(\phi^{n+1})|\nabla\phi^{n+1}|^2 - g(\phi^{n})|\nabla\phi^n|^2\Big) d\x
	
	\\ \hueco\dis
	+\dis \lambda\left(\nabla \sigma^{n+\frac12},\frac{\phi^{n+1} - \phi^n}{\Delta t}\right)
	=\left(\frac{\phi^{n+1} - \phi^n}{\Delta t},\mu^{n+\frac12}\right)\,.
	\ea
	\eeq
	We can subtract equation \eqref{eq:scheme}$_3$ at time $t^{n+1}$ with the same equation at $t^n$ and divide by $\Delta t$ to obtain
	\beq\label{eq:proof3}
	\left(\frac{\sigma^{n+1} - \sigma^{n}}{\Delta t},\bar\sigma\right)
	=\left(\frac{\nabla\phi^{n+1} - \nabla\phi^{n}}{\Delta t},\nabla\bar\sigma\right)
	\eeq
	and taking $\bar\sigma=\sigma^{n+\frac12}$ in \eqref{eq:proof3} we obtain
	\beq\label{eq:proof4}
	\frac1{2\Delta t}\int_\Omega \Big(|\sigma^{n+1}|^2 - |\sigma^{n}|^2\Big)d\x
	=\left(\frac{\nabla\phi^{n+1} - \nabla\phi^{n}}{\Delta t},\nabla\sigma^{n+\frac12} \right)\,.
	\eeq
	Then, by adding \eqref{eq:proof1}, \eqref{eq:proof2} and \eqref{eq:proof4} we obtain relation \eqref{eq:discreteEL}.
	\\
	Finally, taking $\bar\mu=1$ we obtain \eqref{eq:discreteconservation}.
\end{proof}
%As a consequence of the discrete energy law~\eqref{eq:discreteEL}, we can derive the following uniform a priori estimate similar to~\textcolor{magenta}{[insert Natasha ME paper reference] (unsure if we need to place this reference since the proof is straightforward we can discuss :-) )}.

%\begin{remark}
%	Scheme \eqref{eq:scheme} does not introduce any numerical dissipation in the energy law \eqref{eq:discreteEL}.
%\end{remark}

\begin{corollary}
	Let $(\phi^{n+1},\mu^{n+1},\sigma^{n+1})$ solve~\eqref{eq:scheme}. If there exists a constant $E_0$ (which is independent of $h$) such that $E(\phi^0,\sigma^0)\le E_0$ holds, then we can find a constant $C^*=2C_P^4\lambda^{-1}E_0$ independent of $h, \ \Delta t $ and $T$ such that
	\begin{align}\label{eq:uni-bd}
		\|\phi^{n+1}\|^2_{L^2}\le C^*\,.
	\end{align}
\end{corollary}
\begin{proof}
	We first begin by summing~\eqref{eq:discreteEL} over $l=0,\cdots, n$ to obtain
	\begin{align*}
		E(\phi^{n+1},\sigma^{n+1}) 
		+ M \Delta t \sum_{l=0}^n\|\nabla\mu^{l+\frac12}\|^2_{L^2}
		\,=\, E(\phi^{0},\sigma^{0}).   
	\end{align*}
	Then, combining the above with Poincare's inequality with constant $C_P$ and using~\eqref{eq:scheme}$_3$, we have
	\begin{align}
		\|  \phi^{n+1}\|_{L^2}^2
		\le C_P^2\|\nabla \phi^{n+1}\|_{L^2}^2 =  
		C_P^2 \big(\sigma^{n+1},\phi^{n+1}\big)\le
		\frac{C_P^4}{2}\|\sigma^{n+1}\|_{L^2}^2 
		+
		\frac{1}{2}\|\phi^{n+1}\|_{L^2}^2\,,
		%\nonumber\\
		% &\le 
		% \frac{C_P^2}{2g_0} E(\phih^{n+1}, \sigmah^{n+1})\le   \frac{C_P^2}{2g_0}C_0
	\end{align}
	from which we derive the following estimate:
	\begin{align}
		\|  \phi^{n+1}\|_{L^2}^2
		\le
		C_P^4\|\sigma^{n+1}\|_{L^2}^2
		\le
		\frac{2C_P^4}{\lambda}E_0\,.        
	\end{align}
\end{proof}
\subsubsection{Existence of solution of the numerical scheme}

The main idea to show the existence of a solution of the fully discrete numerical scheme \eqref{eq:scheme} is to take advantage of Lemma~$1.4$ in Chapter 2 in Temam's book \cite{Temam} (this Lemma is a consequence of Brouwer's Fixed Point Theorem) since the system \eqref{eq:scheme} can be understood as a nonlinear coercive problem in finite dimension. These types of arguments have already been used to show the existence of solutions for nonlinear schemes in the phase field context \cite{guillentierra2014}.
%{
	\begin{lemma}[\cite{Temam}]\label{Lem_tem}
		Let $X$ be a finite dimensional Hilbert space with scalar product $(\cdot,\cdot)$ and with norm denoted by $|\cdot|$. Let $P$ be a continuous mapping from $X$ into itself, that is, $P: X\rightarrow X$. Assume that there exists  a radius $\rho > 0$ such that
		$(P(\xi),\xi)>0$ for $|\xi|=\rho>0$. Then, there exists $\xi \in X, |\xi| \le \rho$, such that $P(\xi) = 0$.
	\end{lemma}
	
	\begin{lemma}
		Given $(\phi^n,\mu^n, \sigma^n)$ then there exists at least one solution\\ $(\phi^{n+1}, \mu^{n+1},\sigma^{n+1})$ of scheme \eqref{eq:scheme}.
	\end{lemma}
	\begin{proof}
		%The proof is established by standard nonlinear techniques for Cahn-Hilliard type equations as it was done in \cite{guillentierra2014}. To this end, 
		We first introduce the following change of variables: 
		%an equivalent formulation of~\eqref{eq:scheme} expressed in terms of the following variables:
		$$
		\psi
		\,=\,
		\phi^{n+1} - \phi^{n}\,, 
		\quad
		w
		\,=\,
		\mu^{n+\frac12} - \oint_{\Omega} \mu^{n+\frac12} d\x
		\,=\,
		\mu^{n+\frac12} - \oint_{\Omega} \mathcal{L}(\phi^{n+1} , \phi^n) d\x
		$$
		and
		$$
		z=\sigma^{n+1}+ \sigma^{n},
		$$
		where we use the shorthand $\mathcal{L}(\phi^{n+1} , \phi^n)$ defined as
		\beq
		\mathcal{L}(\phi^{n+1} , \phi^n)
		\,=\,
		\beta\widehat{f_0}'(\phi^{n+1},\phi^n)
		+ g'(\phi^{n+\frac12})\frac{|\nabla\phi^{n+1}|^2 + |\nabla\phi^n|^2}{4}\,.
		\eeq
		We can now define the mean value spaces:
		$$
		\widetilde\Phi_h
		\,=\,\left\{
		\psi\in\Phi_h
		\mbox{ such that }
		\oint_{{\Omega}} \psi d\x
		=0
		\right\}\,,
		\widetilde M_h
		\,=\,\left\{
		w \in M_h
		\mbox{ such that }
		\oint_{{\Omega}} w d\x
		=0
		\right\},
		$$
		and
		$$
		\widetilde\Sigma_h
		\,=\,\left\{
		z\in\Sigma_h
		\mbox{ such that }
		\oint_{{\Omega}} z d\x
		=0
		\right\}\,.
		$$
		In terms of these newly introduced variables and spaces, we rewrite the scheme~\eqref{eq:scheme} as follows:
		Find $(\psi,w,z)\in\widetilde\Phi_h\times\widetilde M_h
		\times\widetilde\Sigma_h
		$ satisfying
		\beq\label{eq:eq-scheme}
		\left\{\ba{rcl}\dis
		\left(\psi,\tw\right)
		+ \Delta t M(\nabla w,\nabla\tw)
		&=&0\,,
		\\ \hueco\dis
		\left(\mathcal{L}(\psi +\phi^n, \phi^n ),\tpsi\right)
		+ \frac{\lambda}{2}(\nabla z,\nabla\tpsi)
		\\ \hueco\dis		
		+\frac14\Big(g(\psi + \phi^{n}) + g(\phi^{n})\nabla(\psi 		+ 2\phi^{n}),\nabla\tpsi\Big)		
		&=&\dis
		\left(w + \oint_{\Omega} \mathcal{L}(\psi +\phi^n , \phi^n) \ d\x ,\tpsi\right)\,,
		\\ \hueco\dis
		(z-\sigma^n ,\tz)
		&=&
		(\nabla\psi+ \nabla \phi^n,\nabla\tz)\,,
		\ea
		\right.
		\eeq
		for all $(\tw,\tpsi,\tz)\in\widetilde\Phi_h\times\widetilde M_h
		\times\widetilde\Sigma_h$. 
		Notice that any solution of \eqref{eq:eq-scheme} satisfies 
		$$
		\int_{\Omega}\psi \,d\x 
		\,=\, 
		\int_{\Omega}w \,d\x
		\,=\,
		\int_{\Omega}z \,d\x
		\,=\,
		0\,.
		$$
		Also note that if $(\psi,w,z)$ solves~\eqref{eq:eq-scheme} then the variables $(\phi^{n+1},\mu^{n+1},\sigma^{n+1})$ satisfies  scheme~\eqref{eq:scheme}. Hence we just need to prove existence of solution of \eqref{eq:eq-scheme}.  The remainder of this proof will be devoted to this task.
		We set $\chi=(\psi,w,z)$,
		$\tchi=(\tpsi,\tw,\tz)$, we define the space
		$
		X=\widetilde\Phi_h\times\widetilde M_h
		\times\widetilde\Sigma_h
		$
		with scalar product
		$
		(\chi,\tchi)_X
		=
		\Big((\psi,w,z),(\tpsi,\tw,\tz)\Big)_X
		=
		(\psi,\tpsi) + (w,\tw) + (z,\tz)\,,
		$
		and
		we introduce the following functional $\mathcal{G}(\chi)$.
		$$
		\ba{rcl}\dis
		\Big(\mathcal{G}(\chi),\tchi\Big)
		& =&\dis  \left(\psi,\tw\right)
		+ M \Delta t (\nabla w,\nabla\tw) 
		+ \left(\mathcal{L}(\psi +\phi^n, \phi^n ),\tpsi\right)
		\\ \hueco\dis	
		&+&\dis\frac14\Big(g(\psi + \phi^{n}) + g(\phi^{n})\nabla(\psi + 2\phi^{n}),\nabla\tpsi\Big)
		\\ \hueco\dis
		&-& (w ,\tpsi)+ \dis\frac{\lambda}{2}(\nabla z,\nabla\tpsi)
		+ \frac{\lambda}{2}  (z-\sigma^n ,\tz)
		- \frac{\lambda}{2}(\nabla\psi 
		+ \nabla \phi^n,\nabla\tz)\,.
		\ea
		$$
		Then, following the arguments in~\cite{Temam}, we use Lemma~\ref{Lem_tem} to show the existence of the solution by establishing the existence of $r>0$ such that
		\begin{align*}
			\Big(\mathcal{G}(\chi),\chi\Big)>0 
			\quad\text{ for }\quad
			|\chi|=r\,.
		\end{align*}
		To this end, consider
		$$
		\ba{rcl}\dis
		\Big(\mathcal{G}(\chi),\chi\Big)
		&=&\dis
		M \Delta t \|\nabla w\|^2_{L^2} 
		+ \big(\mathcal{L}(\psi+\phi^n, \phi^n ),\psi\big)
		\\ \hueco\dis		
		&+&\dis\frac14\Big(g(\psi + \phi^{n}) + g(\phi^{n})\nabla(\psi + 2\phi^{n}),\nabla\psi\Big)
		\\ \hueco\dis
		&+&\dis \frac{\lambda}{2}(\nabla z,\nabla\psi)
		+ \frac{\lambda}{2}  (z-\sigma^n ,z)
		- \frac{\lambda}{2}(\nabla\psi + \nabla \phi^n,\nabla z)\,.
		\ea
		$$
		We can write
		$$
		\ba{rcl}\dis
		\beta \left(\widehat{f_0}'(\psi +\phi^n,\phi^n) ,\psi\right)
		&=&\dis
		\beta\int_\Omega
		\frac{f_0(\psi +\phi^n) - f_0(\phi^n)}{\psi +\phi^n - \phi^n}\psi d\x
		\\ \hueco\dis		
		&=&\dis
		\beta\int_\Omega
		\Big(f_0(\psi +\phi^n) - f_0(\phi^n)\Big)d\x\,.
		\ea
		$$
		Using the following relations:
		$$
		\ba{rcl}\dis
		g_2(\psi + 2\phi^{n})\psi
		=g_2(\psi+\phi^n)^2 - g_2(\phi^n)^2 
		&=& \dis g(\psi+\phi^n) - g(\phi^n)\,,\\
		\hueco\dis
		%		\left(\nabla(\psi + 2\phi^{n}),\nabla\psi\right)
		%		=
		\left(\nabla(\psi + 2\phi^{n}),\nabla\psi\right)
		\pm (\nabla \phi^n, \nabla \phi^n)
		&=& \|\nabla(\psi + \phi^n)\|^2_{L^2}  - \|\nabla\phi^n\|^2_{L^2} \,,
		\ea
		$$
we can write
		$$
		\ba{rcl}\dis
		&&\frac12\left(g'\left(\frac{\psi+2\phi^{n}}2\right)\frac{|\nabla(\psi + \phi^n)|^2 + |\nabla\phi^n|^2}{2},\psi\right)
		\\ \hueco
		&+&\dis
		\left(\frac{g(\psi+ \phi^n) +
			g(\phi^{n})}{2}\nabla\left(\frac{\psi+2\phi^{n}}2\right),\nabla\psi\right)
		\\ \hueco\dis
		&=&\dis\frac14 \int_\Omega g_2(\psi + 2\phi^{n})\psi\Big(|\nabla(\psi + \phi^n)|^2  + |\nabla\phi^n|^2\Big)d\x
		\\ \hueco\dis
		&+&
		\dis\frac14 \int_\Omega\big(g(\psi + \phi^n) + g(\phi^n)\big) \nabla(\psi + 2\phi^{n}) \nabla\psi
		d\x
		\\ \hueco\dis
		&=&\dis\frac14 \int_\Omega 
		\big(g(\psi+\phi^n) - g(\phi^n)\big)\Big(|\nabla(\psi + \phi^n)|^2  + |\nabla\phi^n|^2\Big)d\x
		\\ \hueco\dis
		&+&\dis \frac14 \int_\Omega \big(g(\psi + \phi^n) + g(\phi^n)\big)
		\Big(|\nabla(\psi + \phi^{n})|^2 - |\nabla\phi^{n}|^2\Big)
		d\x
		\\ \hueco\dis
		&=&\dis  \frac12 \int_\Omega
		\Big[g(\psi + \phi^n)|\nabla(\psi + \phi^{n})|^2
		-g(\phi^n)|\nabla\phi^{n}|^2\Big]d\x\,.
		\ea
		$$
		Consequently,
		$$
		\ba{rcl}\dis
		&&\dis\Big(\mathcal{L}(\psi +\phi^n, \phi^n ),\psi\Big)
		+
		\left(\frac{g(\psi+ \phi^n) +
			g(\phi^{n})}{2}\nabla\left(\frac{\psi+2\phi^{n}}2\right),\nabla\psi\right)
		\\
		\hueco \dis
		&=&\dis
		\int_\Omega
		\left(\beta f_0(\psi +\phi^n) + \frac12 g(\psi + \phi^n)|\nabla(\psi + \phi^{n})|^2 \right)d\x
		\\ \hueco\dis
		&-&\dis
		\int_\Omega
		\left[\beta f_0(\phi^n)
		+ \frac12g(\phi^n)|\nabla\phi^{n}|^2\right]d\x\,.
		\ea
		$$
		From \eqref{eq:eq-scheme}$_3$ in the previous time step and taking $\tz=z$ we know that
		$$
		(\sigma^{n},z)
		-(\nabla\phi^{n},\nabla z)
		=0\,.
		$$
		In view of this relation, we can express the remaining three terms of $\Big(\mathcal{G}(\chi),\chi\Big)$ as
		%$$
		%\ba{c}\dis
		%\frac{\lambda}{2}(\nabla z,\nabla\psi)
		%+ \frac{\lambda}{2}  (z-\sigma^n ,z)
		%- \frac{\lambda}{2}(\nabla\psi + \nabla \phi^n,\nabla z)
		%\\ \hueco\dis
		%=
		%\frac{\lambda}{2}  (z-\sigma^n ,z)
		%- \frac{\lambda}{2}(\nabla \phi^n,\nabla z)
		%-\frac{\lambda}2\Big((\sigma^{n},z)
		%-(\nabla\phi^{n},\nabla z)\Big)
		%\\ \hueco\dis
		%=\frac{\lambda}{2}  (z-2\sigma^n ,z)
		%\\ \hueco\dis
		%=\frac{\lambda}{2}  \|z - \sigma^n\|^2_{L^2}
		%-\frac{\lambda}{2}  \|\sigma^n\|^2_{L^2}\,.
		%\ea
		%$$
		$$
		\ba{rcl}\dis
		&&\dis\frac{\lambda}{2}(\nabla z,\nabla\psi)
		+ \frac{\lambda}{2}  (z-\sigma^n ,z)
		- \frac{\lambda}{2}(\nabla\psi + \nabla \phi^n,\nabla z)
		\\ \hueco\dis
		&=&\dis\frac{\lambda}{2}  (z-\sigma^n ,z)
		- \frac{\lambda}{2}(\nabla \phi^n,\nabla z)
		-\frac{\lambda}2\Big((\sigma^{n},z)
		-(\nabla\phi^{n},\nabla z)\Big)
		\\ \hueco\dis
		&=&\dis\frac{\lambda}{2}  (z-2\sigma^n ,z)
		=\frac{\lambda}{2}  \|z - \sigma^n\|^2_{L^2}
		-\frac{\lambda}{2}  \|\sigma^n\|^2_{L^2}\,.
		\ea
		$$
		Hence, we have obtained
		$$
		\ba{rcl}
		\mathcal{G}(\chi,\chi) 
		&=&\dis M \Delta t \|\nabla w\|^2_{L^2} 
		+ \frac{\lambda}{2}  \|z - \sigma^n\|^2
		%\\ \hueco
		%&+&\dis
		\\ \hueco
		&+&\dis\int_\Omega
		\left(\beta\widehat{f_0}(\psi +\phi^n) + \frac12 g(\psi + \phi^n)|\nabla(\psi + \phi^{n})|^2 \right)d\x
		\\ \hueco
		&-&\dis 
		\int_\Omega
		\left[\beta\widehat{f_0}(\phi^n)
		+ \frac12g(\phi^n)|\nabla\phi^{n}|^2\right]d\x
		-\frac{\lambda}{2}  \|\sigma^n\|^2_{L^2}\,,
		\ea
		$$
		and so the hypothesis of Lemma~\ref{Lem_tem} holds for $\psi$ large enough. Consequently, there exists 
		$(\psi,w,z)\in X$ such that $\mathcal{G}(\psi,w,z)=0$ and therefore $(\psi,w,z)$ is a solution of scheme \eqref{eq:eq-scheme}.
		
	\end{proof}
	{
		\begin{remark}
			The numerical scheme \eqref{eq:scheme} satisfies a discrete version of energy law \eqref{eq:energylaw} without introducing any additional numerical dissipation. This indicates that the dynamics of the discrete energy will follow exactly the dynamics of the continuous energy. On the other hand, the resulting scheme is nonlinear, which might lead to a time-step constraint for rigorously showing unique solvability, although in the numerical results presented in Section~\ref{sec:simulations} we do not observe any unreasonable constraint on the size of the time step.
		\end{remark}
	}
	\subsection{Iterative Algorithm}
	The numerical scheme derived in this work {is} a nonlinear algebraic system. Hence, we consider an iterative algorithm to approximate the scheme. We detail here the iterative algorithm that we have considered to approximate scheme \eqref{eq:scheme}.
	\\\textbf{Initialization ($\ell=0$).}
	\\ Set $(\phi^{\ell=0},\mu^{\ell=0},\sigma^{\ell=0})=(\phi^n,\mu^n,\sigma^n)\in\Phi_h\times M_h\times \Sigma_h$.
	\\\textbf{Algorithm.}
	\\Given $(\phi^\ell,\mu^\ell,\sigma^\ell)\in\Phi_h\times M_h\times \Sigma_h$, find $(\phi^{\ell+1},\mu^{\ell+1},\sigma^{\ell+1})\in\Phi_h\times M_h\times \Sigma_h$ such that for all $(\bar\phi,\bar\mu,\bar\sigma)\in \Phi_h\times M_h\times \Sigma_h$:
	\beq\label{eq:iterativealgorithm}
	\left\{\ba{rcl}\dis
	\left(\frac{\phi^{\ell+1} - \phi^n}{\Delta t},\bar\mu\right)
	+ M(\nabla\mu^{\ell+\frac12},\nabla\bar\mu)
	&=&0\,,
	\\ \hueco\dis
	%\beta\Big(\widehat{f_0}'(\phi^{\ell},\phi^n) + \widehat{f_0}''(\phi^{\ell},\phi^n)(\phi^{\ell+1} - \phi^{\ell}) ,\bar\phi\Big)
	\beta\Big(\widehat{f_0}'(\phi^{\ell},\phi^n) ,\bar\phi\Big)
	+ \frac12\left(g'(\phi^{\ell+\frac12})\frac{|\nabla\phi^{\ell}|^2 + |\nabla\phi^n|^2}{2},\bar\phi\right)&&
	\\ \hueco\dis
	+\left(\frac{g(\phi^{\ell}) + g(\phi^{n})}{2}\nabla\phi^{\ell+\frac12},\nabla\bar\phi\right)
	+ \lambda(\nabla \sigma^{\ell+\frac12},\nabla\bar\phi)
	&=&(\mu^{\ell+\frac12},\bar\phi)\,,
	\\ \hueco\dis
	(\sigma^{\ell+1},\bar\sigma)
	&=&(\nabla\phi^{\ell+1},\nabla\bar\sigma)\,.
	\ea\right.
	\eeq
	We iterate until all the relative $L^2$-errors reach the required tolerance \texttt{TOL}:
	$$
	\frac{\|(\phi^{\ell+1} - \phi^{\ell},\mu^{\ell+1} - \mu^{\ell},\sigma^{\ell+1} - \sigma^{\ell})\|_{L^2\times L^2\times L^2}}{\|(\phi^{\ell+1} ,\mu^{\ell+1} ,\sigma^{\ell+1} ,)\|_{ L^2\times L^2\times L^2}}
	\,\leq\,
	\texttt{TOL}\,.
	$$
	In particular, the tolerance will be fixed in all the examples as $\texttt{TOL}=10^{-7}$ unless otherwise mentioned. 
	Moreover, we impose $50$ as the maximum number of iterations the iterative algorithm can perform to achieve the required tolerance. However, none of the reported simulations require so many iterations at any time. 
	
	%where 
	%
	%$$
	%\ba{rcl}
	%\widehat{f_0}''(\phi^{n+1},\phi^n)
	%&=&\dis
	%5(\phi^{n+1})^4 + 4(\phi^{n+1})^3\phi^n + 3(\phi^{n+1})^2(\phi^n)^2 + 2\phi^{n+1}(\phi^n)^3 + (\phi^n)^4\,.
	%\\ \hueco
	%&+&
	%(h_0 - 2)\Big(3(\phi^{n+1})^2 + 2\phi^{n+1}\phi^n + (\phi^n)^2 \Big)
	%\\ \hueco
	%&+&
	%(1 - 2h_0)\,.
	%\ea
	%$$
	
	\begin{lemma}\label{lem:iterativealgorithm}
		If $1\in\Phi_h$, $\Phi_h\subset M_h$ and $\Phi_h\subset\Sigma_h$, then there exists a unique solution of iterative scheme \eqref{eq:iterativealgorithm} if $\dis\Delta t 
		\leq
		\frac{3\lambda^2}{2|g_0|^3 M}$ whenever $g_0\ne 0$.
	\end{lemma}
	\begin{proof}
		Let us consider two solutions to the algorithm \eqref{eq:iterativealgorithm}, call them \\$(\phi_1^{\ell+1},\mu_1^{\ell+1},\sigma_1^{\ell+1})$ and $(\phi_2^{\ell+1},\mu_2^{\ell+1},\sigma_2^{\ell+1})$, and denote their difference by $(\phi,\mu,\sigma)$. That is,
		$$
		(\phi,\mu,\sigma)
		\,=\,
		(\phi_1^{\ell+1},\mu_1^{\ell+1},\sigma_1^{\ell+1})
		-(\phi_2^{\ell+1},\mu_2^{\ell+1},\sigma_2^{\ell+1})
		\in\Phi_h\times M_h\times \Sigma_h\,.
		$$
		Then, taking the difference between the system for $(\phi_1^{\ell+1},\mu_1^{\ell+1},\sigma_1^{\ell+1})$ and \\$(\phi_2^{\ell+1},\mu_2^{\ell+1},\sigma_2^{\ell+1})$, we obtain
		\beq\label{eq:iterativealgorithmproof}
		\left\{\ba{rcl}\dis
		\frac1{\Delta t}(\phi,\bar\mu)
		+ \frac{M}2(\nabla\mu,\nabla\bar\mu)
		&=&0\,,
		\\ \hueco\dis
		%\beta\Big(\widehat{f_0}'(\phi^{\ell},\phi^n) + \widehat{f_0}''(\phi^{\ell},\phi^n)(\phi^{\ell+1} - \phi^{\ell}) ,\bar\phi\Big)
		\frac{g_2}4\Big( \phi\big(|\nabla\phi^{\ell}|^2 + |\nabla\phi^n|^2\big),\bar\phi\Big)
		+\frac{1}{4}\left(\big(g(\phi^{\ell}) + g(\phi^{n})\big)\nabla\phi,\nabla\bar\phi\right)
		&&
		\\ \hueco\dis		
		+ \frac{\lambda}2(\nabla \sigma,\nabla\bar\phi)
		- \frac12(\mu,\bar\phi)
		&=&
		0\,,
		\\ \hueco\dis
		(\sigma,\bar\sigma)
		-(\nabla\phi,\nabla\bar\sigma)
		&=&
		0\,.
		\ea\right.
		\eeq
		Testing \eqref{eq:iterativealgorithmproof} by 
		$(\bar\phi,\bar\mu,\bar\sigma)=(4\phi,2\mu,2{\lambda}\sigma)$
		we obtain
		$$
		\ba{rcl}\dis
		\Delta t M\|\nabla\mu\|^2_{L^2}
		&+&\dis{g_2}\int_\Omega\phi^2\big(|\nabla\phi^{\ell}|^2 + |\nabla\phi^n|^2\big)d\x
		\\ \hueco		
		&+&\dis\int_\Omega\big(g(\phi^{\ell}) + g(\phi^{n})\big)|\nabla\phi|^2 d\x
		+ 2{\lambda}\|\sigma\|^2_{L^2}
		\,=\,
		0\,.
		\ea
		$$
		In particular, using the definition of $g(\phi)$ in \eqref{eq:deffunctiong}, we can write
		\beq\label{eq:iterativealgorithmproof2}
		\ba{rclcl}\dis
		\Delta t M\|\nabla\mu\|^2_{L^2}
		&+&\dis{g_2}\int_\Omega\phi^2\big(|\nabla\phi^{\ell}|^2 + |\nabla\phi^n|^2\big)d\x
		&&
		\\ \hueco\dis
		&+&\dis g_2\int_\Omega \big( (\phi^{\ell})^2 + (\phi^{n})^2\big)|\nabla\phi|^2 d\x		
		+ 2{\lambda}\|\sigma\|^2_{L^2}
		&=&\dis
		-2g_0\|\nabla\phi\|^2_{L^2}
		\\ \hueco\dis
		&&&\leq&\dis
		2|g_0|\|\nabla\phi\|^2_{L^2}\,.
		\ea
		\eeq
		Taking $\bar\sigma=\phi$ in \eqref{eq:iterativealgorithmproof}$_3$ and applying Young's inequality, we can write
		$$
		\|\nabla\phi\|^2_{L^2}
		\,\leq\,
		\frac{\lambda}{4|g_0|}\|\sigma\|^2_{L^2}
		+\frac{|g_0|}{\lambda}\|\phi\|^2_{L^2}\,.
		$$
		Next, taking $\bar\mu=\phi$ in \eqref{eq:iterativealgorithmproof}$_1$ and applying Young's inequality again, we have
		$$
		\frac1{\Delta t}\|\phi\|^2_{L^2}
		\,\leq\,
		\frac{\Delta t|g_0| M^2}{4\lambda}\|\nabla\mu\|^2_{L^2}
		+\frac{\lambda}{4\Delta t|g_0|}\|\nabla\phi\|^2_{L^2}\,.
		$$
		Therefore, combining both the previous expressions, we get
		$$
		\|\nabla\phi\|^2_{L^2}
		\,\leq\,
		\frac{\lambda}{4|g_0|}\|\sigma\|^2_{L^2}
		+\frac{|g_0|}{\lambda}\left(\frac{(\Delta t)^2|g_0| M^2}{4\lambda}\|\nabla\mu\|^2_{L^2}
		+\frac{\lambda}{4|g_0|}\|\nabla\phi\|^2_{L^2}\right)\,,
		$$
		hence
		$$
		\frac{3}{4}\|\nabla\phi\|^2_{L^2}
		\,\leq\,
		\frac{\lambda}{4|g_0|}\|\sigma\|^2_{L^2}
		+\frac{(\Delta t)^2|g_0|^2 M^2}{4\lambda^2}\|\nabla\mu\|^2_{L^2}\,,
		$$
		or, 
		$$
		\|\nabla\phi\|^2_{L^2}
		\,\leq\,
		\frac{\lambda}{3|g_0|}\|\sigma\|^2_{L^2}
		+\frac{(\Delta t)^2|g_0|^2 M^2}{3\lambda^2}\|\nabla\mu\|^2_{L^2}\,.
		$$
		Using this last expression in \eqref{eq:iterativealgorithmproof2} we obtain
		$$
		\ba{c}\dis
		\left(\Delta t M - \frac{2(\Delta t)^2|g_0|^3 M^2}{3\lambda^2}\right)\|\nabla\mu\|^2_{L^2}
		+{g_2}\int_\Omega\phi^2\big(|\nabla\phi^{\ell}|^2 + |\nabla\phi^n|^2\big)d\x
		\\ \hueco\dis
		+g_2\int_\Omega \big( (\phi^{\ell})^2 + (\phi^{n})^2\big)|\nabla\phi|^2 d\x
		\,+\, {\lambda}\|\sigma\|^2_{L^2}
		\leq
		0\,,
		\ea
		$$ %Therefore
		provided 
		$\dis\frac{2(\Delta t)^2|g_0|^3 M^2}{3\lambda^2}
		\leq
		\Delta t M 
		$ or,	
		$\dis\Delta t 
		\leq
		\frac{3\lambda^2}{2|g_0|^3 M}$. Hence, we can deduce that
		$\nabla\mu=\bf{0}$ (which implies that $\mu=C$), $\phi=0$ and $\sigma=0$.
		%\textcolor{magenta}{\Big(NSS: Did you mean $\bf{0}$ instead of $C$ in the previous sentence? Because that would mean that $\mu=C$ thus in view of \eqref{eq:iterativealgorithmproof}$_2$, we would have
			%\[
			%0=(C,\bar\phi)= C(1,\bar\phi), \quad \pato\bar\phi\in\Phi_h.\Big)
			%\]
			%}
		Finally, checking \eqref{eq:iterativealgorithmproof}$_2$ we obtain
		$$
		0
		\,=\,
		(\mu,\bar\phi)
		\,=\,
		(C,\bar\phi)
		\,=\,
		C(1,\bar\phi)
		\quad
		\pato\bar\phi\in\Phi_h .
		$$ 
		Thus, if $1\in\Phi_h$, then $C=\mu=0$.
	\end{proof}
	
	\begin{remark}
		If $g_0=0$, iterative scheme \eqref{eq:iterativealgorithm} is uniquely solvable.
	\end{remark}
	
	{
		\begin{remark}
	 Although we have not rigorously established the convergence of the solution of iterative algorithm \eqref{eq:iterativealgorithm} to the solution of the nonlinear system \eqref{eq:scheme}, the proposed iterative algorithm performs well in simulations as it can be seen in the results presented in Section~\ref{sec:simulations}. Studying this convergence is important and non-trivial, but in our case standard arguments, such as taking difference between the nonlinear scheme and the iterative algorithm to construct an equation for the error and then take advantage of the discrete energy stability structure of the problem, would require an extremely small time step to control the terms arising from the variational derivative of $g(\phi)|\nabla\phi|^2$, making the assumptions unpractical. In fact, one might design nonlinear schemes and iterative algorithms with this difficulty in mind to obtain a less restrictive requirement for the time step, but then the other good properties of the scheme would be lost (i.e. the second order accuracy and the discrete energy stability). Designing new numerical schemes that balance all these difficulties at the same time is an interesting goal but it is out of the scope of this work.
	 
	 %and a more detailed study should be considered \textcolor{blue}{}, which is something we plan to address in the future but it is out of the scope of this work.
	\end{remark}}
	{
		\begin{remark}
			The following choice of the discrete spaces satisfy all the requirements of Lemma~\ref{lem:iterativealgorithm} and Lemma~\ref{lem:stabilityandconservation}: 
			$$
			\Phi_h\times M_h\times \Sigma_h 
			\,=\,
			\mathbb{P}_k\times\mathbb{P}_k\times\mathbb{P}_k\,,
			$$
			where $\mathbb{P}_{k}$ denotes the subspace of $H^1(\Omega)$ consisting of elementwise polynomials of degree at most $k\ge0$.
		\end{remark}
	}	
	
	%\newpage
	%\begin{lem}
	%There exists a unique solution of scheme \eqref{eq:scheme} if ...
	%\end{lem}
	\section{Simulations}\label{sec:simulations}
	In this section, we present the results of several numerical experiments to demonstrate the performance of the proposed numerical scheme. All the simulations have been carried out using \textit{FreeFem++} software \cite{freefem}, the color plots have been generated using \textit{Paraview} \cite{Paraview}, and the rest of the plots have been generated using \textit{MATLAB} \cite{matlab}. We need to mention that in order to be consistent in the presentation of results, for all the color plots, we have used a color scale bar for $\phi$ between $[-1,1]$ and at the same time for all the simulations we present the evolution in time of the maximum and minimum value of $\phi$, to illustrate that depending on the choice of the parameters the interval where the values of $\phi$ lies can change. In all the $2D$ plots the horizontal axis corresponds with the $x$-axis and the vertical with the $y$-axis, while in the $3D$ ones we have incorporated a small plot of the axis in the corner of the images to indicate the orientation of each axis.
	
	{The goal of the simulations presented in this section are twofold: first we investigate the experimental order of convergence in time of the proposed numerical scheme and then we illustrate its ability in capturing the micro-emulsion phase. To this end, we recall that on a macroscopic level, the micro-emulsion phase is a single phase structured fluid consisting of regions of oil-rich and water-rich phases which form a complicated, intertwined structure \cite{D0RA08092F}. In order to resemble the conditions studied experimentally, we set initial conditions consisting on droplets of oil and water dispersed in a \textit{continuous} medium of surfactant.	In particular, the dynamics that we observe in our simulations are consistent with numerical studies previously reported in other works~\cite{diegelsharma-me23, hoppelinsenmann-me18}. 
	}

	{For all the simulations in the next subsections we consider the following choice of the discrete spaces, which satisfy the requirements of Lemma~\ref{lem:stabilityandconservation} and Lemma~\ref{lem:iterativealgorithm}: 
		$$
		\Phi_h\times M_h\times \Sigma_h 
		\,=\,
		\mathbb{P}_1\times\mathbb{P}_1\times\mathbb{P}_1\,,
		$$}	
	After that, we present the results of several $2D$ simulations to illustrate the influence of the different physical parameters on the system's dynamics to understand the model's limitations.
	
	Finally, we present some $3D$ numerical results to evidence the applicability of the scheme proposed in this work. %We rely on linear Lagrange finite elements in our experiments for our spatial discretization. 
	
	%	\textcolor{magenta}{
		%%\begin{remark}
		%%	
		%%\end{remark}	
		%	
		%	}

	\subsection{Experimental of order of convergence (EOC)}

	In this section we compute numerically the order of convergence in time of the numerical scheme \eqref{eq:scheme}. We consider a square domain $\Omega={(}0,32{)}^2$ discretized using a structured mesh of size $h=1/100$, with temporal interval $[0,10^{-5}]$. Moreover the choice of the parameters values is $M=1$, $g_2=1$, $g_0=-1$, $h_0=0.5$, $\lambda=0.5$ and $\beta=1$. In order to report the results we need to introduce some additional notation. We denote the individual absolute errors  {at final time $T$} and their convergence rates between two 
	consecutive 
	time steps of size $\Delta t$ and $\widehat{\Delta t}$ as
	\begin{align*}
		&		e_2(\psi):=\|\psi_{exact} - \psi_s\|_{L^2(\Omega)},
		\,e_1(\psi):=\|\psi_{exact} - \psi_s\|_{H^1(\Omega)}
		\mbox{ and }\\
		&		r_i(\cdot):=\left[\log\left(\frac{e_i(\cdot)}{\widehat{e}_i(\cdot)}\right)\right]\Big/ \left[\log\left(\frac{\Delta t}{\widehat{\Delta t}}\right)\right]\,.
	\end{align*}
	The EOC (Experimental Order of Convergence) is computed using as exact solution the one obtained by solving the system using the time step  $\Delta t = 10^{-8}$ and the initial condition
	% \eqref{eq:initialcondex1} 
	\begin{align*}
		\phi^0(x,y)&=
		-\tanh\left(\frac{\sqrt{(x-7)^2 + (y-7)^2}-3}{\sqrt{2\lambda}}\right)
		\\ &		
		-\tanh\left(\frac{\sqrt{(x-20)^2 + (y-20)^2}-6}{\sqrt{2\lambda}}\right)
		+1\,.
	\end{align*}
	The configuration considered as reference/exact solution is presented in Figure~\ref{fig:Ex1FinalConfig}. The results of this study are presented in Table~\ref{tab:Ex1_order}. We can observe how as expected the proposed scheme achieves second order in time for all of the unknowns both in the $L^2(\Omega)$ and $H^1(\Omega)$ norms.
	
	\begin{figure}[h]
		\begin{center}
			\includegraphics[width=0.25\textwidth]{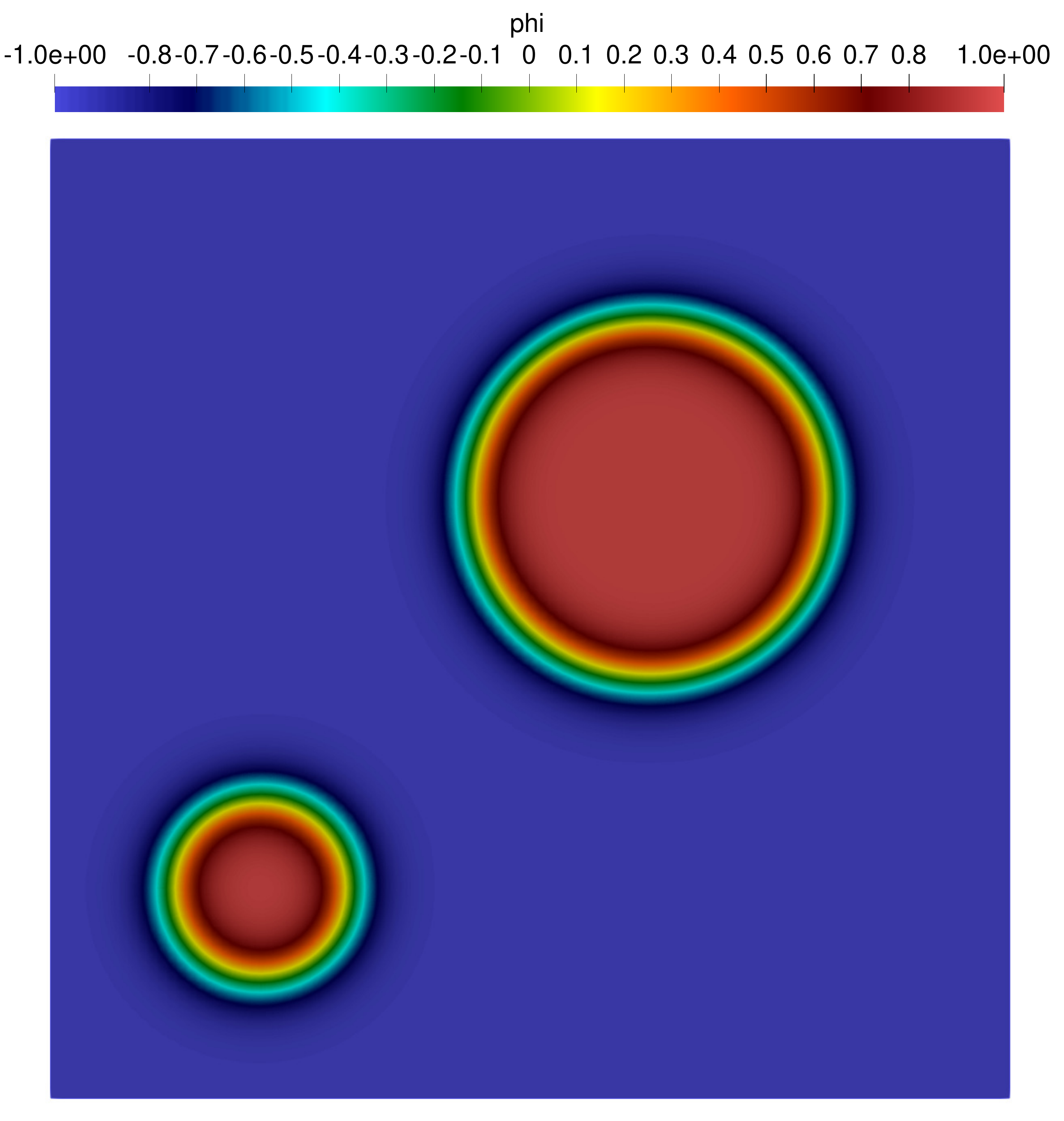}
			\includegraphics[width=0.25\textwidth]{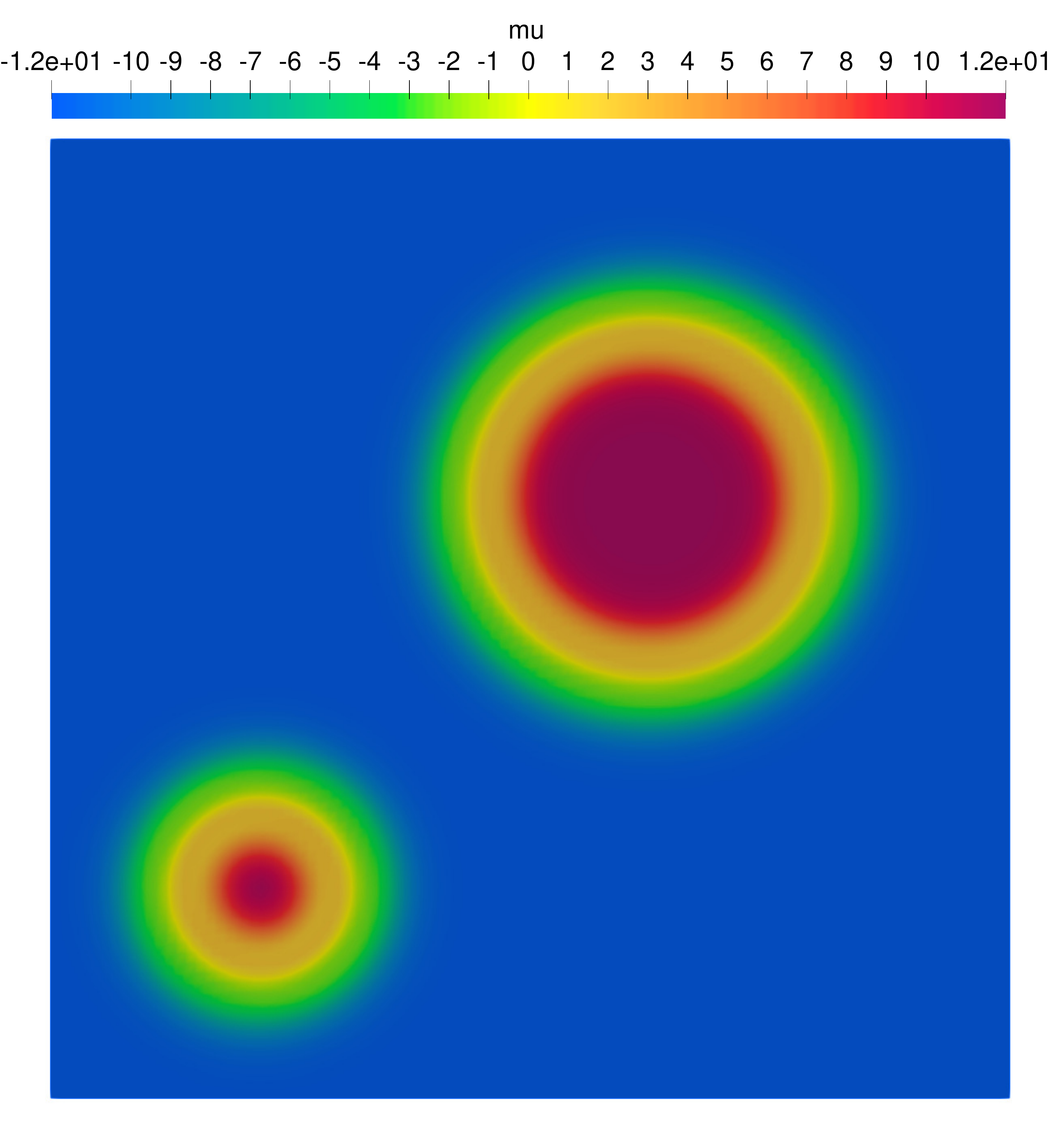}
			\includegraphics[width=0.25\textwidth]{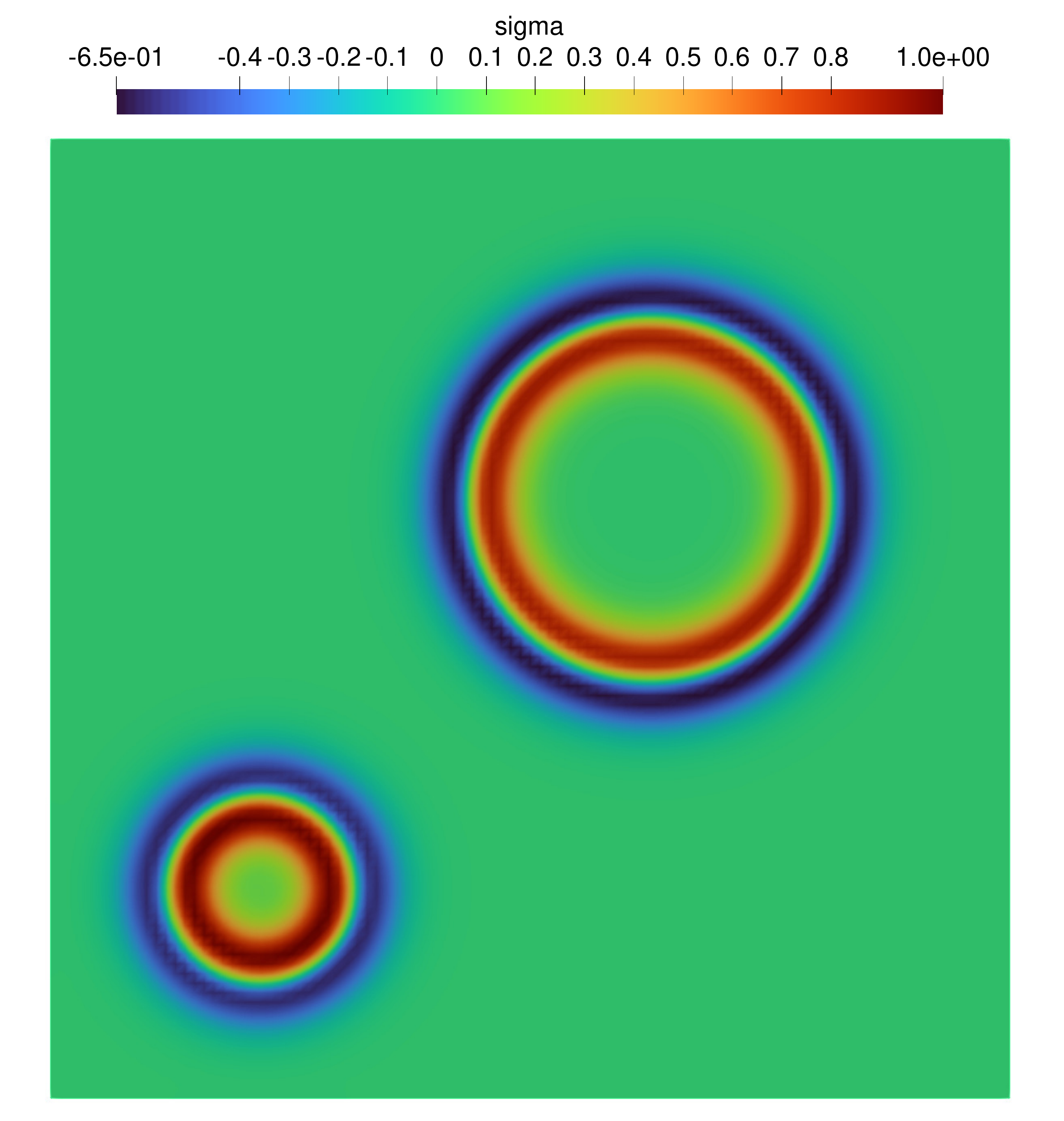}
			\caption{Example I. Experimental Order of Convergence. Reference solution for $\phi$ (left), $\mu$ (center) and $\sigma$ (right).} \label{fig:Ex1FinalConfig}
		\end{center}
	\end{figure}

	\begin{table}[h]
		%\begin{center}
		\begin{tabular}{|c|cc|cc|}                      
			\hline                                                                                                                                                                                                                                                              
			$\Delta t $& $e_2(\phi)$  & $r_2(\phi)$  & $e_1(\phi)$ & $r_1(\phi)$  
			\\ \hline                                                                                                               
			$1\times10^{-6}$ 
			& $0.2013\times10^{-7}$ & $ - $ 
			& $0.5656\times10^{-5}$ & $ - $ 
			\\ 
			$2{^{-1}}\times10^{-6}$ 
			& $0.0494\times10^{-7}$ & $ 2.0283 $ 
			& $ 0.1329\times10^{-5}$ & $ 2.0894 $ 
			\\ 
			$3{^{-1}}\times10^{-6}$ 
			& $0.0219\times10^{-7}$ & $ 2.0010 $ 
			& $ 0.0590\times10^{-5}$ & $ 2.0009 $ 
			\\ 
			$4{^{-1}}\times10^{-6}$ 
			& $0.0123\times10^{-7}$ & $ 2.0036 $ 
			& $ 0.0332\times10^{-5}$ & $ 2.0037 $ 
			\\ 
			$5{^{-1}}\times10^{-6}$ 
			& $0.0079\times10^{-7}$ & $ 2.0033 $ 
			& $ 0.0212\times10^{-5}$ & $ 2.0026 $ 
			\\ \hline
			%\end{tabular}
			%\\                                             
			%\begin{tabular}{|c|cc|cc|cc|}                                                               
			%\hline                                                                                                                                                                                                                                               
			$\Delta t$ & $e_2(\mu)$ & $r_2(\mu)$
			&  $e_1(\mu)$ & $r_1(\mu)$ 
			\\ \hline                                           
			$1\times10^{-6}$ 
			& $0.1272\times10^{-3}$ & $ - $ 
			& $0.0598$ & $ - $ 
			\\ 
			$2{^{-1}}\times10^{-6}$ 
			& $0.0134\times10^{-3}$ & $ 3.2453 $ 
			& $0.0038$ & $ 3.9702 $ 
			\\ 
			$3{^{-1}}\times10^{-6}$ 
			& $0.0060\times10^{-3}$ & $ 2.0037 $ 		
			& $0.0017$ & $ 2.0035 $ 
			\\ 
			$4{^{-1}}\times10^{-6}$ 
			& $0.0033\times10^{-3}$ & $ 2.0030 $ 		
			& $0.0010$ & $ 2.0025 $ 
			\\ 
			$5{^{-1}}\times10^{-6}$ 
			& $0.0021\times10^{-3}$ & $ 2.0047 $ 		
			& $0.0006$ & $ 2.0039 $
			\\ \hline
			%\end{tabular}
			%\\                                             
			%\begin{tabular}{|c|cc|cc|cc|}                                                               
			%\hline                                                                                                                                                                                                                                               
			$\Delta t$ 	& $e_2(\sigma)$  & $r_2(\sigma)$  & $e_1(\sigma)$ & $r_1(\sigma)$
			\\ \hline                                           
			$1\times10^{-6}$ 
			& $0.1702\times10^{-5}$ & $ - $ 
			& $0.6135\times10^{-3}$ & $ - $ 
			\\ 
			$2{^{-1}}\times10^{-6}$ 
			& $0.0355\times10^{-5}$ & $ 2.2623 $ 
			& $0.0984\times10^{-3}$ & $ 2.6406 $ 
			\\ 
			$3{^{-1}}\times10^{-6}$ 
			& $0.0158\times10^{-5}$ & $ 2.0013 $ 
			& $0.0437\times10^{-3}$ & $ 2.0011 $ 
			\\ 
			$4{^{-1}}\times10^{-6}$ 
			& $0.0089\times10^{-5}$ & $ 2.0025 $ 
			& $0.0246\times10^{-3}$ & $ 2.0024 $ 
			\\ 
			$5{^{-1}}\times10^{-6}$ 
			& $0.0057\times10^{-5}$ & $ 2.0040 $ 
			& $0.0157\times10^{-3}$ & $ 2.0040 $ 
			\\ \hline
		\end{tabular}
		\caption{Experimental absolute errors and order of convergences with reference solution computed using $\Delta t=10^{-8}$. }\label{tab:Ex1_order}
		%\end{center}
	\end{table}

	\subsection{Study of the dynamics of the system for different choice of parameters}
	In this subsection, we perform simulations to understand the role of each of the physical parameters in the behavior of the system. This becomes particularly crucial since the properties of the surfactant/amphiphile and its concentration appear indirectly in the model through the parameters:~$\lambda$, $\beta$, and $g(\phi)$ (see~\cite{PawlowZajaczkowski11} for a detailed discussion). For ease of benchmarking, we perform simulations using initial conditions similar to those considered previously by other researchers in~\cite{diegelsharma-me23,hoppelinsenmann-me18}. The initial condition is chosen to consist of several regions (or balls) of the oil-rich phase ($\phi=-1,$ blue in the plots), several regions (also balls) of water-rich phases ($\phi=1,$ red in the plots) and the rest of the domain filled with the microemulsion phase ($\phi=0,$ green in the plots). In the context of modeling the dynamics of ternary mixtures, this initial configuration represents the dispersion of droplets of oil and water in a `continuous' surfactant/amphiphile media.
	%To this end we have performed simulations using as initial condition a configuration with several regions (or balls) of oil-rich phase ($\phi=-1$, blue in the plots), several regions (also balls) of water-rich phase ($\phi=1$, red in the plots) and the rest of the domain filled with micro-emulsion phase ($\phi=0$, green in the figures). We chose this configuration because is the type of configuration that has been considered previously by other researchers in \cite{} (\textcolor{red}{Natasha, you know more than me of which works to cite, please add some references from your advisor and your work with Amanda}).
	
	For all the simulations we consider a square domain $\Omega={(}-5,5{)}^2$ discretized using a structured spatial mesh of size $h=1/100$, the temporal interval set to $[0,T]=[0,0.5]$ and the time step set to $\Delta t=10^{-5}$. We first run the simulations using the following values of the physical parameters: $M=0.1$, $g_2=1$, $g_0=-4$, $h_0=0.5$, $\lambda=0.1$ and $\beta=1$. The obtained dynamics are presented in Figure~\ref{fig:standarddynamics} and they will be used as a reference profile to compare with other possible values of the parameters considered in the next subsections. 
	%We observe how the micro-emulsion phase ($\phi=0$) seems to be unstable, {\underline{in the sense that the system does not allow to have big regions of this phase,}} it prefers to have large regions of oil-rich ($\phi=-1$) and water-rich ($\phi=1$) phases that are separated by interfaces of certain width of micro-emulsions ($\phi=0$). \\
	%	{\textcolor{red}{I might be misunderstanding the underlined statement but it is at odds with one of my simulations please see by clicking \href{https://youtu.be/PQ2Y6cLy_Xw?si=YnudzjvT993UYbJ1}{here}. I took the $\lambda=1, \ \beta=1/6, \ h_0=0.5, g_0=-4, g_2=1$. We need to clarify this. Additionally, according to Pawlow et al.~\cite{PawlowZajaczkowski11}, ``a microemulsion with a large amount of internal interfaces can become stable.''  }}

	In Figure~\ref{fig:Standardplots} we present more information of this simulation. We can observe how the energy of the system is decreasing in time and how the volume ($\int_\Omega\phi^{n+1} d\x$) is conserved in time, which is consistent with the theoretical results in Lemma~\ref{lem:stabilityandconservation}. Moreover, we plot 
	the evolution  in time of the maximum and the minimum of $\phi$ ($\max_{\x\in\Omega}\phi(\x)$ and $\min_{\x\in\Omega}\phi(\x)$) over the domain to illustrate that the solution of the system does not remain in the interval $[-1,1]$, which is a drawback of the model (which does not satisfy a maximum principle result at the continuous level), rather than the numerical scheme. Furthermore, we also report on the number of iterations that the system needs to perform to achieve the required tolerance $\texttt{TOL}=10^{-7}$. This information is presented so as to show that the number of iterations depends on the choice of model parameters and that this number can change in accordance with the constraint of well-posedness of the iterative algorithm presented in Lemma~\ref{lem:iterativealgorithm}. We show that by either reducing the value of $\lambda$ or increasing the value of $|g_0|$ or $M$ makes it more challenging to approximate the nonlinear system~\eqref{eq:scheme}.
	%Through the reducing the value of $\lambda$ or increasing the value of $|g_0|$ or $M$ results in being more challenging to approximate the nonlinear system \eqref{eq:scheme}. 
	Finally, we omit to plot the evolution of the volume for the rest of the simulations in this section since all of them satisfy the conservation property as proved in Lemma~\ref{lem:iterativealgorithm}.

	\begin{figure}[H]
		\begin{center}
			\includegraphics[width=0.19\textwidth]{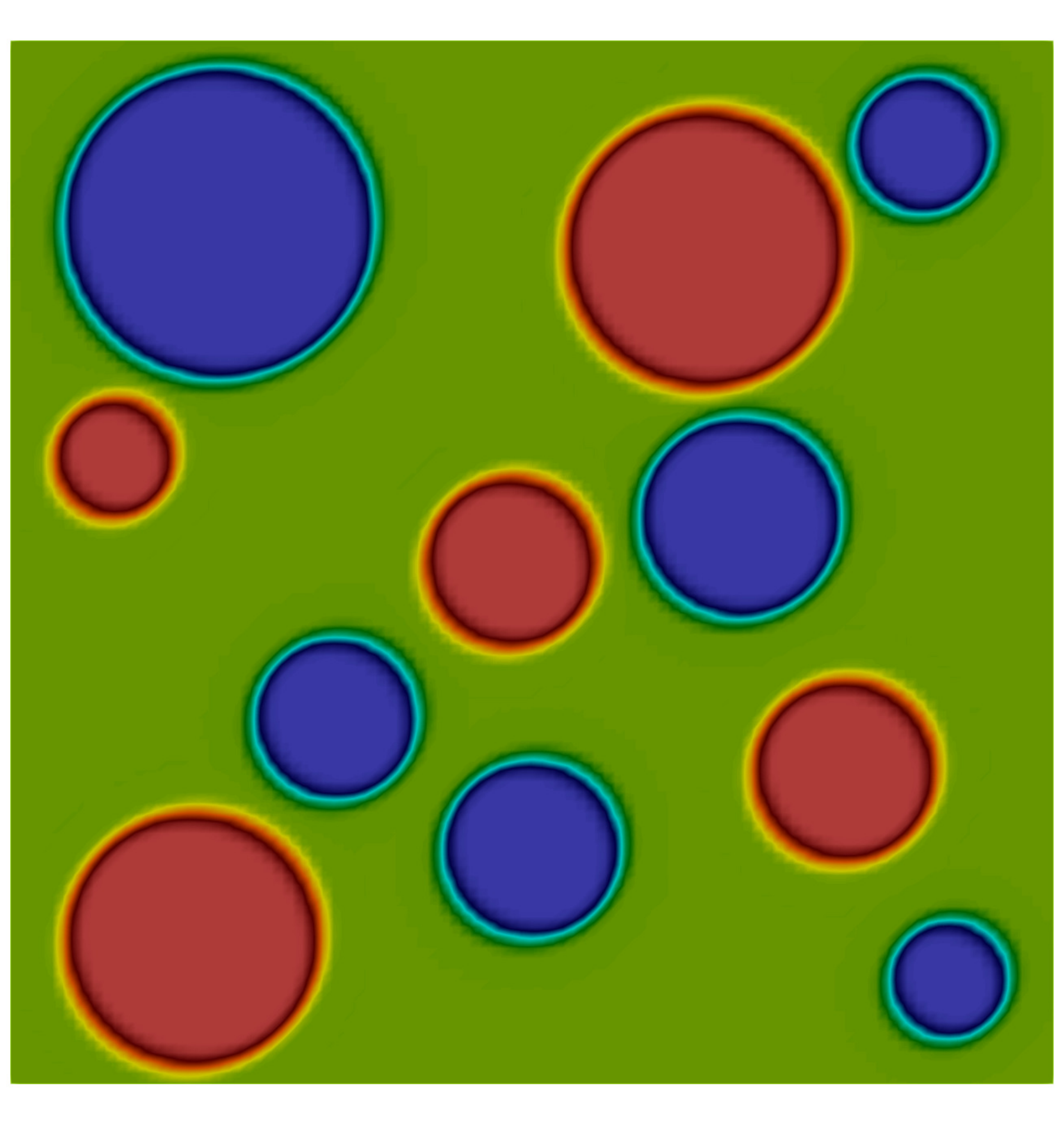}
			\includegraphics[width=0.19\textwidth]{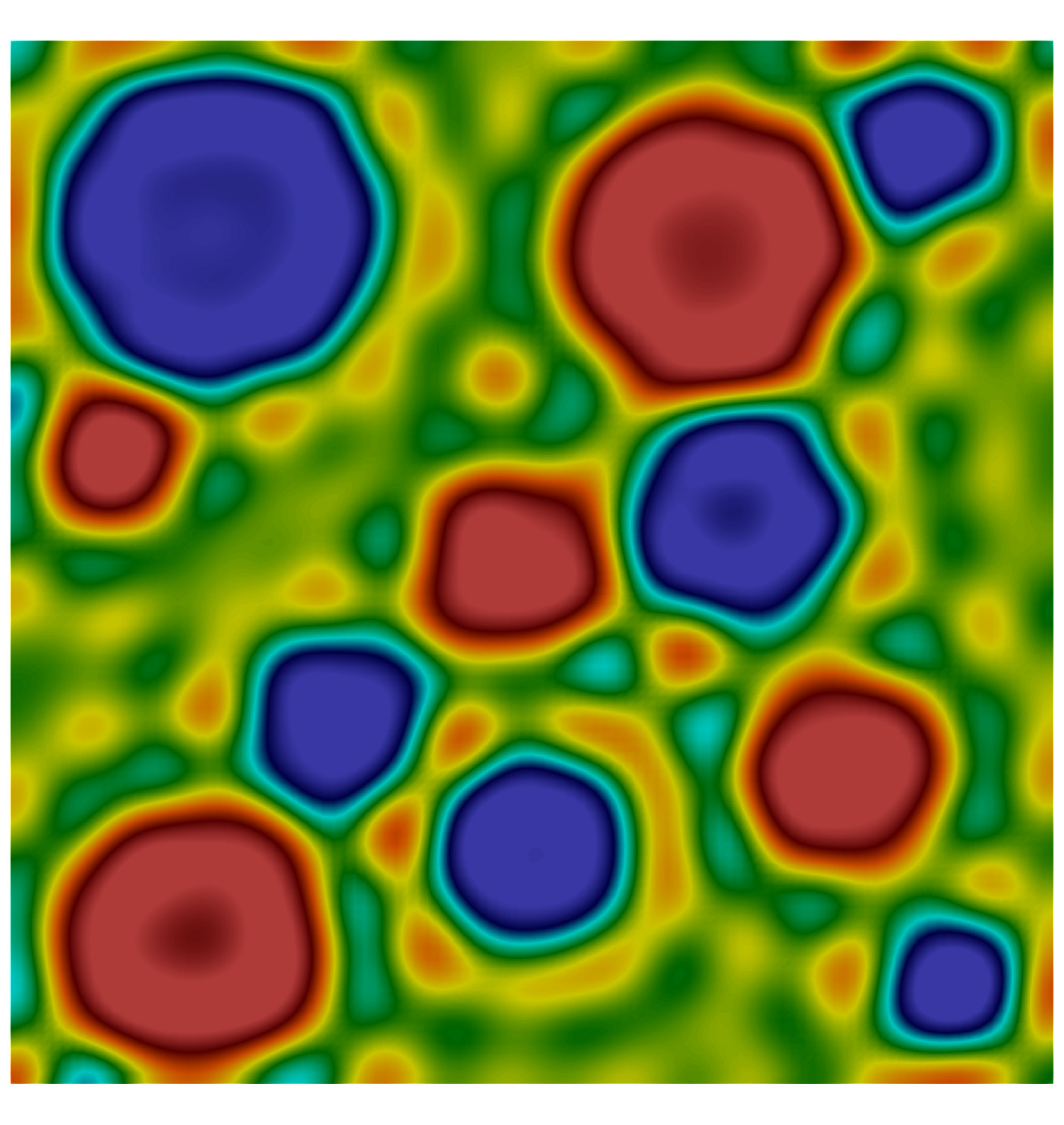}
			\includegraphics[width=0.19\textwidth]{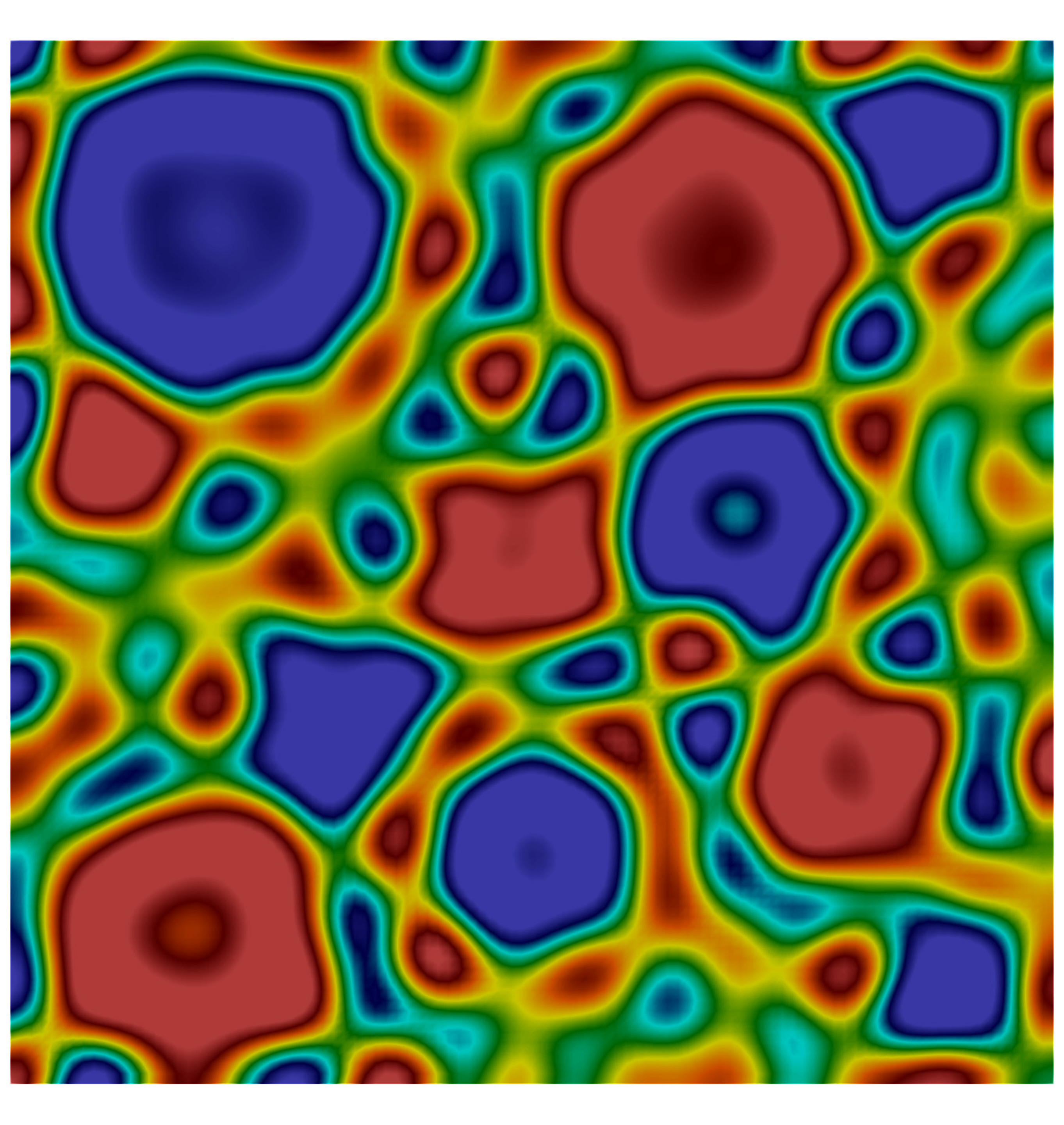}
			\includegraphics[width=0.19\textwidth]{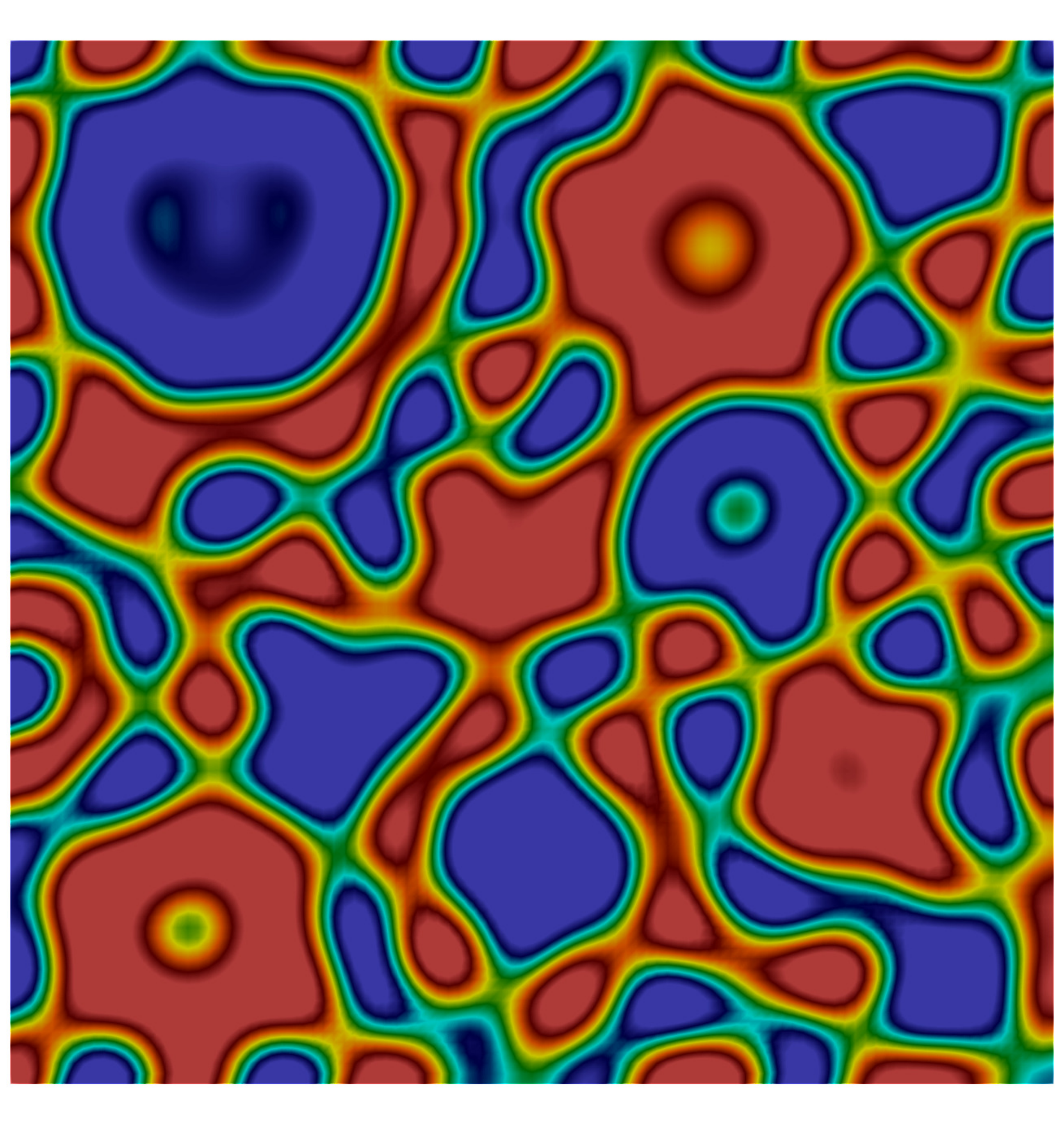}
			\includegraphics[width=0.19\textwidth]{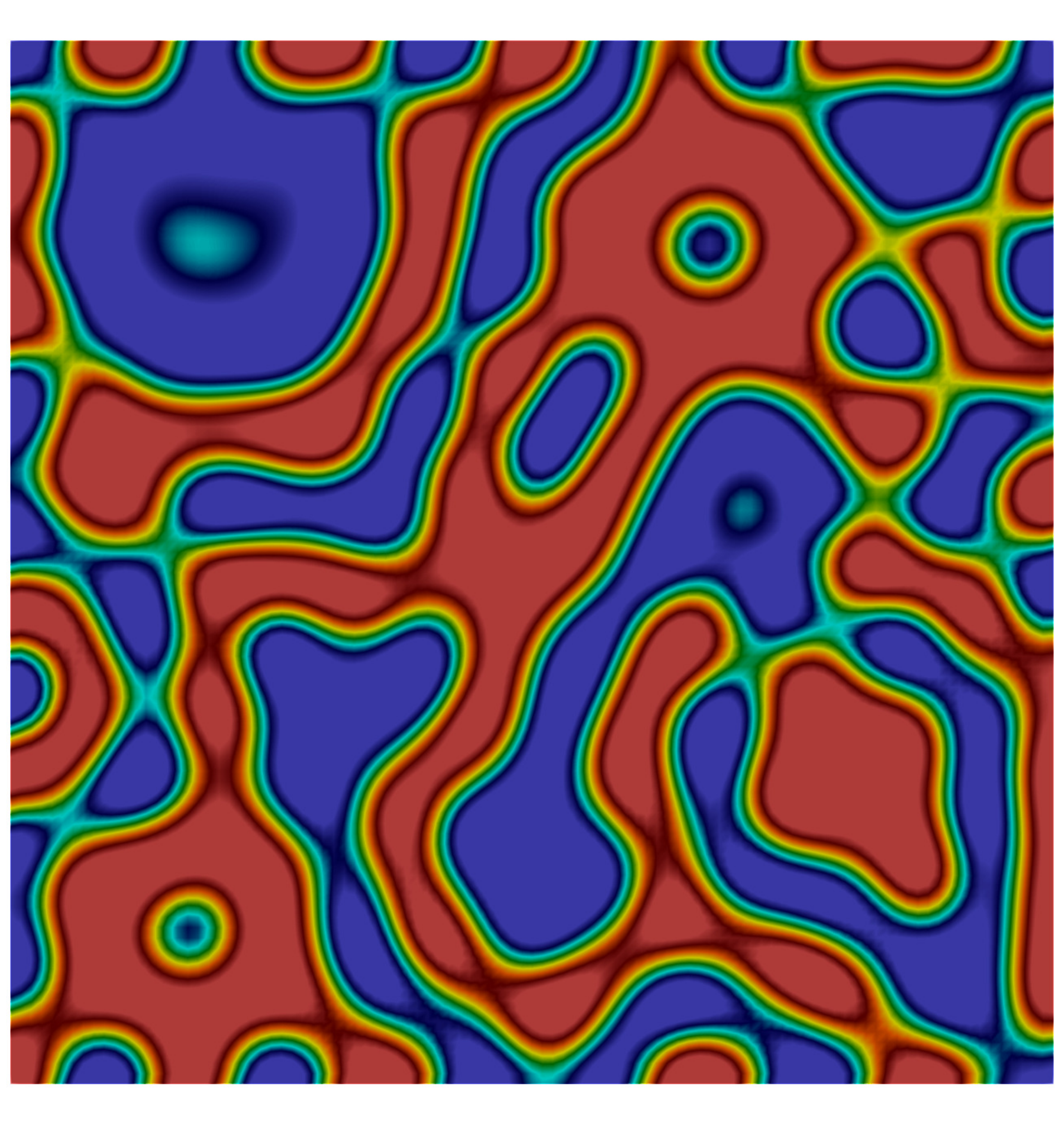}
			\\
			\includegraphics[width=0.19\textwidth]{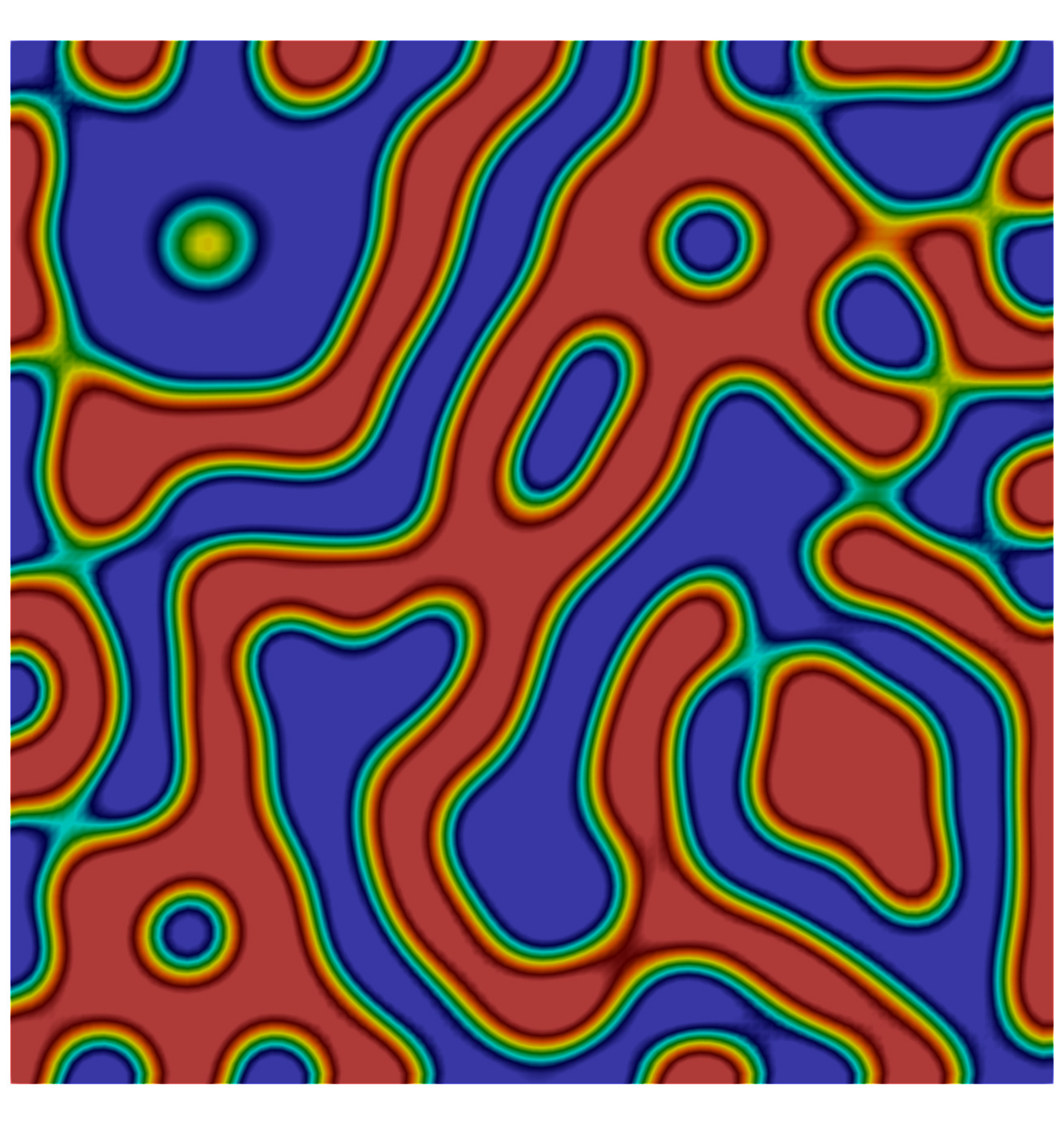}
			\includegraphics[width=0.19\textwidth]{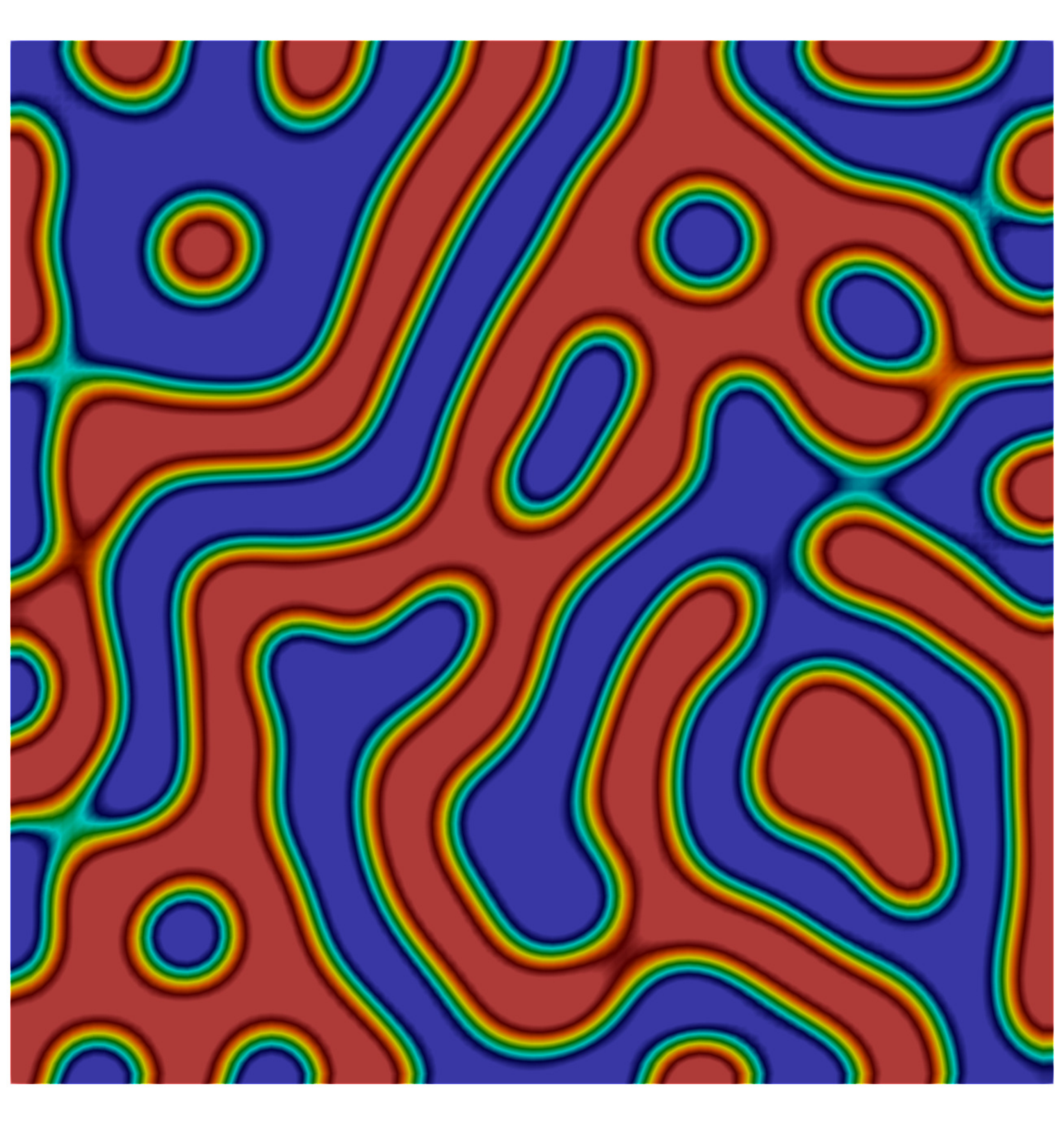}
			\includegraphics[width=0.19\textwidth]{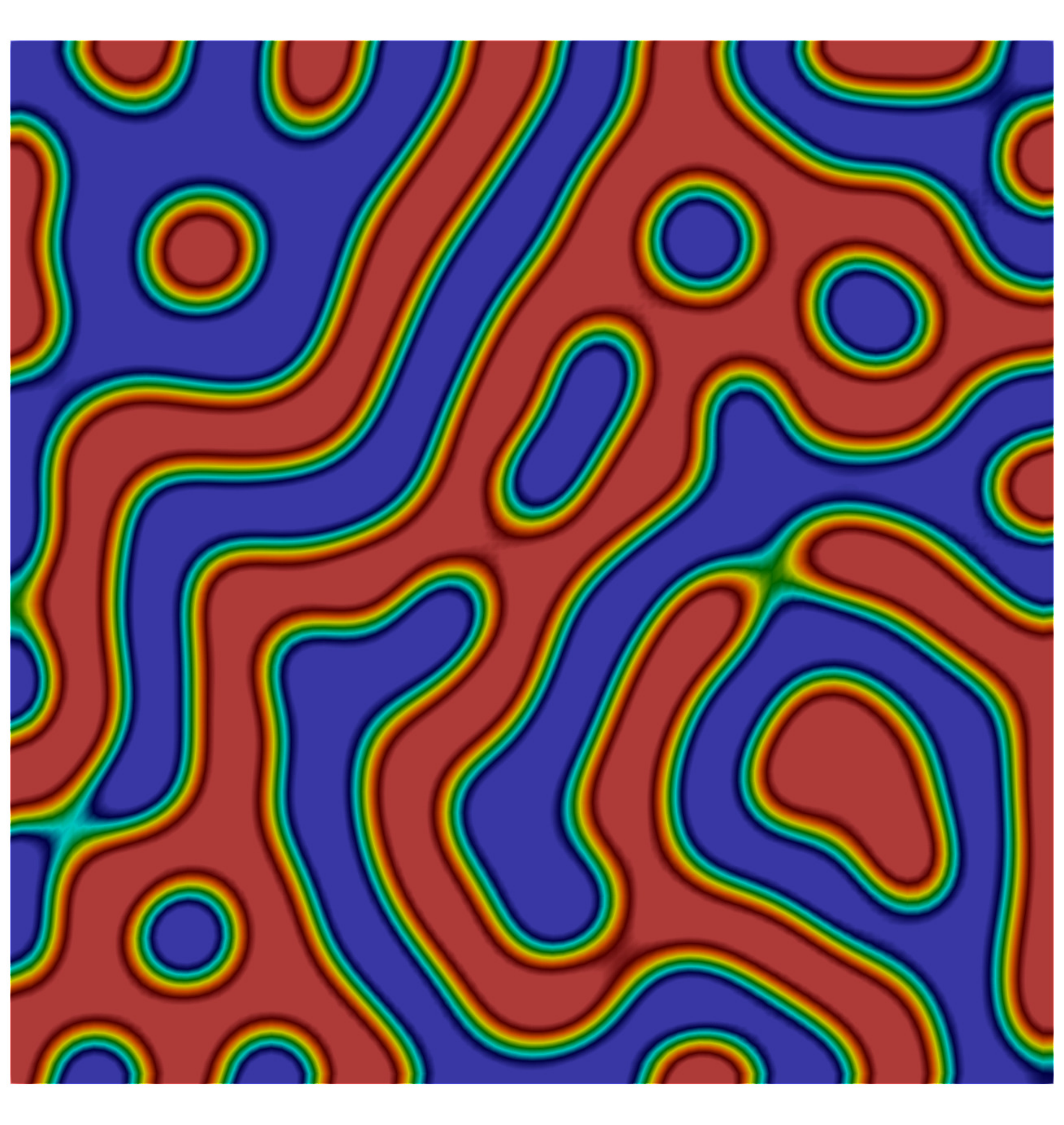}
			\includegraphics[width=0.19\textwidth]{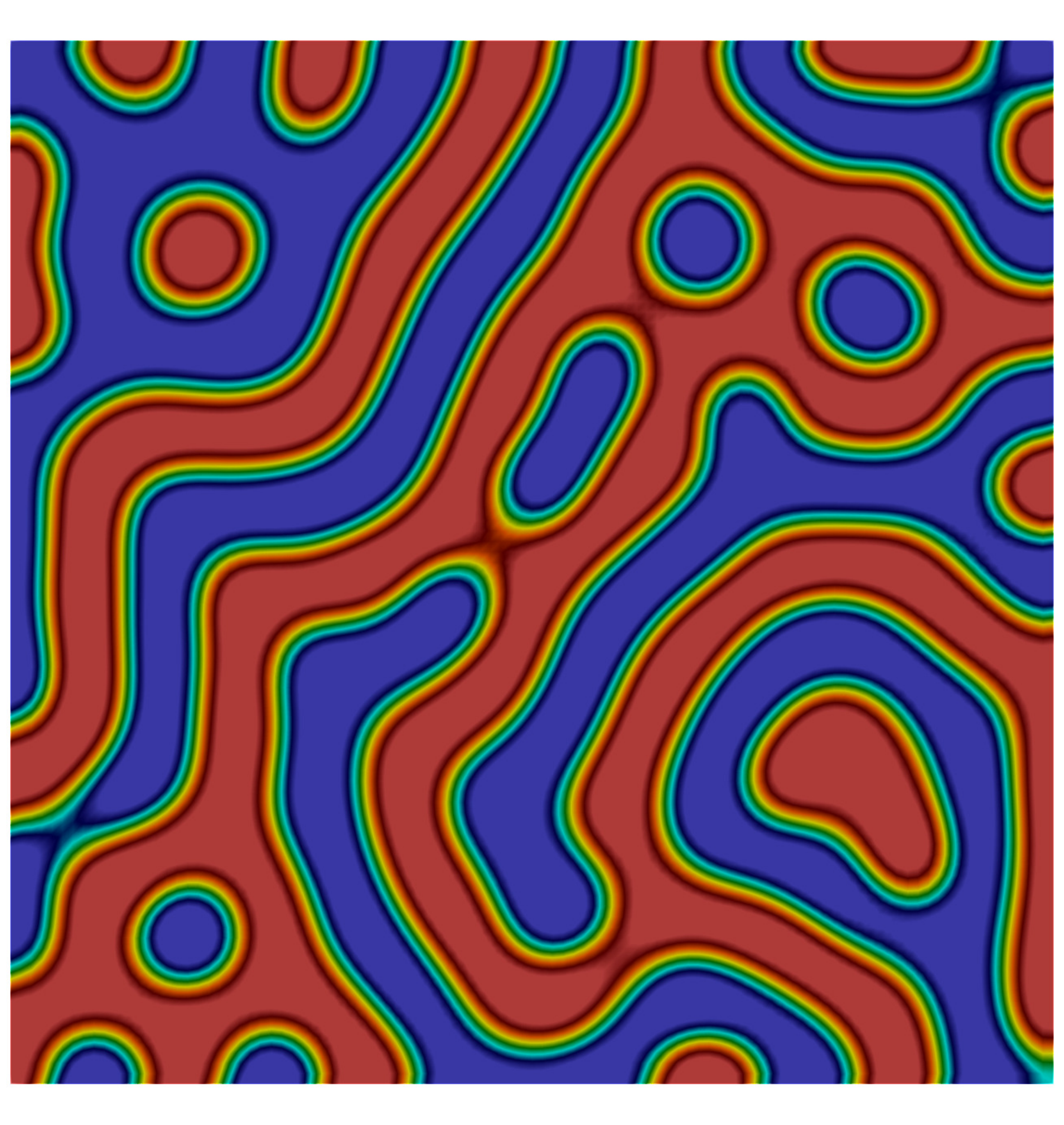}
			\includegraphics[width=0.19\textwidth]{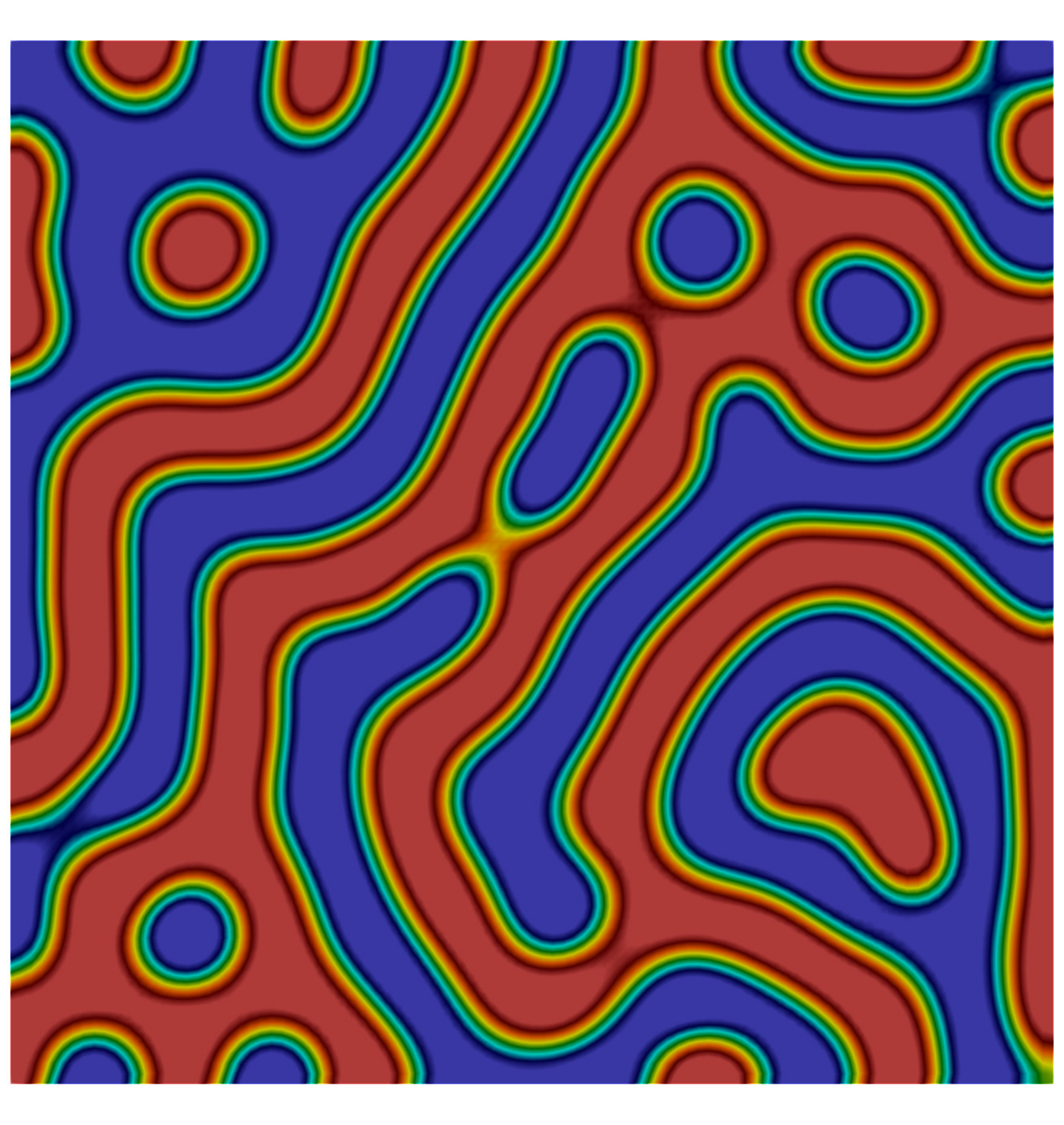}
			\caption{Evolution of $\phi$ at times $t=0, 0.01, 0.025, 0.05, 0.1, 0.15, 0.25, 0.35, 0.45$ and $0.5$ (from left to right and top to bottom) taking $M=0.1$, $g_2=1$, $g_0=-4$, $h_0=0.5$, $\lambda=0.1$ and $\beta=1$.} \label{fig:standarddynamics}
		\end{center}
	\end{figure}

	\begin{figure}[h]
		\begin{center}
			\includegraphics[width=0.45\textwidth]{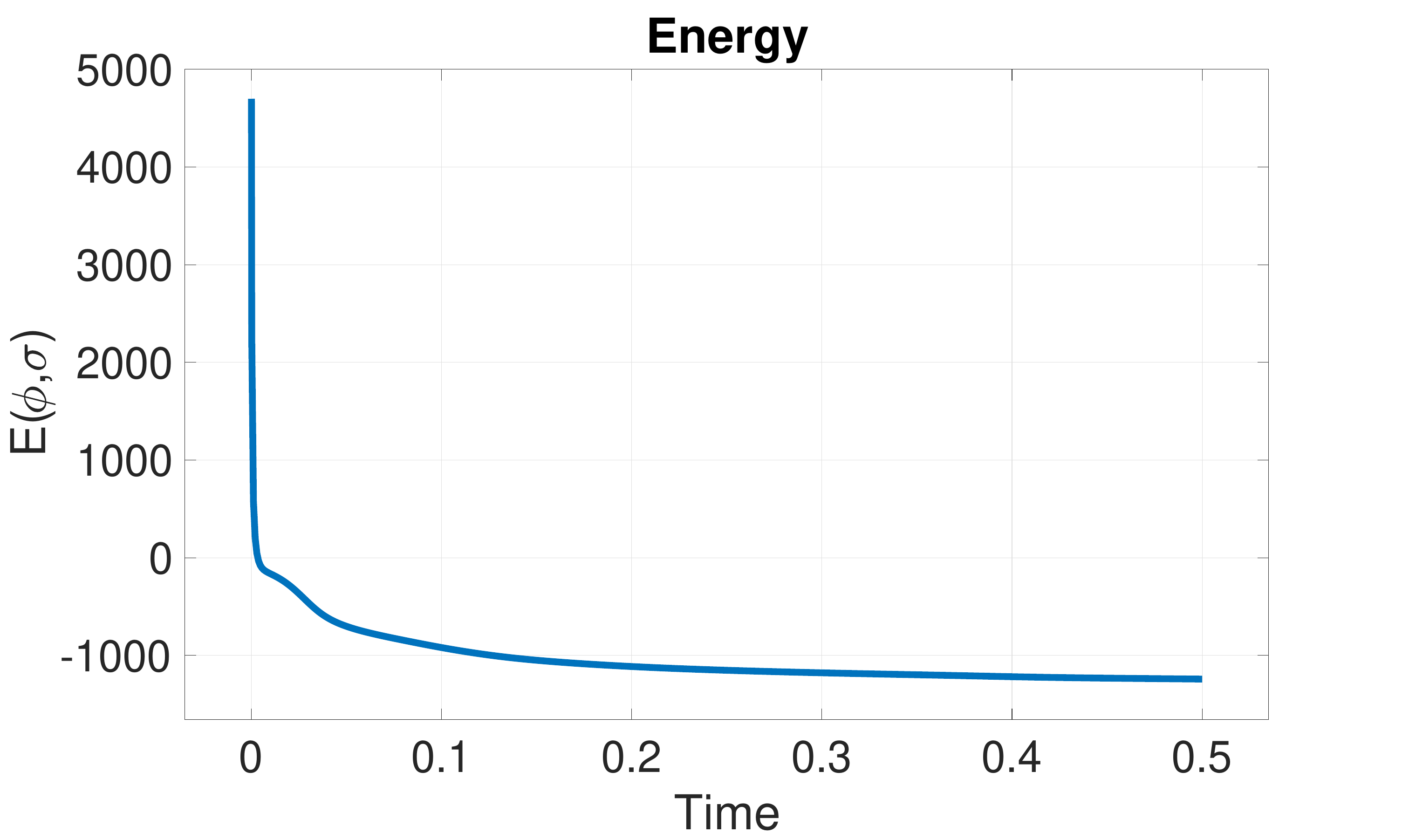}
			\includegraphics[width=0.45\textwidth]{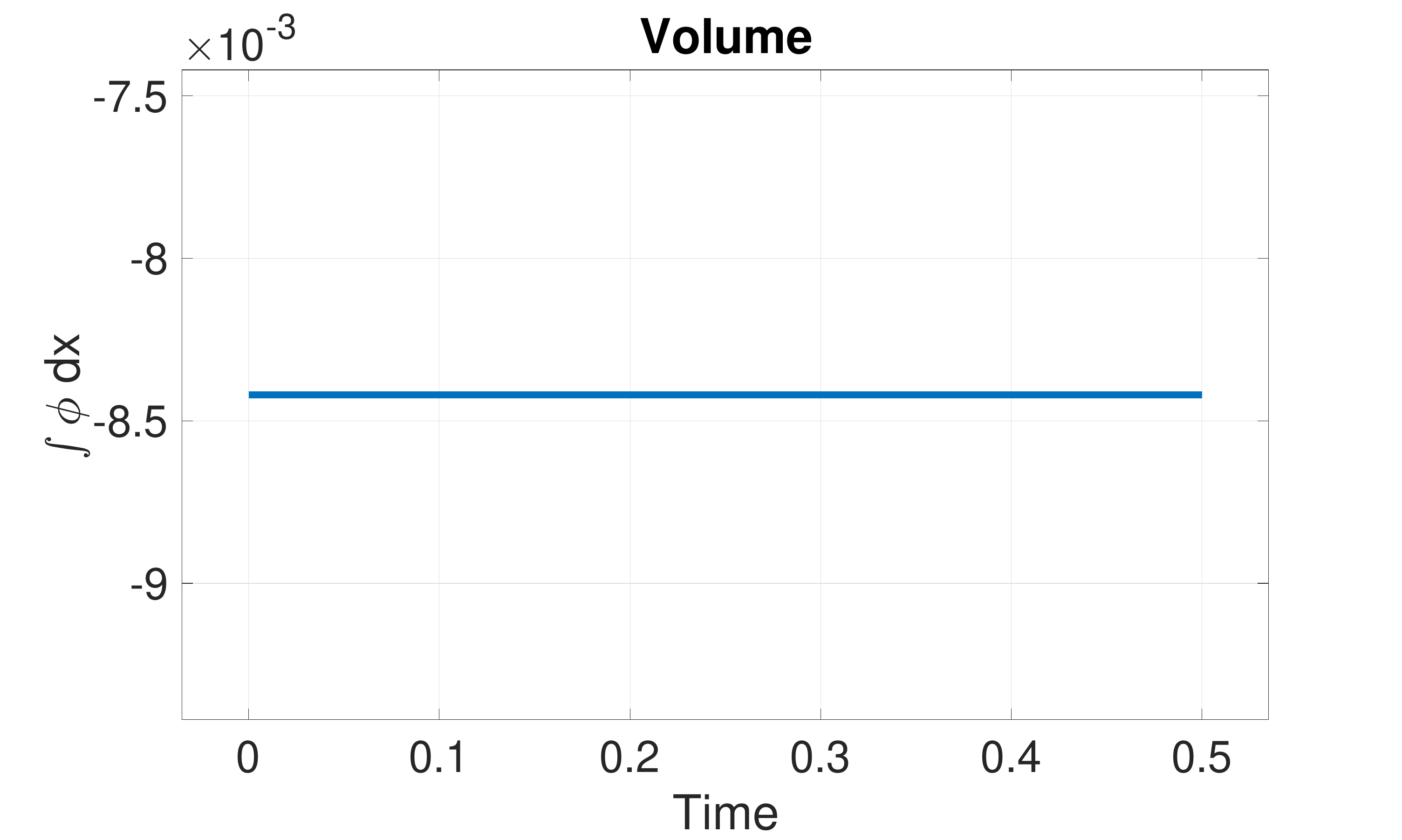}
			\\
			\includegraphics[width=0.45\textwidth]{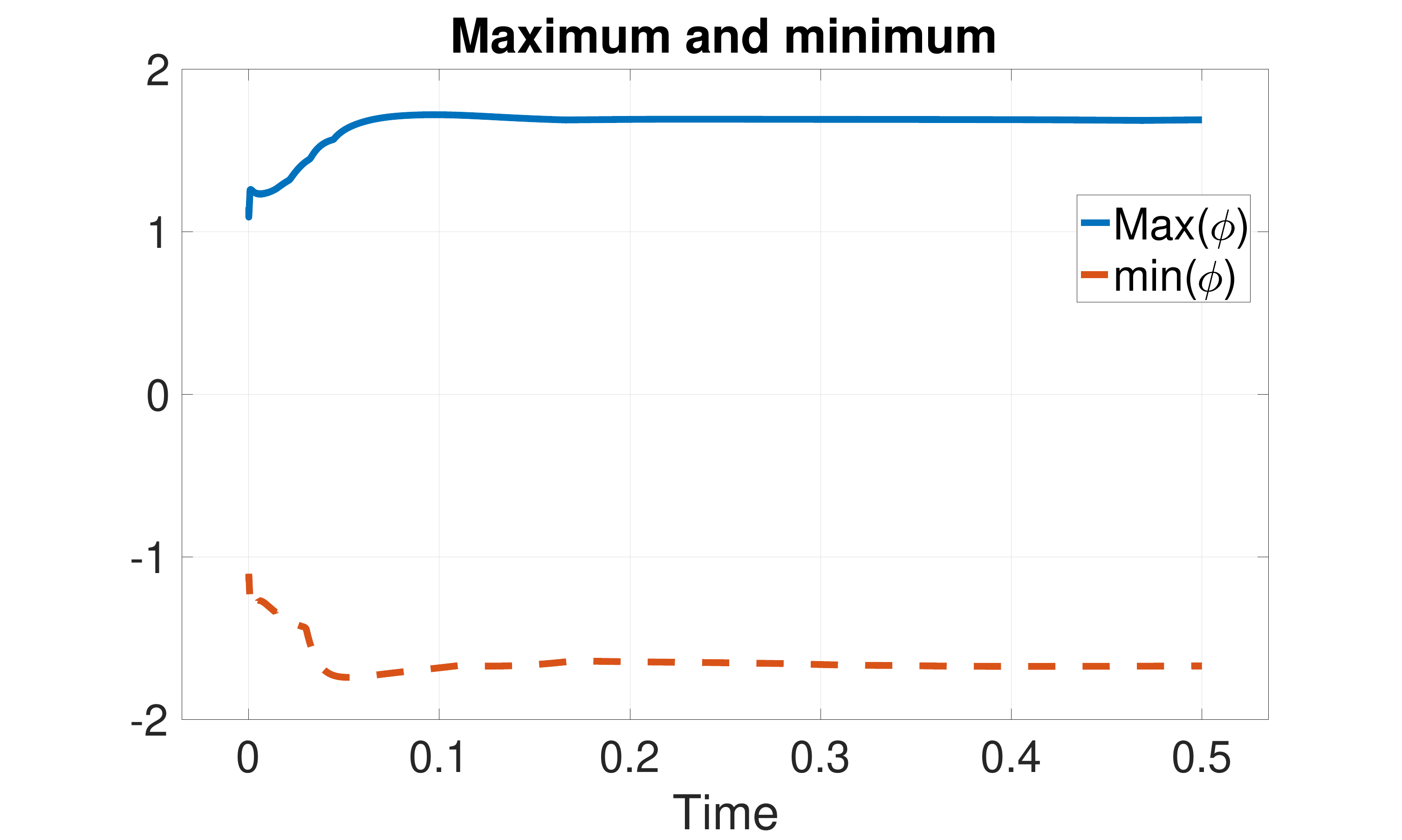}
			\includegraphics[width=0.45\textwidth]{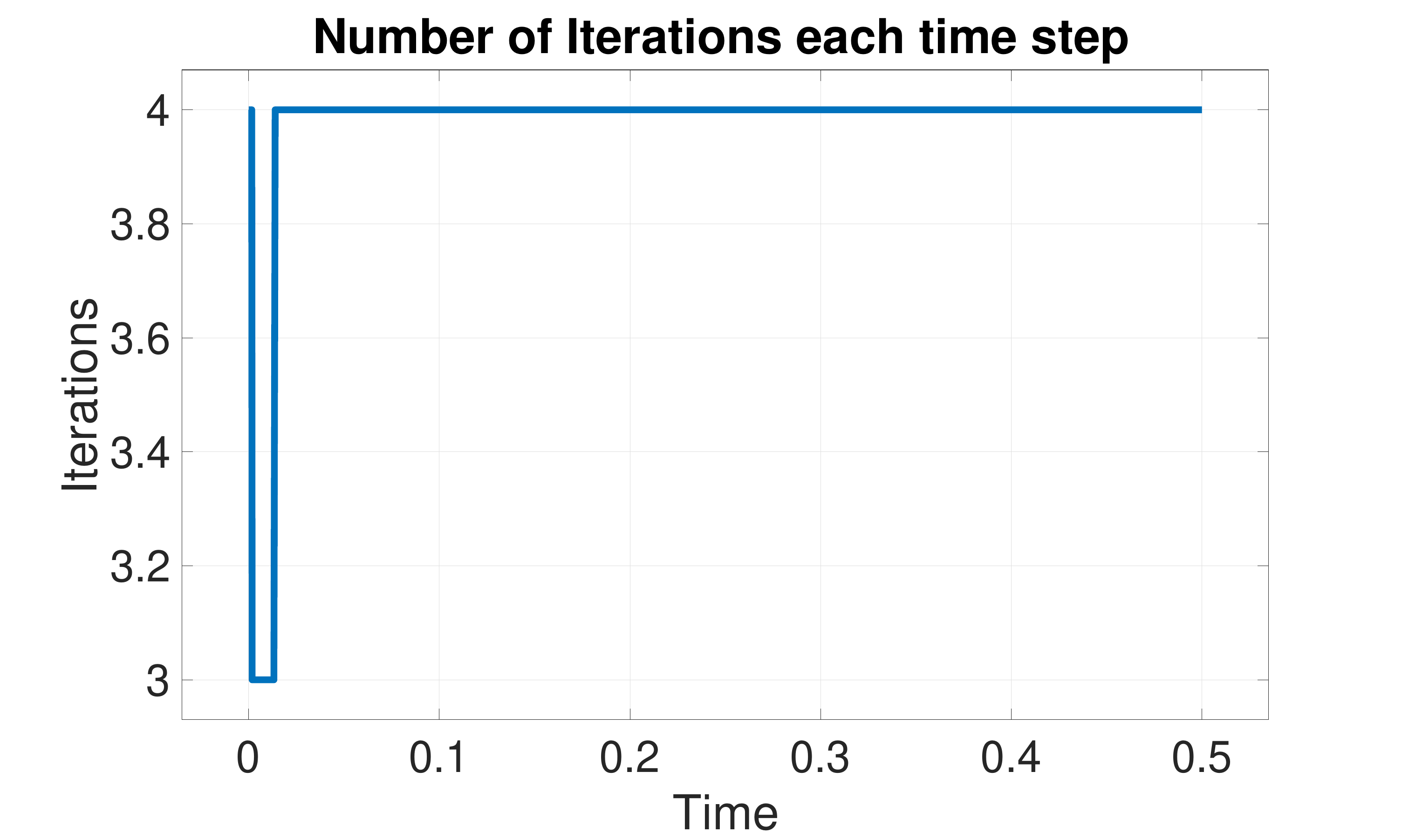}
			\caption{Evolution in time of the energy (top left), the volume (top right), maximum and minimum of $\phi$ (bottom left) and the number of iterations to achieve tolerance $\texttt{TOL}=10^{-7}$ (bottom right) taking $M=0.1$, $g_2=1$, $g_0=-4$, $h_0=0.5$, $\lambda=0.1$ and $\beta=1$.} \label{fig:Standardplots}
		\end{center}
	\end{figure}
	
	\subsubsection{Study of influence of parameter $\beta$ on the dynamics of the system}\label{sub:beta}
	In this example, we fix the parameters to
	$M=0.1$, $g_2=1$, $g_0=-4$, $h_0=0.5$, $\lambda=0.1$ and  we consider different values of parameter $\beta$, namely $\beta=1$, $\beta=10$ and $\beta=10^2$. The dynamics associated to these simulations are presented in Figure~\ref{fig:Exbetas} and the evolution of the energies, maximum of $\phi$, minimum of $\phi$ and number of iterations are presented in Figure~\ref{fig:betasplot}. 
	%Parameter $\beta$ plays the role of the weight of the six order polynomial $f_0(\phi)$ in the energy $E(\phi,\sigma)$, meaning that the larger the value of $\beta$, the more important that this term becomes for determine the behavior of the system. 
	In view of the discrete energy functional~\eqref{eq:energy},
	%\textcolor{red}{(G: I don't think we need to plot the energy again)}
	%\begin{align*}
	%	E(\phi, \sigma)
	%	:=
	%	\int_\Omega
	%	\left(
	%	\beta f_0(\phi)
	%	+\frac12g(\phi)|\nabla\phi|^2
	%	+\frac{\lambda}2(\sigma)^2
	%	\right)d\x\,,
	%\end{align*}
	it is clear that the coefficient $\beta$ influences the tendency of the ternary system to morph into water-rich and oil-rich regions, hence increasing the value of $\beta$ leads to more significant bulk regions of water and oil. This is evident from the profiles of $\phi$ in the bottom row subfigures of Figure~\ref{fig:Exbetas} which possess relatively larger bulk regions of oil and water. The other impact of this coefficient is discussed in~\cite{PawlowZajaczkowski11} and it is that at the microemulsion phase ($\phi=0$),
	%the properties of the amphiphile and its concentration are related to the choice of $\beta$ in the following manner. T
	the value of $\beta f_0(\phi)$ is inversely proportional to the concentration of the amphiphile/surfactant. That is, $\beta f_0(0)= \beta h_0$ is low for high concentrations of the amphiphile/surfactant and vice versa. 
	
	%(It is unclear to me \emph{how} our simulations are capturing the strength of the amphiphile or its concentration for that matter(it has the largest volume initially). I am inclined to either a) omit this other impact of $\beta$ or b) mention that it is unclear how to measure ``high'' amphiphile concentrations.)
	
	In fact, we can observe that larger values of $\beta$ makes the system to achieve values that are closer to the minima of $f_0(\phi)$ but without having an specific effect on the width of the interface. In all the simulations the total energy decreases as expected and the number of iterations of the iterative algorithm increases at the beginning of the dynamics as the value of $\beta$ increases, but without being a problem for the choice of discretization parameters considered.

	\begin{figure}[h]
		\begin{center}
			\includegraphics[width=0.19\textwidth]{images/Ex2/standard/Ex2_Standard_0}
			\includegraphics[width=0.19\textwidth]{images/Ex2/standard/Ex2_Standard_50}
			\includegraphics[width=0.19\textwidth]{images/Ex2/standard/Ex2_Standard_150}
			\includegraphics[width=0.19\textwidth]{images/Ex2/standard/Ex2_Standard_350}
			\includegraphics[width=0.19\textwidth]{images/Ex2/standard/Ex2_Standard_500}
			\\
			\includegraphics[width=0.19\textwidth]{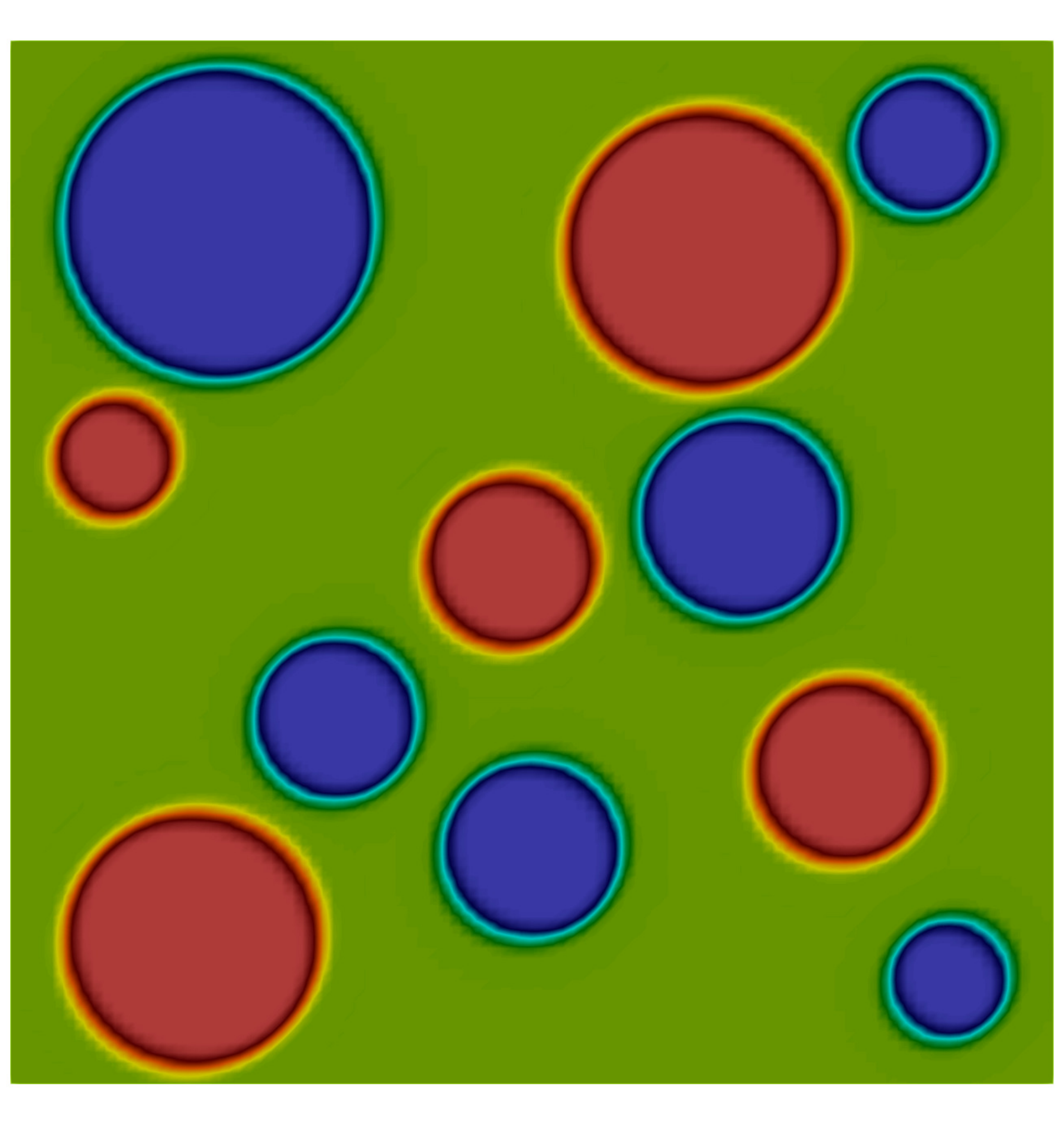}
			\includegraphics[width=0.19\textwidth]{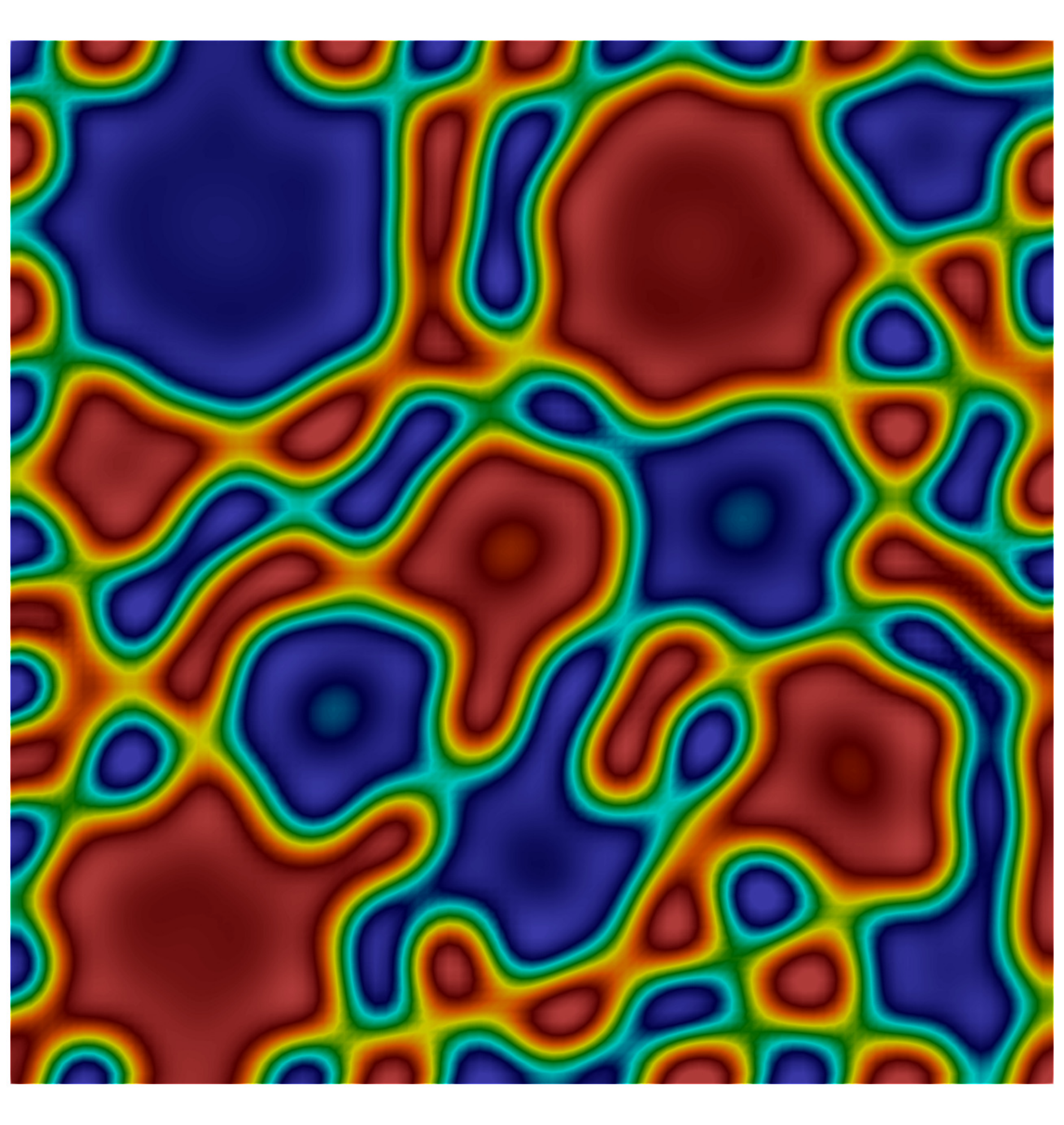}
			\includegraphics[width=0.19\textwidth]{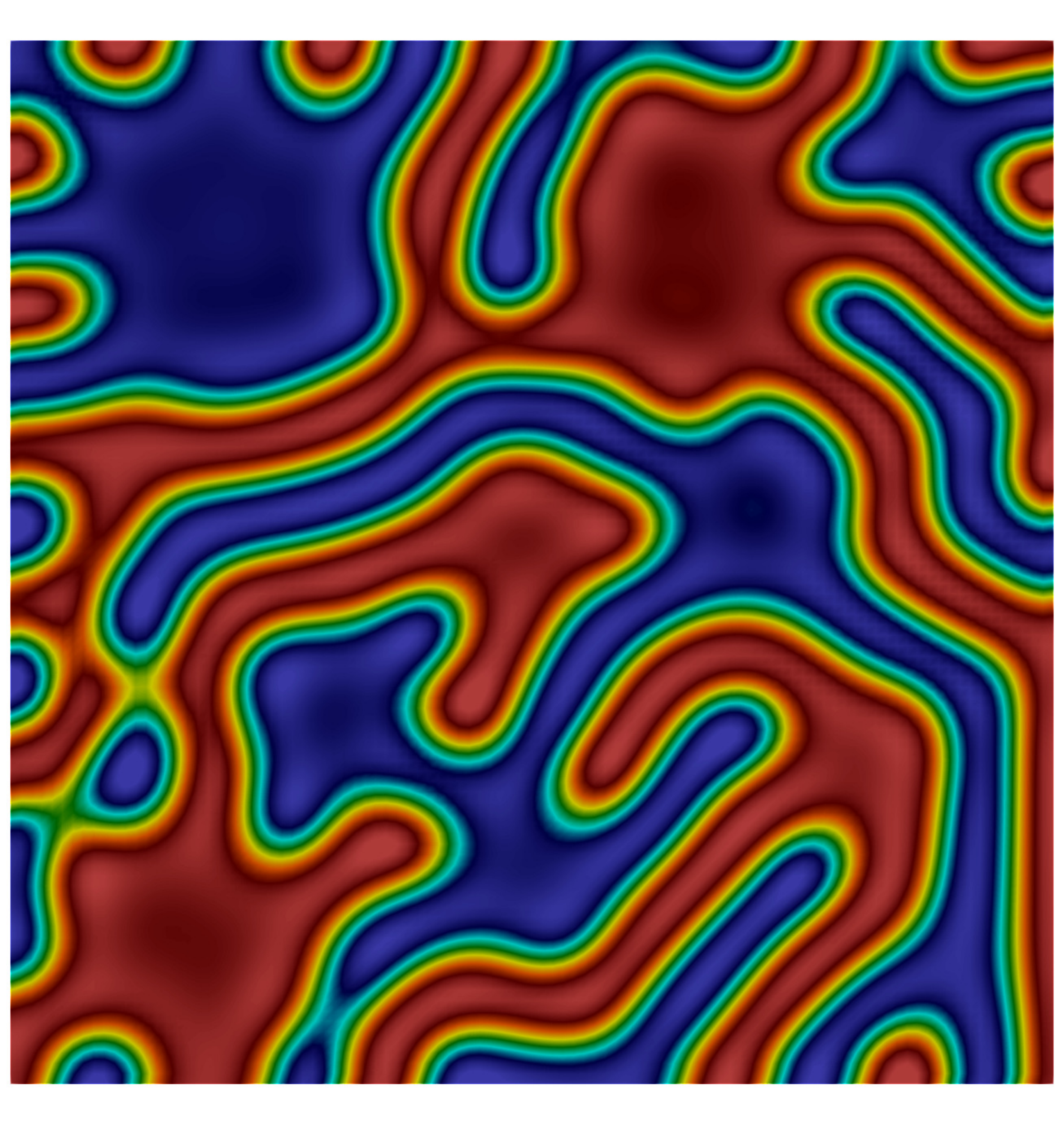}
			\includegraphics[width=0.19\textwidth]{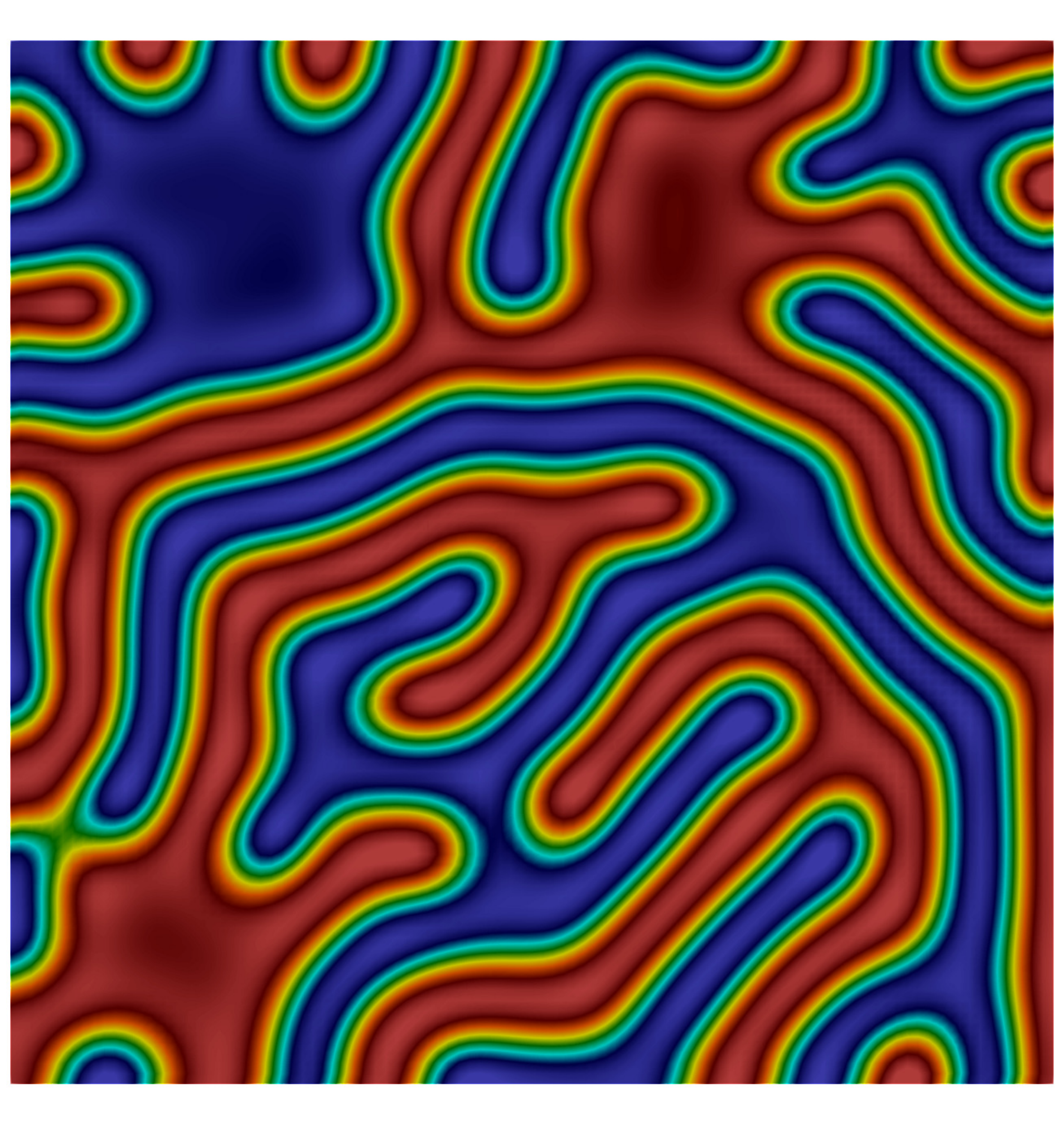}
			\includegraphics[width=0.19\textwidth]{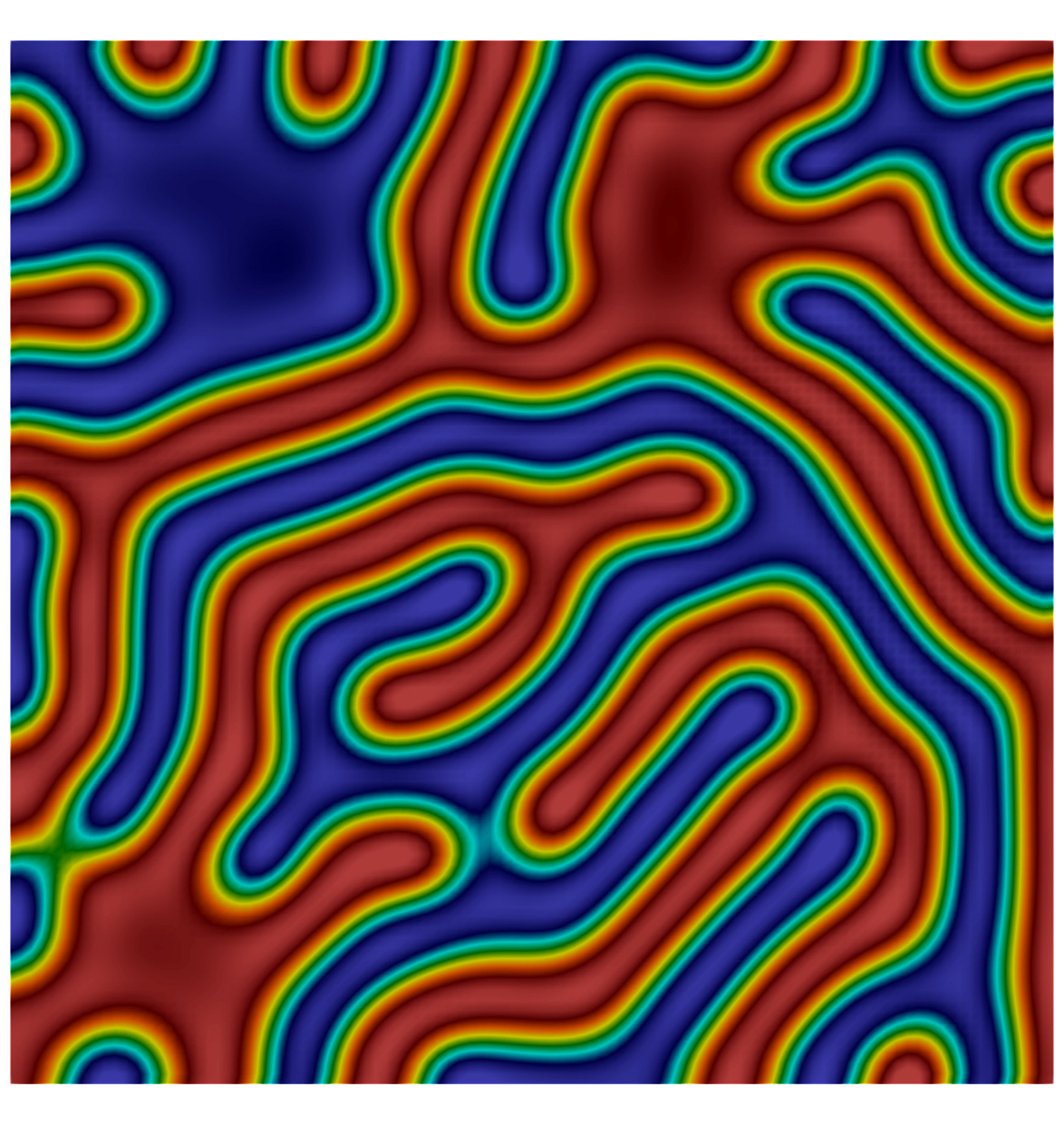}
			\\
			\includegraphics[width=0.19\textwidth]{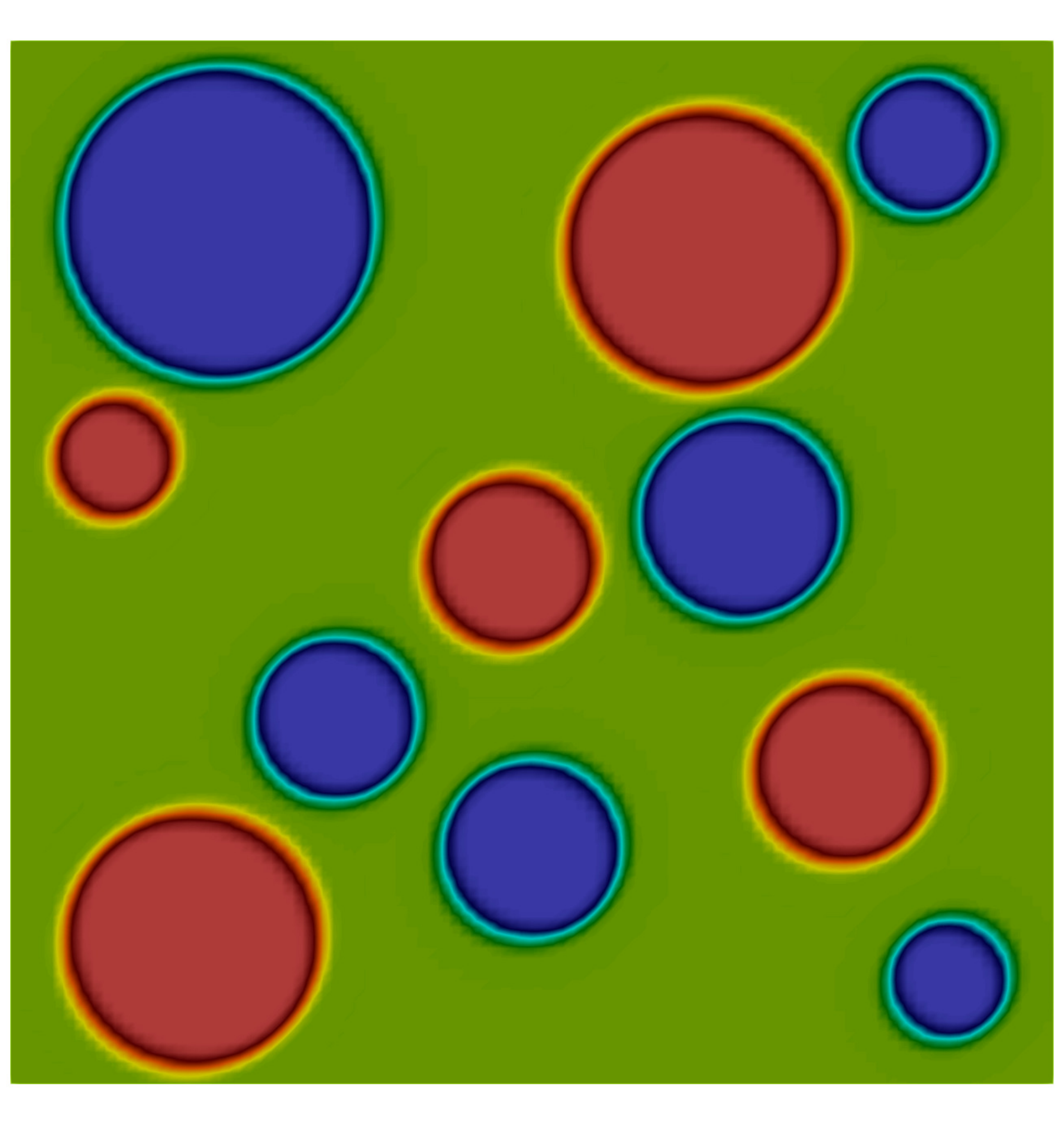}
			\includegraphics[width=0.19\textwidth]{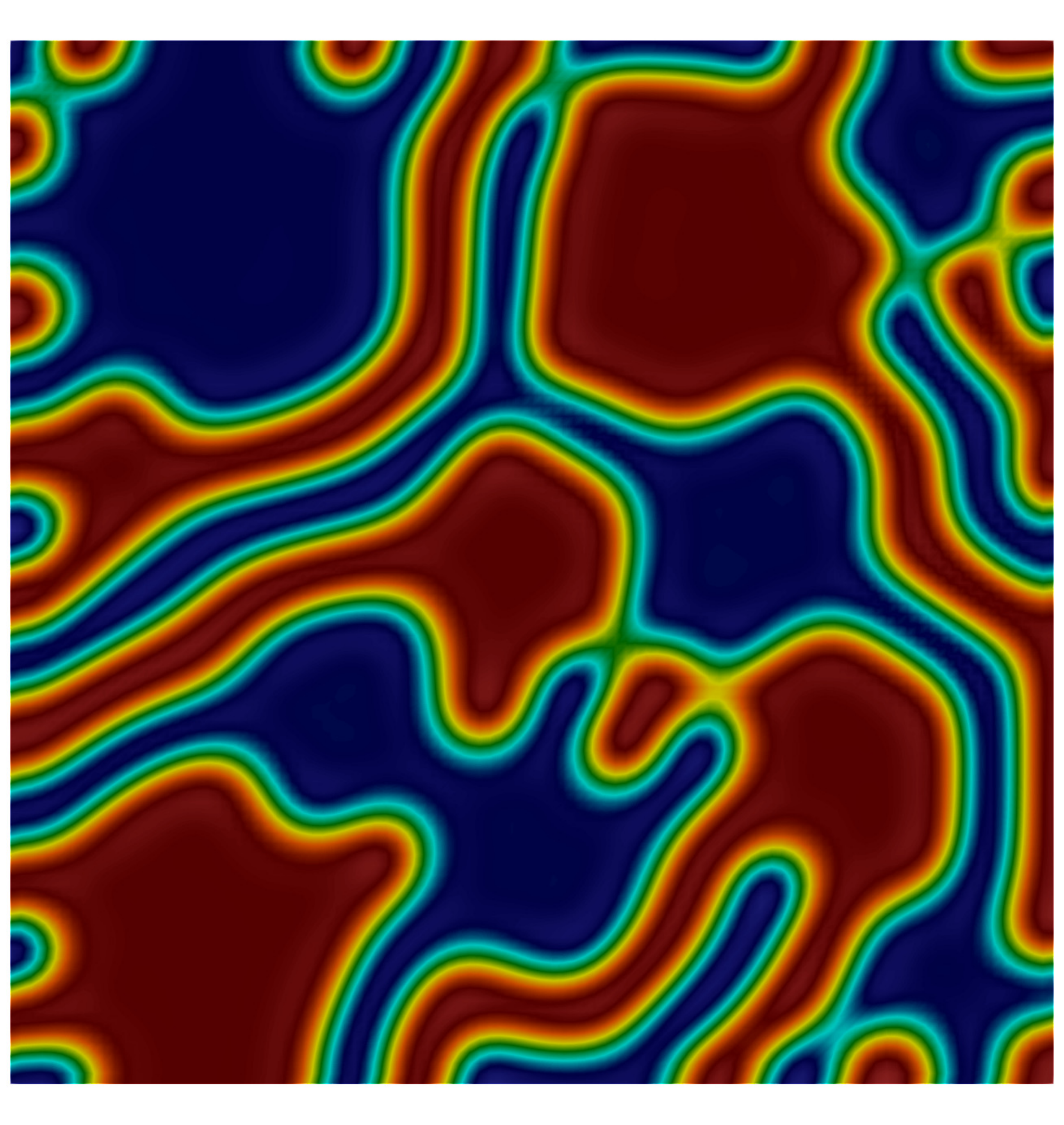}
			\includegraphics[width=0.19\textwidth]{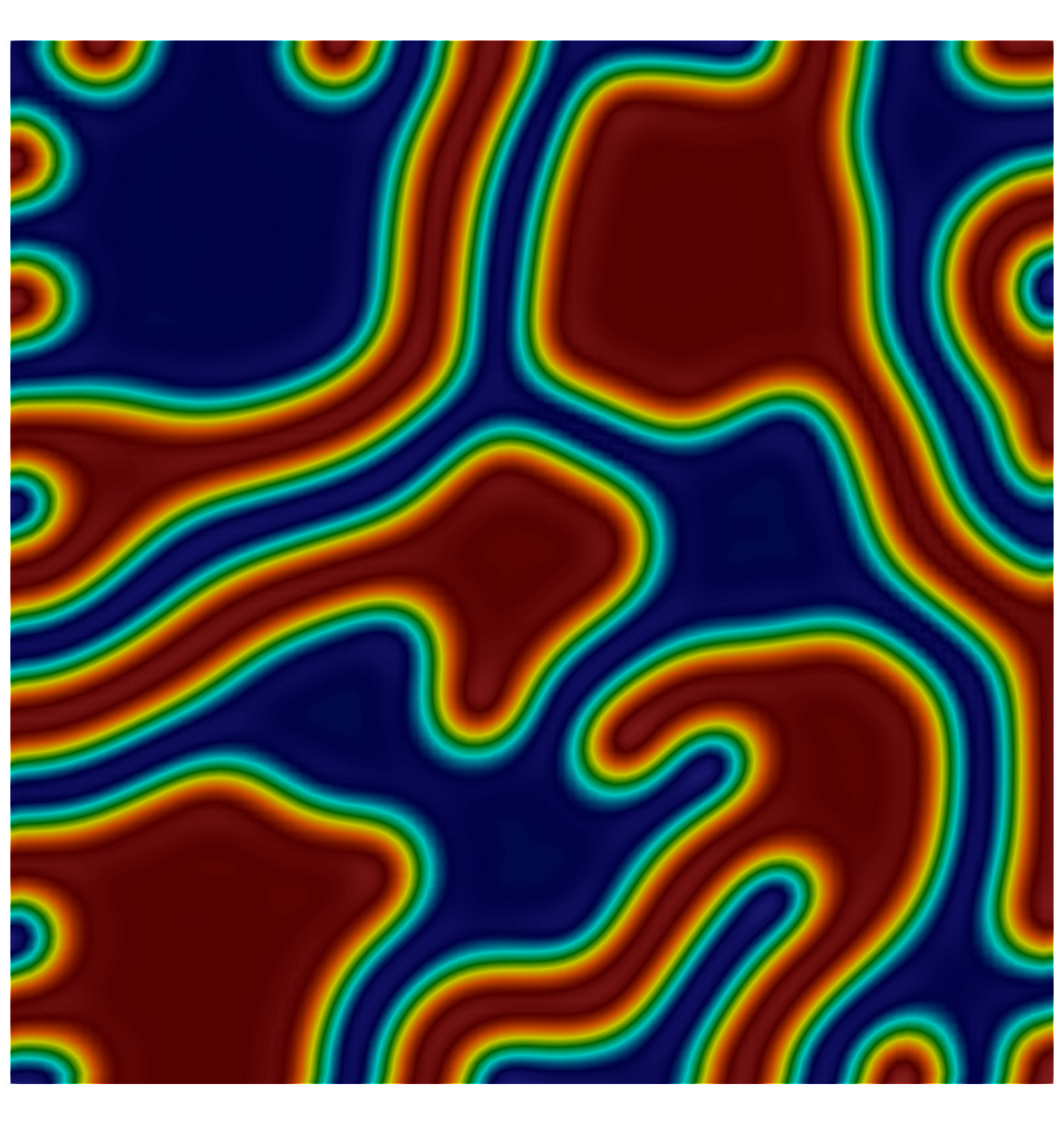}
			\includegraphics[width=0.19\textwidth]{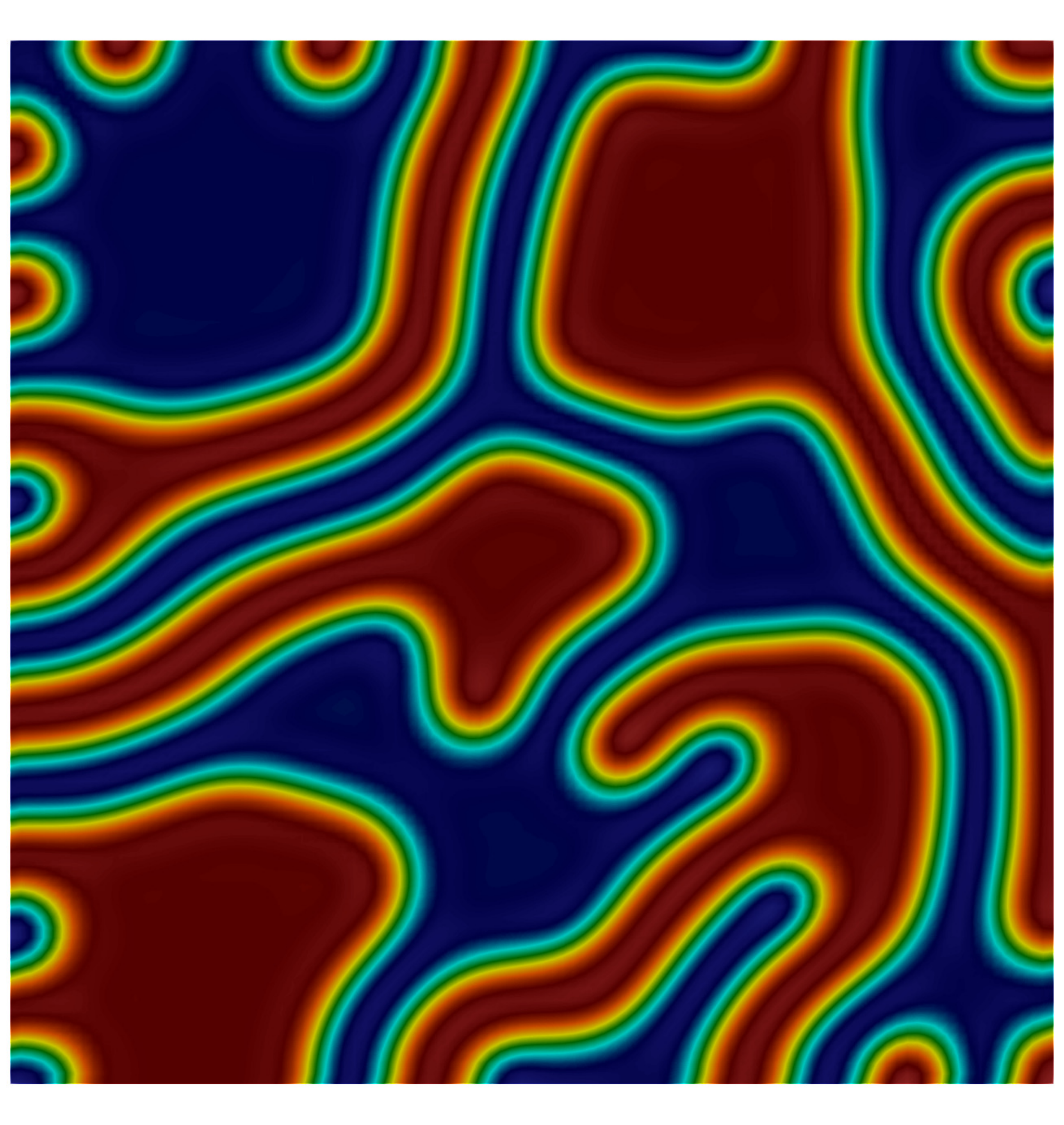}
			\includegraphics[width=0.19\textwidth]{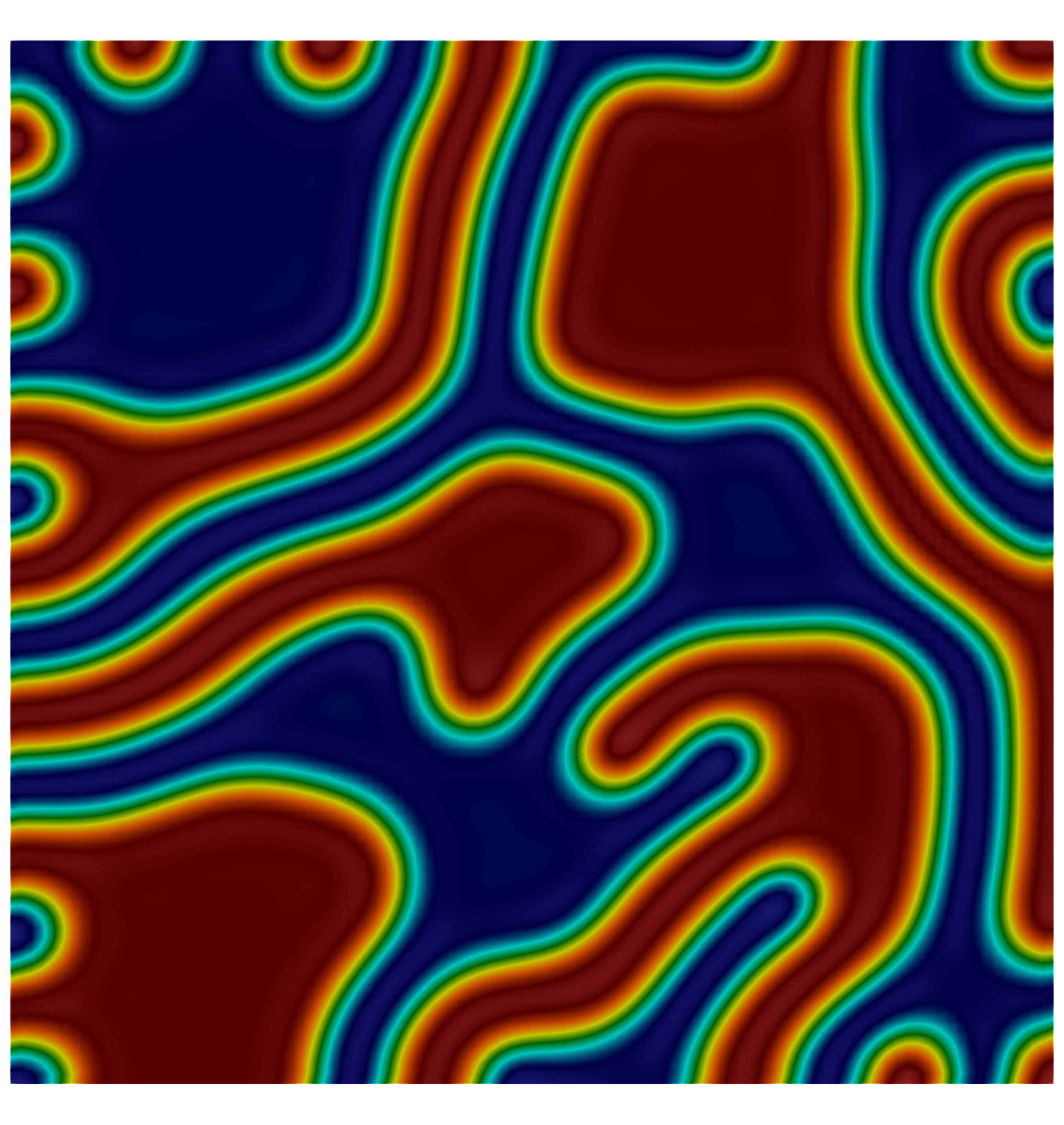}
			\caption{Evolution of $\phi$ at times $t=0, 0.05, 0.15, 0.35$ and $0.5$ (from left to right) taking $M=0.1$, $g_2=1$, $g_0=-4$, $h_0=0.5$ and $\lambda=0.1$ with $\beta=1$ (top row), $\beta=10$ (center row) and $\beta=10^2$ (bottom row).} \label{fig:Exbetas}
		\end{center}
	\end{figure}

	\begin{figure}[h]
		\begin{center}
			\includegraphics[width=0.45\textwidth]{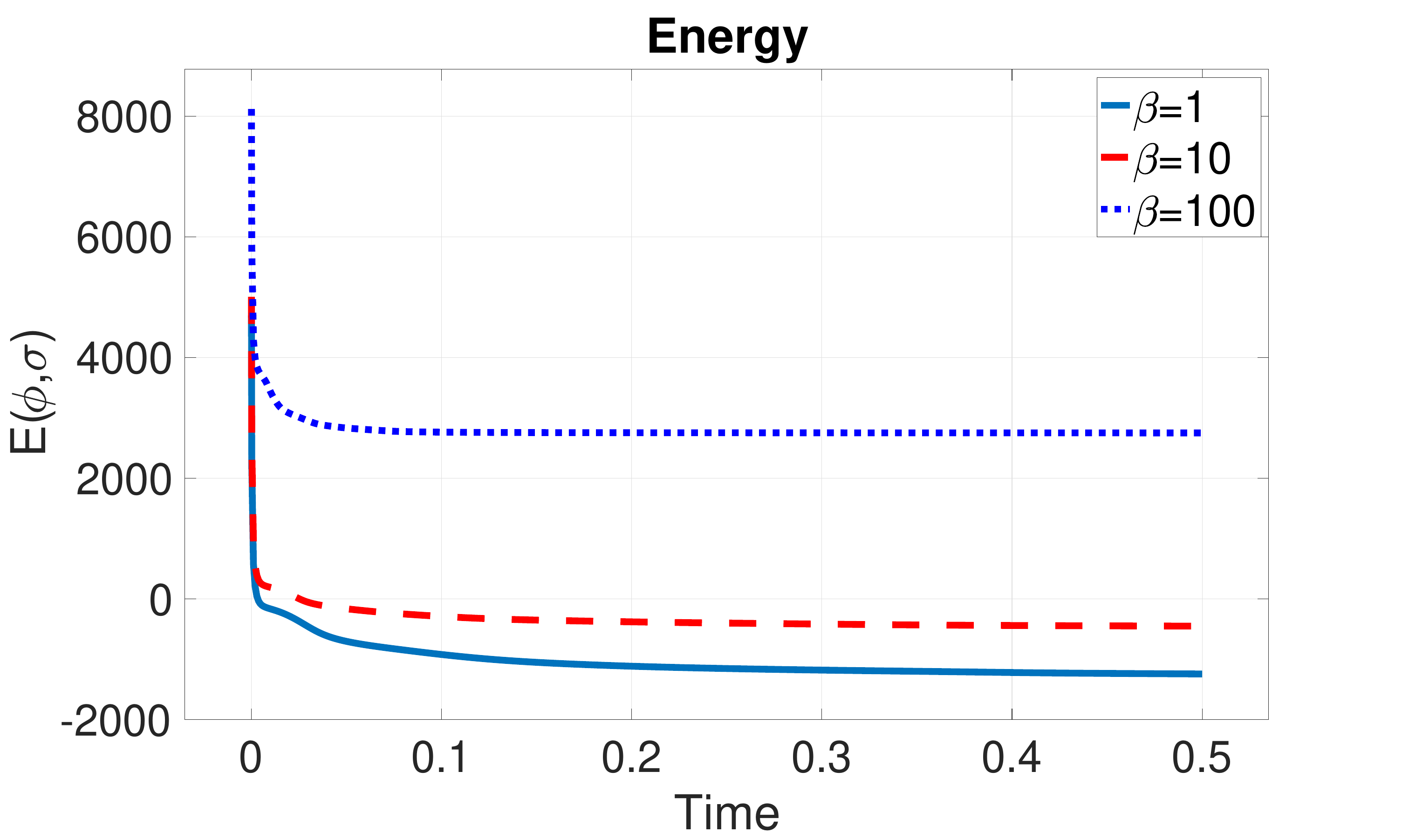}
			\includegraphics[width=0.45\textwidth]{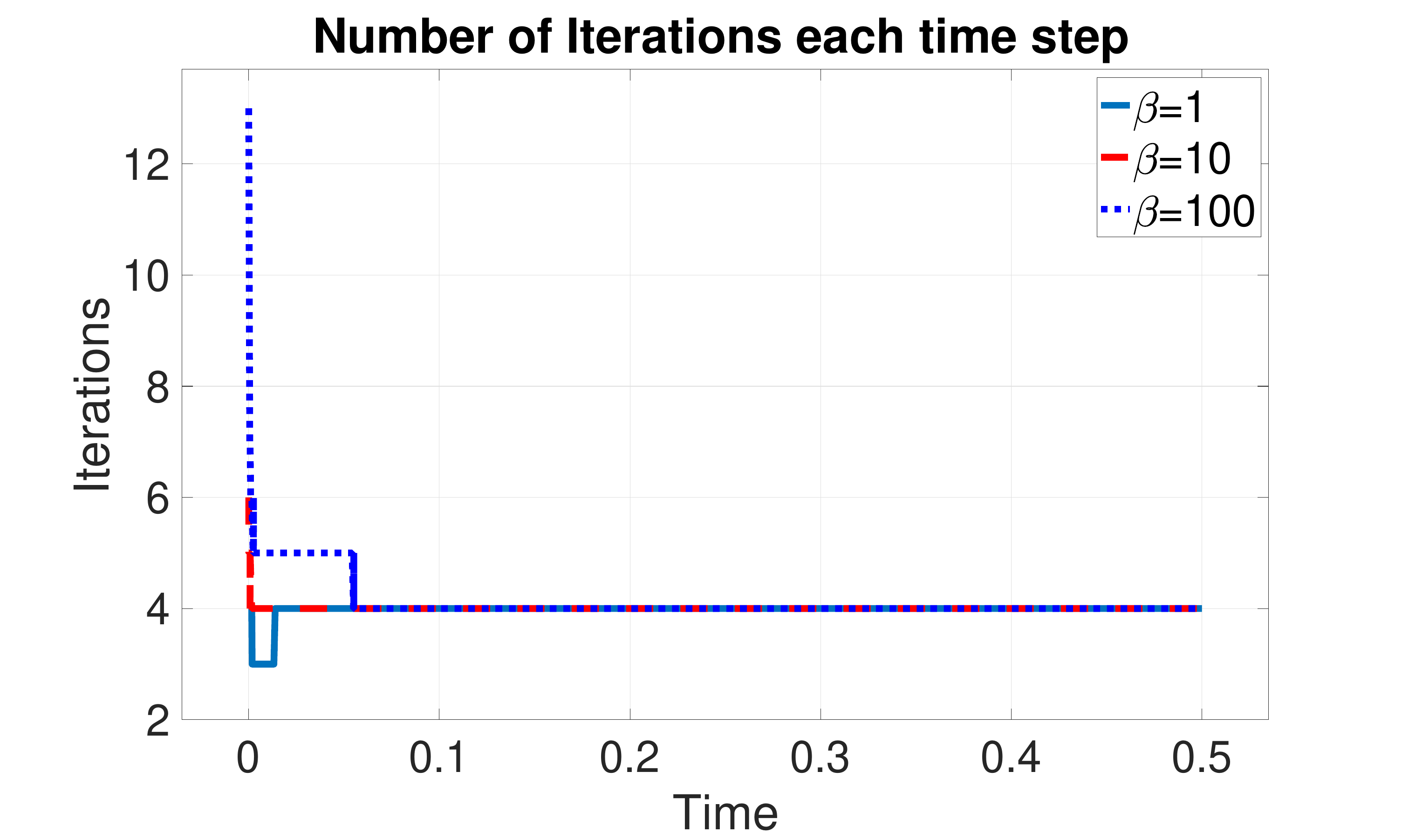}
			\\
			\includegraphics[width=0.45\textwidth]{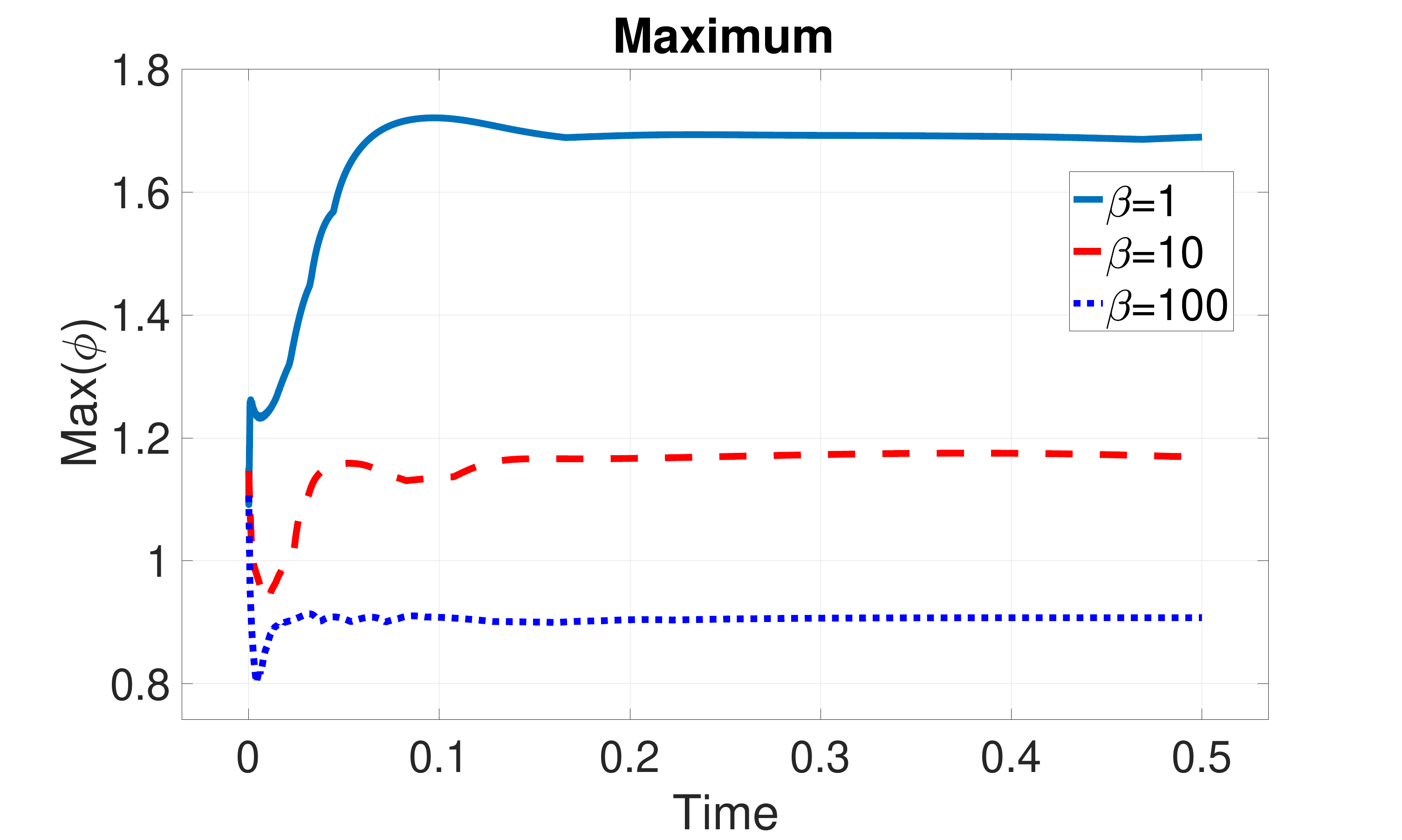}
			\includegraphics[width=0.45\textwidth]{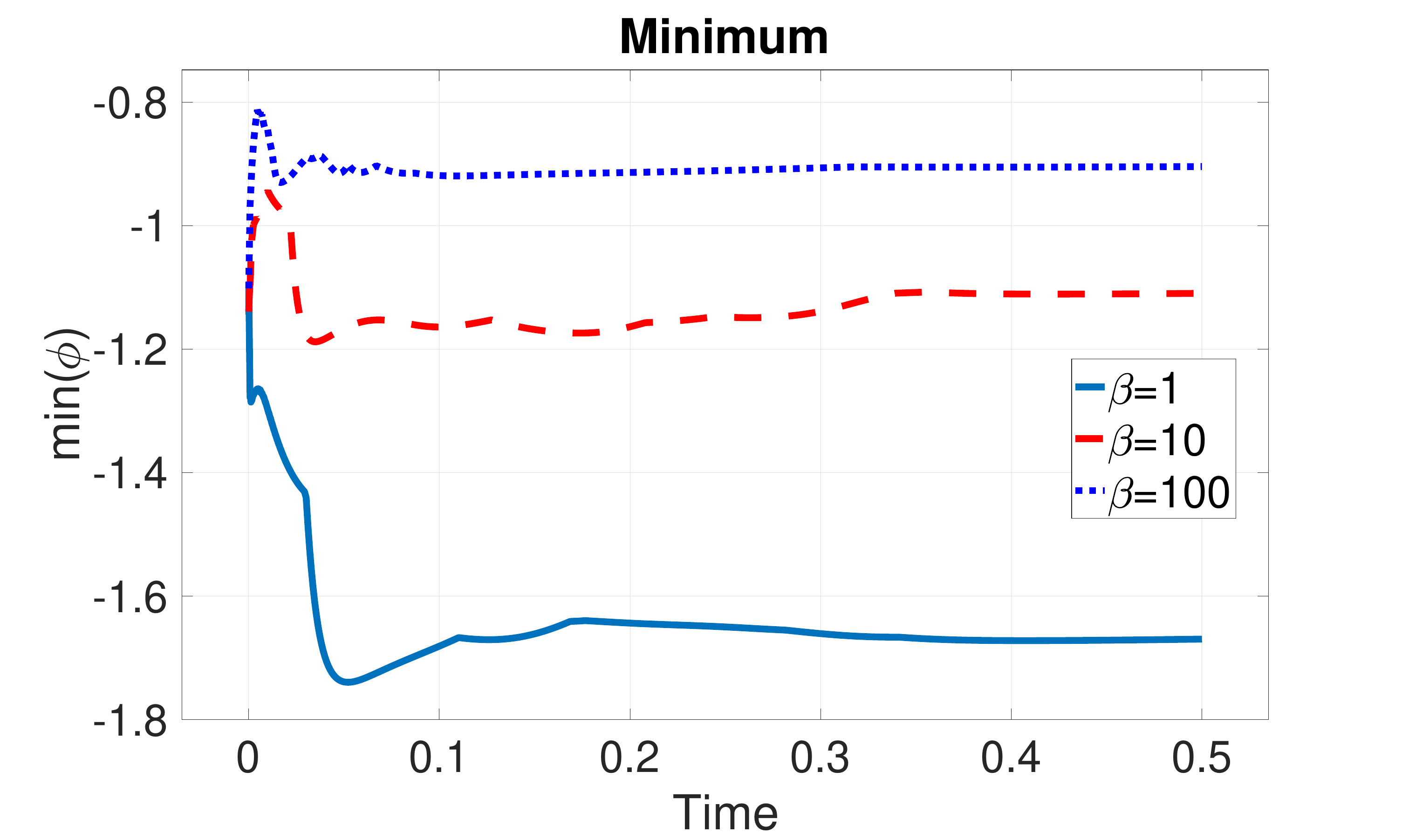}
			\caption{Evolution in time of the energies (top left), the number of iterations to achieve tolerance $\texttt{TOL}=10^{-7}$ (top right), maximum of $\phi$ (bottom left) and minimum of $\phi$ (bottom right) taking $M=0.1$, $g_2=1$, $g_0=-4$, $h_0=0.5$, $\lambda=0.1$ and $\beta=1, 10$ and $10^2$.} \label{fig:betasplot}
		\end{center}
	\end{figure}
	
	\subsubsection{Study of influence of parameter $g_2$ on the dynamics of the system}
	In this example, we fix the parameters to
	$M=0.1$, $g_0=-4$, $h_0=0.5$, $\lambda=0.1$, $\beta=1$ and  we consider different values of parameter $g_2$ {in the definition of the function $g(\phi)$ in \eqref{eq:deffunctiong}} { which impacts the amphiphilic presence}, namely $g_2=1$, $g_2=10$ and $g_2=10^2$. The dynamics associated to these simulations are presented in Figure~\ref{fig:g2s} and the evolution of the energies, maximum of $\phi$, minimum of $\phi$ and number of iterations are presented in Figure~\ref{fig:g2splot}. 
	The impact of varying $g_2$ in the phase field-dependent coefficient $g(\phi)=g_2\phi^2 + g_0$ can be explained as follows. As the coefficient $g_2$ is increased, the dynamics of the ternary system leans more towards interactions between the phases rather than separating into the bulk regions. Thus, the influence of increasing the magnitude of $g_2$ is reflected in larger interacting/mixing regions.  This behavior is correctly captured by the approximations generated from our numerical scheme as seen in the bottom row subfigures of Figure~\ref{fig:g2s}. Furthermore, based on the experimental observations discussed in~\cite{PawlowZajaczkowski11}, it seems that $g(\phi)$ must be positive in the oil-rich and water-rich phases, \textit{so this would imply} the condition:~$g_2>-g_0$.
	%(We ought to discuss how to proceed since this is violated in my advisor's paper and in Amanda and my paper. I would like to add more text.)
	We would like to point out that unlike previous numerical simulations in~\cite{diegelsharma-me23,hoppelinsenmann-me18}, we present simulations in both cases, where this condition holds and when it is violated.
	%			 where this condition is violated and when this condition holds and report on its impact.}
%}

\begin{figure}[h]
\begin{center}
	\includegraphics[width=0.19\textwidth]{images/Ex2/standard/Ex2_Standard_0}
	\includegraphics[width=0.19\textwidth]{images/Ex2/standard/Ex2_Standard_50}
	\includegraphics[width=0.19\textwidth]{images/Ex2/standard/Ex2_Standard_150}
	\includegraphics[width=0.19\textwidth]{images/Ex2/standard/Ex2_Standard_350}
	\includegraphics[width=0.19\textwidth]{images/Ex2/standard/Ex2_Standard_500}
	\\
	\includegraphics[width=0.19\textwidth]{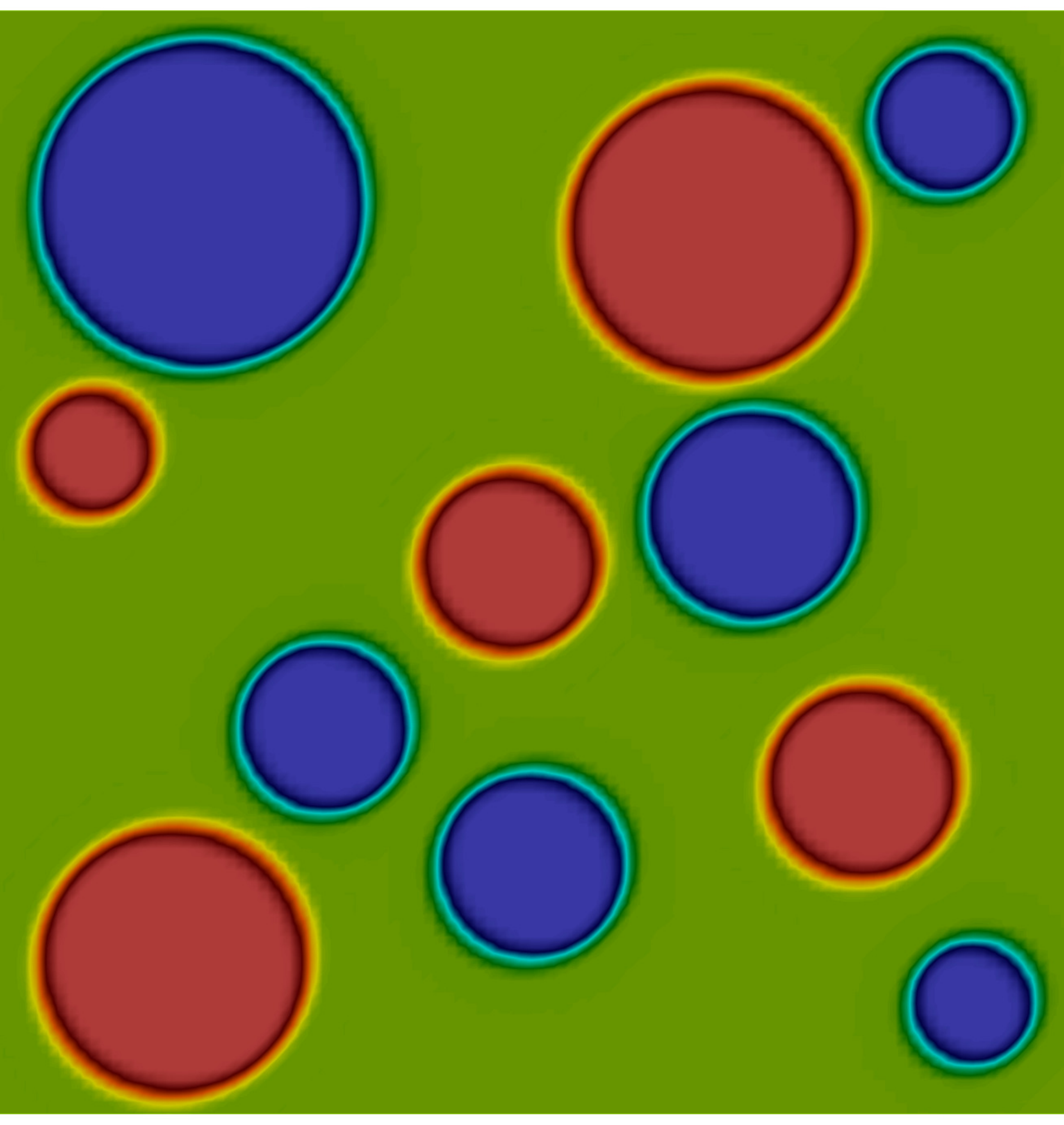}
	\includegraphics[width=0.19\textwidth]{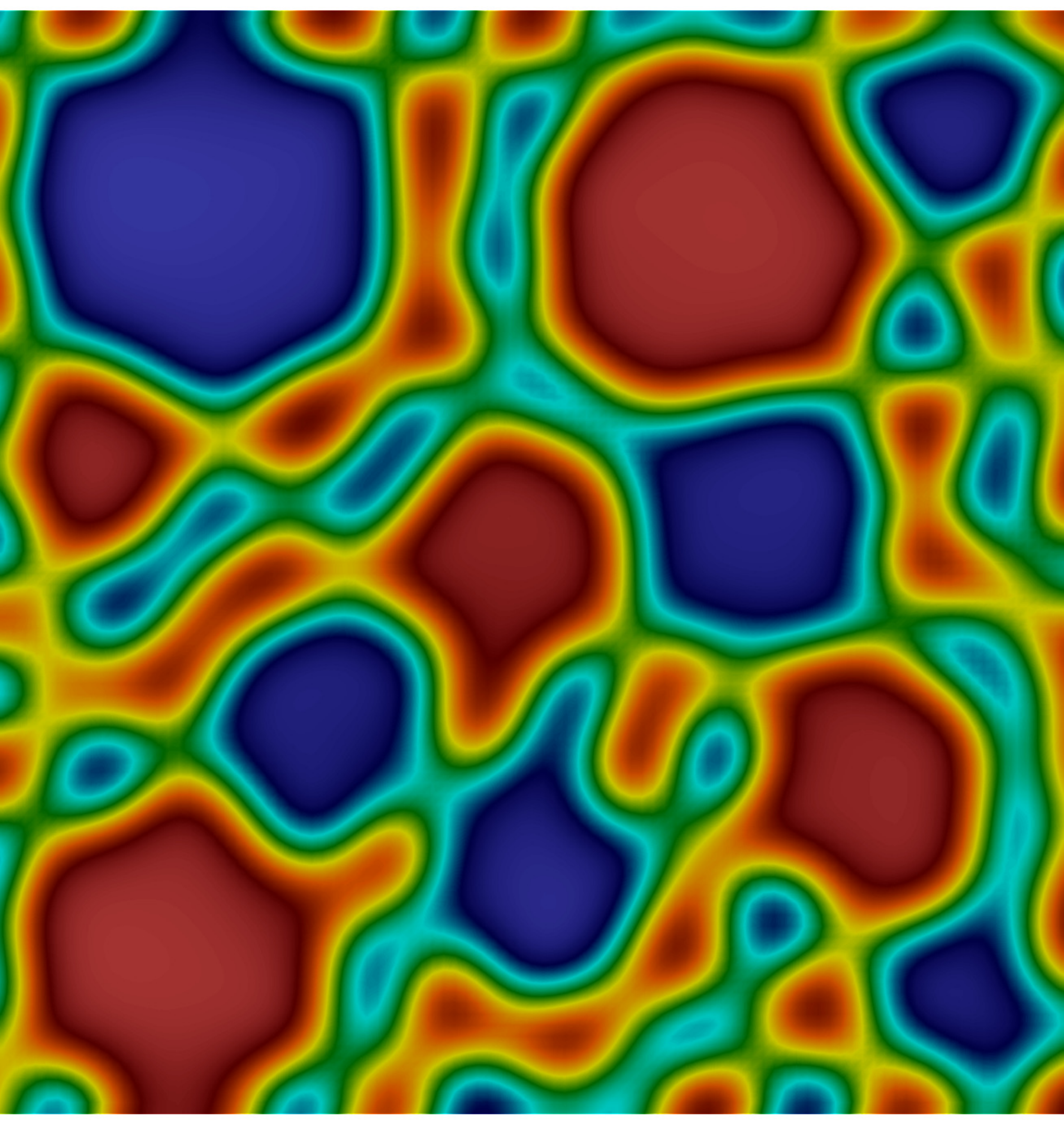}
	\includegraphics[width=0.19\textwidth]{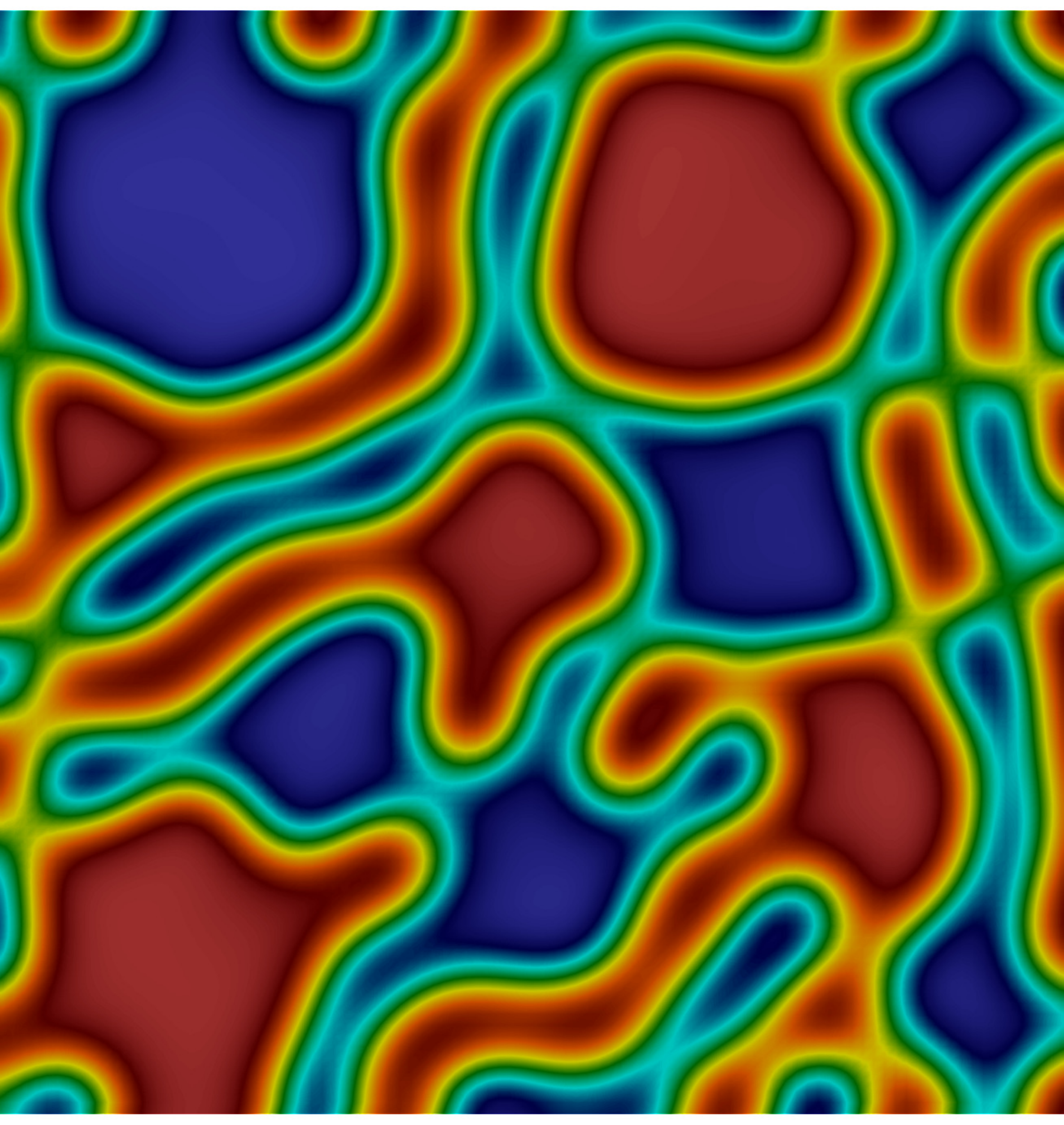}
	\includegraphics[width=0.19\textwidth]{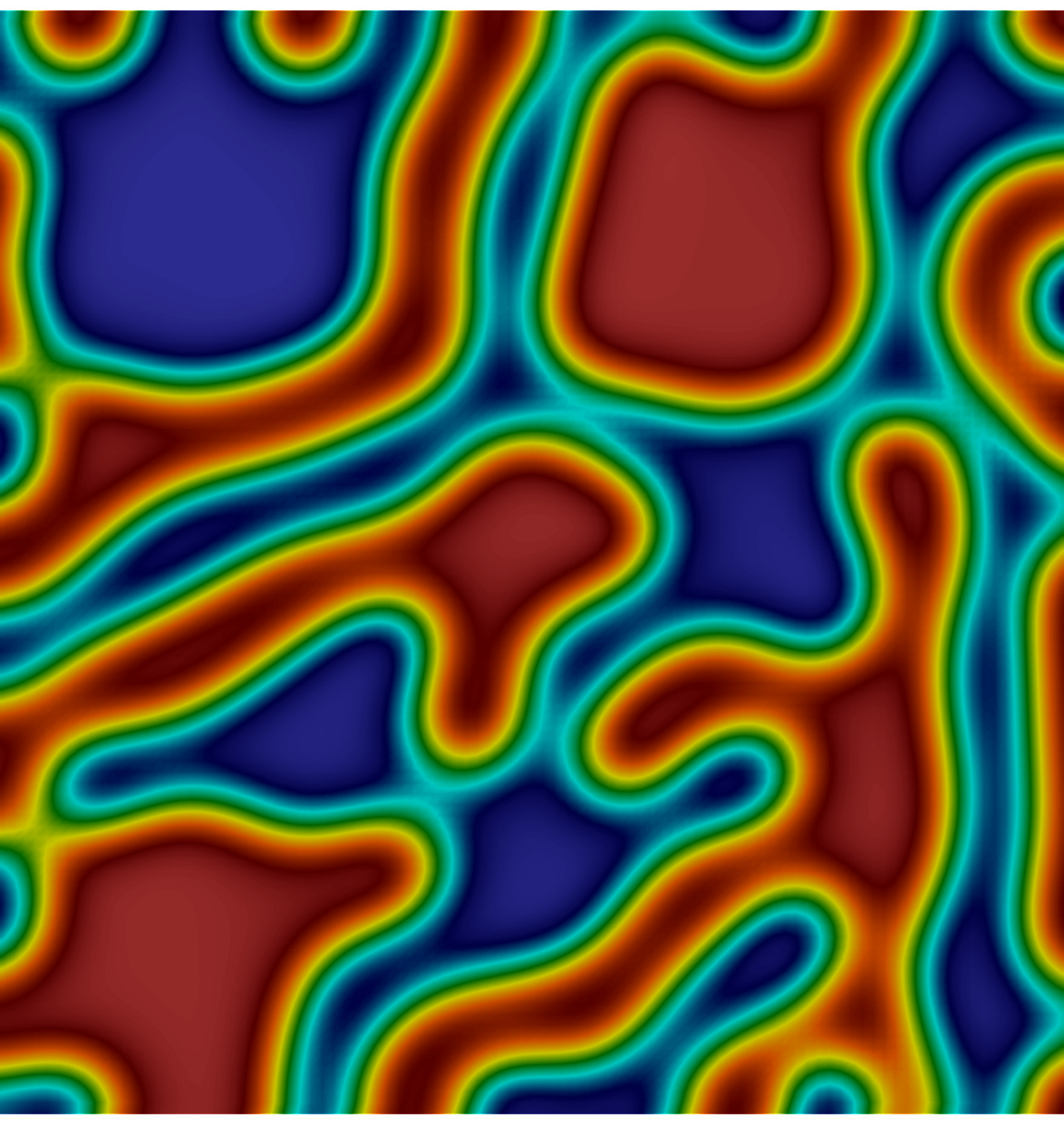}
	\includegraphics[width=0.19\textwidth]{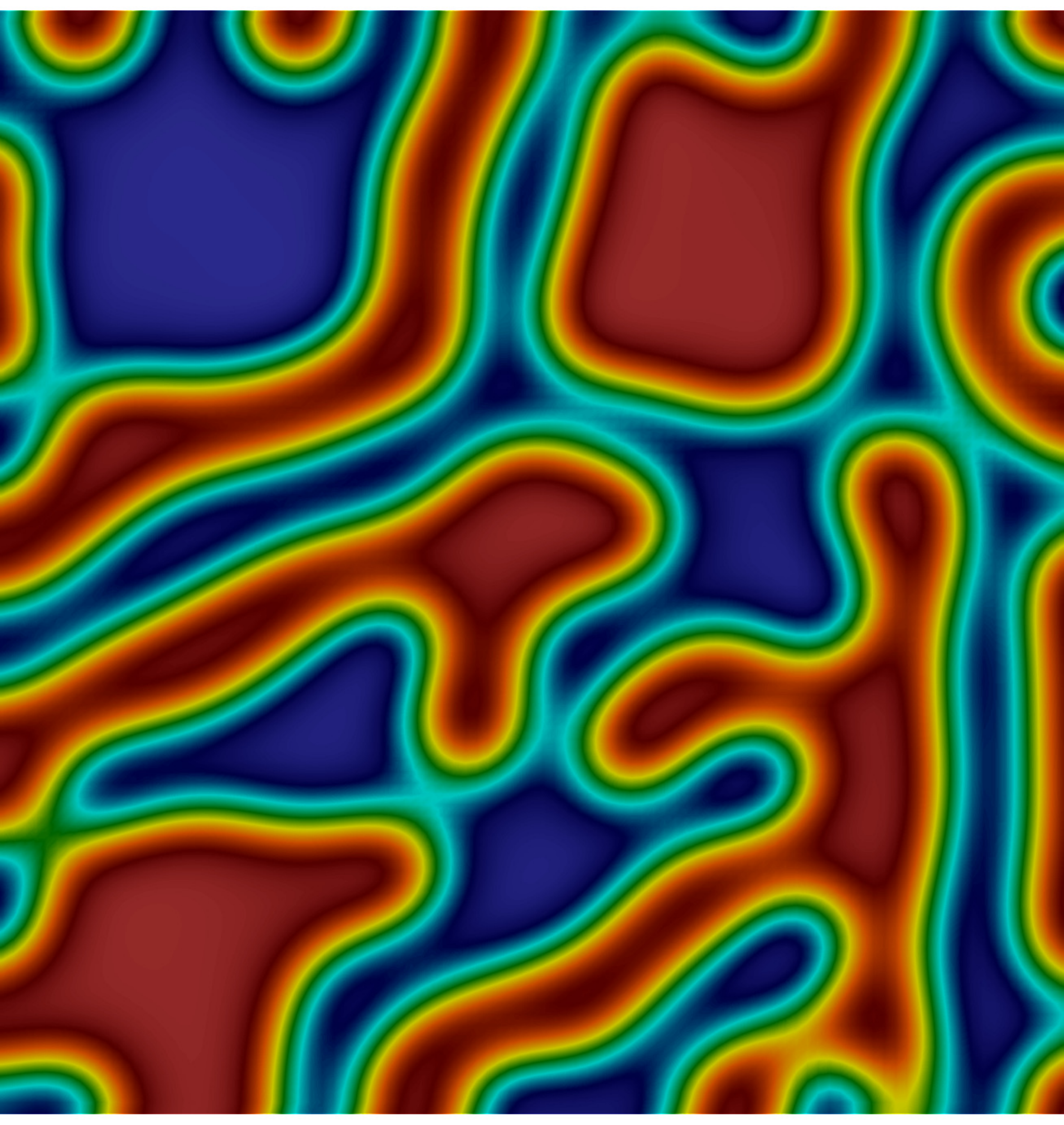}
	\\
	\includegraphics[width=0.19\textwidth]{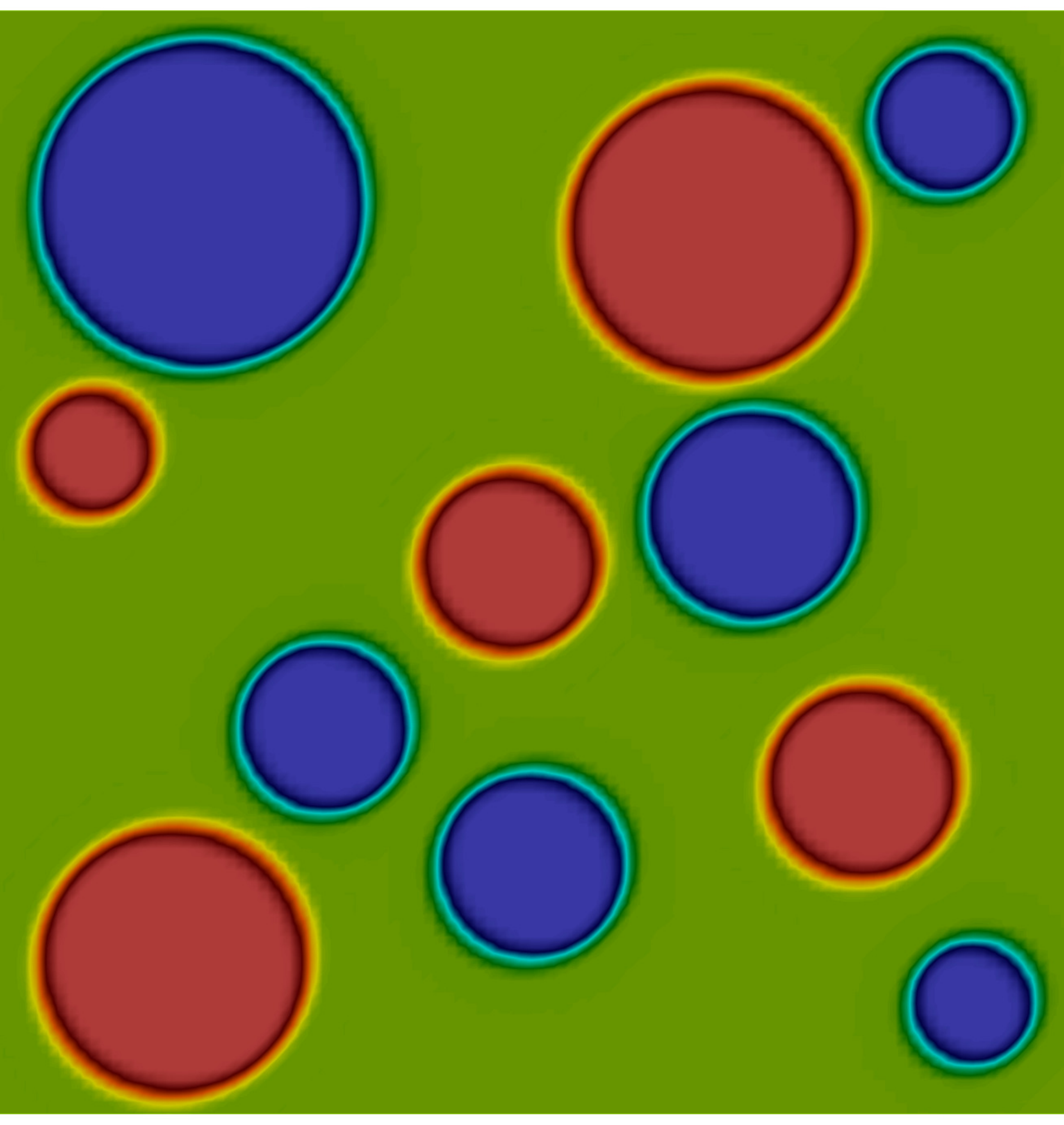}
	\includegraphics[width=0.19\textwidth]{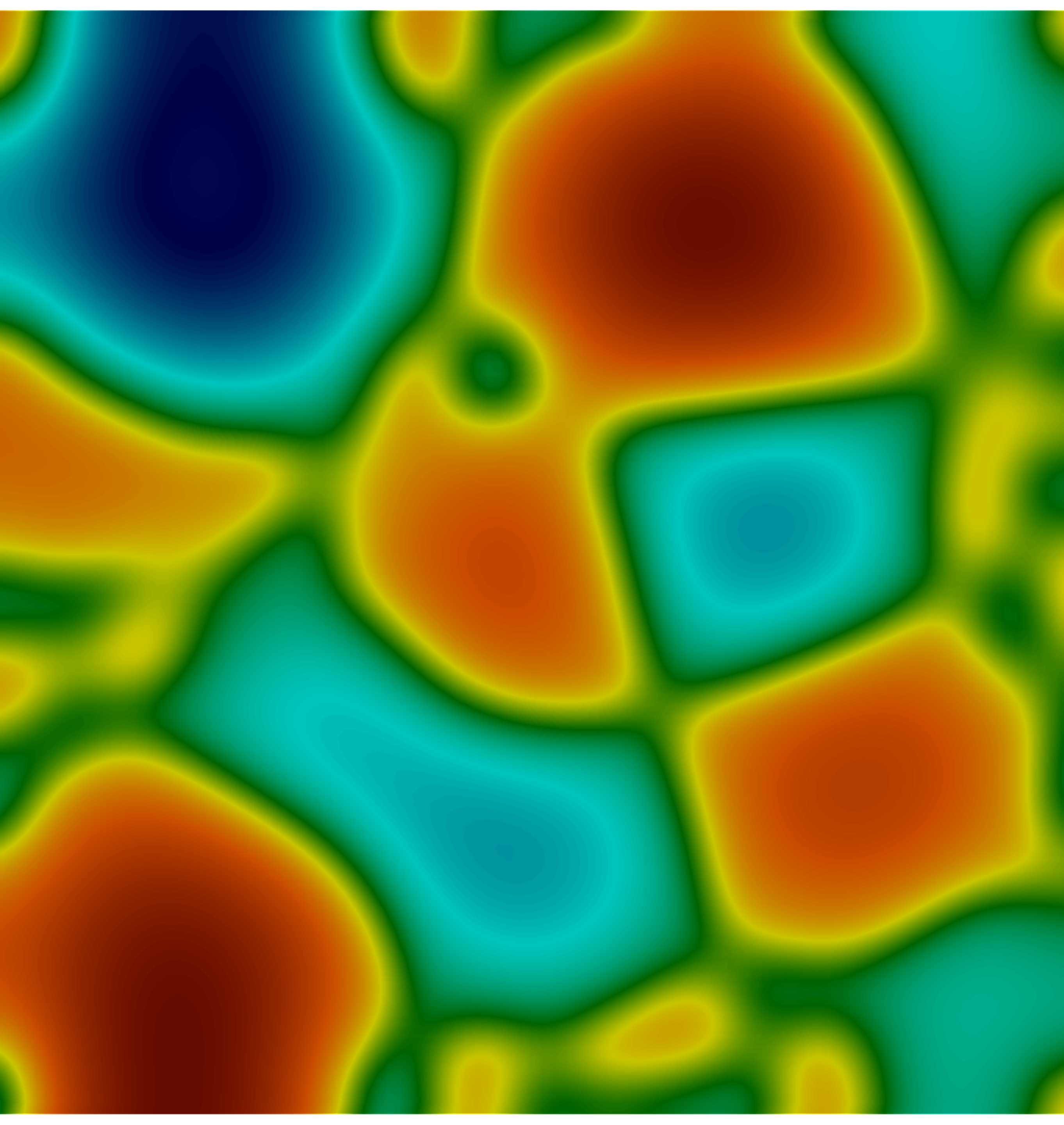}
	\includegraphics[width=0.19\textwidth]{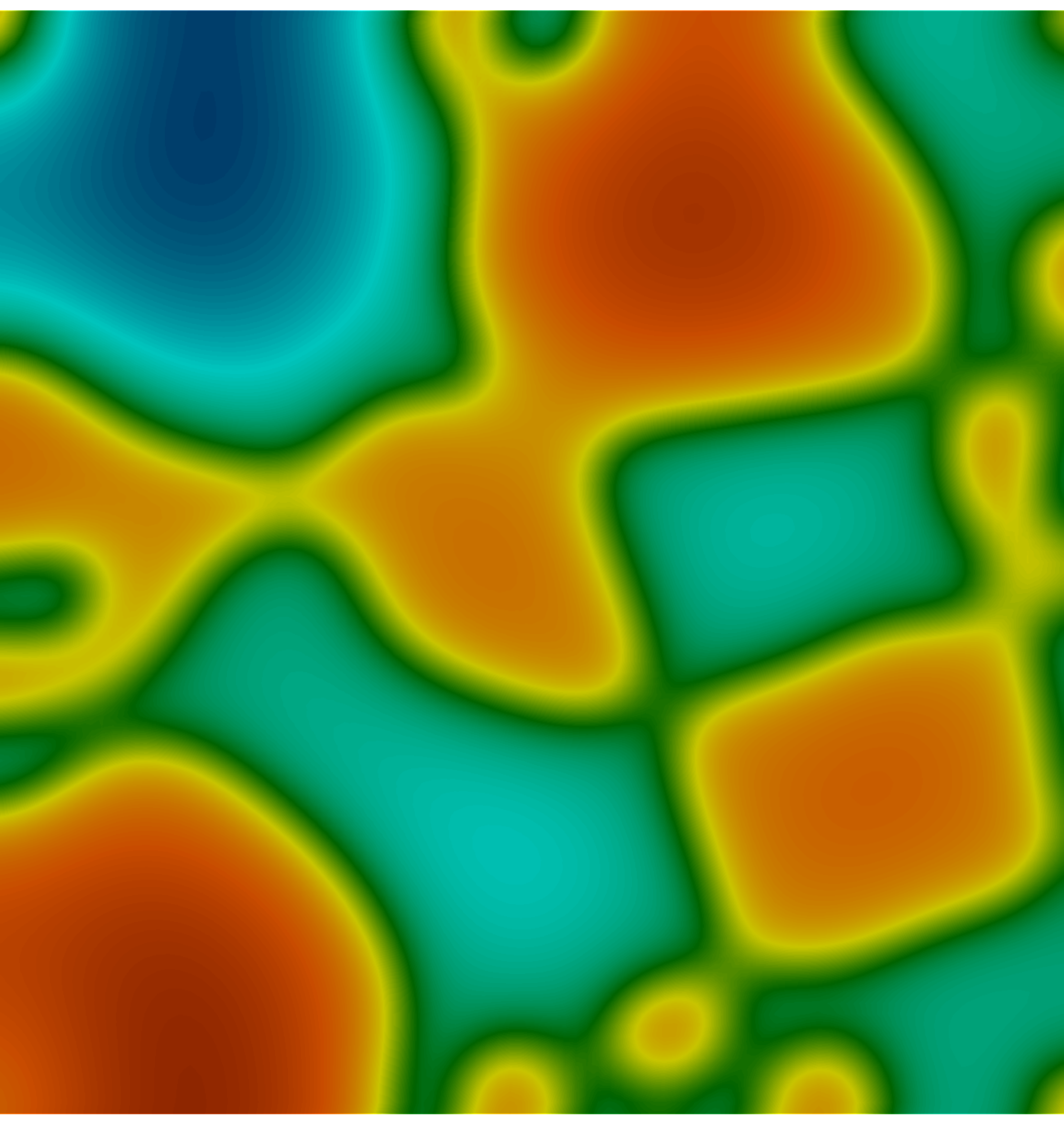}
	\includegraphics[width=0.19\textwidth]{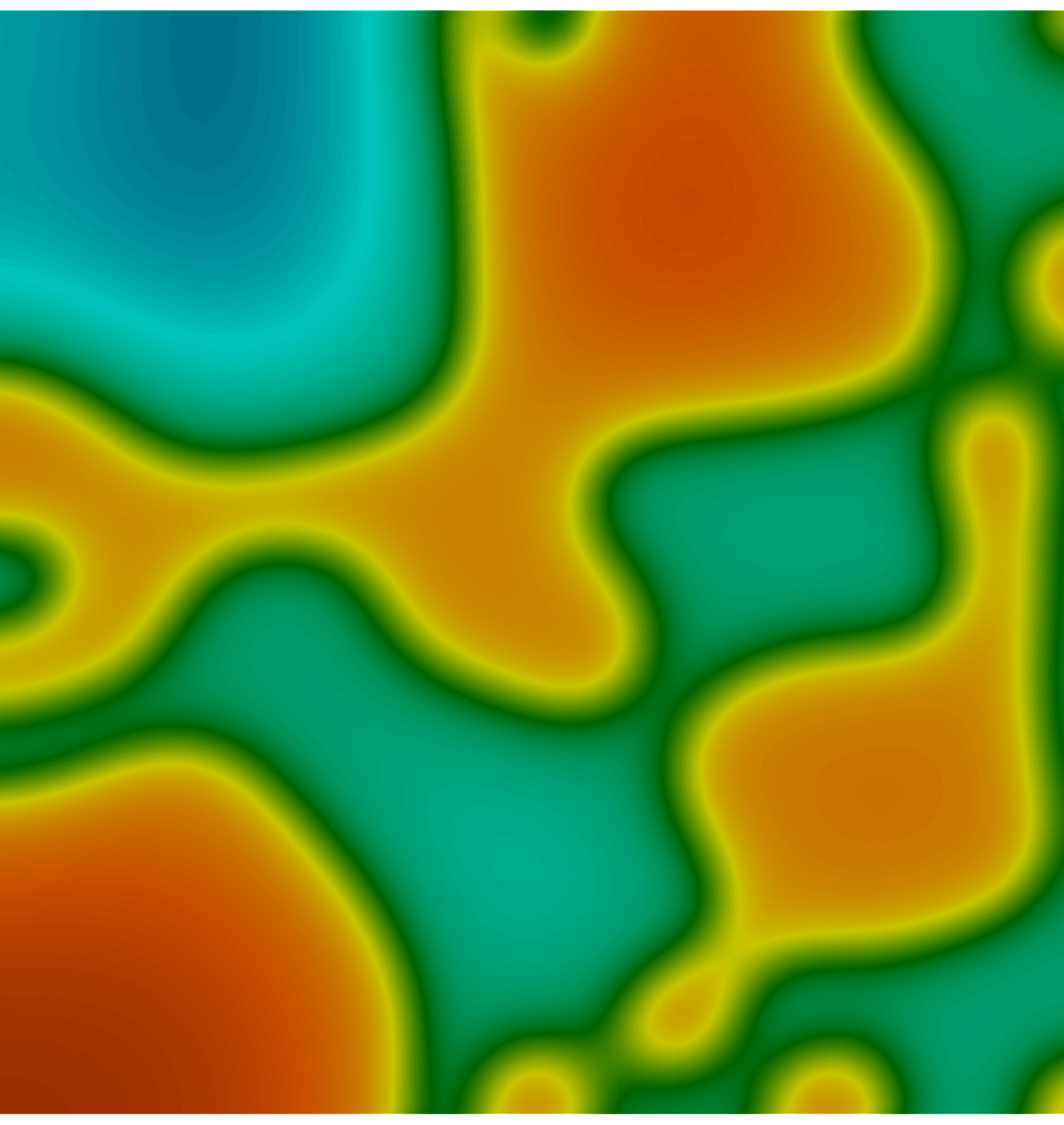}
	\includegraphics[width=0.19\textwidth]{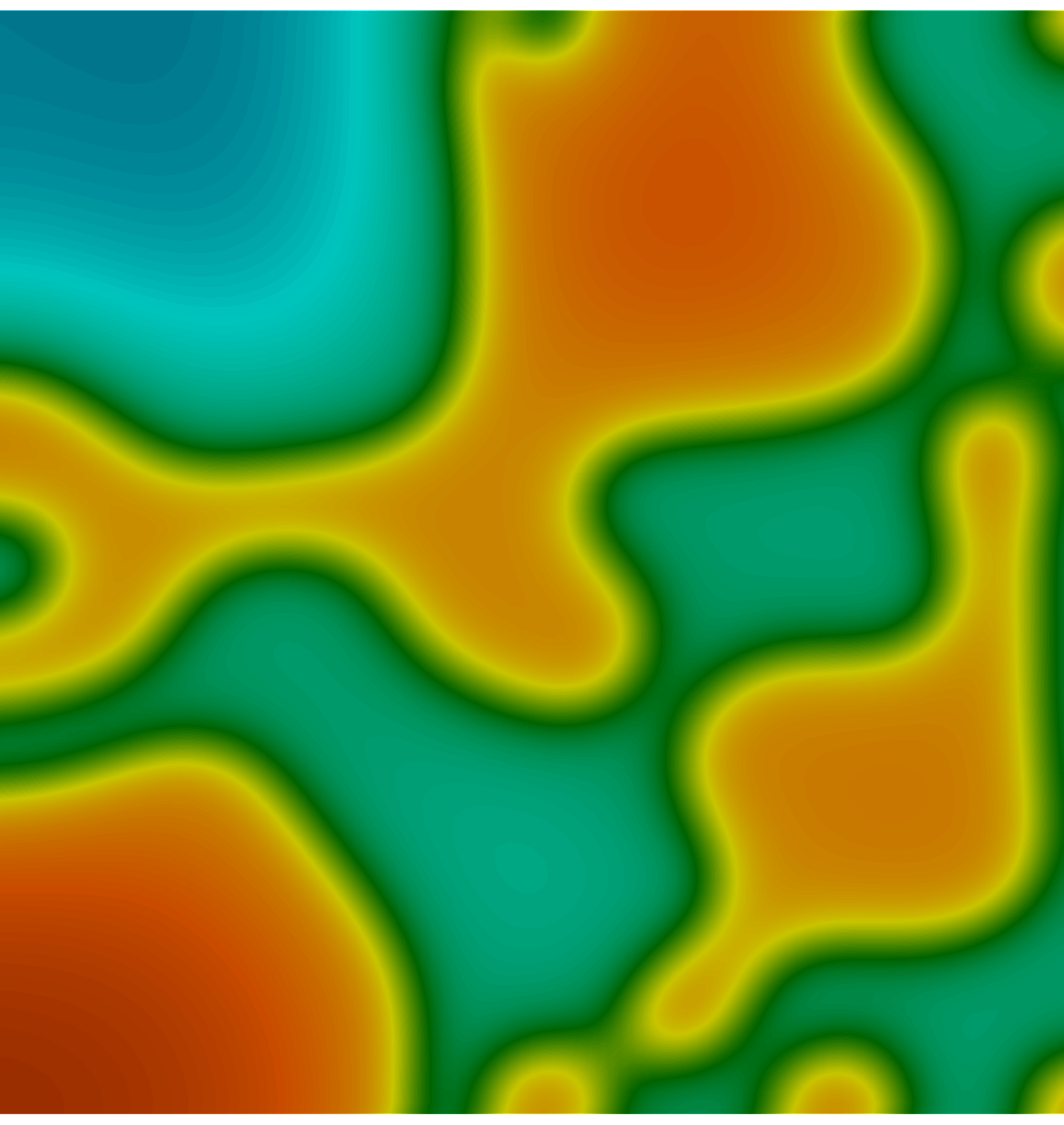}
	\caption{Evolution of $\phi$ at times $t=0, 0.05, 0.15, 0.35$ and $0.5$ (from left to right) taking $M=0.1$, $g_0=-4$, $h_0=0.5$, $\lambda=0.1$ and $\beta=1$ with $g_2=1$ (top row), $g_2=10$ (center row) and $g_2=10^2$ (bottom row).} \label{fig:g2s}
\end{center}
\end{figure}

\begin{figure}[h]
\begin{center}
	\includegraphics[width=0.45\textwidth]{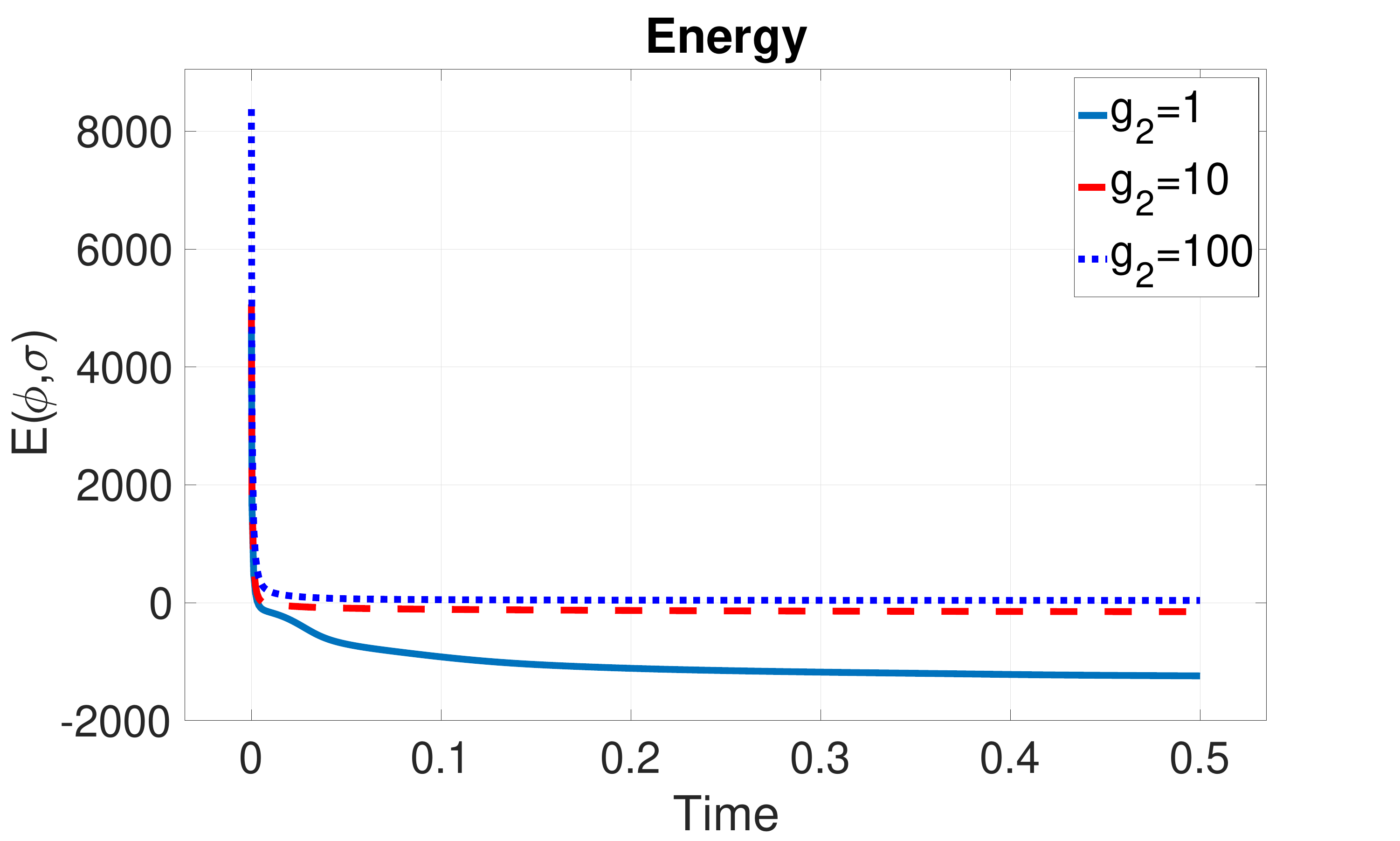}
	\includegraphics[width=0.45\textwidth]{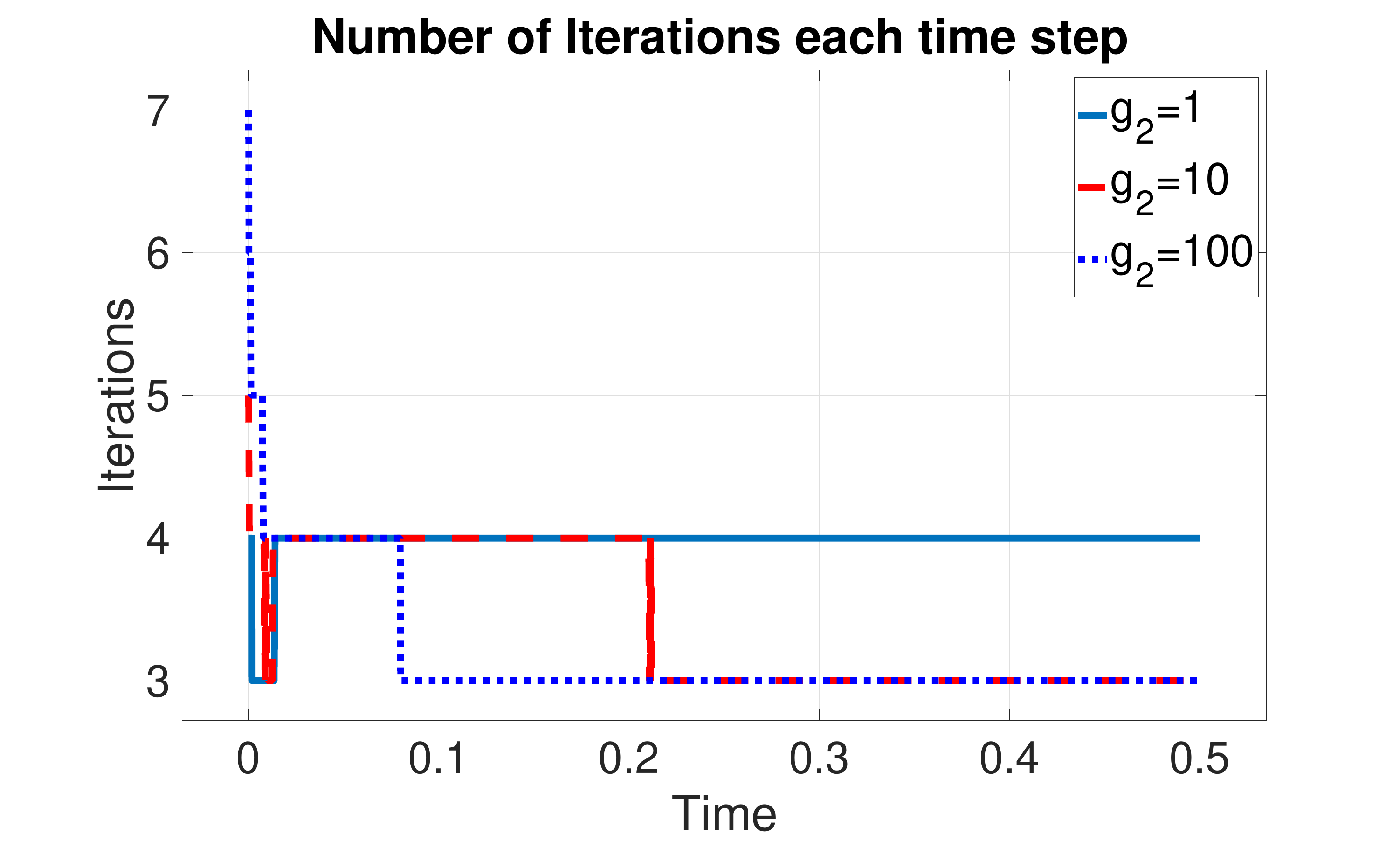}
	\\
	\includegraphics[width=0.45\textwidth]{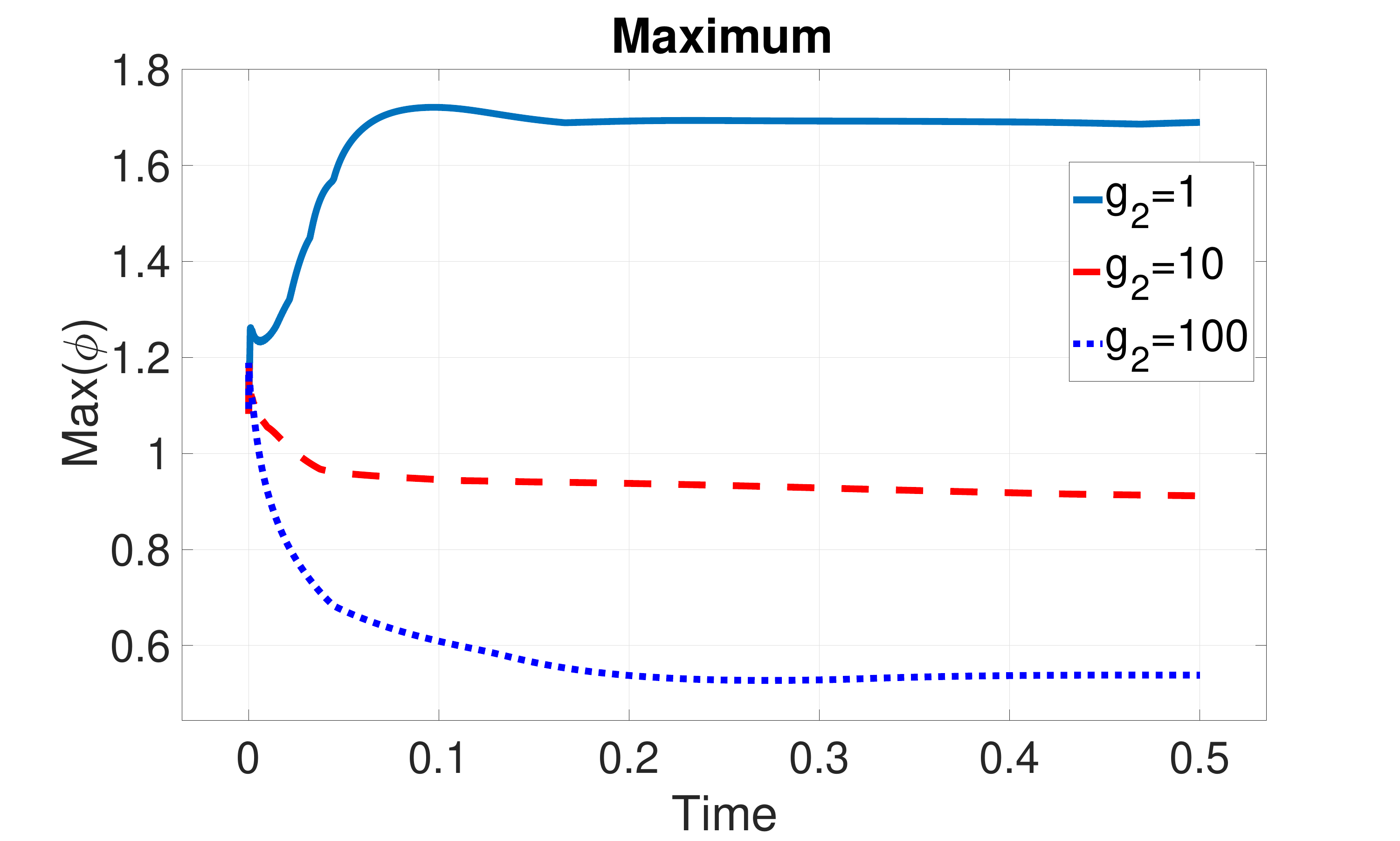}
	\includegraphics[width=0.45\textwidth]{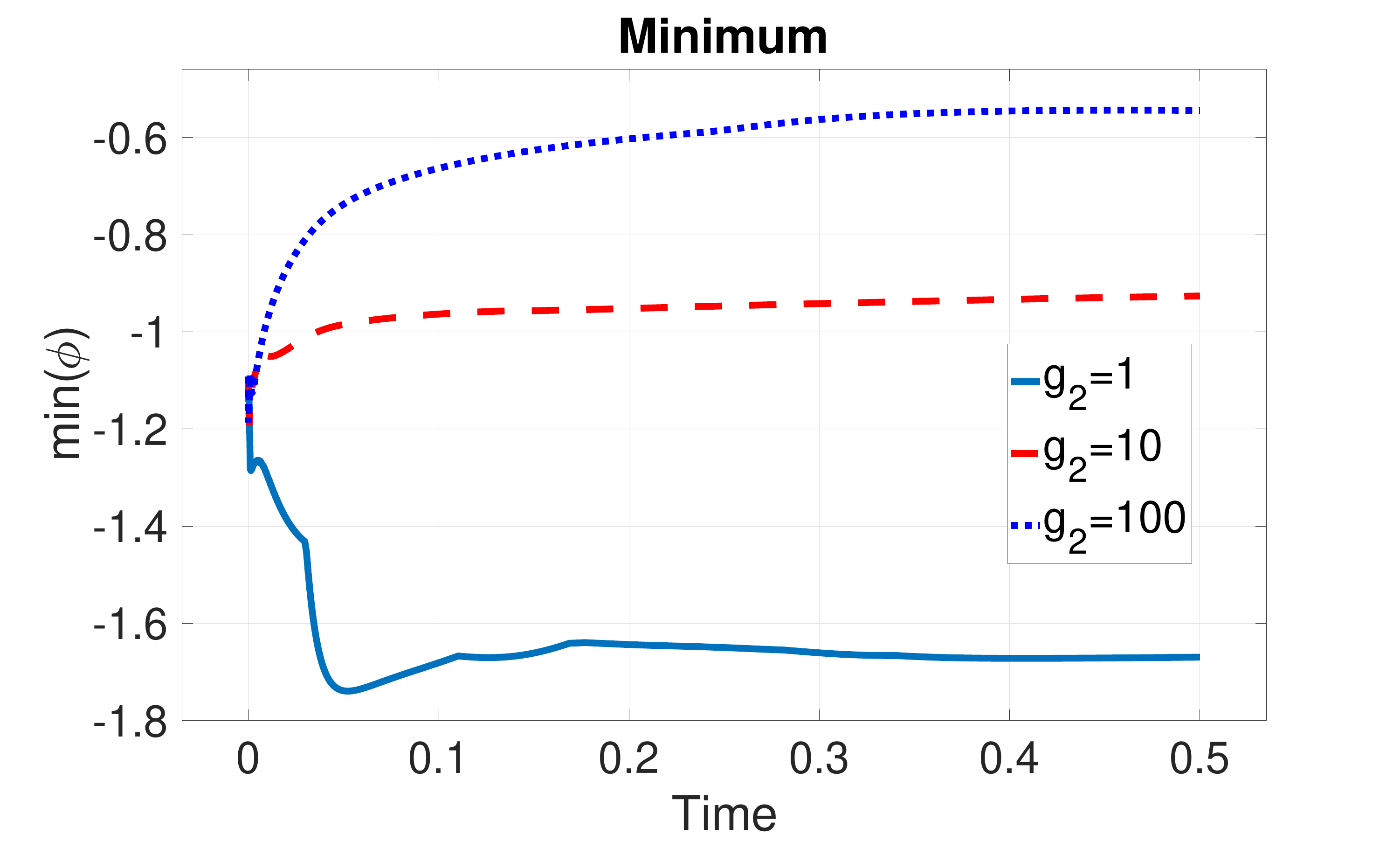}
	\caption{Evolution in time of the energies (top left), the number of iterations to achieve tolerance $\texttt{TOL}=10^{-7}$ (top right), maximum of $\phi$ (bottom left) and minimum of $\phi$ (bottom right) taking $M=0.1$, $g_0=-4$, $h_0=0.5$, $\lambda=0.1$ and $\beta=1$ with $g_2=1, 10$ and $g_2=10^2$.} \label{fig:g2splot}
\end{center}
\end{figure}

\subsubsection{Study of influence of parameter $g_0$ on the dynamics of the system}
In this example we fix the parameters to
$M=0.1$, $g_2=1$, $h_0=0.5$, $\lambda=0.1$, $\beta=1$ and  we consider different values of parameter $g_0$ {in the definition of the function $g(\phi)$ in \eqref{eq:deffunctiong}}, namely $g_0=-4$, $g_0=-10$ and $g_0=-25$. The dynamics associated to these simulations are presented in Figure~\ref{fig:g0s} and the evolution of the energies, maximum of $\phi$, minimum of $\phi$ and number of iterations are presented in Figure~\ref{fig:g0splot}. 
%Parameter $g_0$ is the constant coefficient in $g(\phi)$ which multiply $|\nabla\phi|^2$ in the energy $E(\phi,\sigma)$ and it is not {\underline{completely clear a priori what is its role in the dynamics.}} 
The obtained results suggest that decreasing $g_0$ leads to thinner interfaces and at the same time induces the dynamics of the system to increase the amount of interface in the domain. Larger number of interfaces indicate increased stability of the microemulsion phase~\cite{PawlowZajaczkowski11}. Moreover, 
the values of $\phi$ gets away from the minima of the functional $f_0(\phi)$ as $g_0$ is decreased. In all the simulations, the energies decreases as expected and in accordance with Lemma~\ref{lem:iterativealgorithm}, there is a strong sensitivity on the value of $g_0$ and the number of iterations needed to achieve the required tolerance in the iterative algorithm. As expected, an increase in the value $|g_0|$ leads to an increase in the number of iterations needed for the iterative algorithm to converge.~While the model and the numerical scheme is valid for any arbitrary choice of $g_0$ (regardless of sign), the physically relevant choices of $g_0$ are $g_0<0$. This is because theory dictates that in the microemulsion phase, at a high amphiphile concentration $g(\phi)$ becomes negative (see~\cite{PawlowZajaczkowski11} and the references therein). Furthermore, negative values of $g_0$ creates more interfaces~\cite{GompperSchick94} as is numerically validated in Figure~\ref{fig:g0s}. 
{Finally we consider the case $g_0=0$ in Figure~\ref{fig:g0_0}, and we observe how the dynamics indicate that the negative value of this parameter is important for creating an interphase between the phases.}
%		\textcolor{red}{(NSS:~I daresay the next sentence suggests more stability of the microemulsions as the number of interfaces increases. Again, something to discuss.)We observe how the dynamics indicate that the negative value of this parameter is important for creating an interphase between the phases.}
%}
\begin{figure}[h]
\begin{center}
\includegraphics[width=0.19\textwidth]{images/Ex2/standard/Ex2_Standard_0}
\includegraphics[width=0.19\textwidth]{images/Ex2/standard/Ex2_Standard_50}
\includegraphics[width=0.19\textwidth]{images/Ex2/standard/Ex2_Standard_150}
\includegraphics[width=0.19\textwidth]{images/Ex2/standard/Ex2_Standard_350}
\includegraphics[width=0.19\textwidth]{images/Ex2/standard/Ex2_Standard_500}
\\
\includegraphics[width=0.19\textwidth]{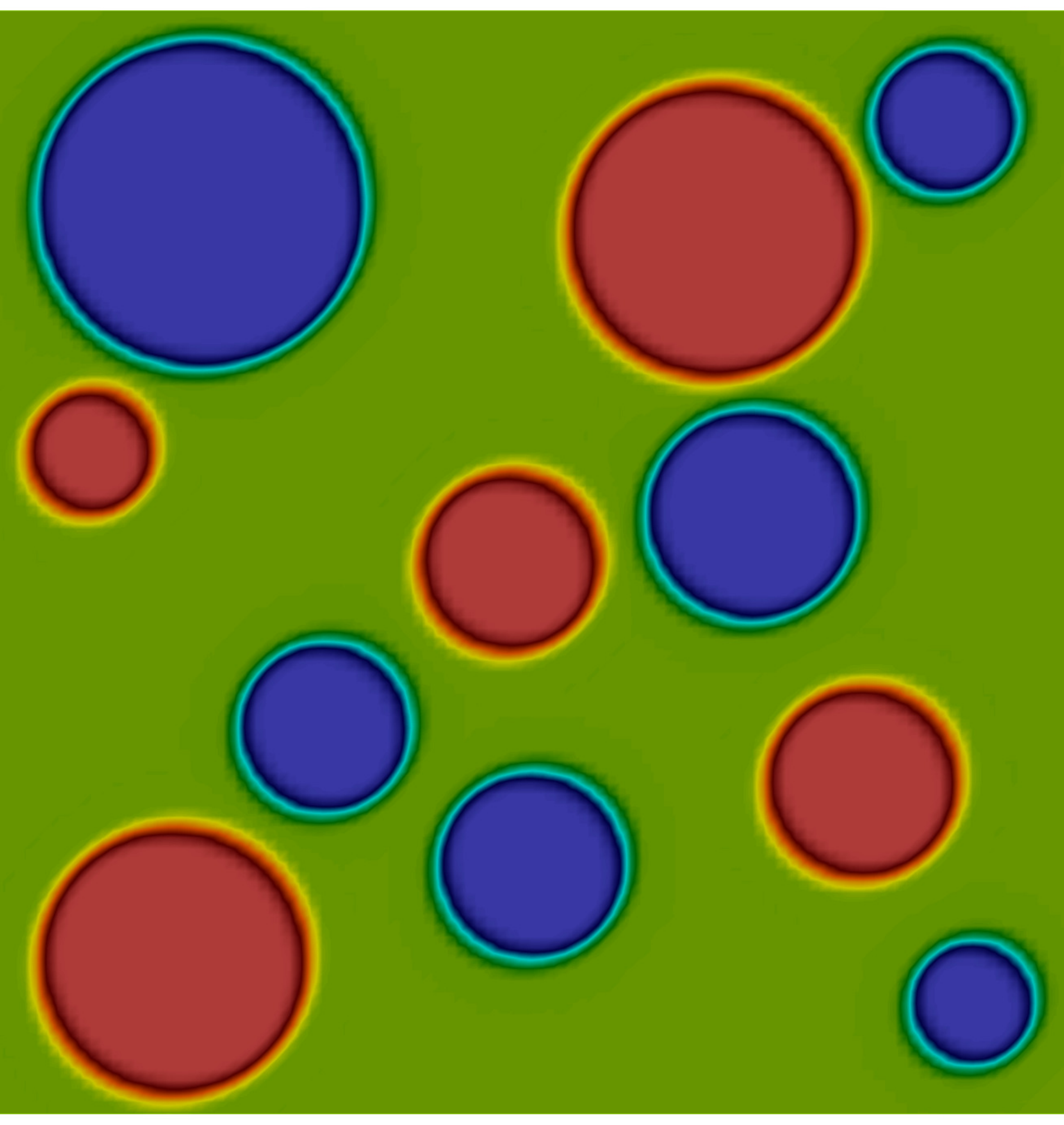}
\includegraphics[width=0.19\textwidth]{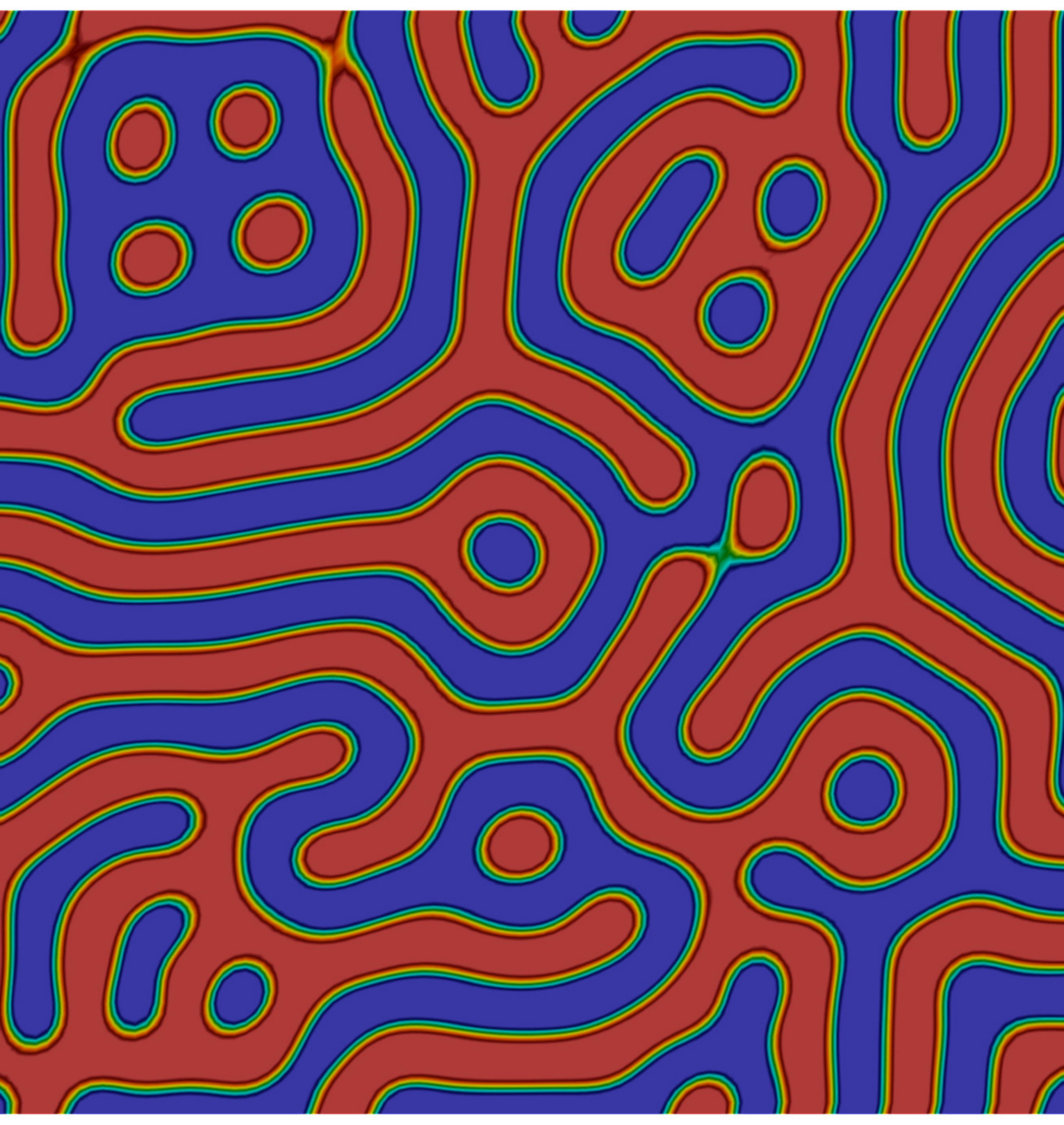}
\includegraphics[width=0.19\textwidth]{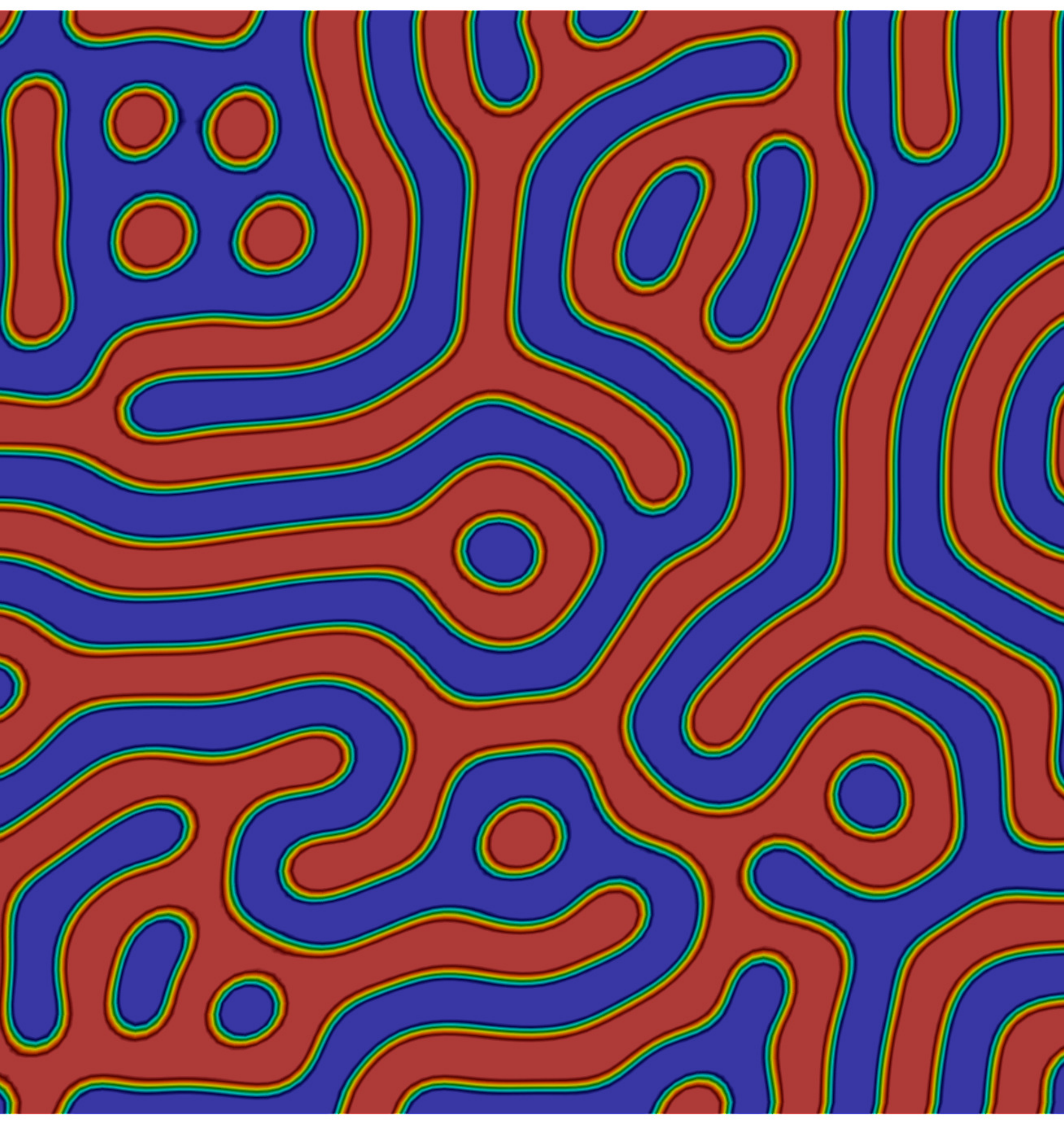}
\includegraphics[width=0.19\textwidth]{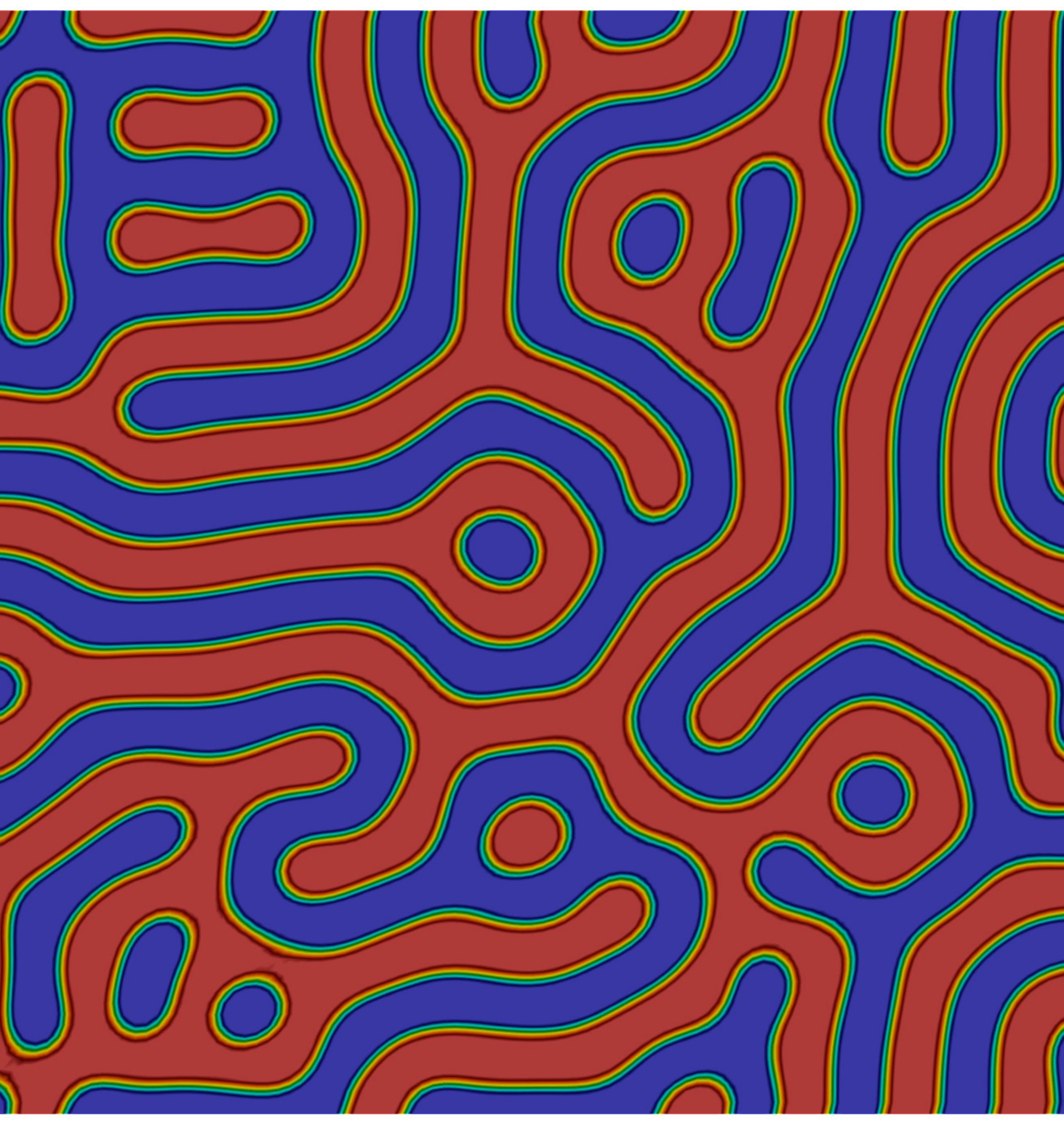}
\includegraphics[width=0.19\textwidth]{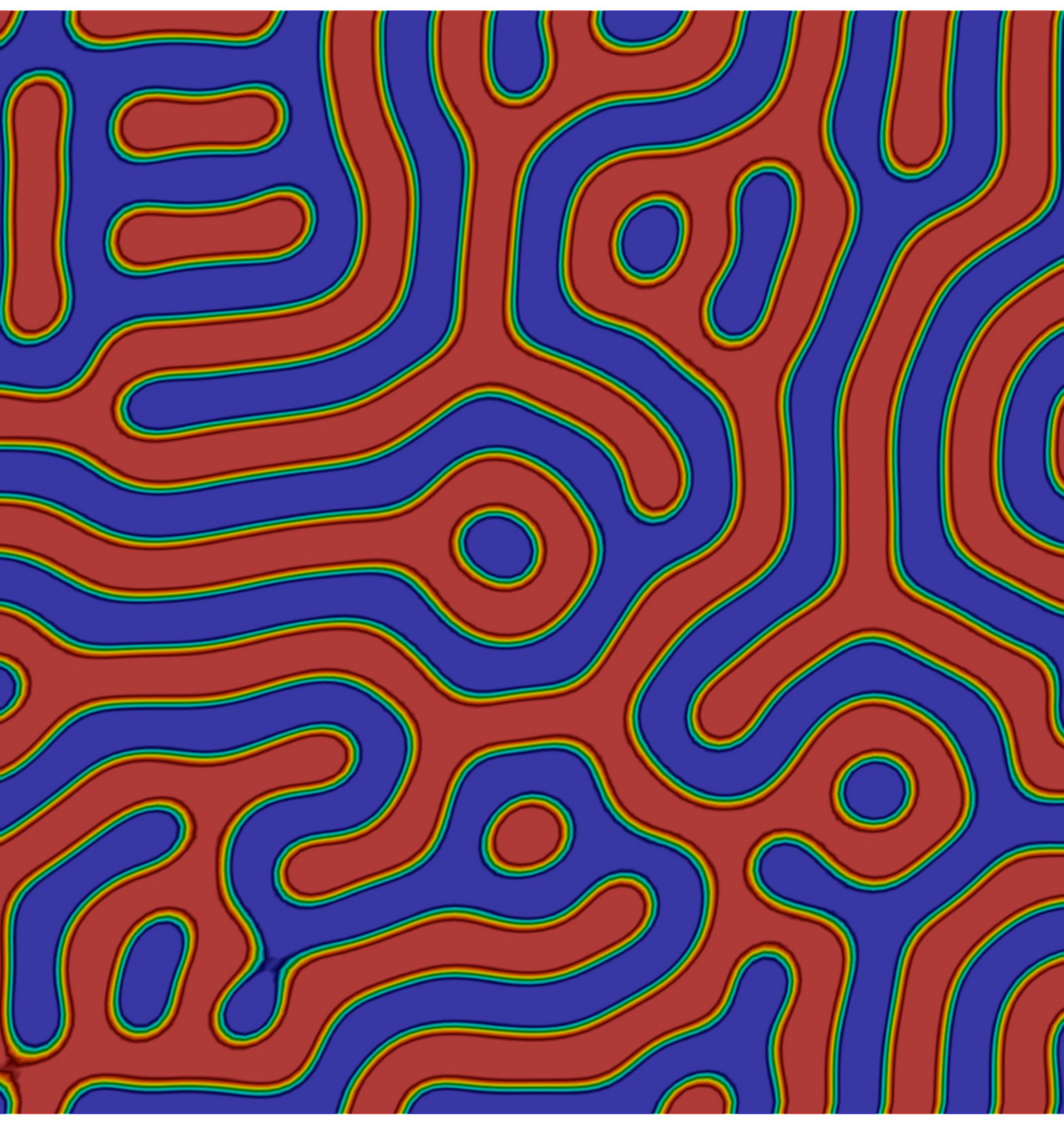}
\\
\includegraphics[width=0.19\textwidth]{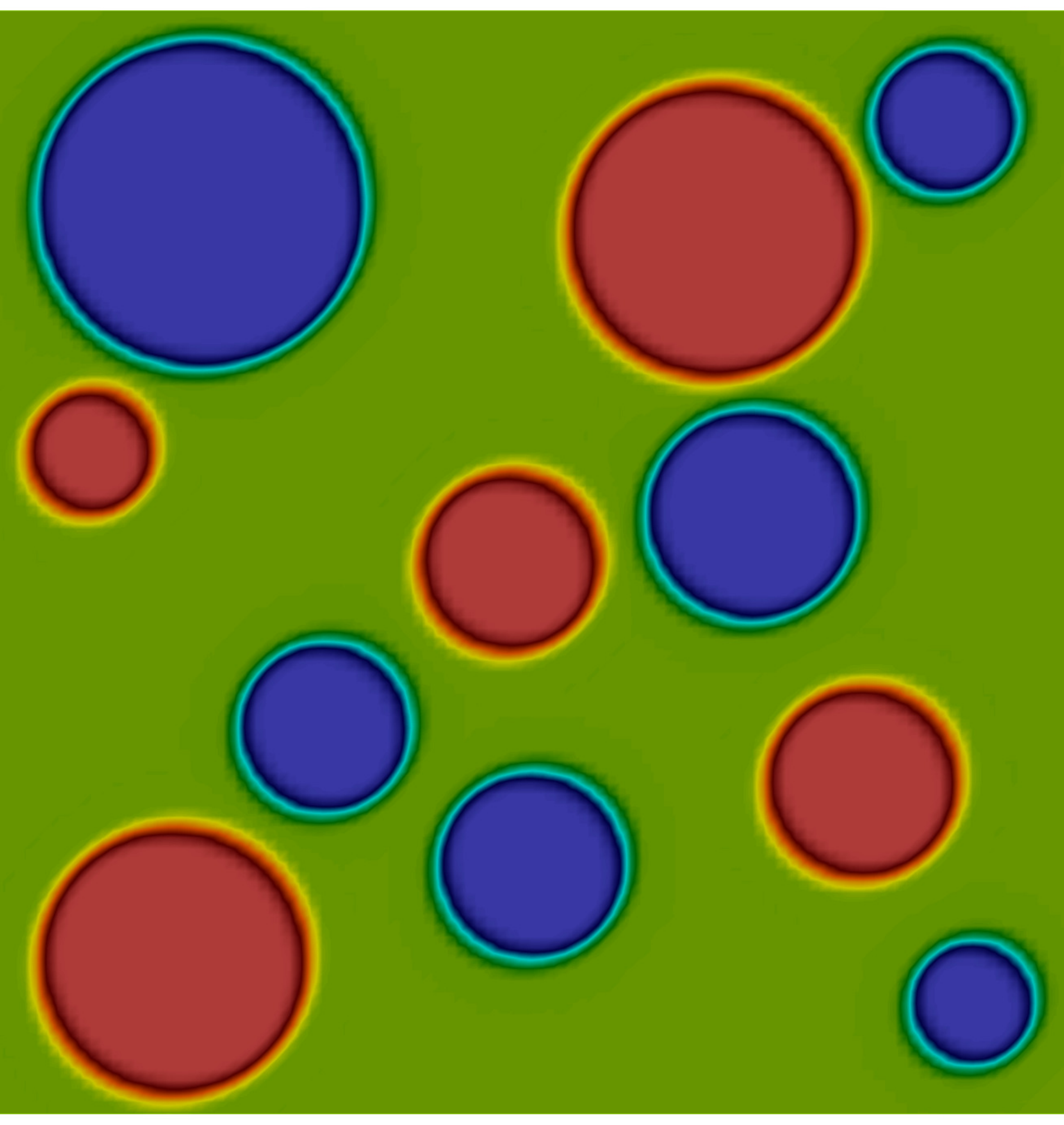}
\includegraphics[width=0.19\textwidth]{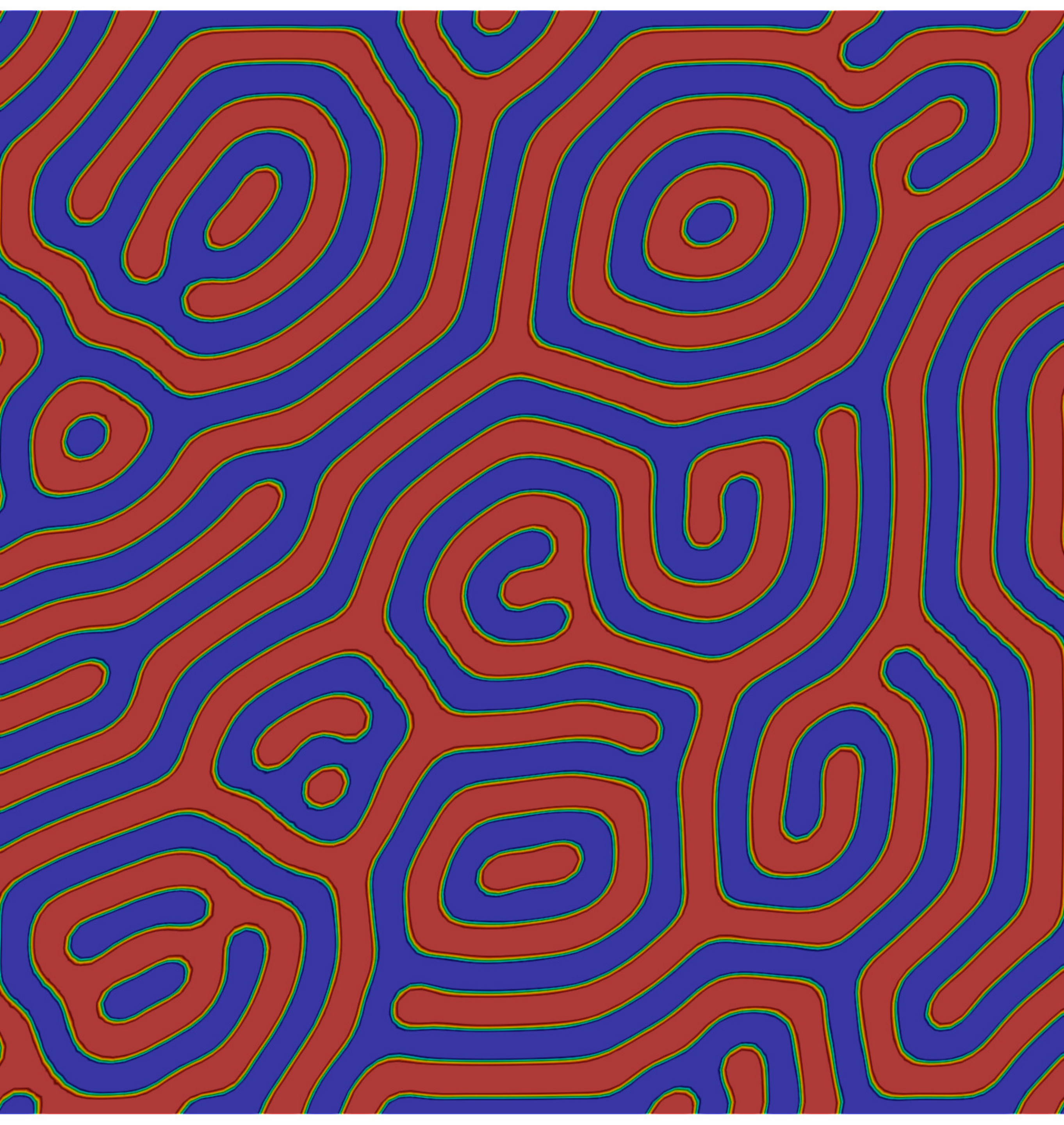}
\includegraphics[width=0.19\textwidth]{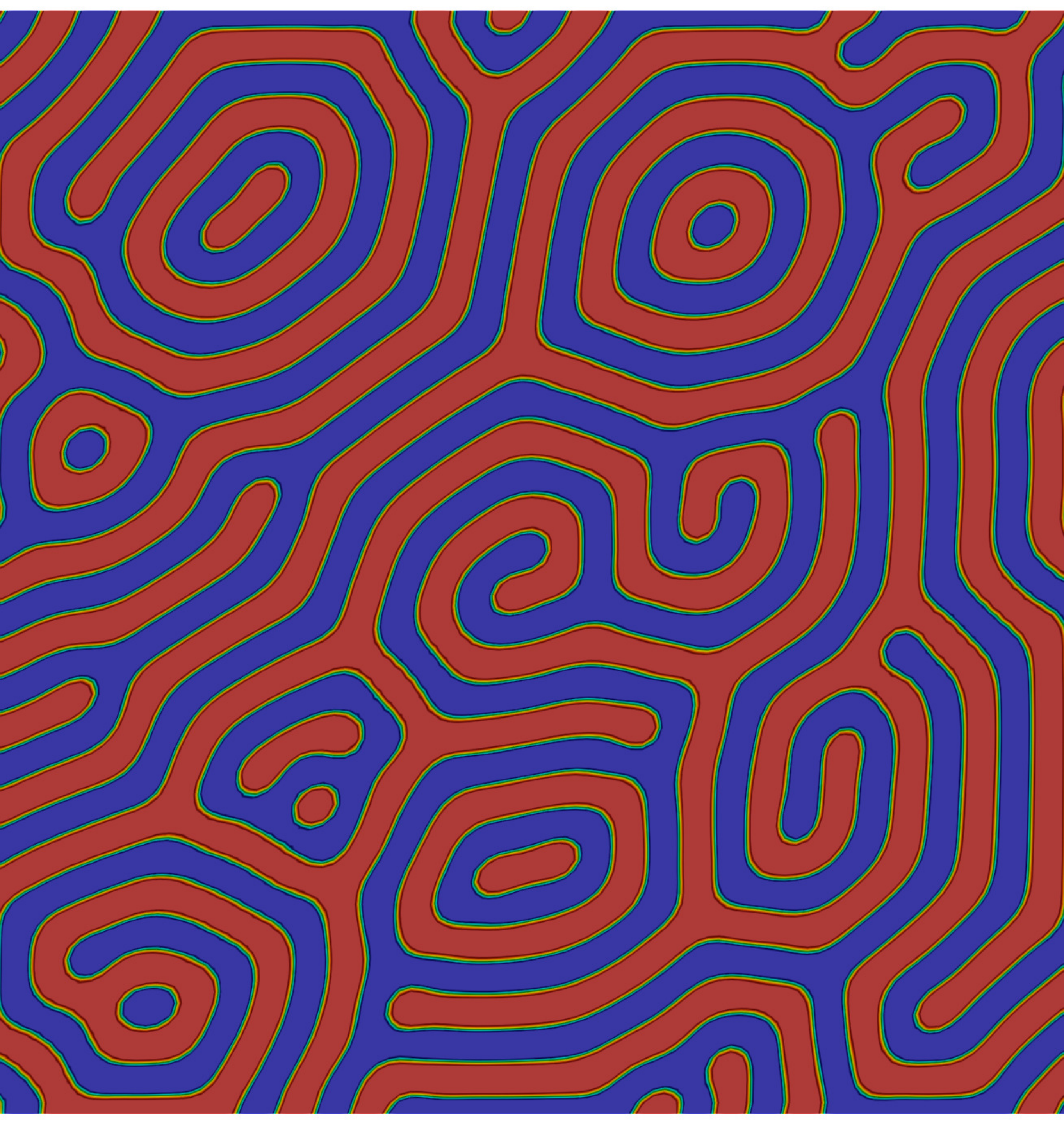}
\includegraphics[width=0.19\textwidth]{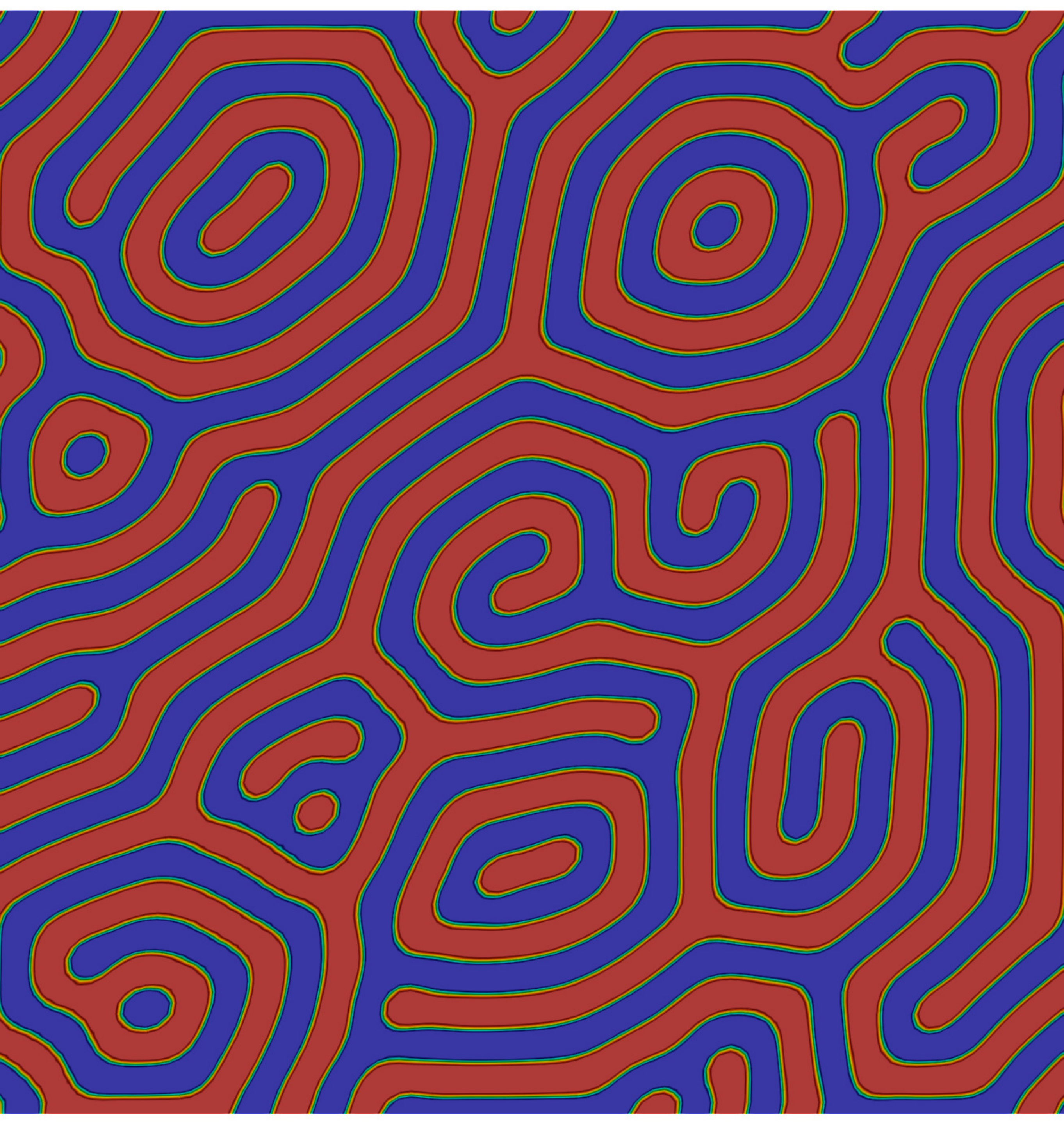}
\includegraphics[width=0.19\textwidth]{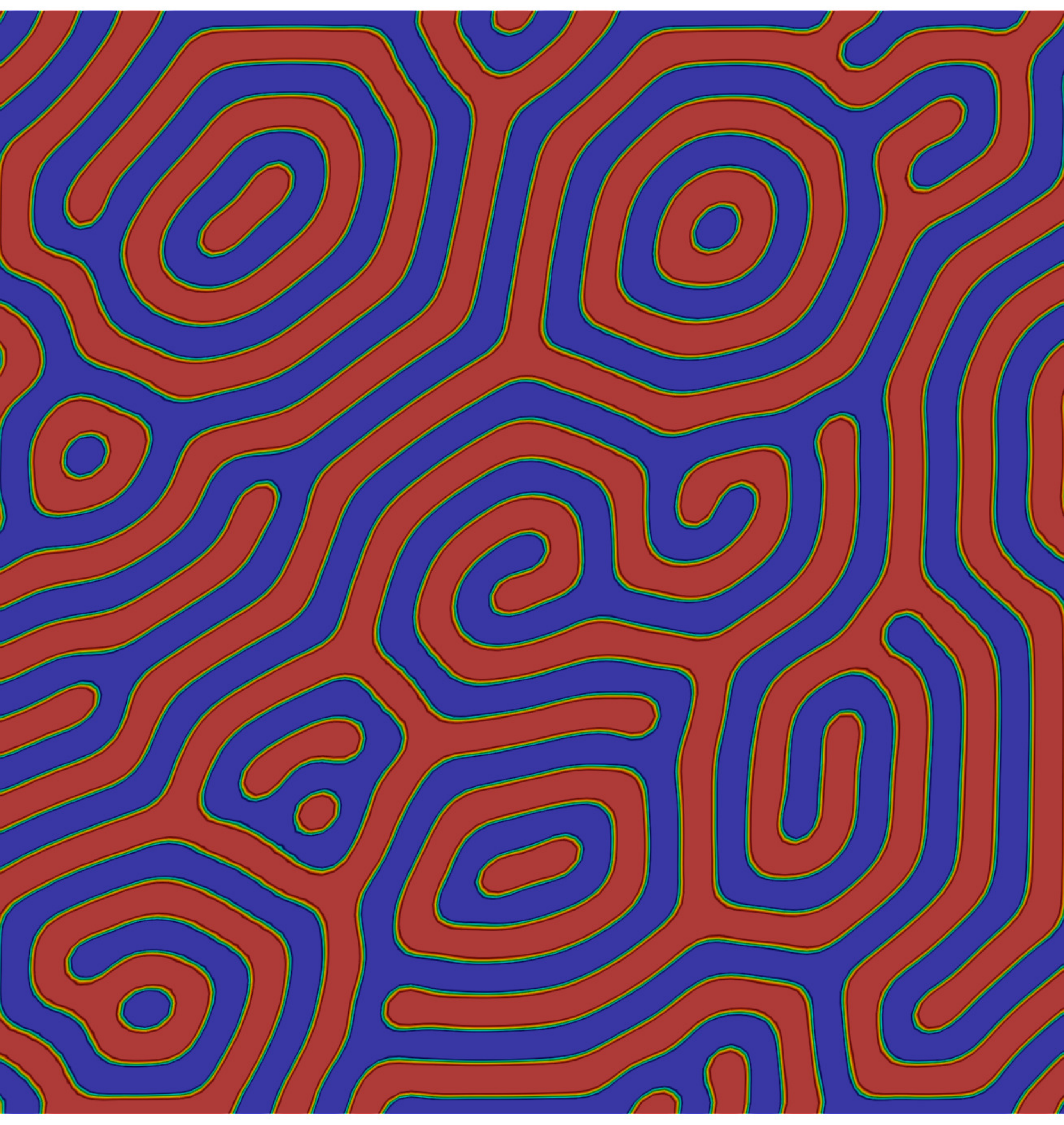}
\caption{Evolution of $\phi$ at times $t=0, 0.05, 0.15, 0.35$ and $0.5$ (from left to right) taking $M=0.1$, $g_2=1$, $h_0=0.5$, $\lambda=0.1$ and $\beta=1$ with $g_0=-4$ (top row), $g_0=-10$ (center row) and $g_0=-25$ (bottom row).} \label{fig:g0s}
\end{center}
\end{figure}

\begin{figure}[h]
\begin{center}
\includegraphics[width=0.45\textwidth]{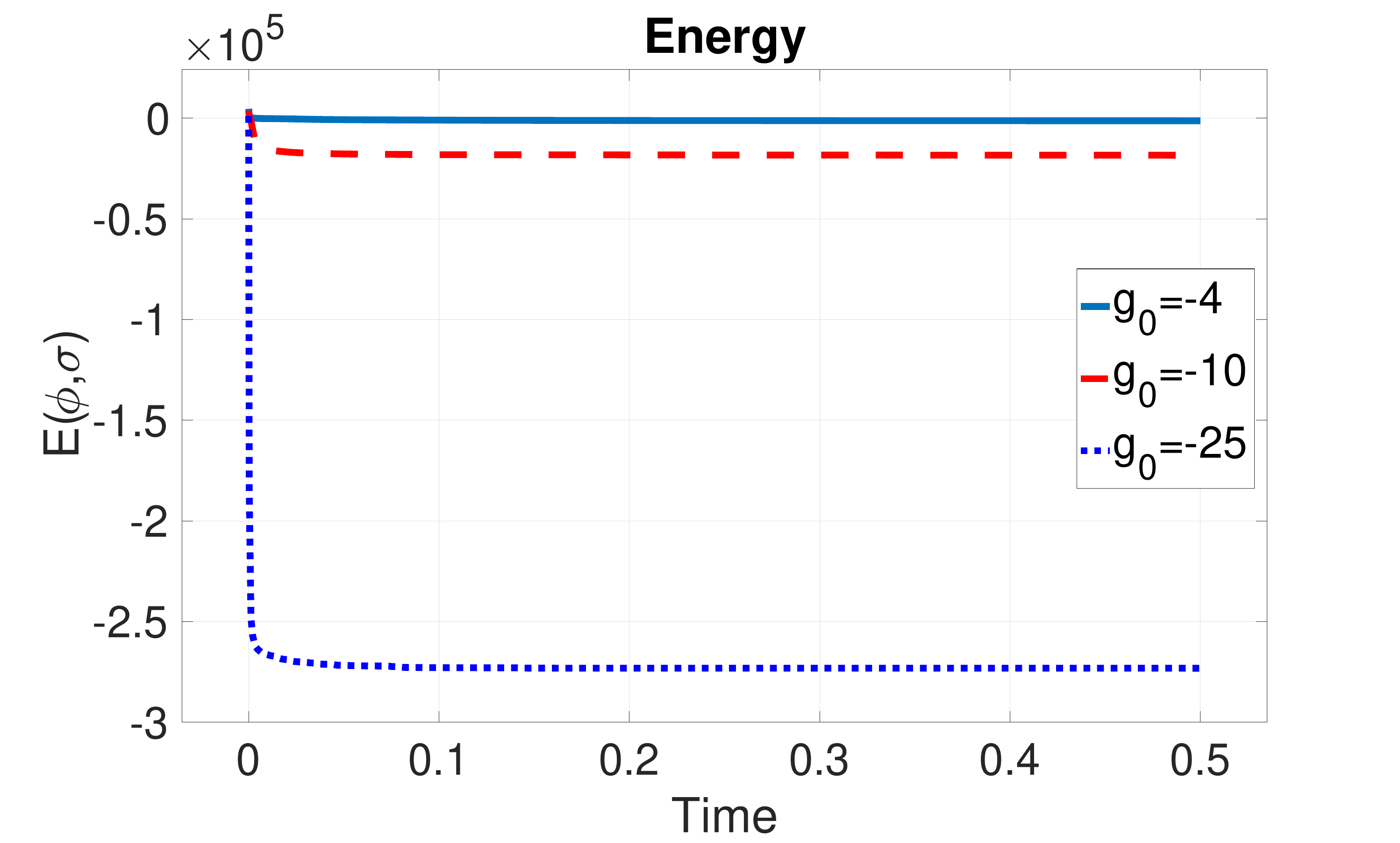}
\includegraphics[width=0.45\textwidth]{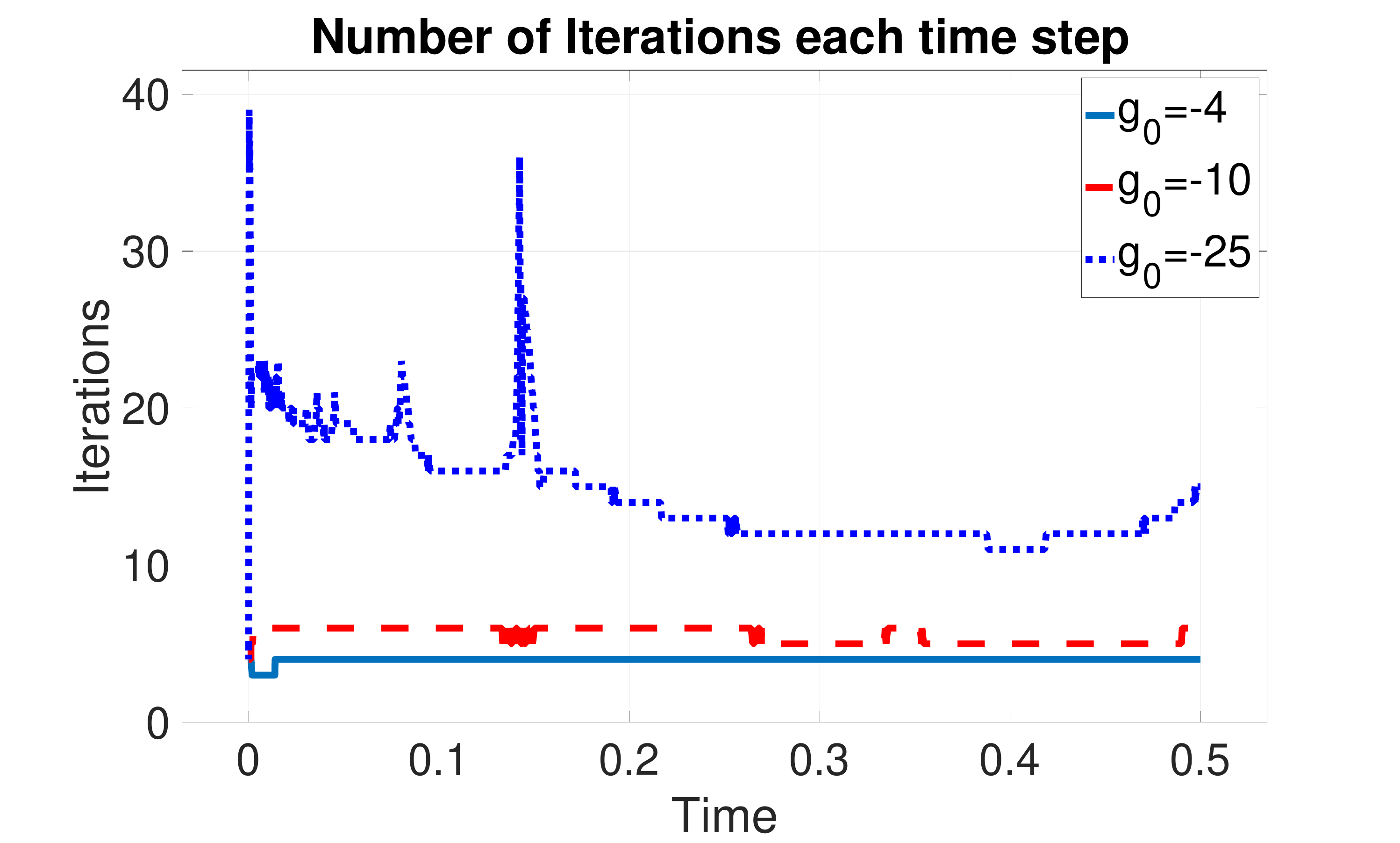}
\\
\includegraphics[width=0.45\textwidth]{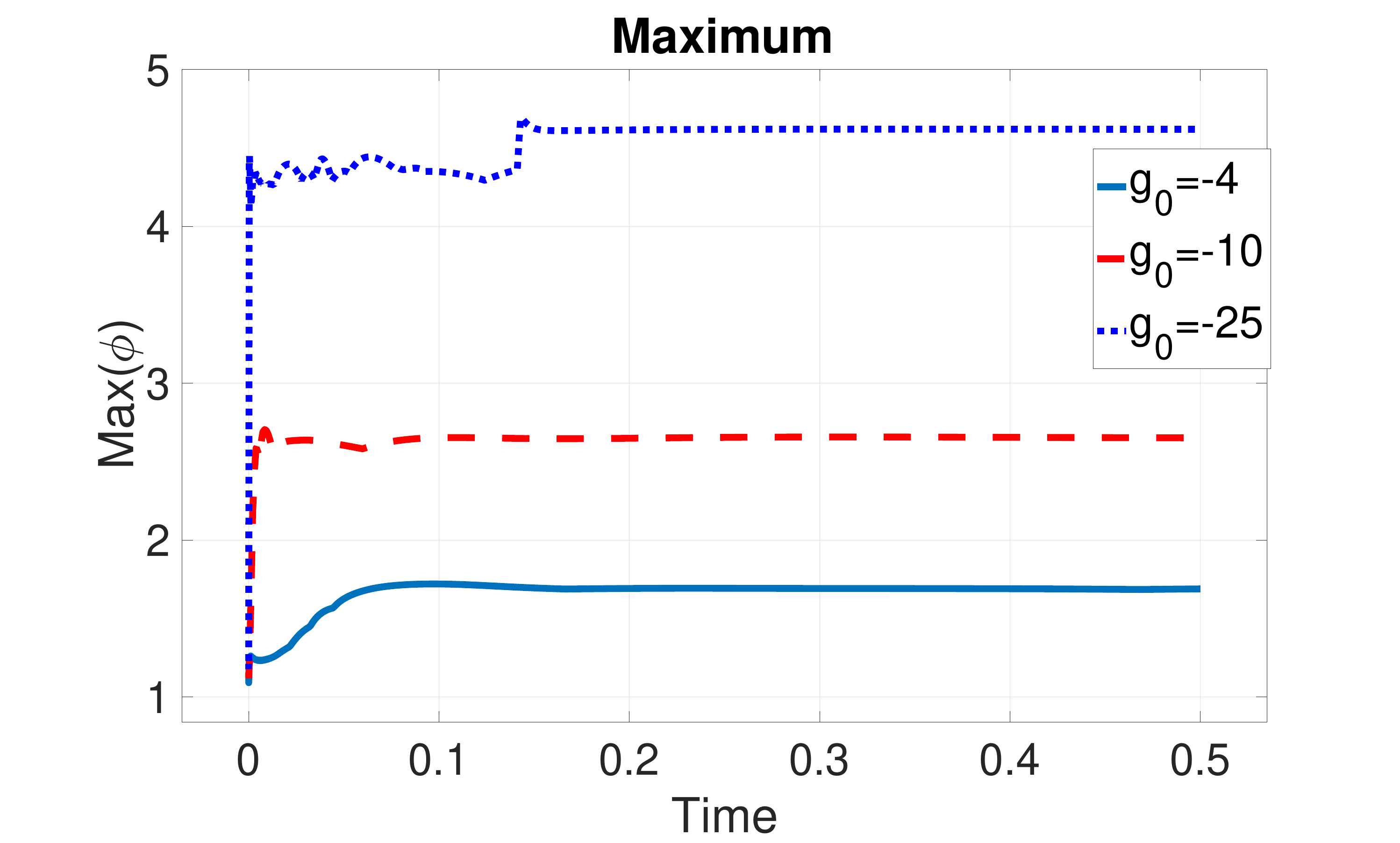}
\includegraphics[width=0.45\textwidth]{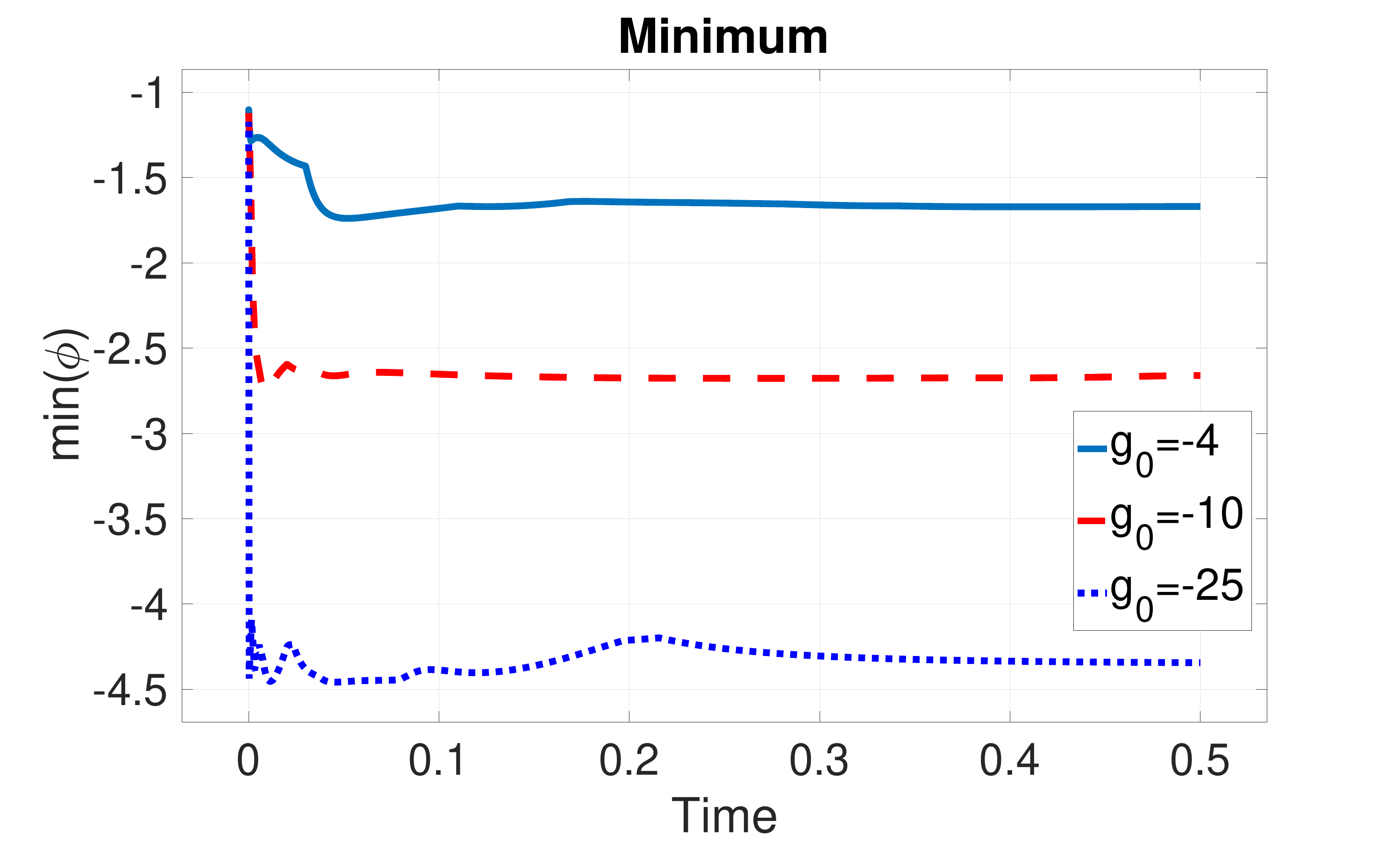}
\caption{Evolution in time of the energies (top left), the number of iterations to achieve tolerance $\texttt{TOL}=10^{-7}$ (top right), maximum of $\phi$ (bottom left) and minimum of $\phi$ (bottom right) taking $M=0.1$, $g_2=1$, $h_0=0.5$, $\lambda=0.1$, $\beta=1$ and $g_0=-4, -10$ and $-100$.} 
\label{fig:g0splot}
\end{center}
\end{figure}

%%%%%%%%%%%%%%%%%%%%%%%%%%%%%%%%%%%%%%%%%%%%%%%%%%%
%nss: commented for now since files are missing
%%%%%%%%%%%%%%%%%%%%%%%%%%%%%%%%%%%%%%%%%%%%%%%%%%%
\begin{figure}[h]
\begin{center}
\includegraphics[width=0.19\textwidth]{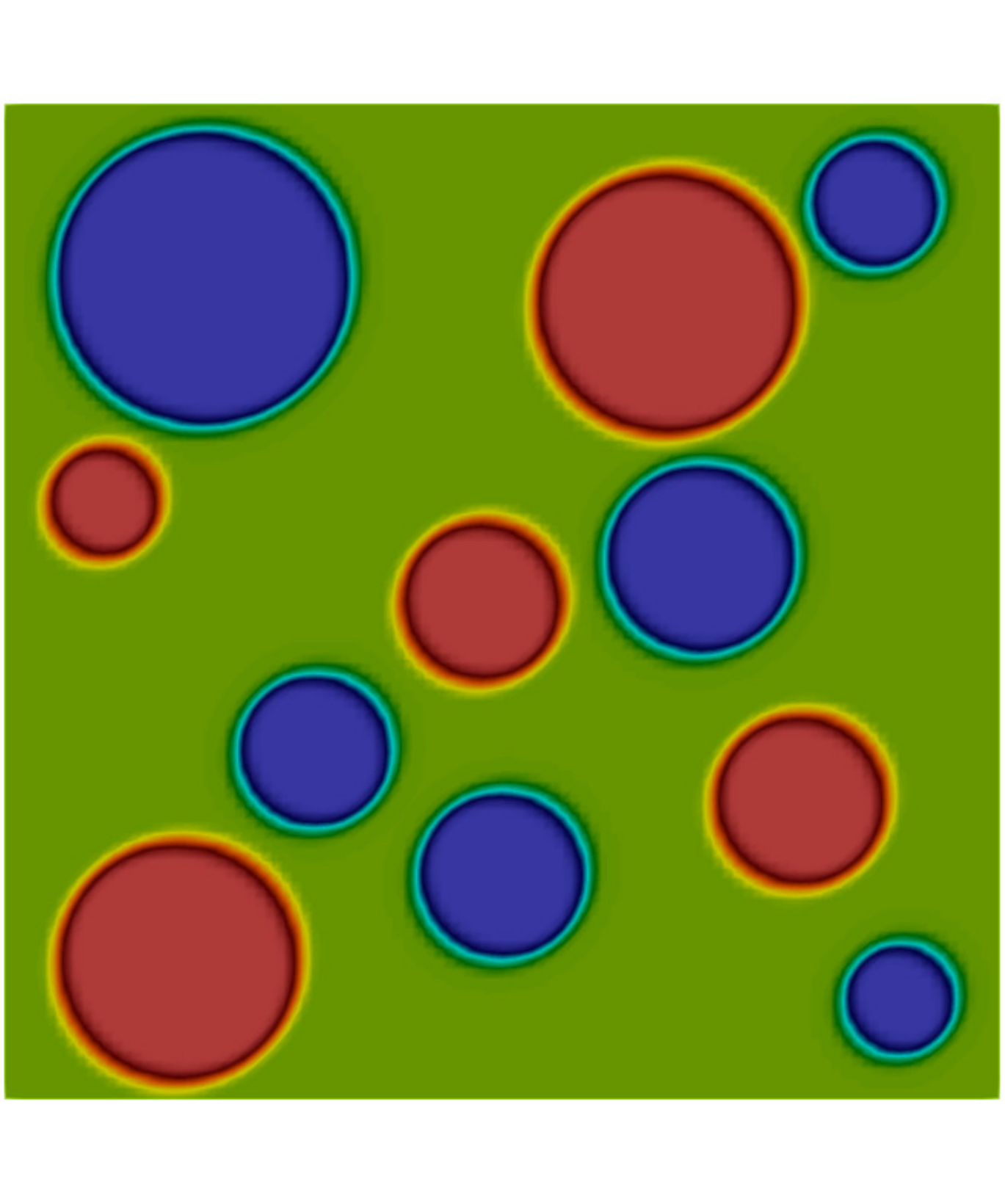}
\includegraphics[width=0.19\textwidth]{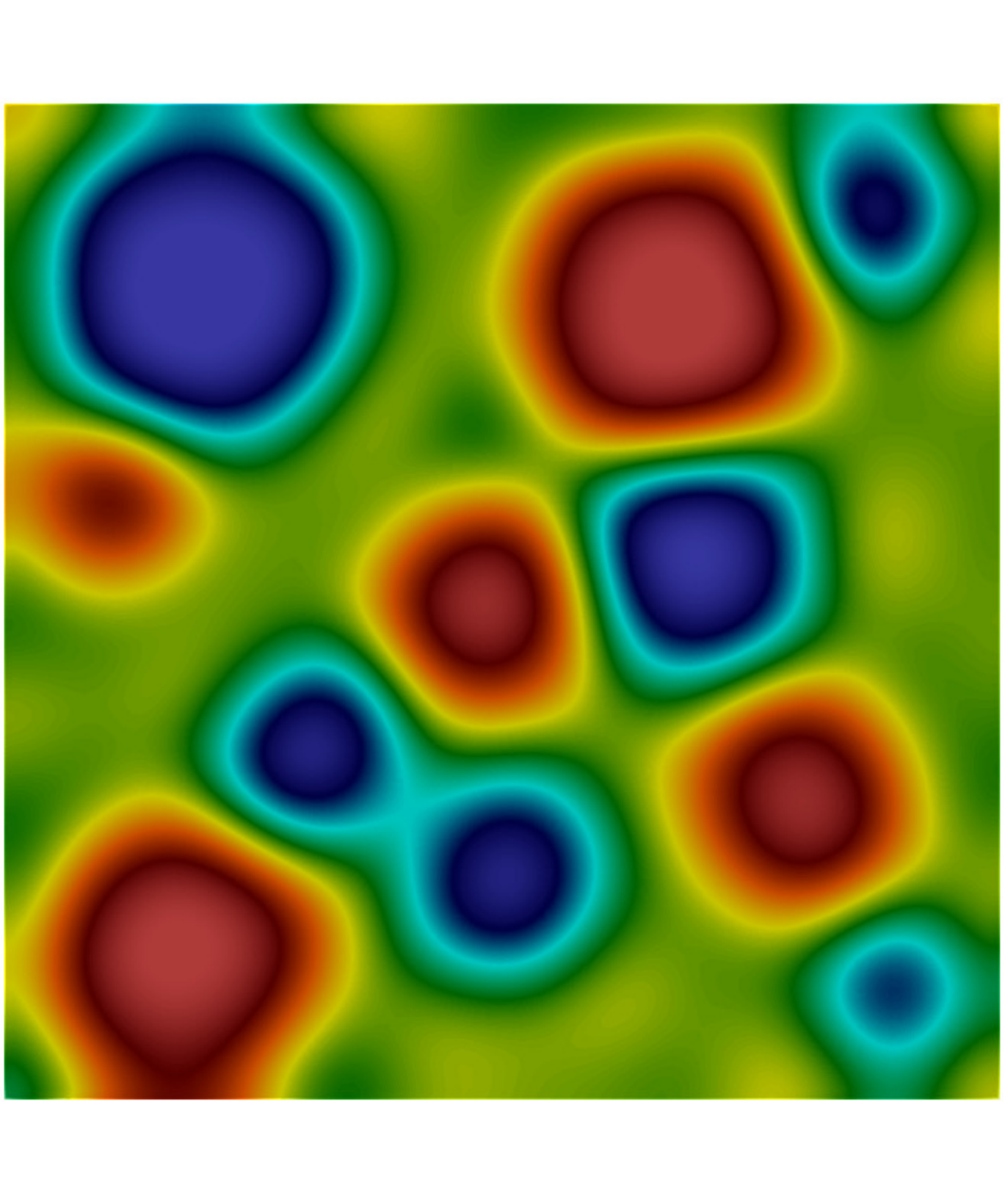}
\includegraphics[width=0.19\textwidth]{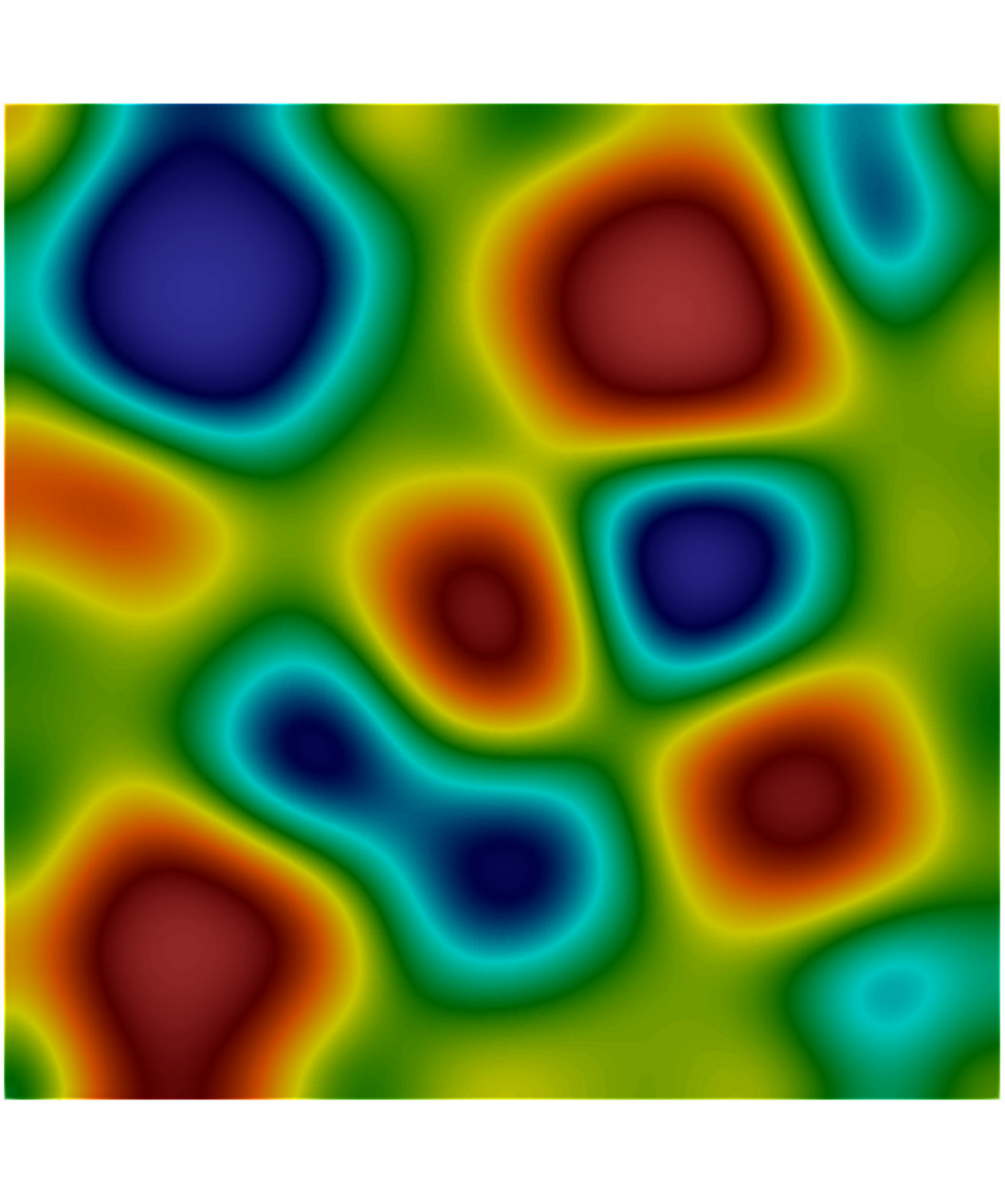}
\includegraphics[width=0.19\textwidth]{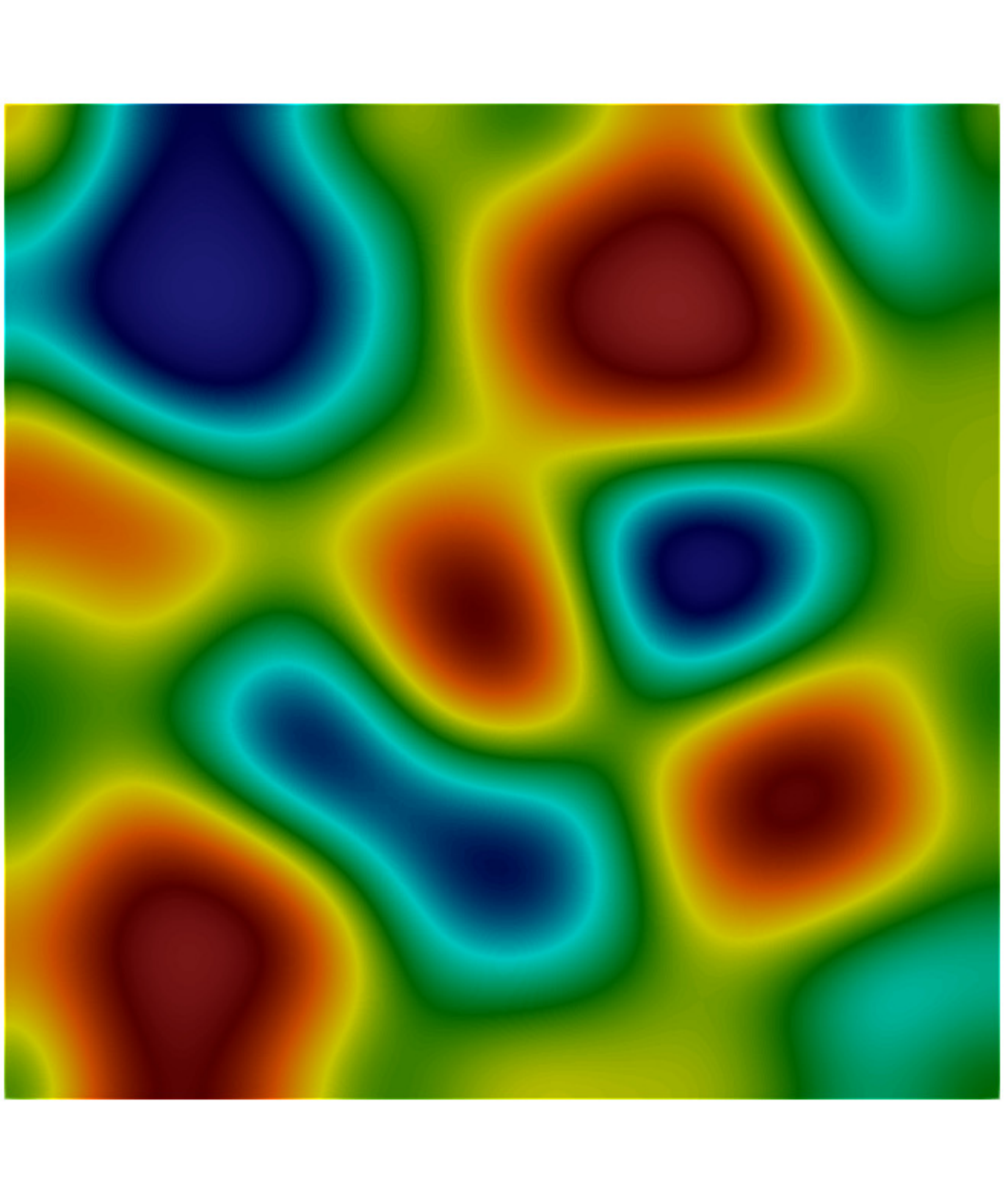}
\includegraphics[width=0.19\textwidth]{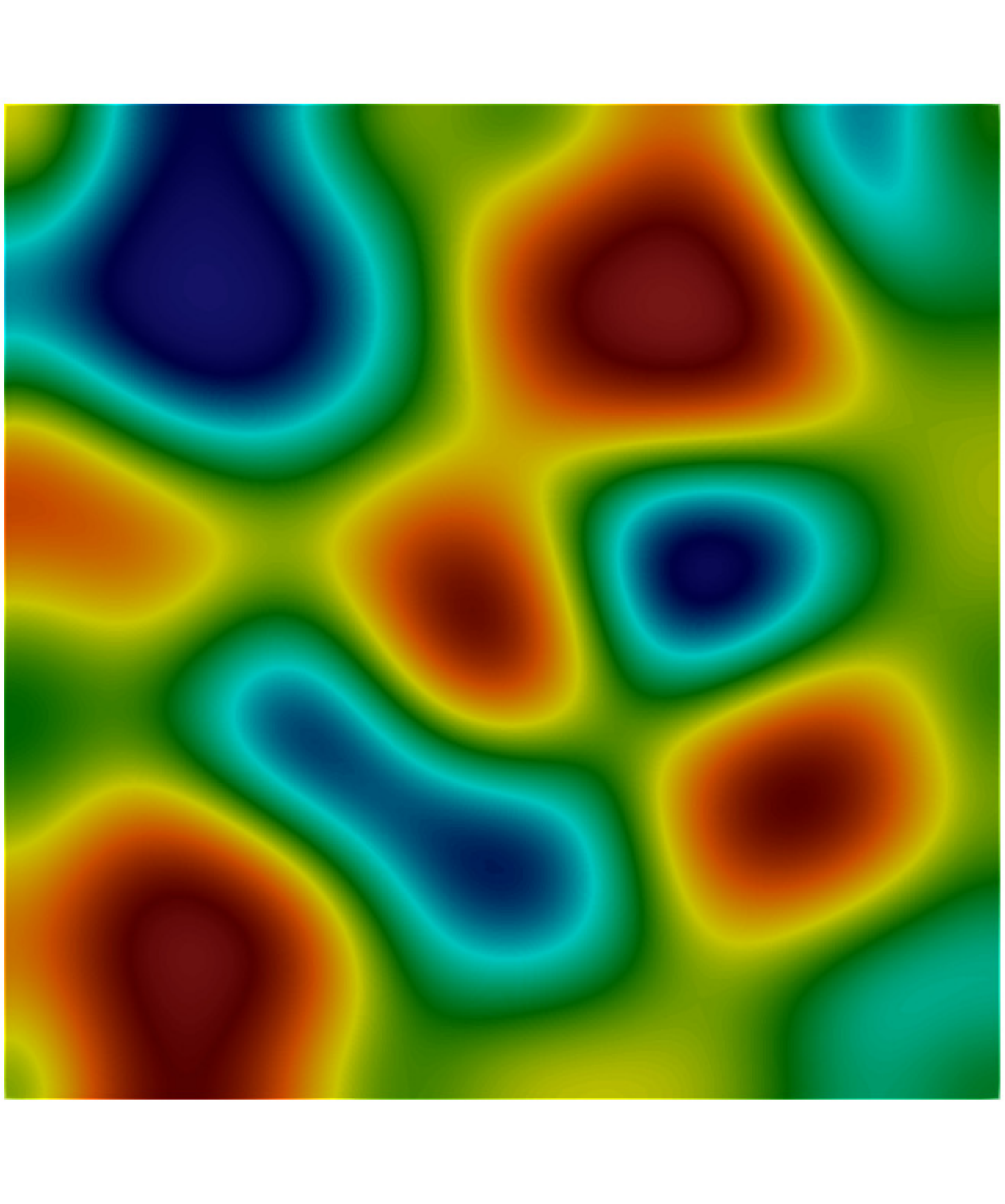}
\caption{{Evolution of $\phi$ at times $t=0, 0.05, 0.15, 0.35$ and $0.5$ (from left to right) taking $M=0.1$, $g_2=1$, $h_0=0.5$, $\lambda=0.1$ and $\beta=1$ with $g_0=0$.}} \label{fig:g0_0}
\end{center}
\end{figure}

\subsubsection{Study of influence of parameter $\lambda$ on the dynamics of the system}\label{sub:lambda}

In this example we fix the parameters to
$M=0.1$, $g_2=1$, $g_0=-4$, $h_0=0.5$, $\beta=1$ and we consider different values of parameter $\lambda$, namely $\lambda=0.1$, $\lambda=0.05$ and $\lambda=0.01$. The dynamics associated to these simulations are presented in Figure~\ref{fig:lambdas} and the evolution of the energies, maximum of $\phi$, minimum of $\phi$ and number of iterations are presented in Figure~\ref{fig:lambdasplot}. The parameter $\lambda$ plays the role of the weight of the term $\sigma^2$ in the energy $E(\phi,\sigma)$, which is associated with a part of the energy associated with the curvature of the interface, meaning that considering larger values of $\lambda$ will induce more importance to minimize the curvature of the interface. The obtained results suggest that there is a clear direct relation between the value of $\lambda$ and the interfacial width, that is, decreasing its value leads to thinner interfaces. Moreover, as $\lambda$ is decreased, the values of $\phi$ gets away from the minima of the functional $f_0(\phi)$. In all the simulations the energies decrease as expected and there is a mild effect on the value of $\lambda$ and the number of iterations needed to achieve the required tolerance in the iterative algorithm, that is, taking smaller values $\lambda$ leads to more difficulties to get the iterative algorithm to converge but in a milder way when compared with varying the value of $g_0$. These results coincide with the results in Lemma~\ref{lem:iterativealgorithm}, where variations on $|g_0|$ has stronger effects on $\Delta t$ than variations on $\lambda$ due to the power $3$ on $|g_0|$.
\\
%\noindent
%\textcolor{red}{(I would like to add the following description:)}
%	\textbf{
%		In view of the discrete energy functional
%		\begin{align*}
%			E(\phi)
%			:=
%			\int_\Omega
%			\left(
%			\beta f_0(\phi)
%			+\frac12g(\phi)|\nabla\phi|^2
%			+\frac{\lambda}2(\Delta\phi)^2
%			\right)d\x\,,
%		\end{align*}
%		it is clear that the coefficient $\beta$ influences the tendency of the ternary system to morph into water-rich and oil-rich regions hence increasing the value of $\beta$ leads to larger bulk regions of water and oil. This is evident from the profiles of $\phi$ in the bottom row subfigures of Figure~\ref{fig:Exbetas} which possess relatively larger bulk regions of oil and water. The other impact of this coefficient is discussed in~\cite{PawlowZajaczkowski11} is that at the microemulsion phase ($\phi=0$),
%		%the properties of the amphiphile and its concentration are related to the choice of $\beta$ in the following manner. T
%		the value of $\beta f_0(\phi)$ is inversely proportional to the concentration of the amphiphile/surfactant. That is $\beta f_0(0)= \beta h_0$ is low for high concentrations of the amphiphile/surfactant and vice versa. 
%	}
%\textcolor{red}{(NSS:~It is unclear to me \emph{how} our simulations are capturing the strength of the amphiphile or its concentration for that matter(it has the largest volume initially). I am inclined to either a) omit this other impact of $\beta$ or b) mention that it is unclear how to measure ``high'' amphiphile concentrations.)
%}
\begin{figure}[h]
\begin{center}
\includegraphics[width=0.19\textwidth]{images/Ex2/standard/Ex2_Standard_0}
\includegraphics[width=0.19\textwidth]{images/Ex2/standard/Ex2_Standard_50}
\includegraphics[width=0.19\textwidth]{images/Ex2/standard/Ex2_Standard_150}
\includegraphics[width=0.19\textwidth]{images/Ex2/standard/Ex2_Standard_350}
\includegraphics[width=0.19\textwidth]{images/Ex2/standard/Ex2_Standard_500}
\\
\includegraphics[width=0.19\textwidth]{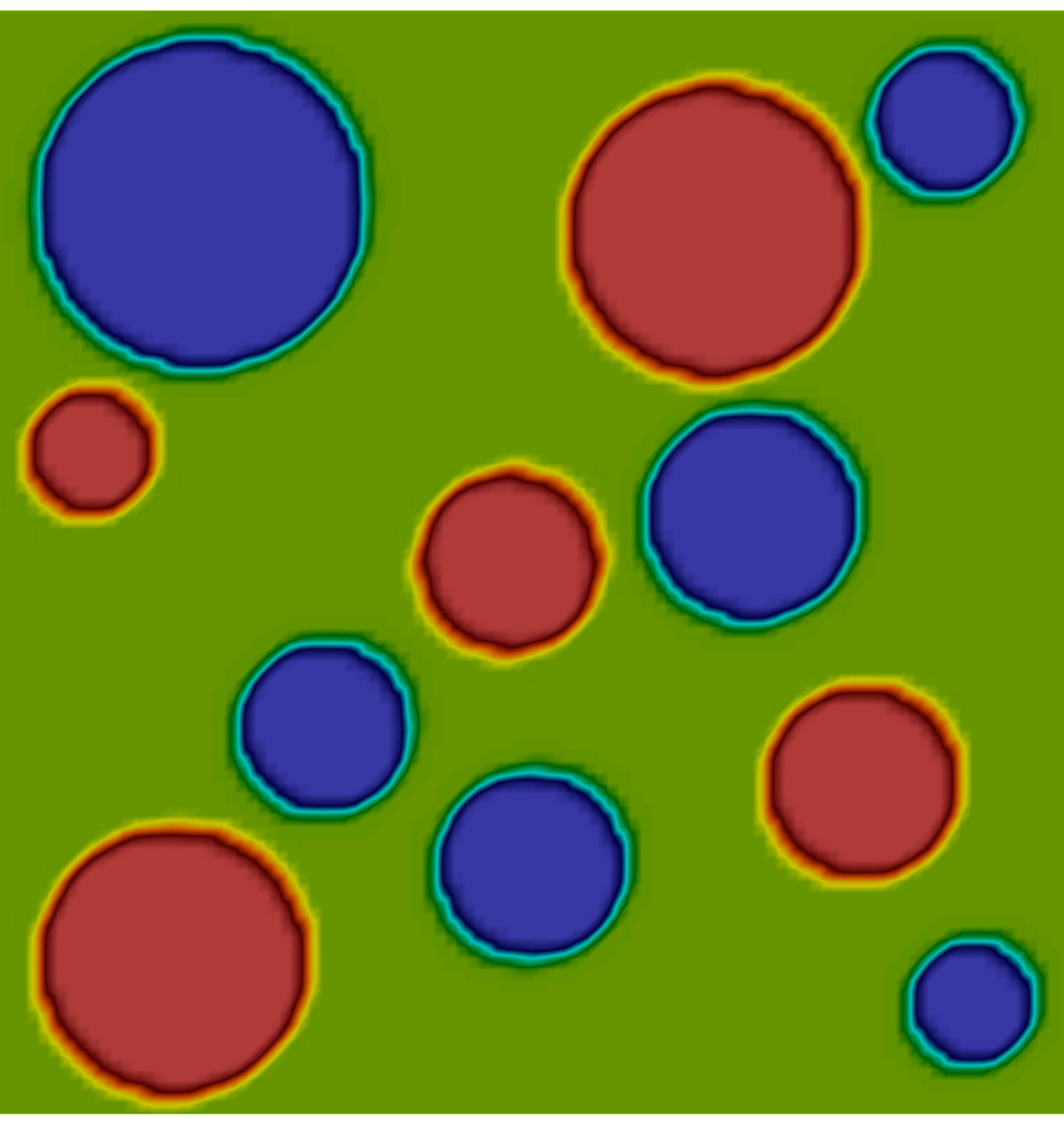}
\includegraphics[width=0.19\textwidth]{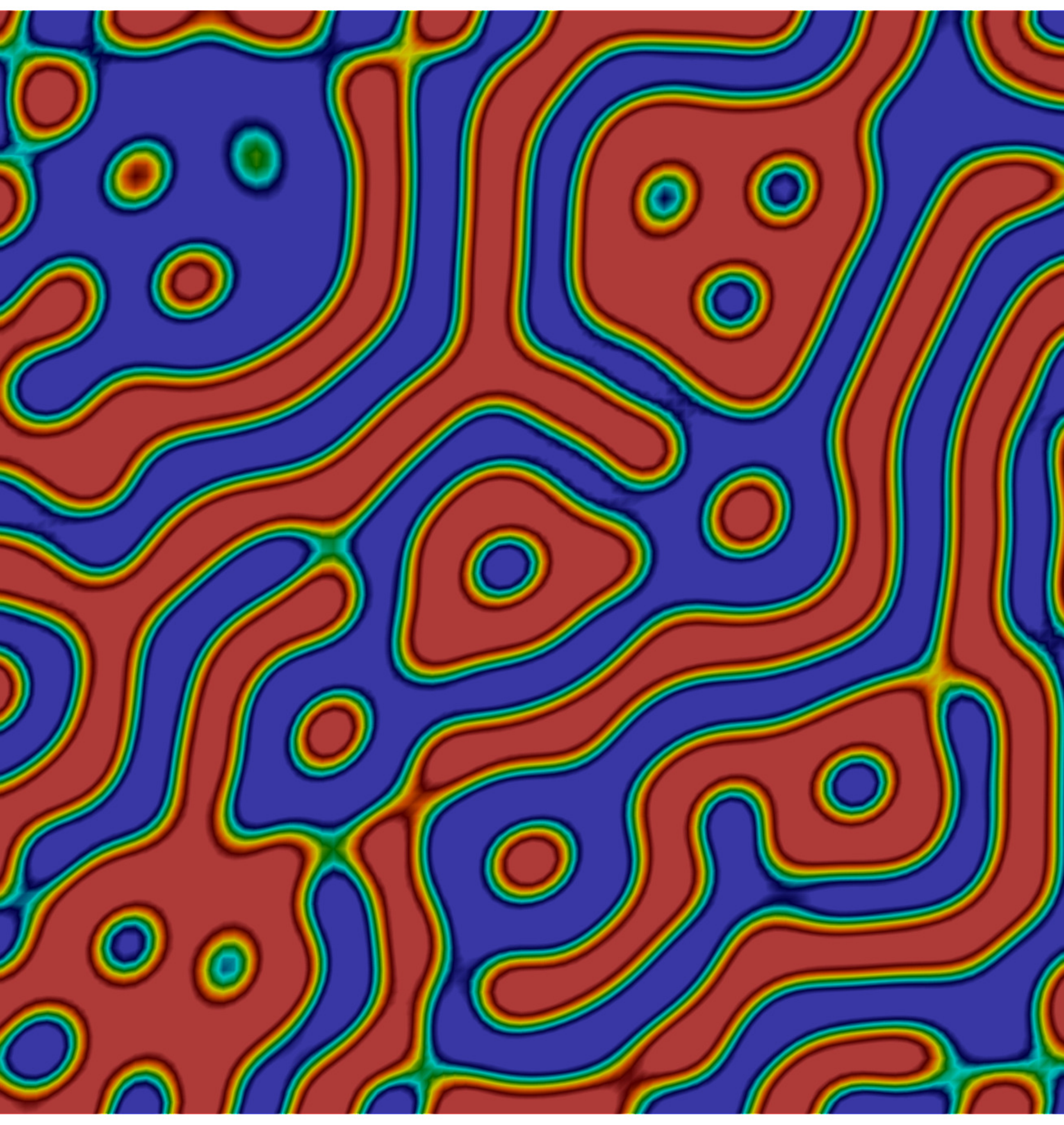}
\includegraphics[width=0.19\textwidth]{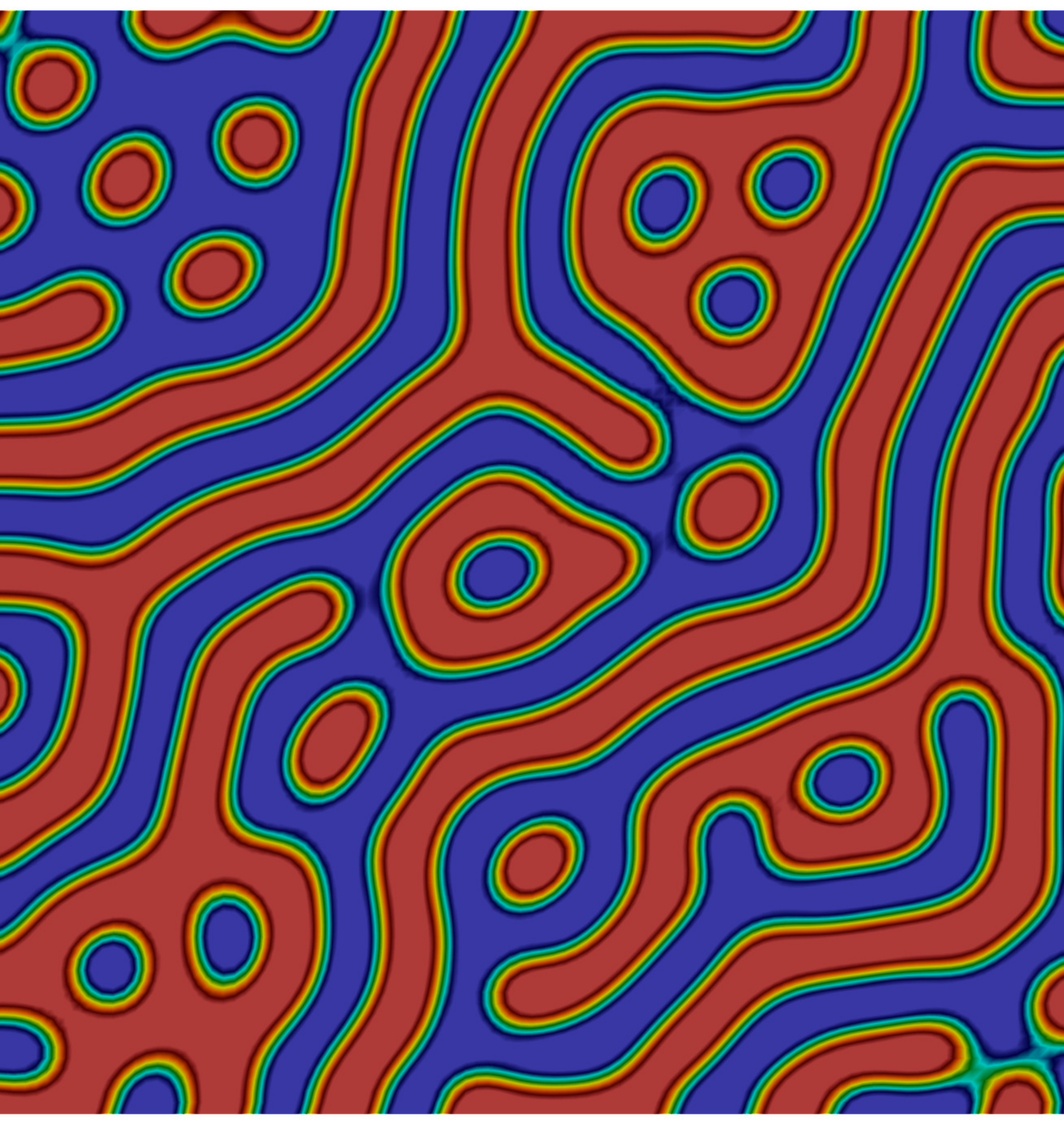}
\includegraphics[width=0.19\textwidth]{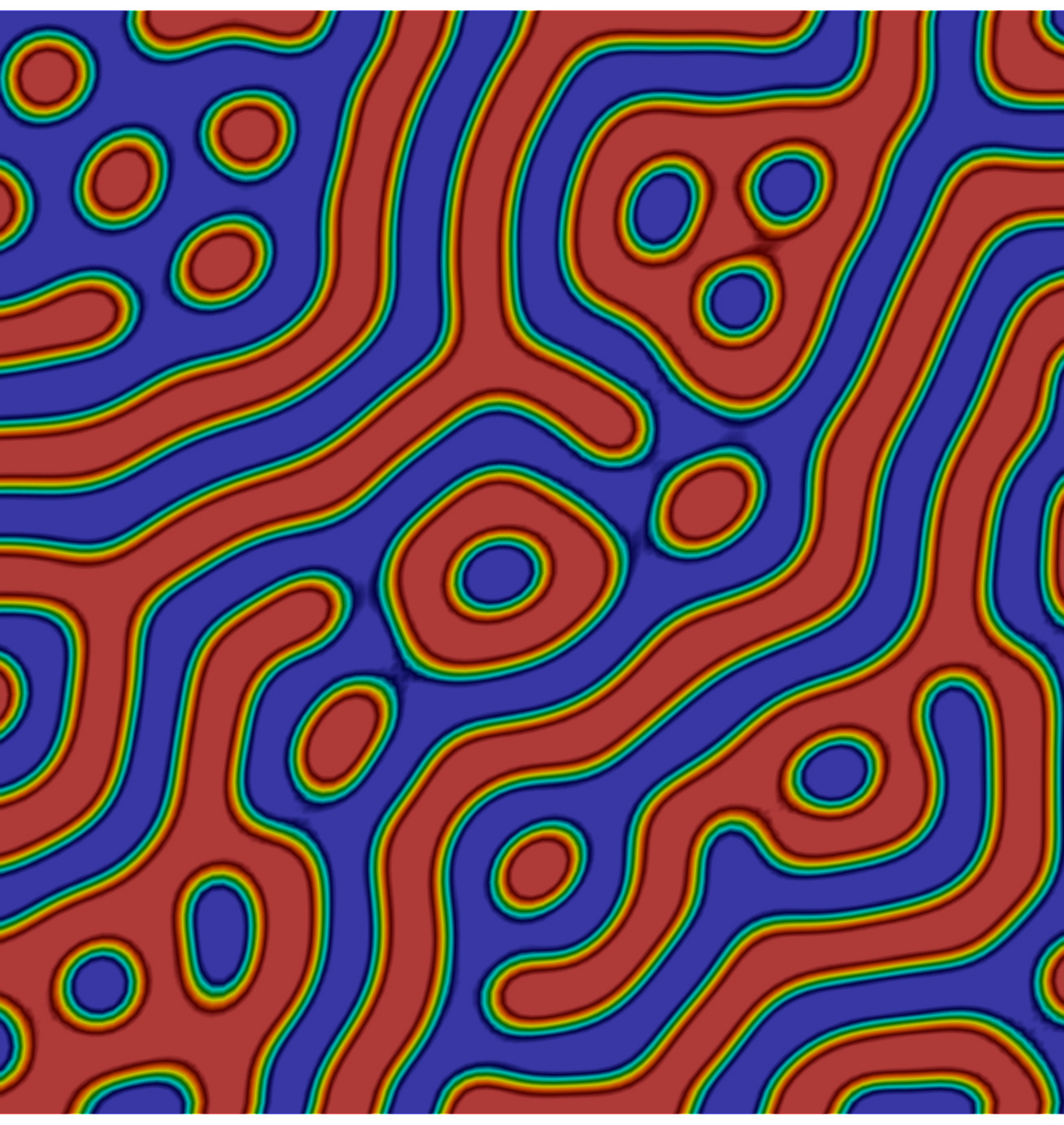}
\includegraphics[width=0.19\textwidth]{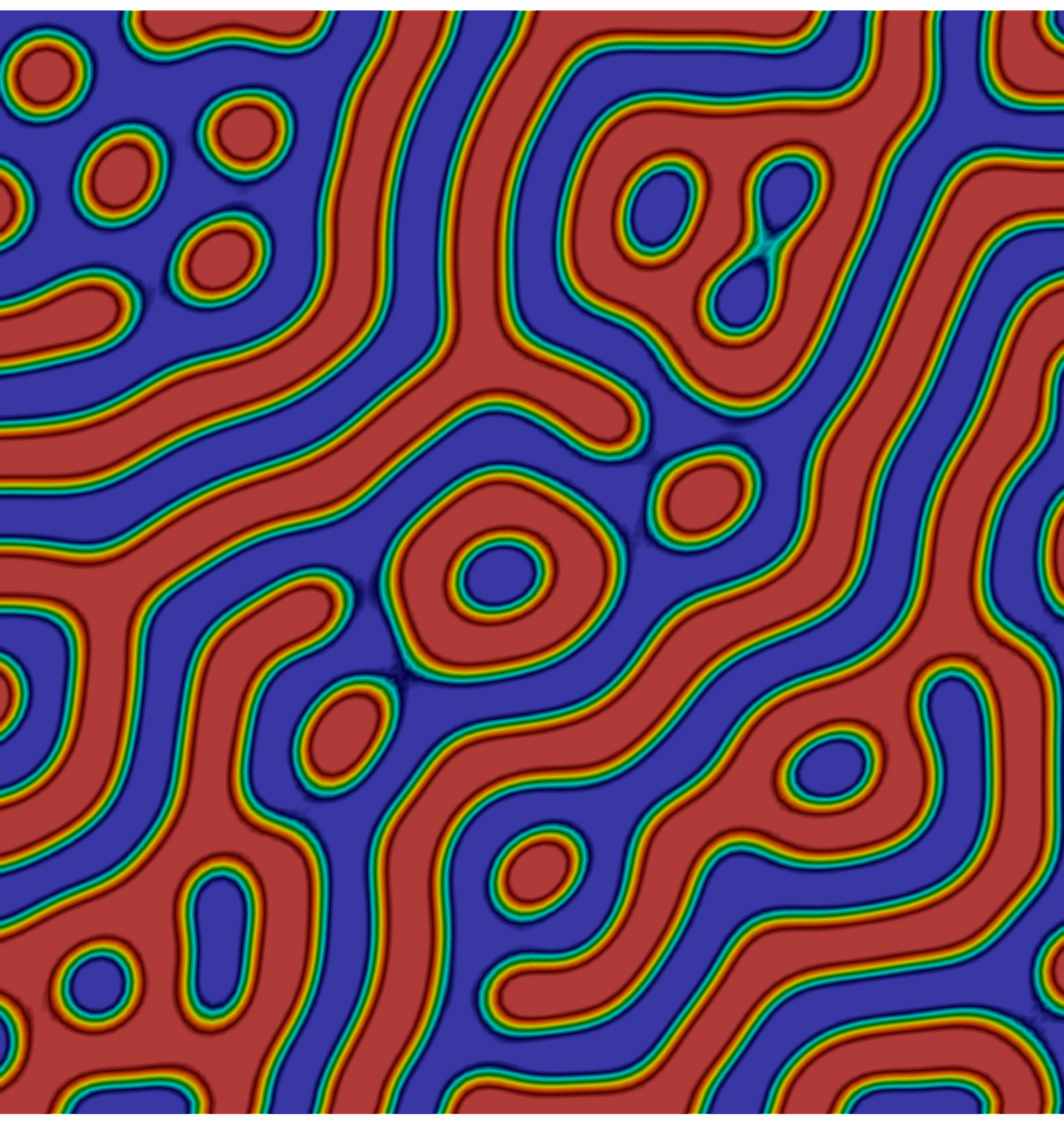}
\\
\includegraphics[width=0.19\textwidth]{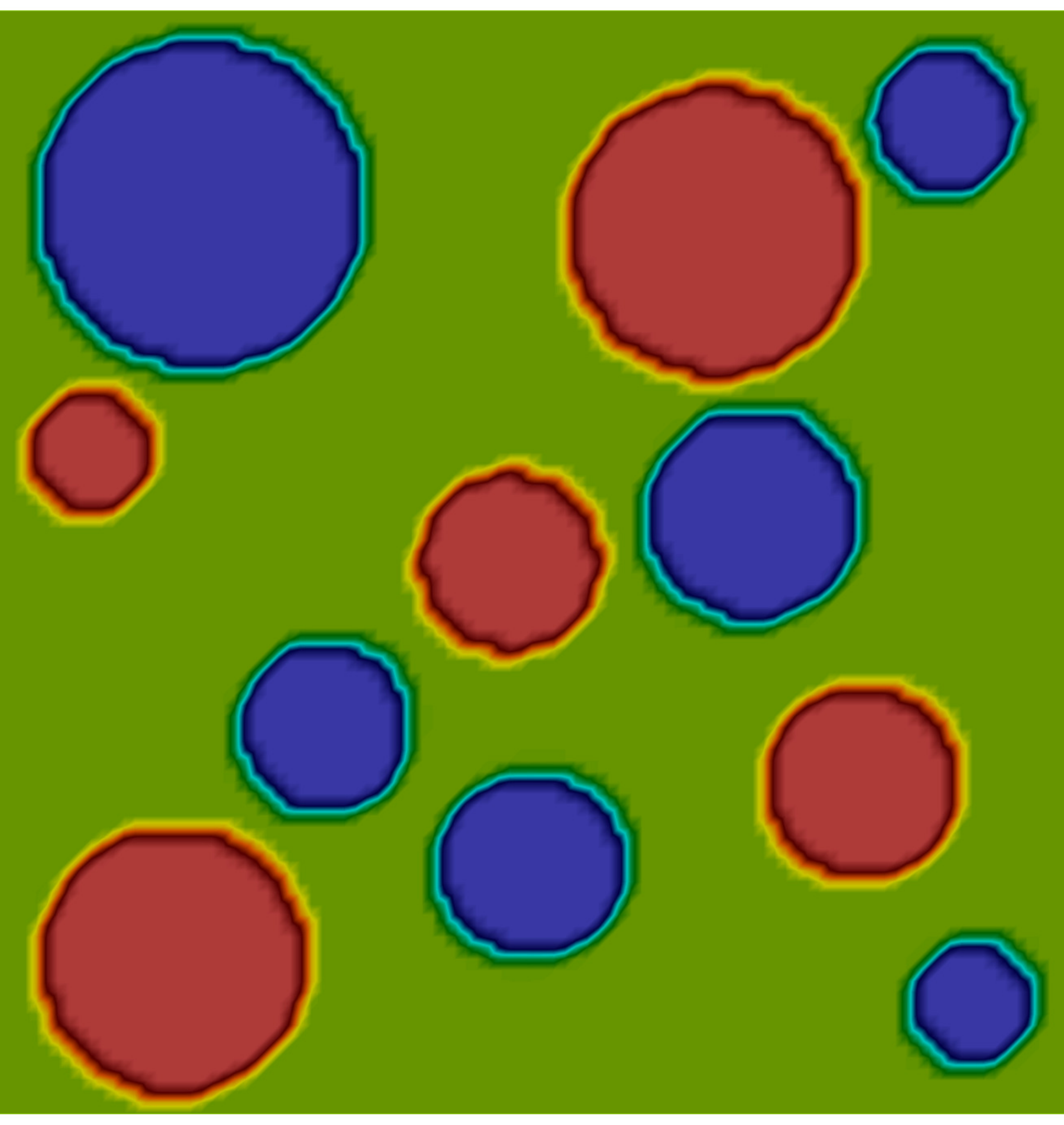}
\includegraphics[width=0.19\textwidth]{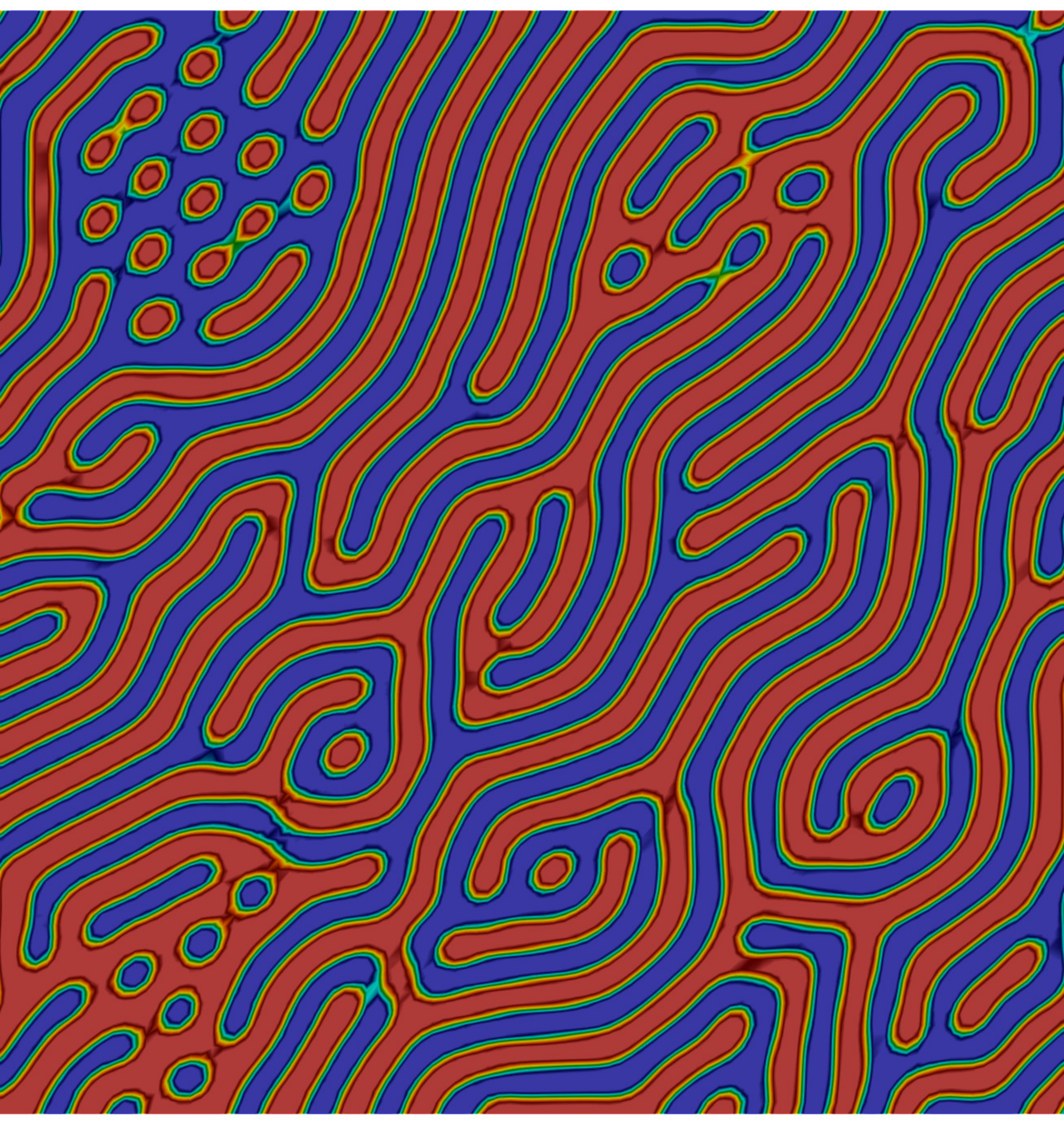}
\includegraphics[width=0.19\textwidth]{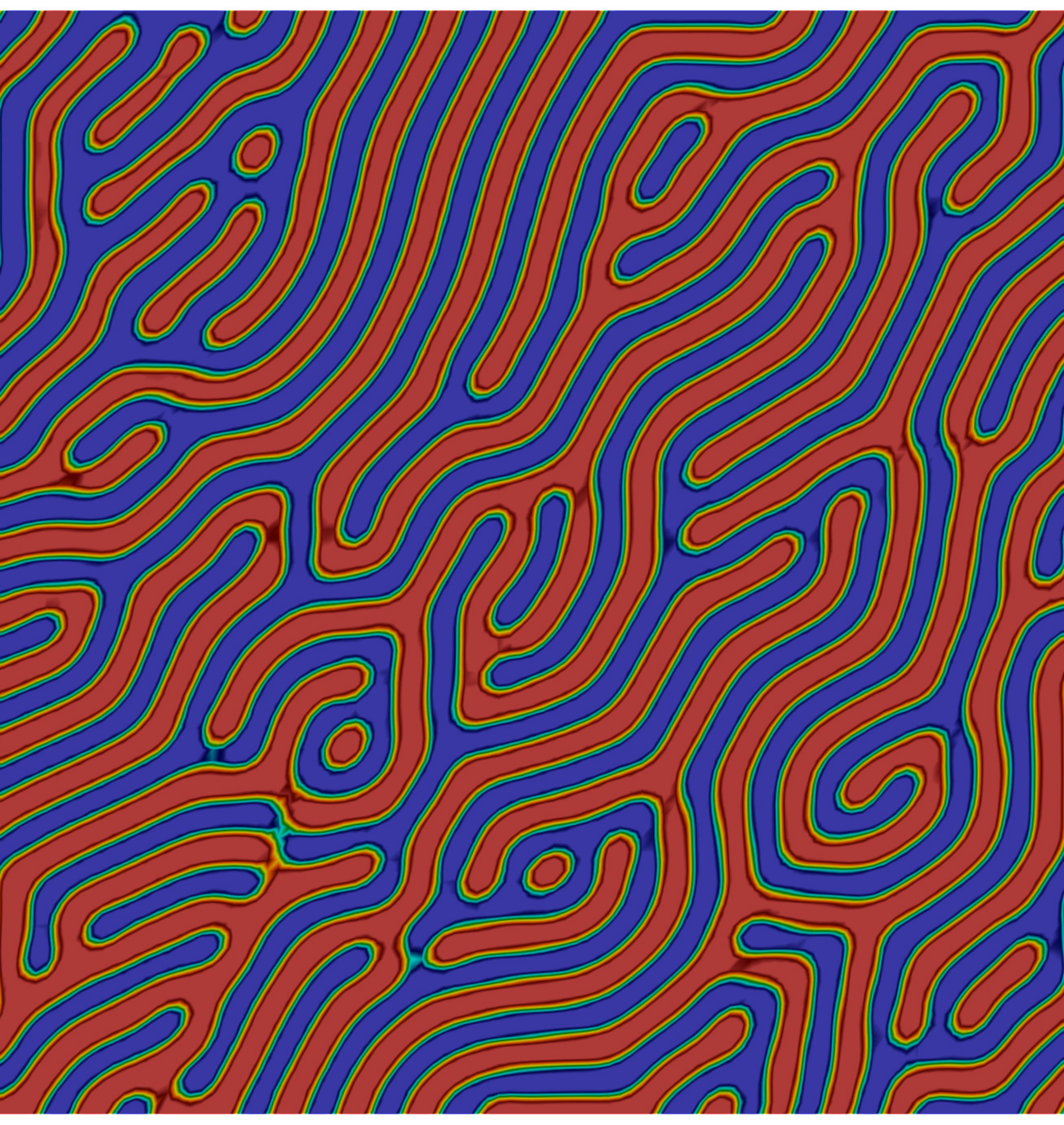}
\includegraphics[width=0.19\textwidth]{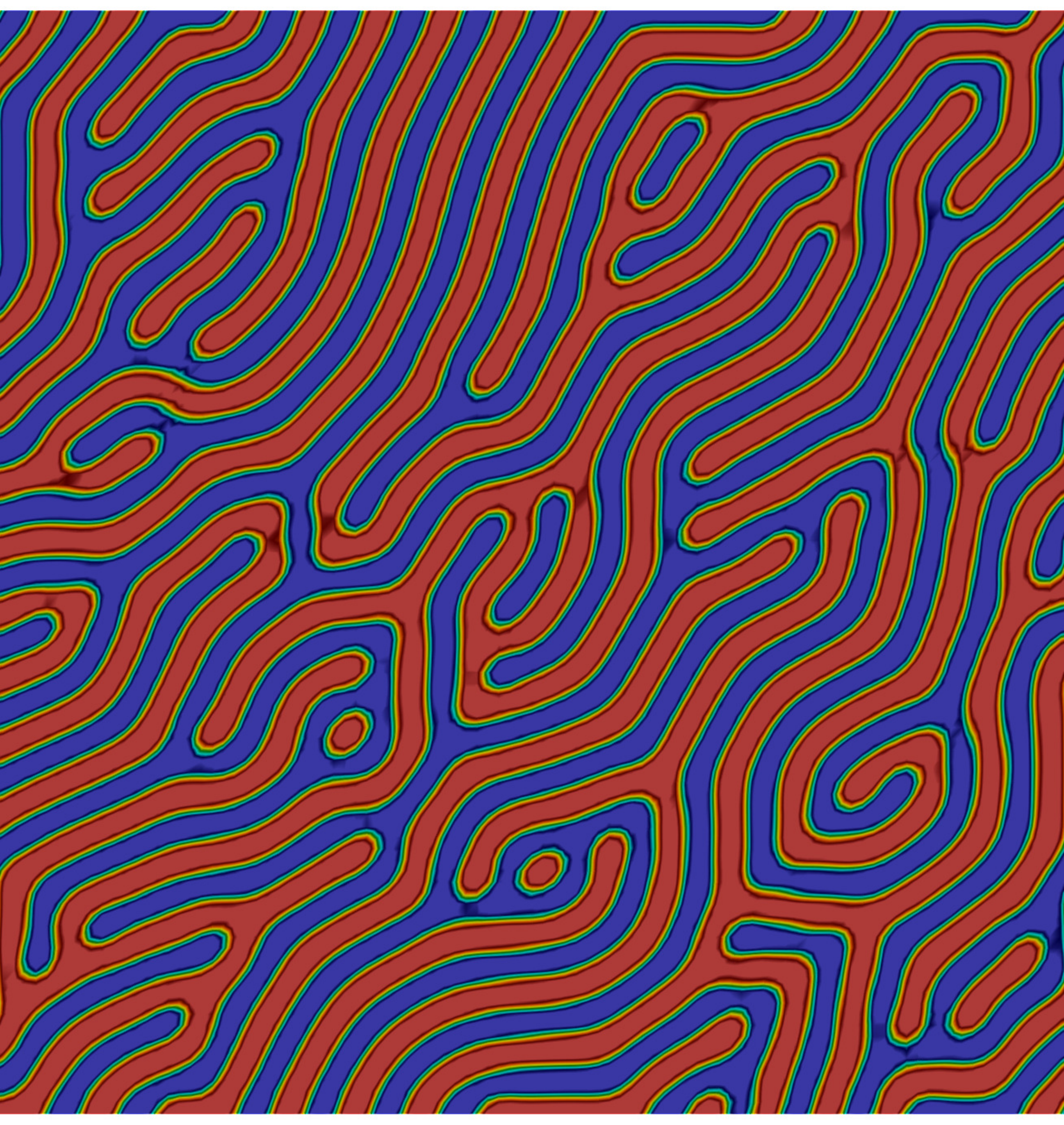}
\includegraphics[width=0.19\textwidth]{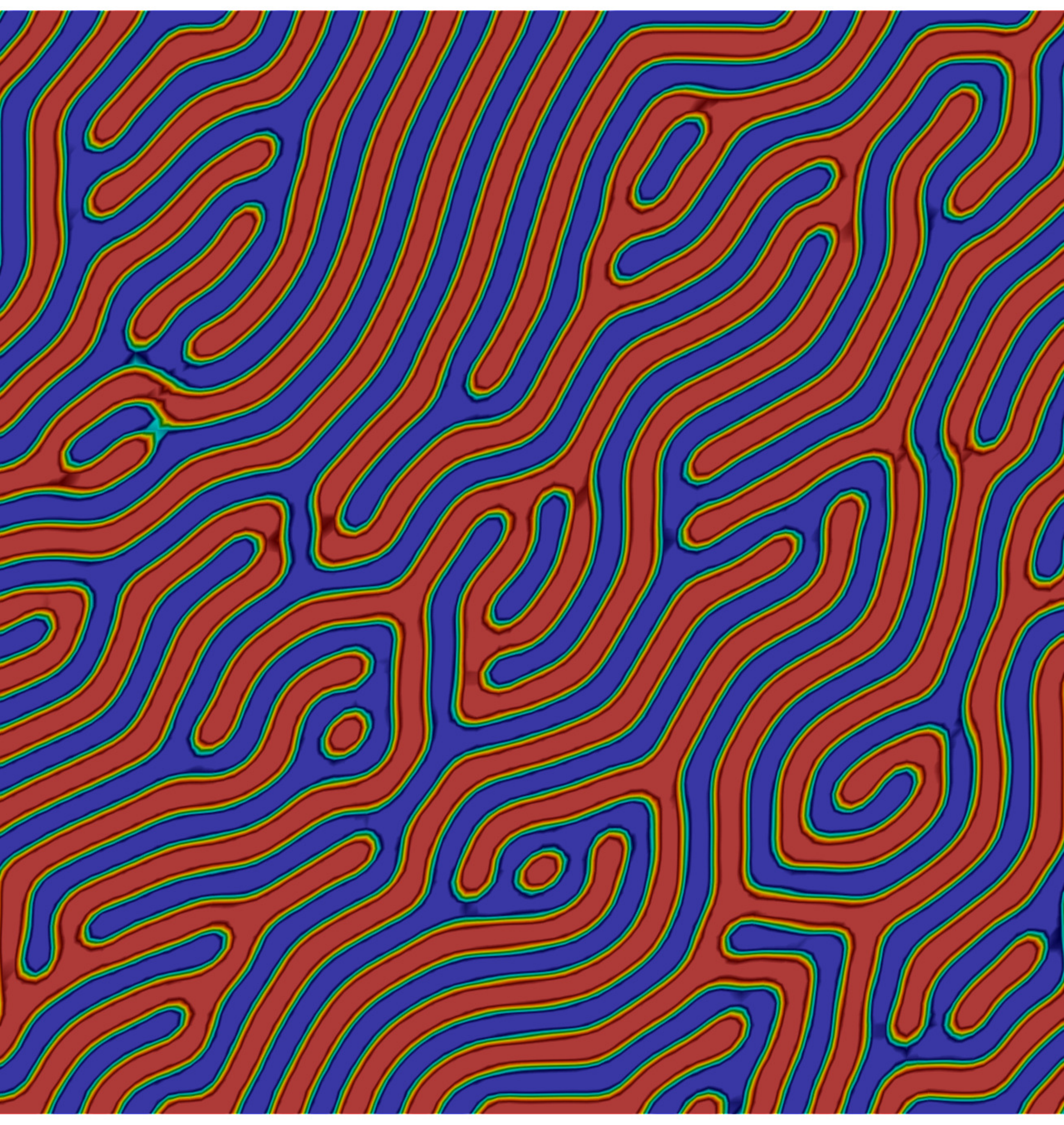}
\caption{Evolution of $\phi$ at times $t=0, 0.05, 0.15, 0.35$ and $0.5$ (from left to right) taking $M=0.1$, $g_2=1$, $g_0=-4$, $h_0=0.5$ and $\beta=1$ with $\lambda=0.1$ (top row), $\lambda=0.05$ (center row) and $\lambda=0.01$ (bottom row).} \label{fig:lambdas}
\end{center}
\end{figure}

\begin{figure}[h]
\begin{center}
\includegraphics[width=0.45\textwidth]{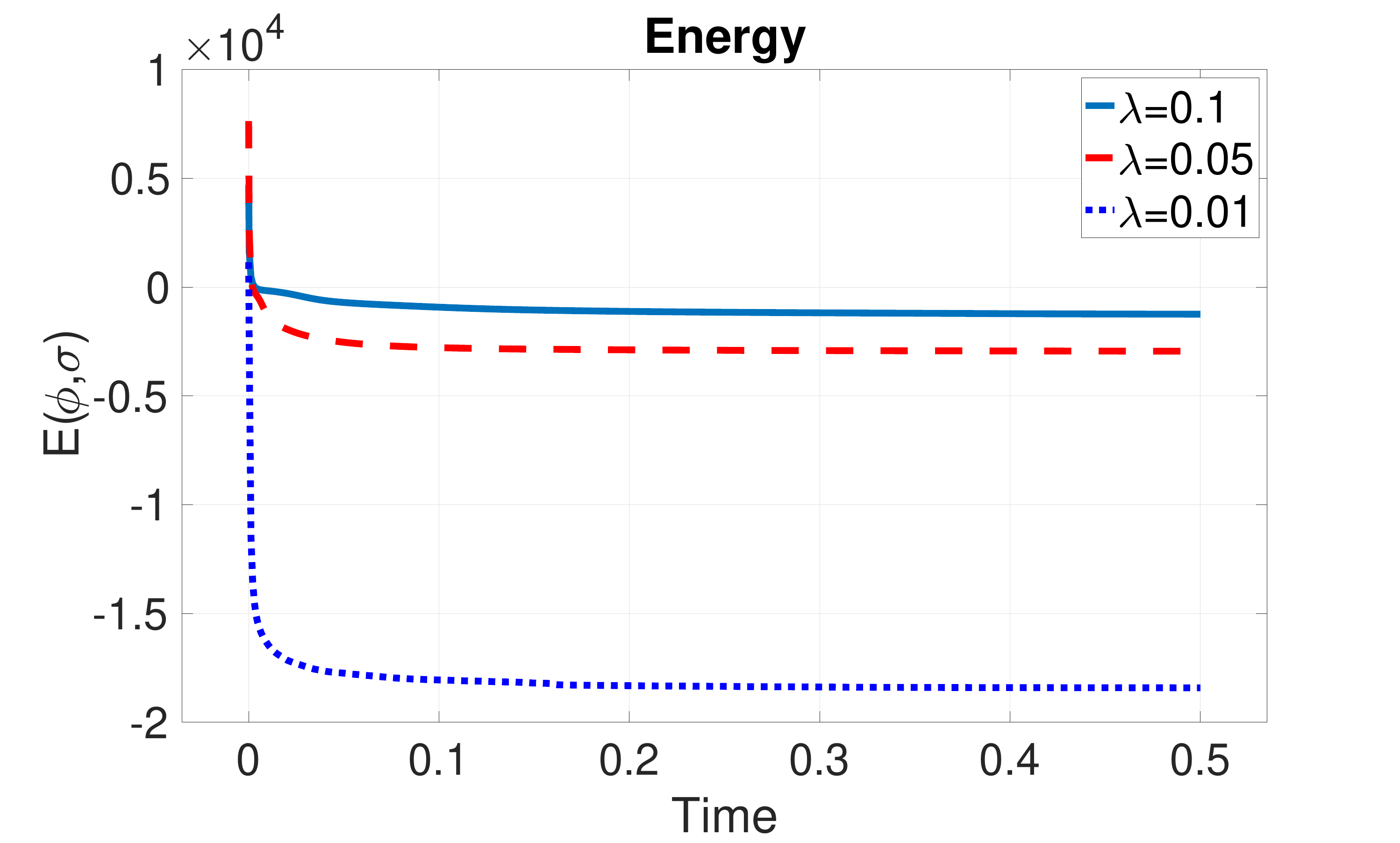}
\includegraphics[width=0.45\textwidth]{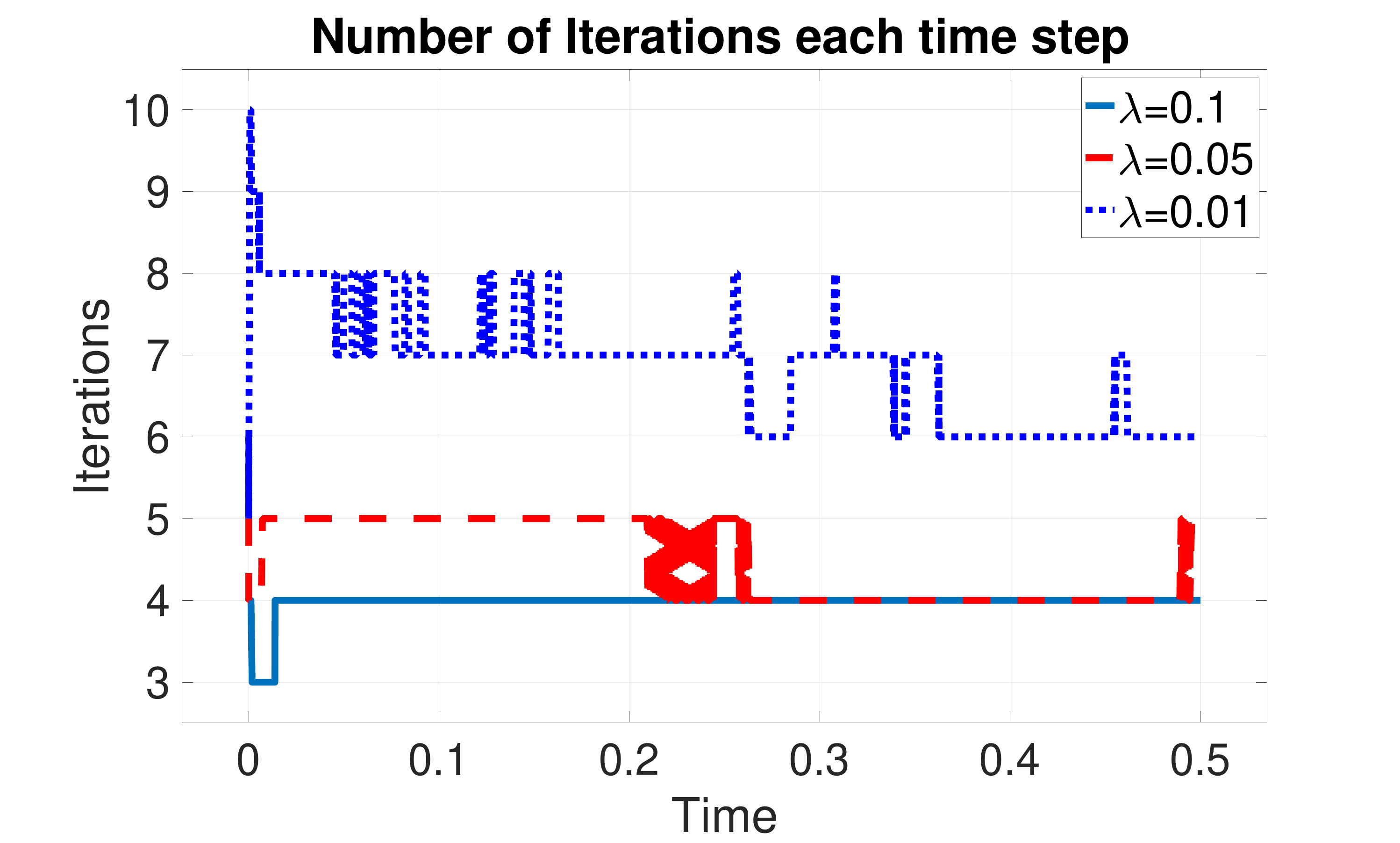}
\\
\includegraphics[width=0.45\textwidth]{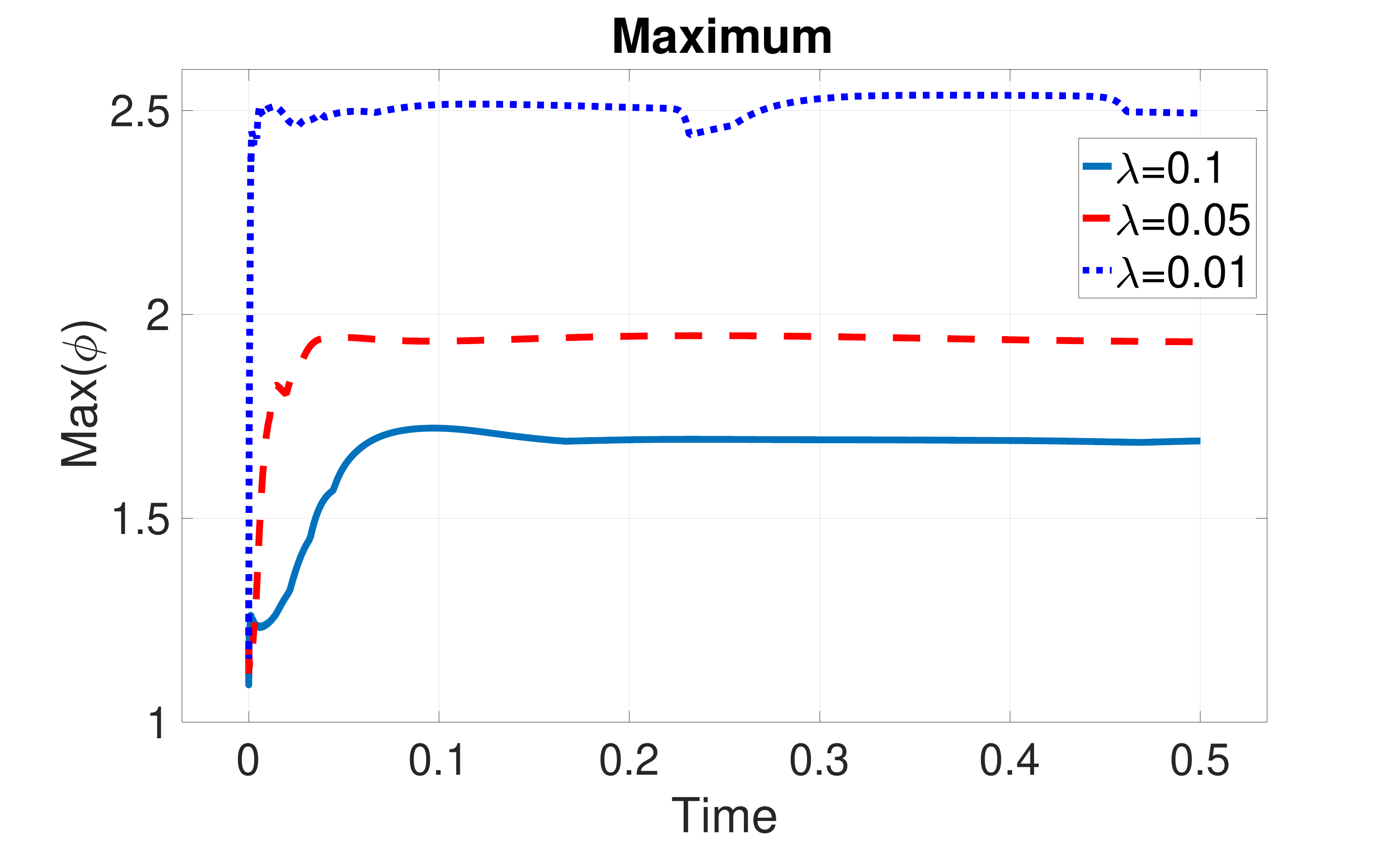}
\includegraphics[width=0.45\textwidth]{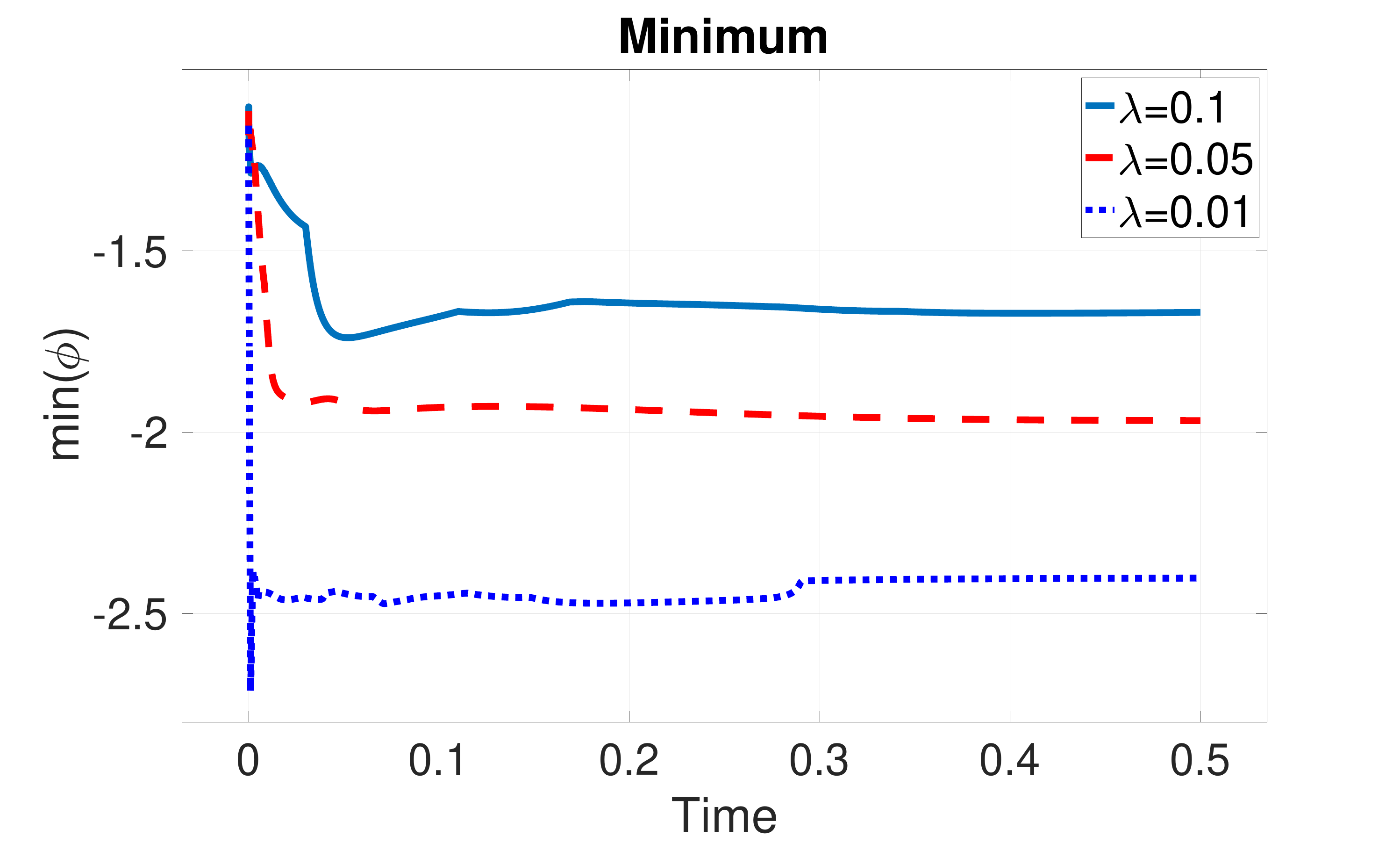}
\caption{Evolution in time of the energies (top left), the number of iterations to achieve tolerance $\texttt{TOL}=10^{-7}$ (top right), maximum of $\phi$ (bottom left) and minimum of $\phi$ (bottom right) taking $M=0.1$, $g_2=1$, $g_0=-4$, $h_0=0.5$ and $\beta=1$ and $\lambda=0.1, 0.05$ and $0.01$.} \label{fig:lambdasplot}
\end{center}
\end{figure}

\subsubsection{Study of influence of considering the choice $\beta=\lambda^{-1}$}

We have observed in subsection~\ref{sub:beta} that high values of parameter $\beta$ are related with $\phi$ being able to achieve values that are close to the minima of $f_0(\phi)$. At the same time, we have seen in subsection~\ref{sub:lambda} that the value of parameter $\lambda$ is associated with the width of the interface, but with the undesired effect of thinner interfaces producing that the values of $\phi$ go away from the minima of $f_0(\phi)$. In order to produce results where the reduction of the width of the interface does not produce the values of $\phi$ being too far away of the minima of $f_0(\phi)$, we propose to link the choice of these two parameters such that $\beta=\lambda^{-1}$. 
In this example we fix the parameters to
$M=0.1$, $g_2=1$, $g_0=-4$ and $h_0=0.5$ and we consider different values of parameter $\beta=\lambda^{-1}$, namely $\beta=\lambda^{-1}=10$, $\beta=\lambda^{-1}=20$ and $\beta=\lambda^{-1}=50$. The dynamics associated to these simulations are presented in Figure~\ref{fig:lambdabetas} and the evolution of the energies, maximum of $\phi$, minimum of $\phi$ and number of iterations are presented in Figure~\ref{fig:lambdabetasplot}. 
We can observe how now decreasing the value of $\lambda$ still reduces the amount of interface in the system; however, this undesirable effect is mitigated with the increasing value of $\beta$  and as observed in the bottom row of subfigures in Figure~\ref{fig:lambdabetasplot}, the maximum and minimum of $\phi$ stay closer to the physical interval $[-1,1]$ than previously observed. %not reaching values far away from the minima $f_0(\phi)$.
In all the simulations, the energies decrease as expected, and there is the same mild effect on the value of $\lambda$ and the number of iterations needed to achieve the required tolerance in the iterative algorithm that we observed in subsection~\ref{sub:lambda}. 
%As a close this subsection, 
%We would like to conclude this subsection % that in order to linking the values of these two para 
We conclude this subsection with the observation that linking the values of these two parameters is needed to keep $\phi$ relatively close to the physical interval $[-1,1]$.
% in order to obtain physically relevant results, where the solution of $\phi$ should not go too far away of the physical interval 

\begin{figure}[h]
\begin{center}
%\includegraphics[width=0.19\textwidth]{images/Ex2/lambdabeta/Ex2_lambdabeta_1_0}
%\includegraphics[width=0.19\textwidth]{images/Ex2/lambdabeta/Ex2_lambdabeta_1_50}
%\includegraphics[width=0.19\textwidth]{images/Ex2/lambdabeta/Ex2_lambdabeta_1_150}
%\includegraphics[width=0.19\textwidth]{images/Ex2/lambdabeta/Ex2_lambdabeta_1_350}
%\includegraphics[width=0.19\textwidth]{images/Ex2/lambdabeta/Ex2_lambdabeta_1_500}

%\\
\includegraphics[width=0.19\textwidth]{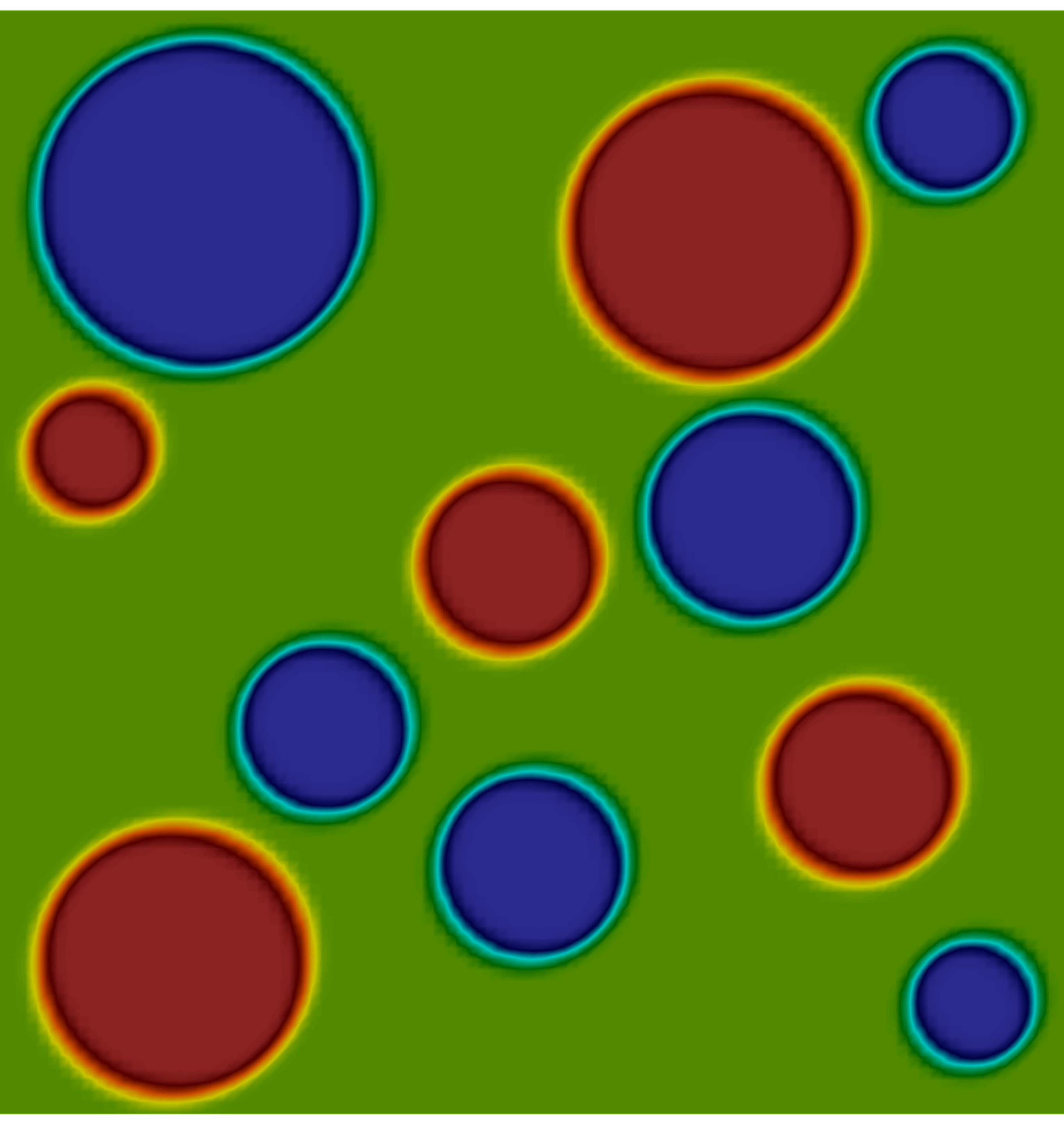}
\includegraphics[width=0.19\textwidth]{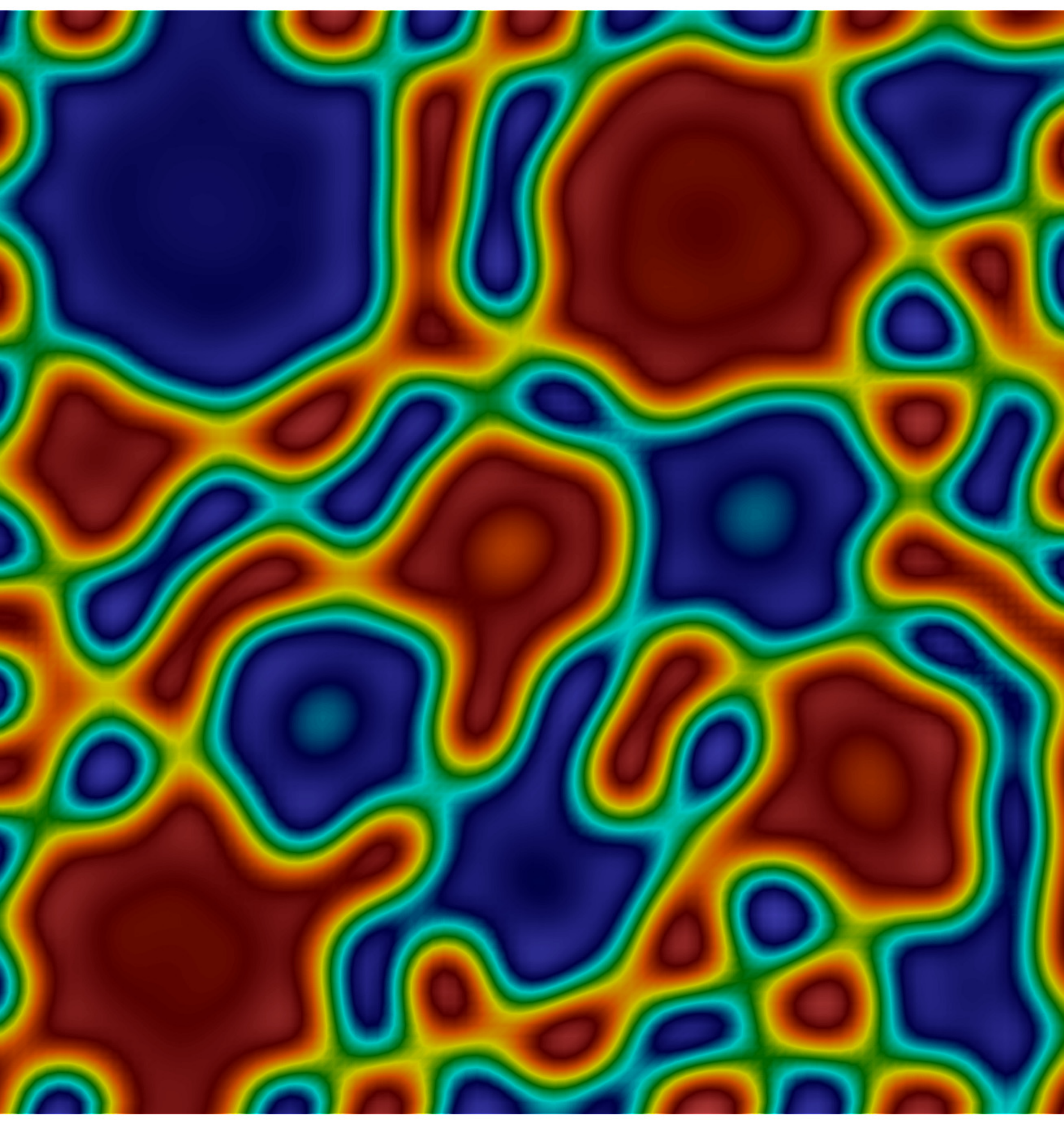}
\includegraphics[width=0.19\textwidth]{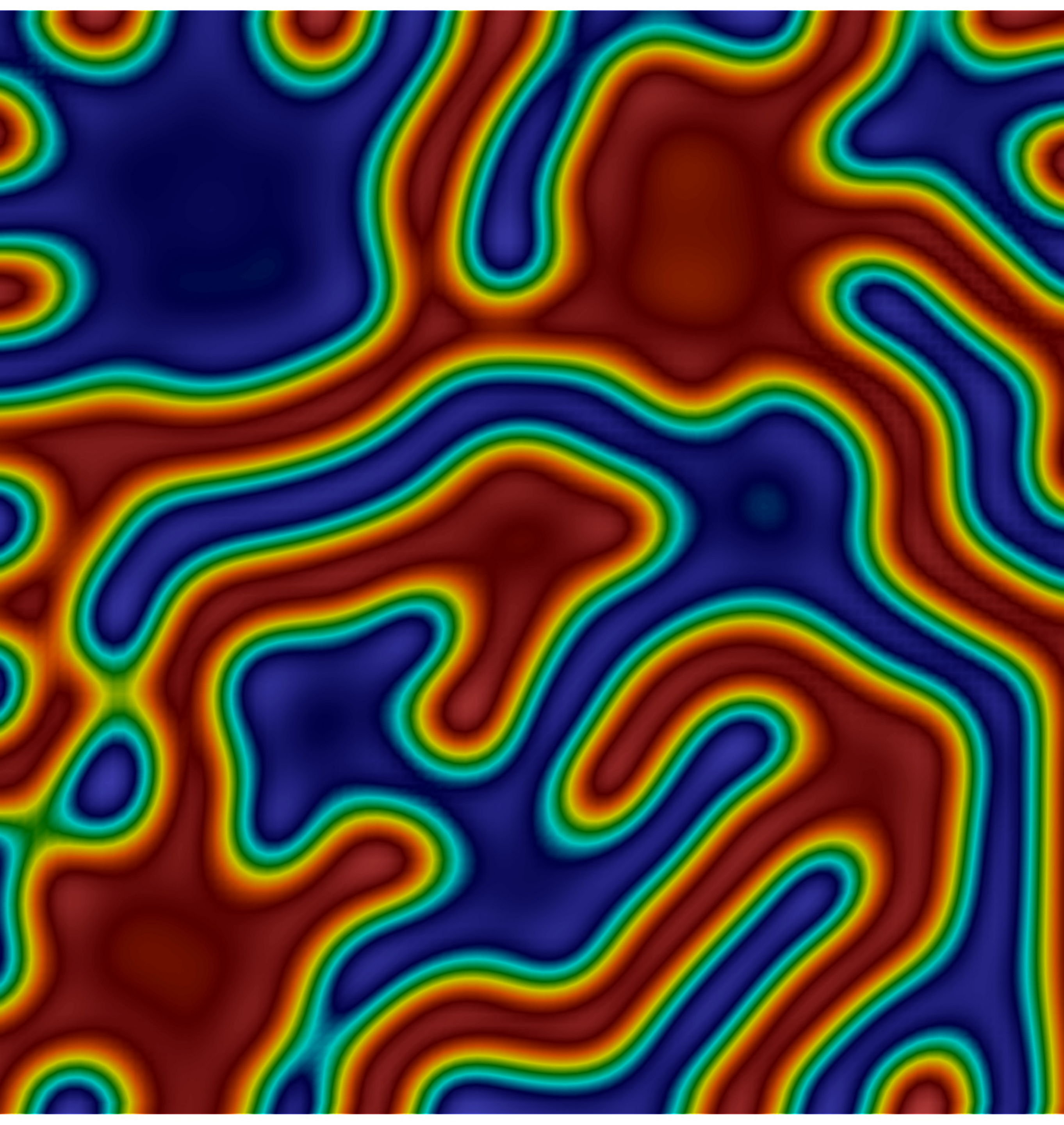}
\includegraphics[width=0.19\textwidth]{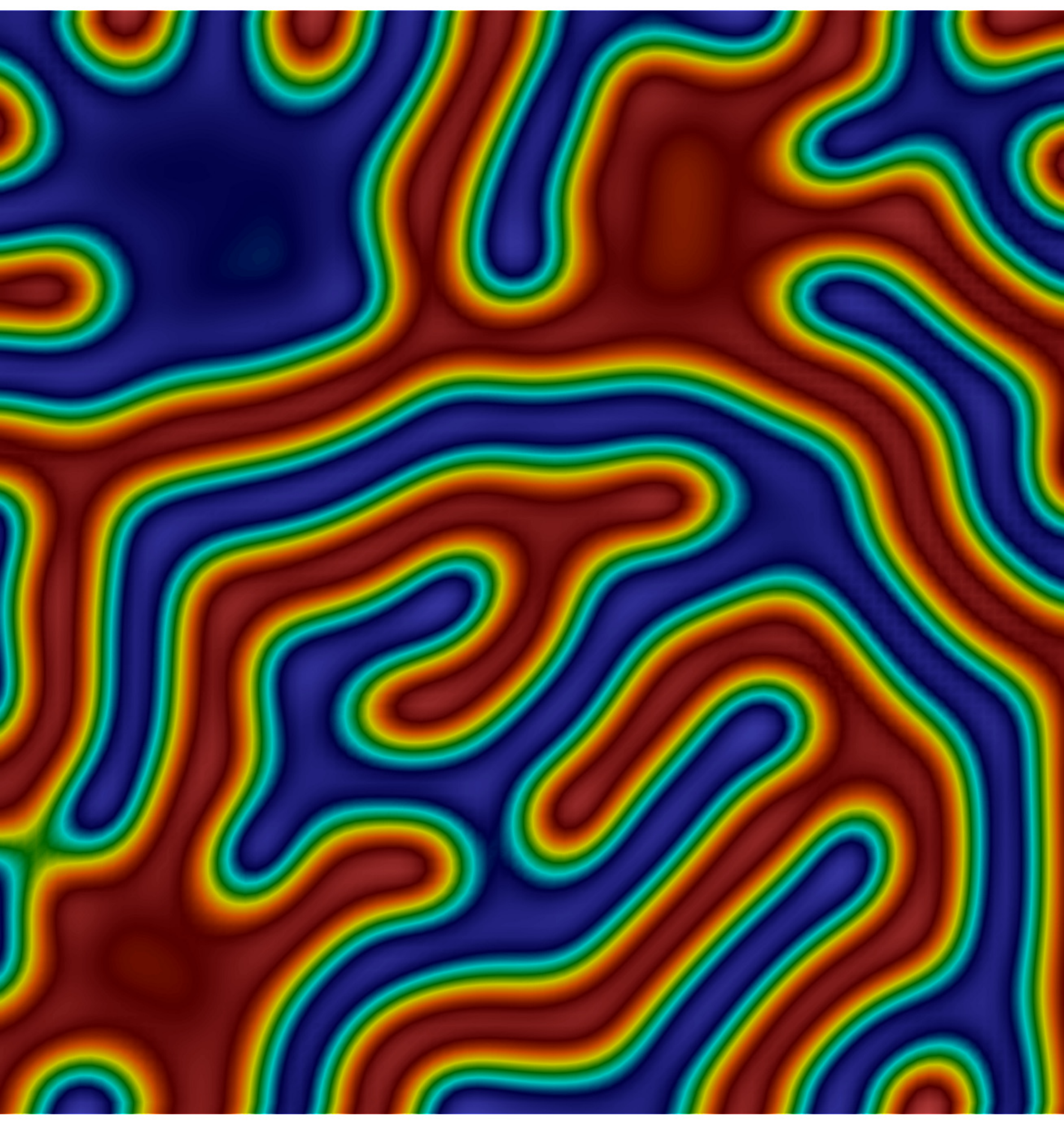}
\includegraphics[width=0.19\textwidth]{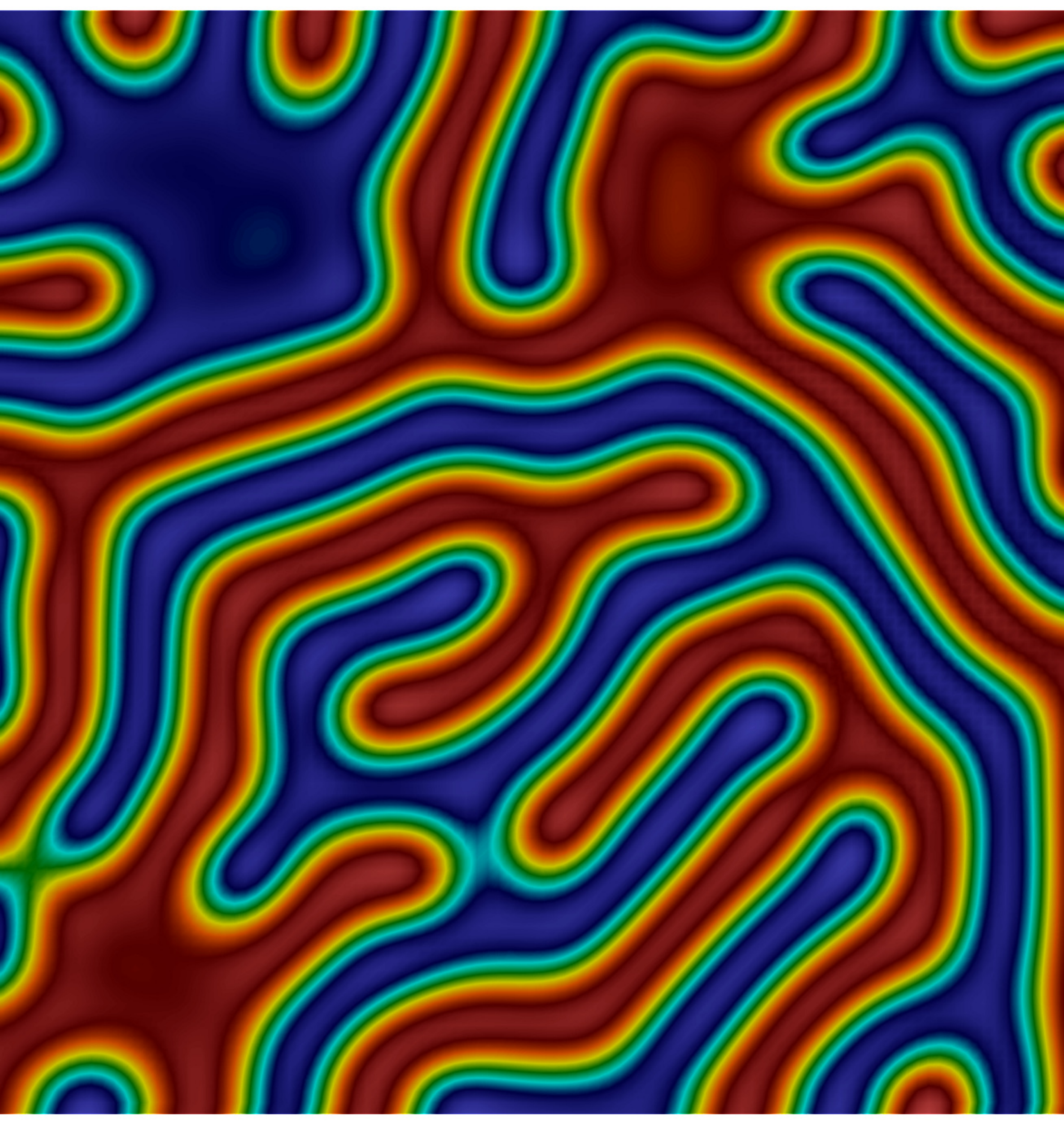}
\\
\includegraphics[width=0.19\textwidth]{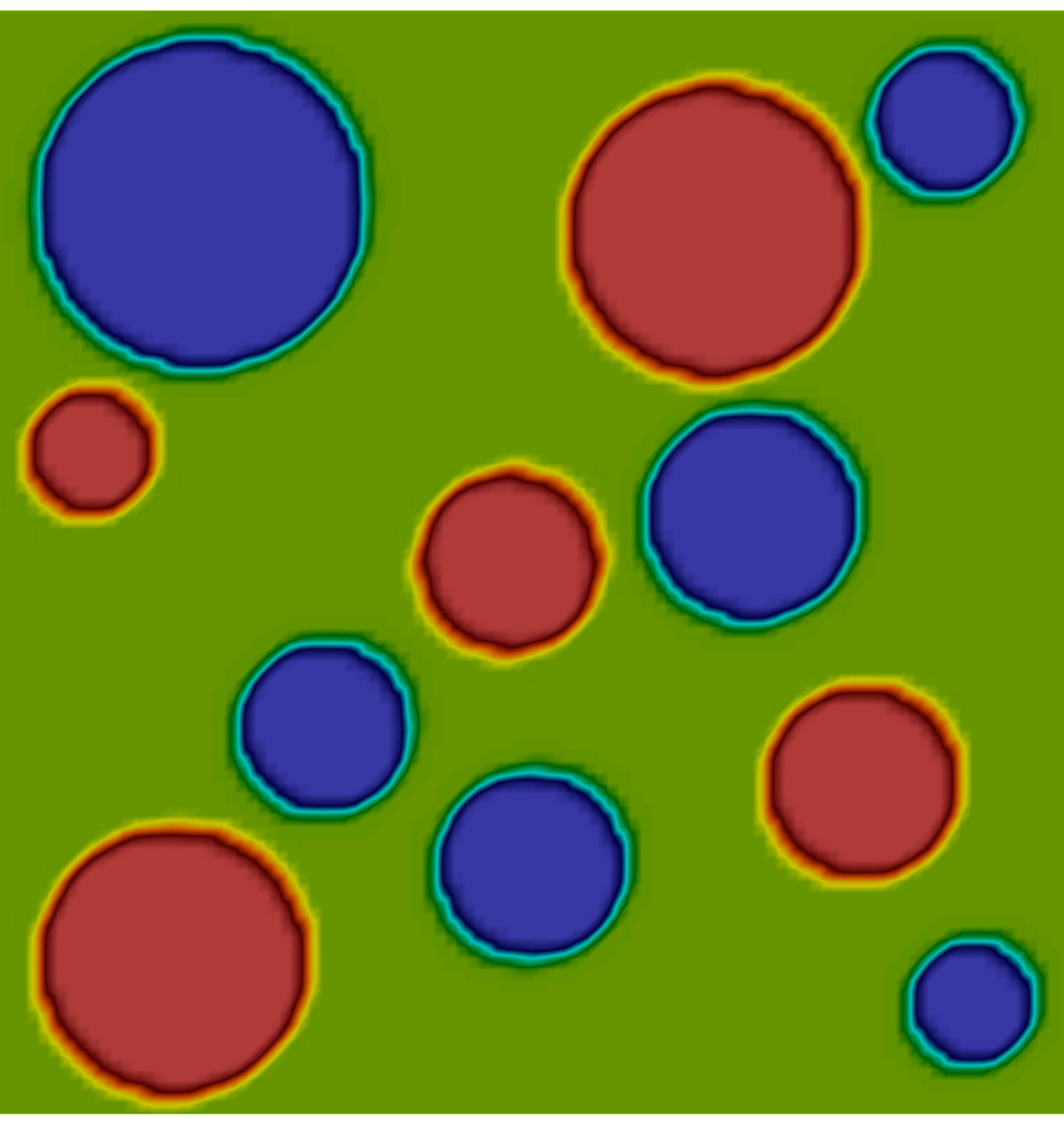}
\includegraphics[width=0.19\textwidth]{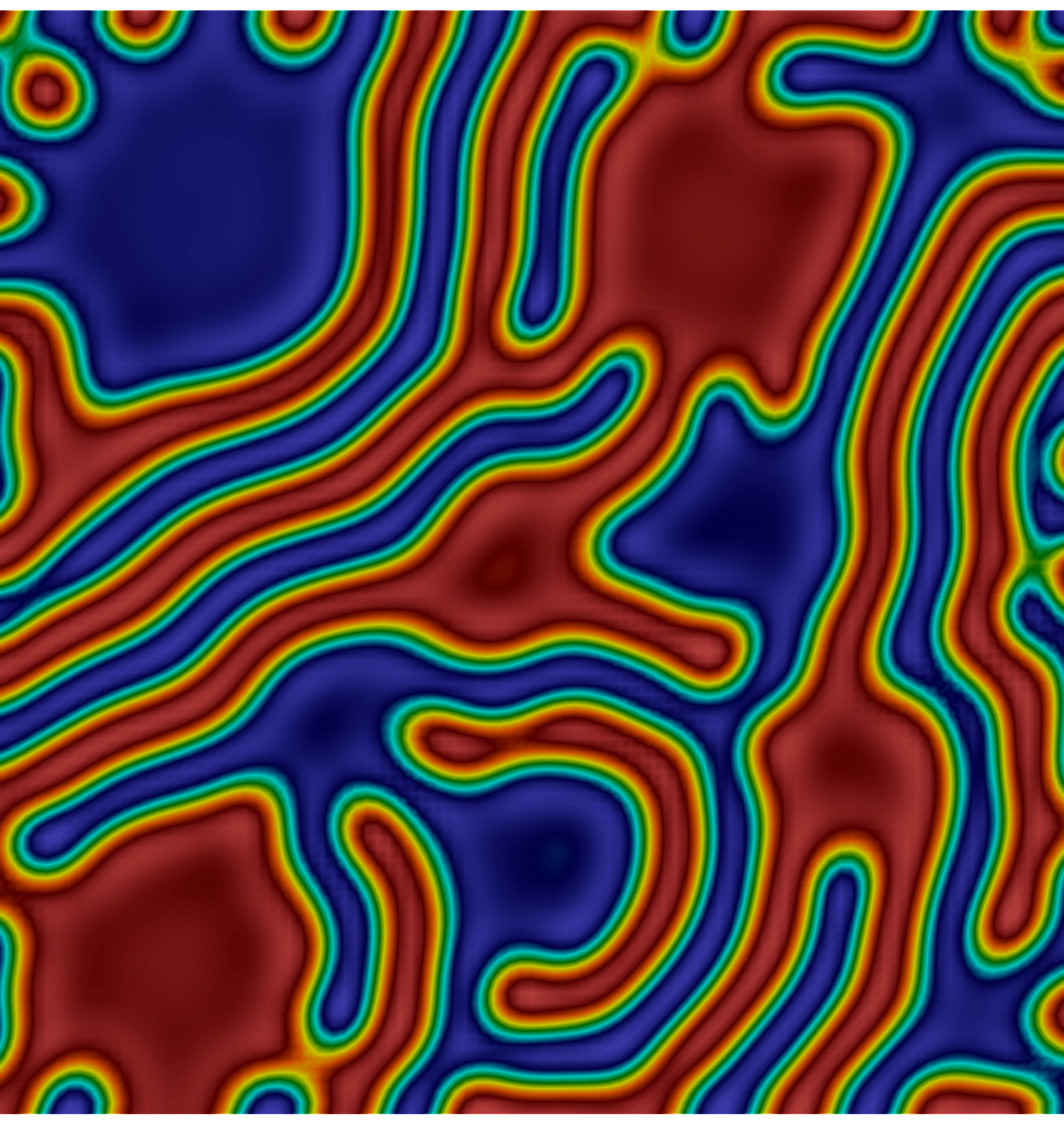}
\includegraphics[width=0.19\textwidth]{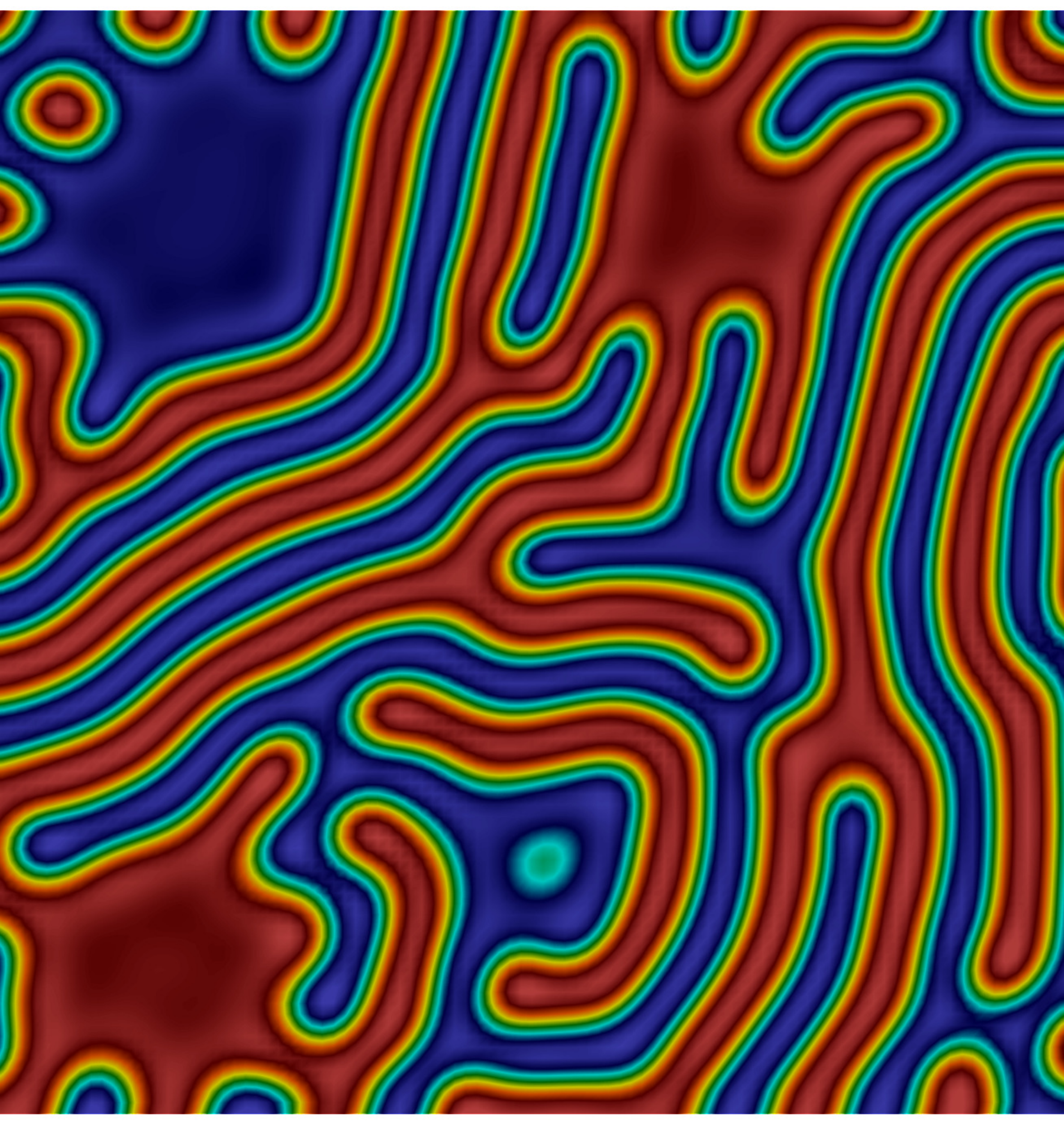}
\includegraphics[width=0.19\textwidth]{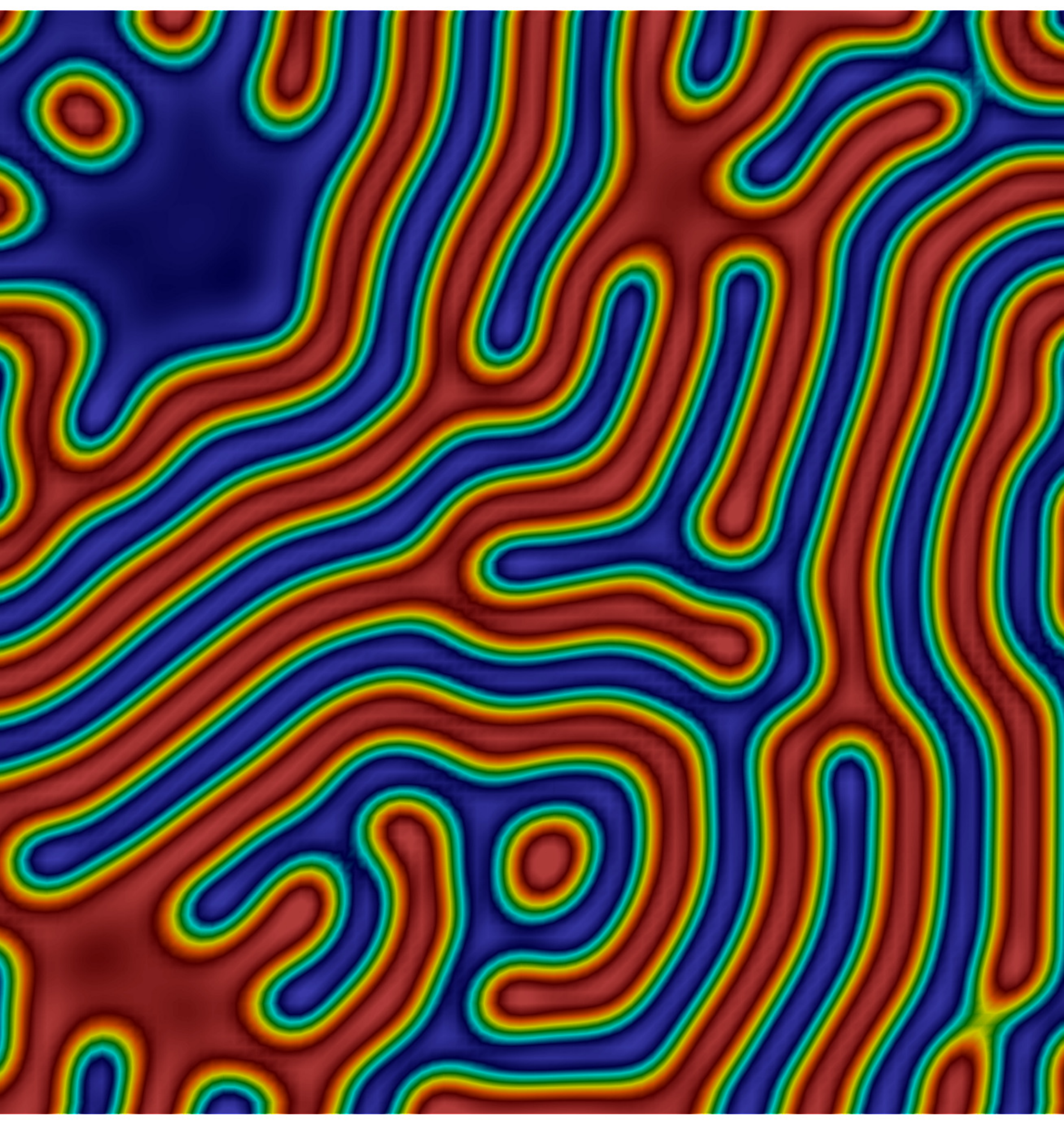}
\includegraphics[width=0.19\textwidth]{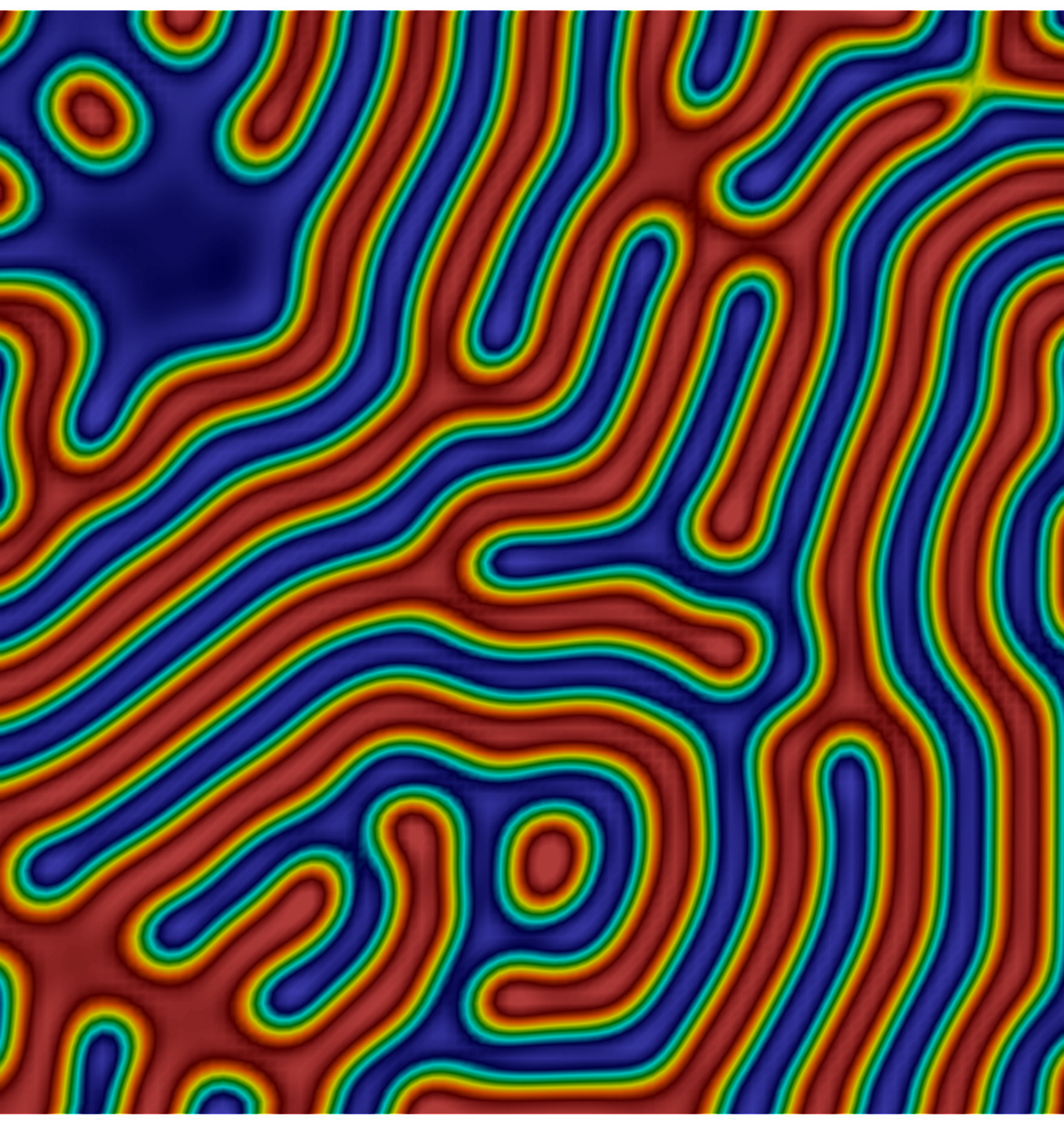}
\\
\includegraphics[width=0.19\textwidth]{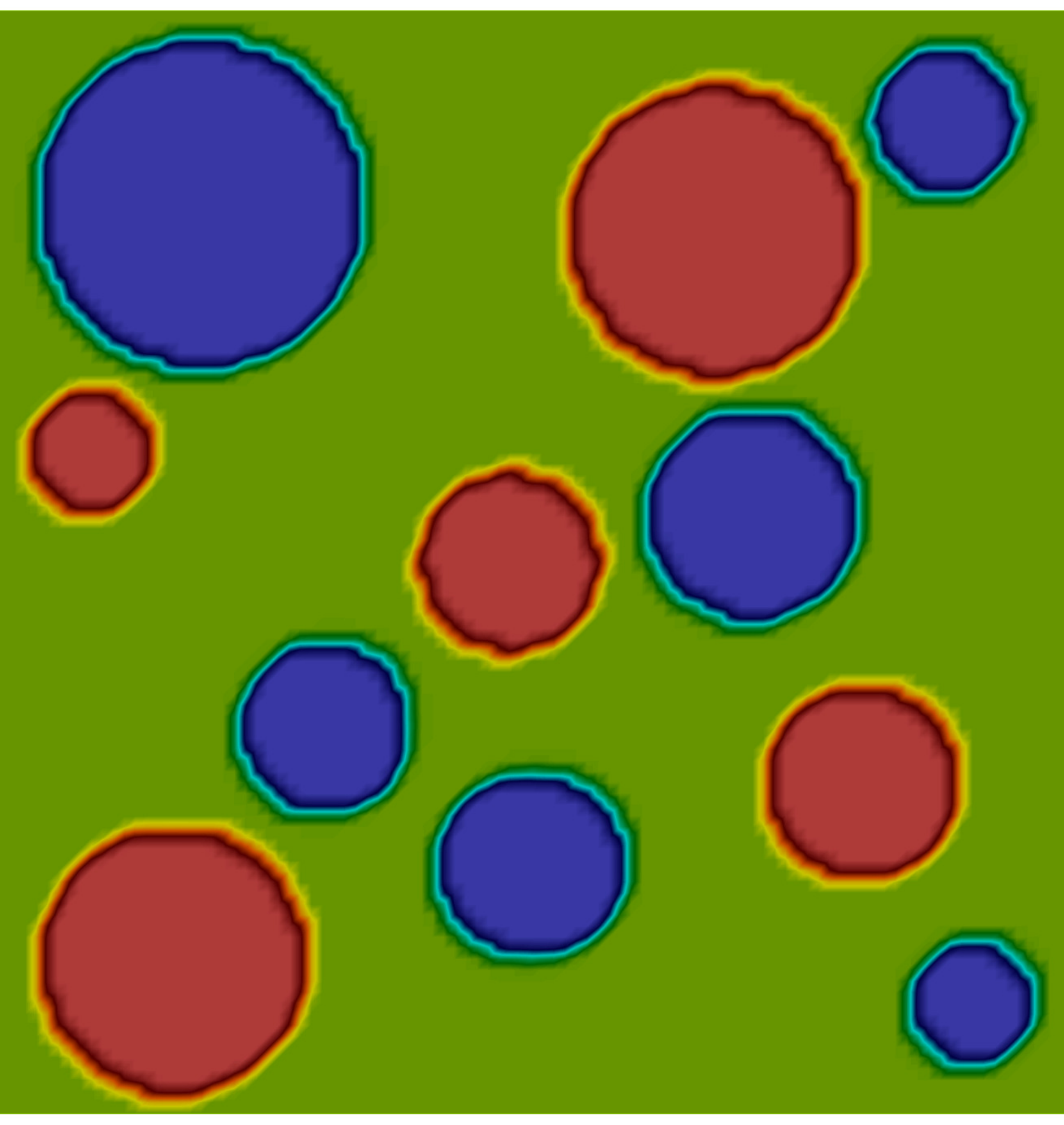}
\includegraphics[width=0.19\textwidth]{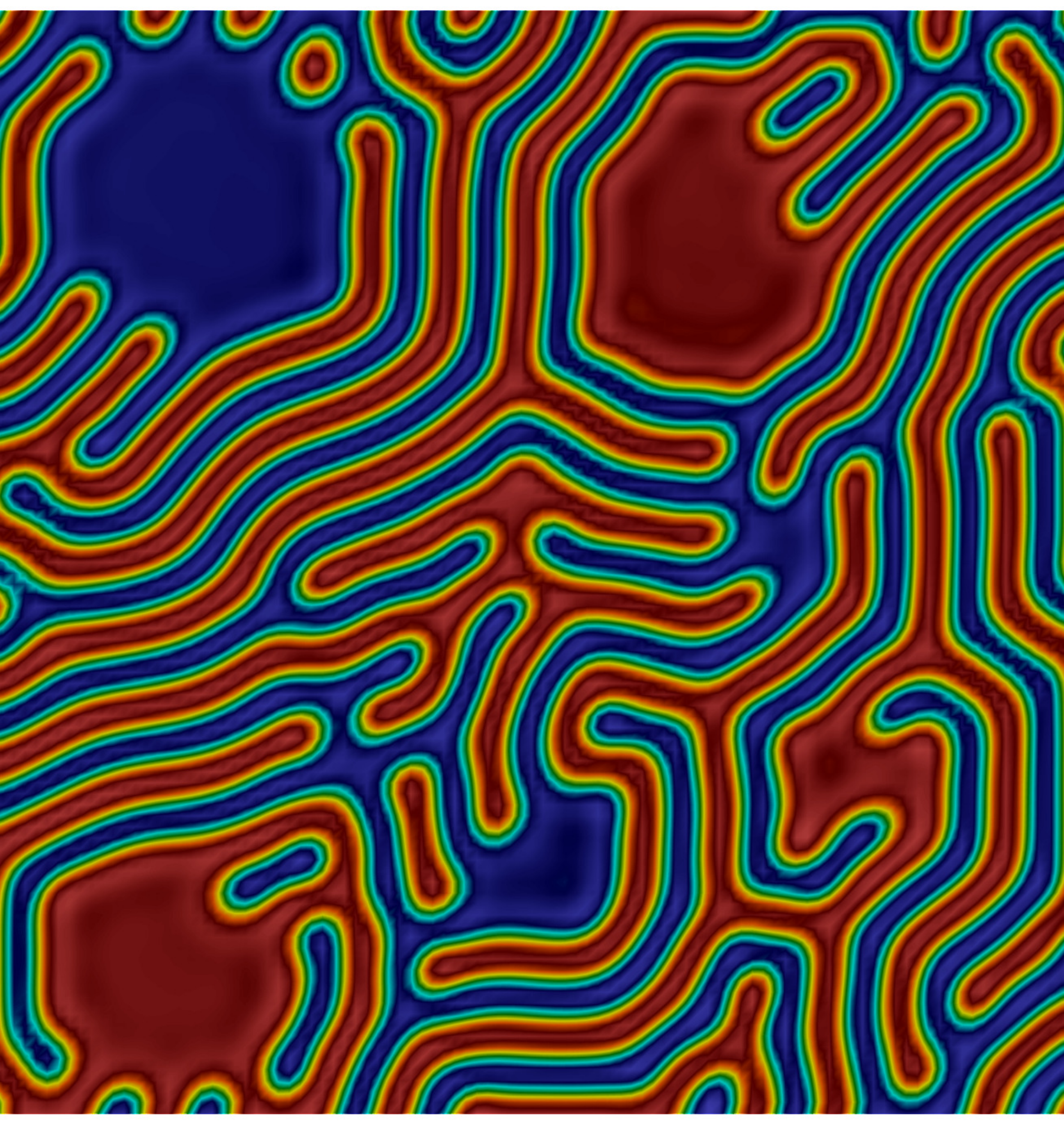}
\includegraphics[width=0.19\textwidth]{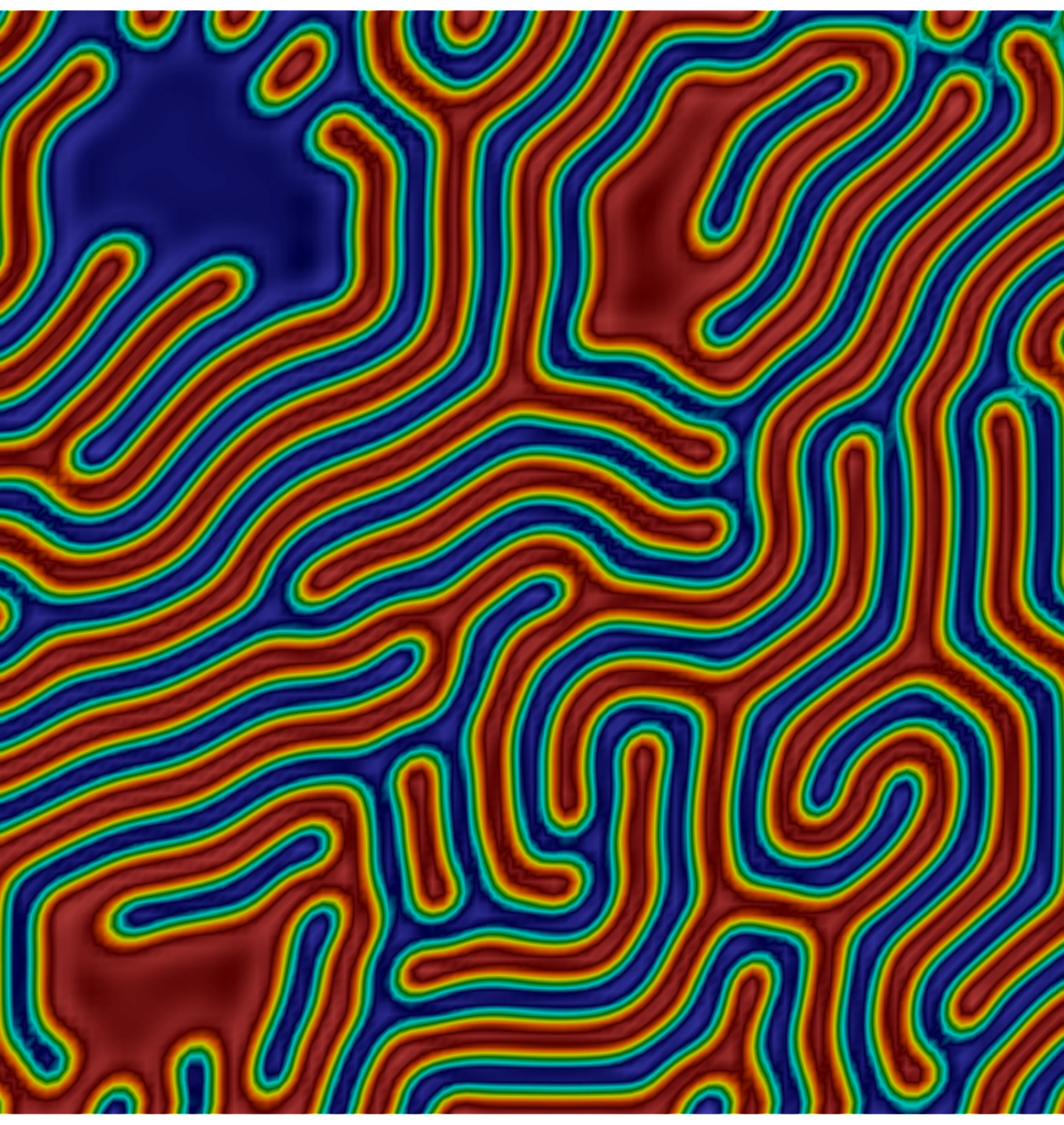}
\includegraphics[width=0.19\textwidth]{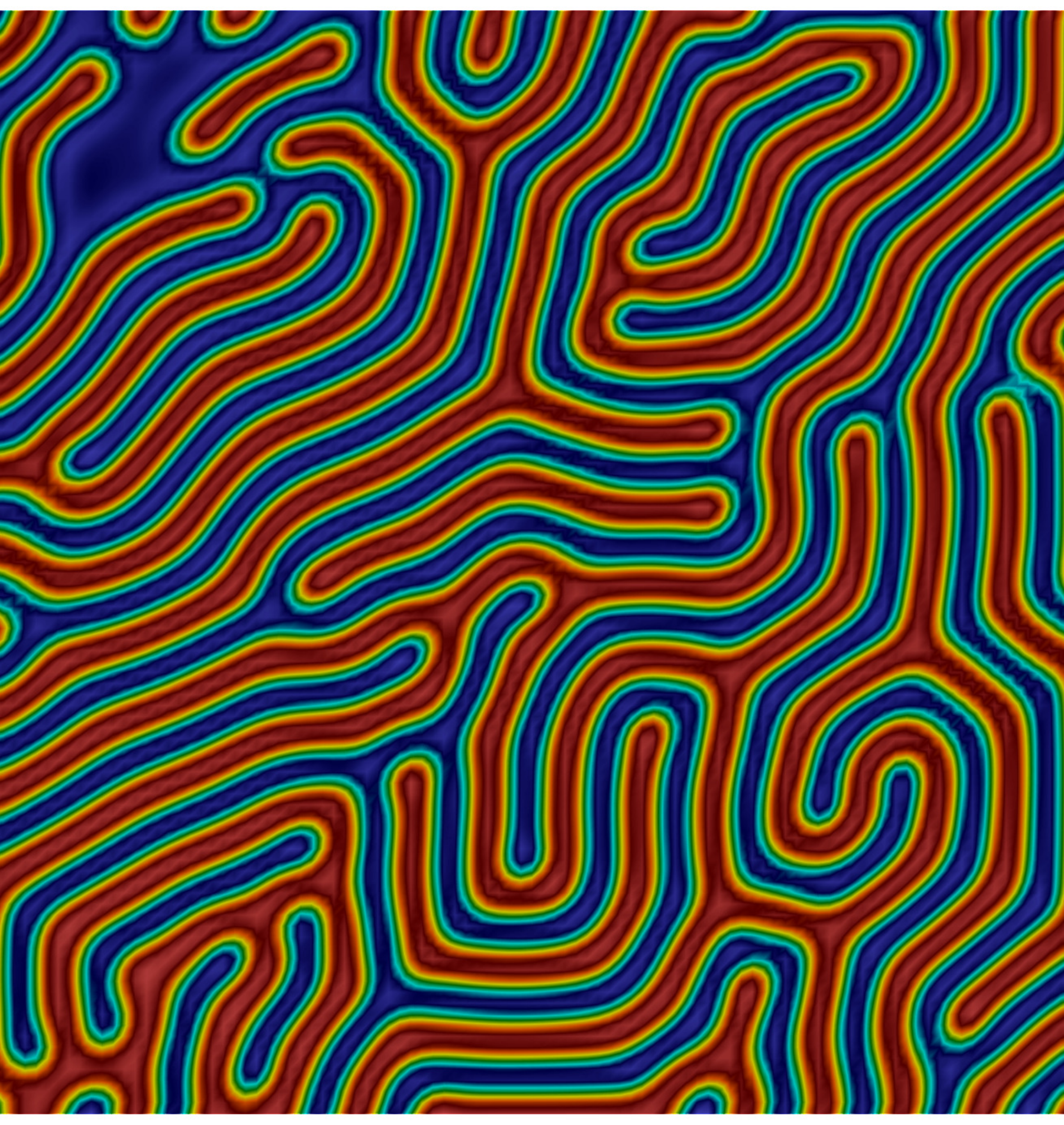}
\includegraphics[width=0.19\textwidth]{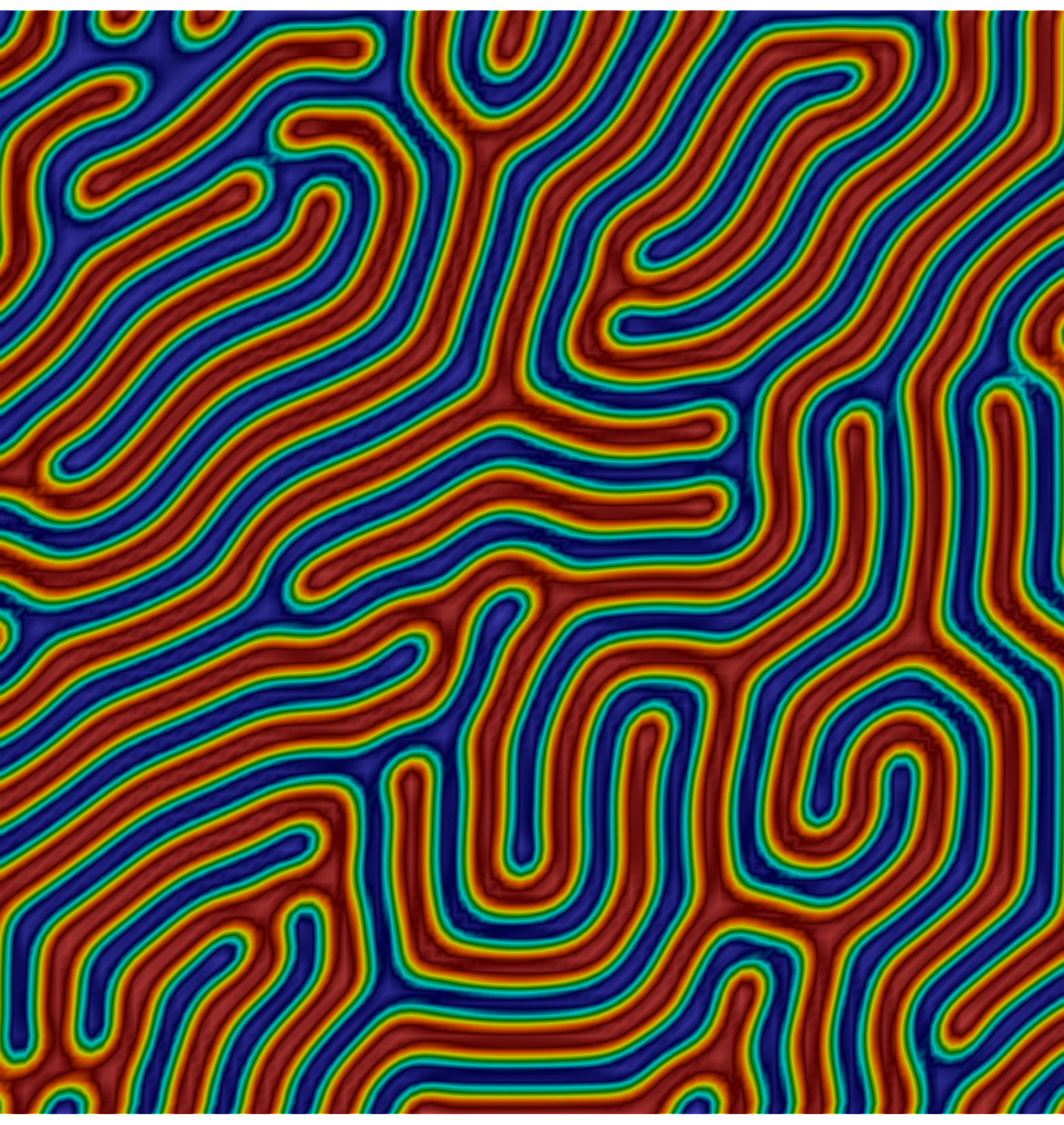}
\caption{Evolution of $\phi$ at times $t=0, 0.05, 0.15, 0.35$ and $0.5$ (from left to right) taking $M=0.1$, $g_2=1$, $g_0=-4$ and $h_0=0.5$ with $\beta=\lambda^{-1}=10$ (top row), $\beta=\lambda^{-1}=20$ (center row) and $\beta=\lambda^{-1}=50$ (bottom row).} \label{fig:lambdabetas}
\end{center}
\end{figure}

\begin{figure}[h]
\begin{center}
\includegraphics[width=0.45\textwidth]{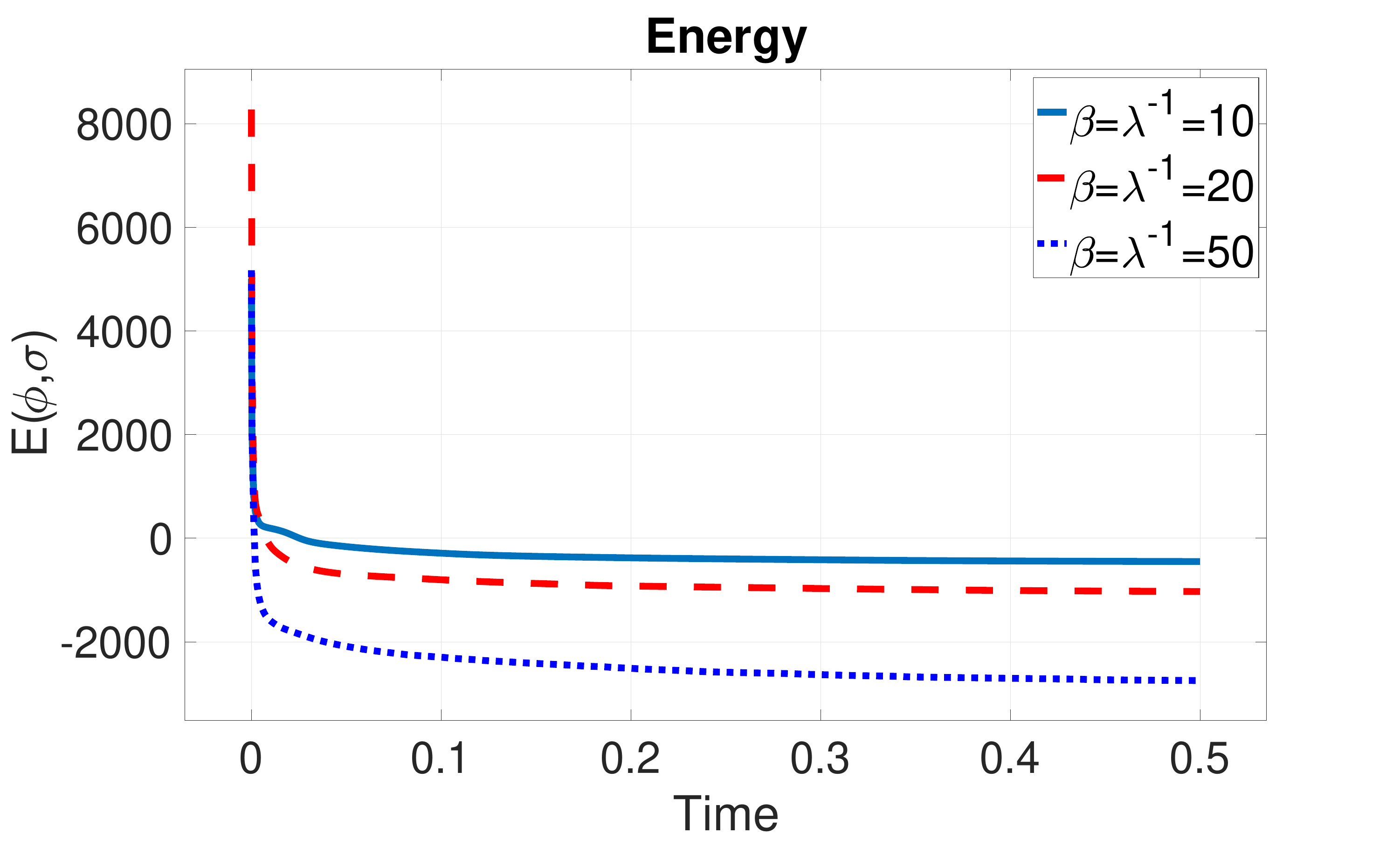}
\includegraphics[width=0.45\textwidth]{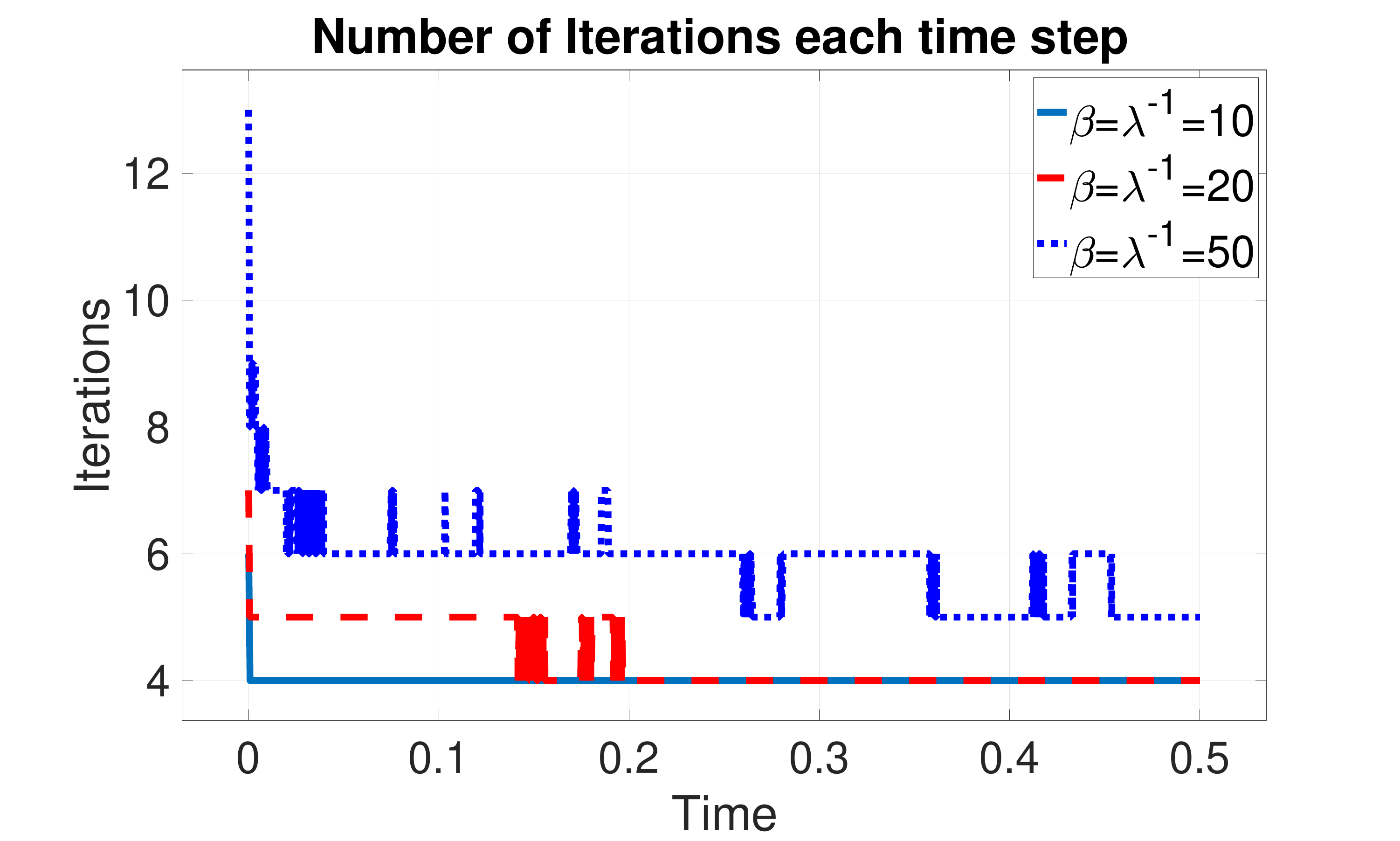}
\\
\includegraphics[width=0.45\textwidth]{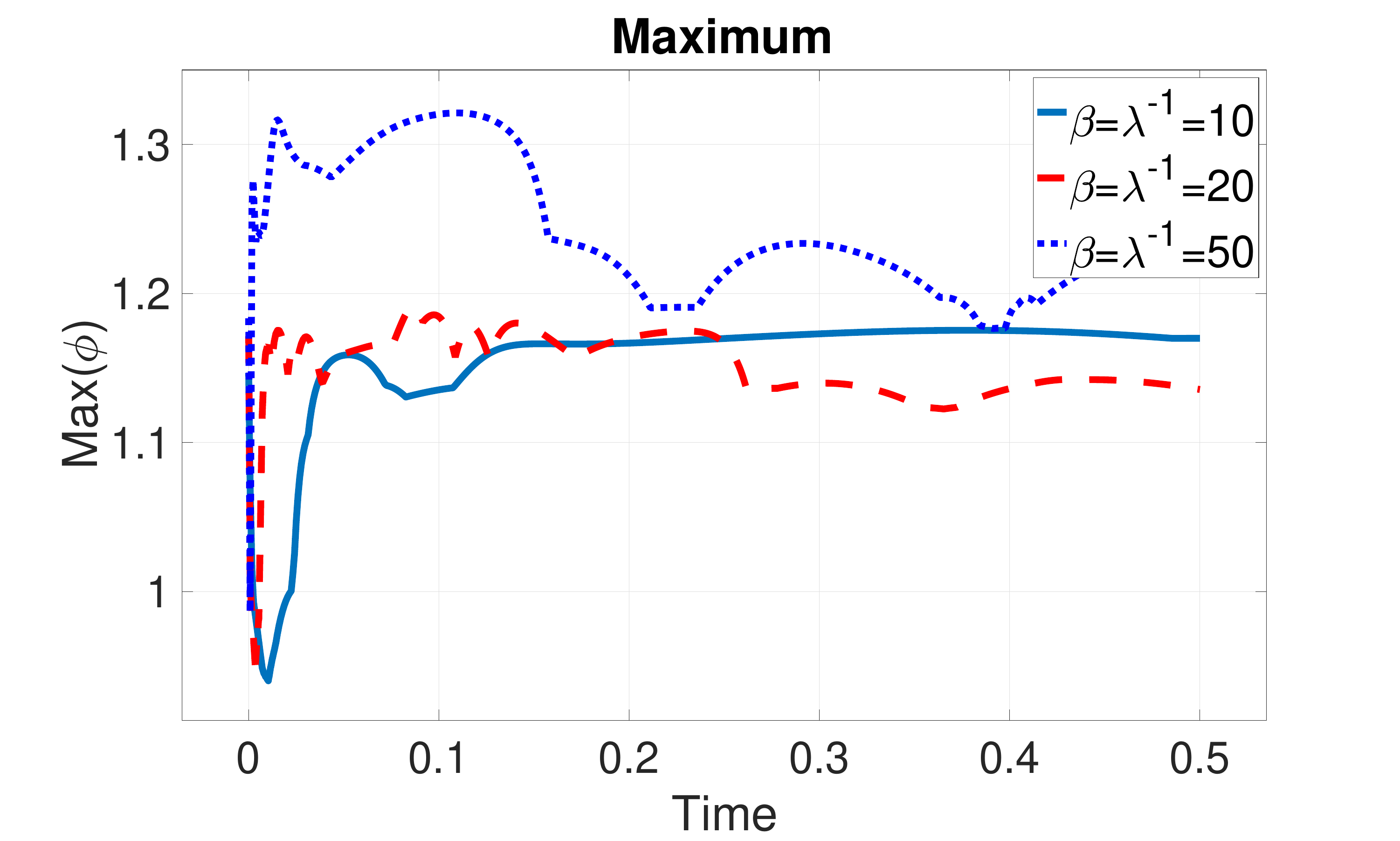}
\includegraphics[width=0.45\textwidth]{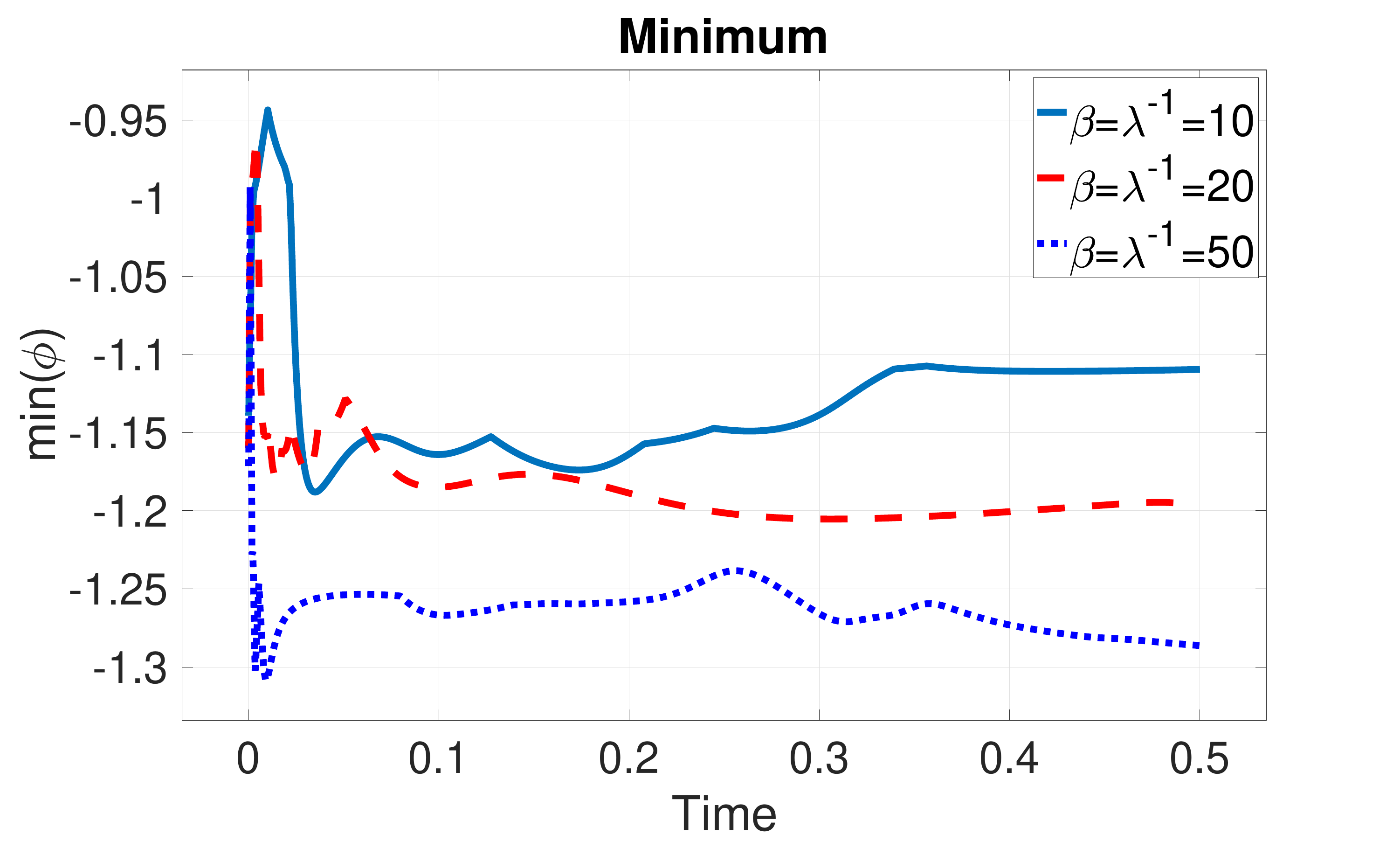}
\caption{Evolution in time of the energies (top left), the number of iterations to achieve tolerance $\texttt{TOL}=10^{-7}$ (top right), maximum of $\phi$ (bottom left) and minimum of $\phi$ (bottom right) taking $M=0.1$, $g_2=1$, $g_0=-4$, $h_0=0.5$ and $\beta=\lambda^{-1}=10, 20$ and $50$.} \label{fig:lambdabetasplot}
\end{center}
\end{figure}

\subsection{3D simulations}
In this section, we present the results obtained by performing simulations in $3D$ to illustrate our numerical scheme's capability to capture relevant results independently of the dimension considered. In particular, we consider a spatial domain $\Omega={(}-2.5,2.5{)}\times
{(}-2.5,2.5{)}
\times{(}-0.5,0.5{)}$
with a structured tetrahedral mesh of $75\times75\times15$.
The choice of the physical parameters are
$M=0.1$, $g_0 = -4$,  $g_2=1$, $h_0=0.5$, $\lambda=0.01$, $\beta=5$, the time step is fixed to $\Delta t=10^{-5}$ and the tolerance of the iterative algorithm set to $\texttt{TOL}=10^{-5}$.
We consider as initial condition a configuration similar to the one used in the $2D$ examples, and we run the simulations over the time interval $[0,T]=[0,0.1]$. The dynamics associated to this simulation is presented in Figure~\ref{fig:3Ddynamics} and the evolution of the energies, the volume, the maximum and minimum of $\phi$ and number of iterations are presented in Figure~\ref{fig:3Dplot}. We observe how the proposed numerical scheme is able to capture the complicated dynamics of the system while minimizing the energy and conserving the volume.

\begin{figure}[h]
\begin{center}
\includegraphics[width=0.24\textwidth]{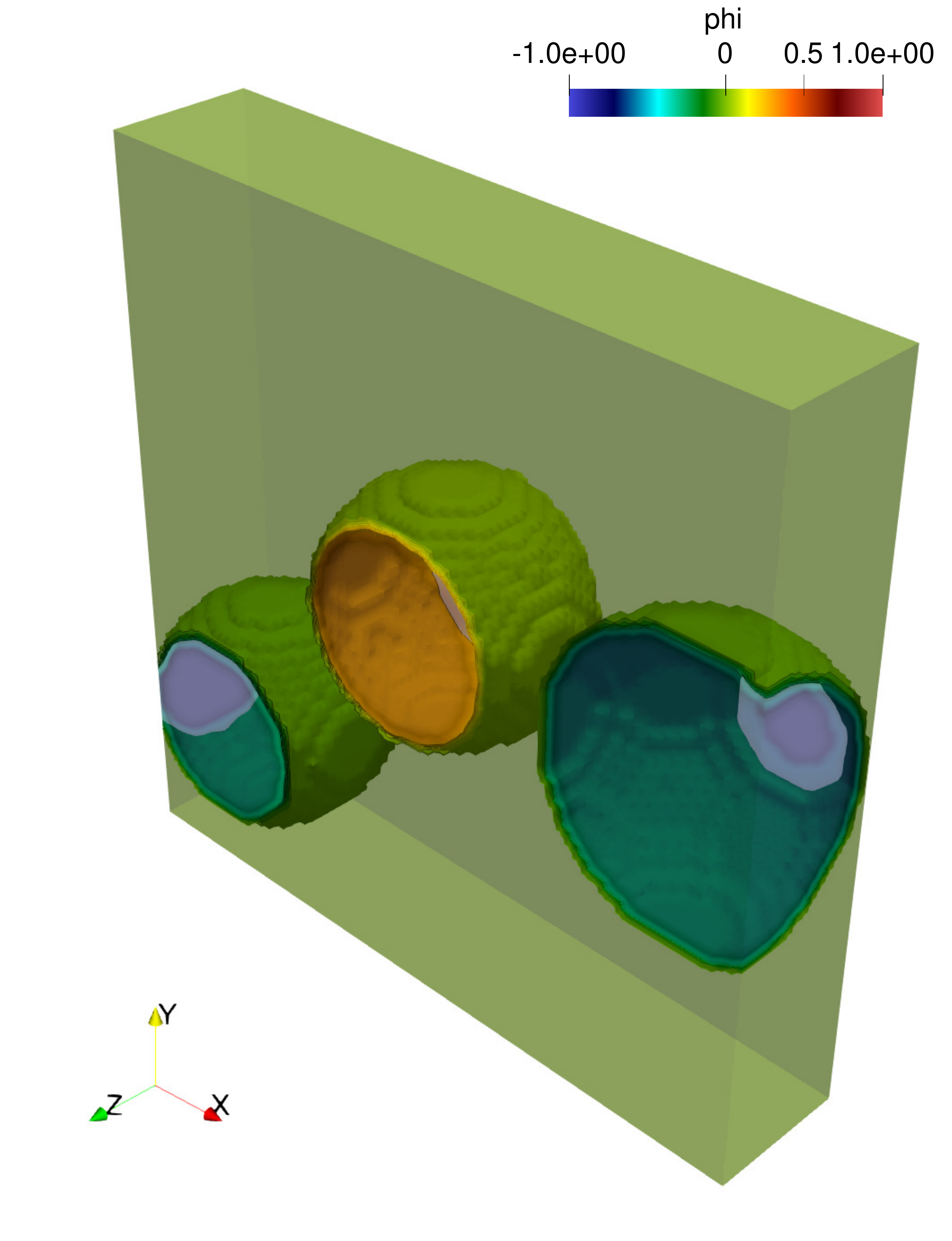}
\includegraphics[width=0.24\textwidth]{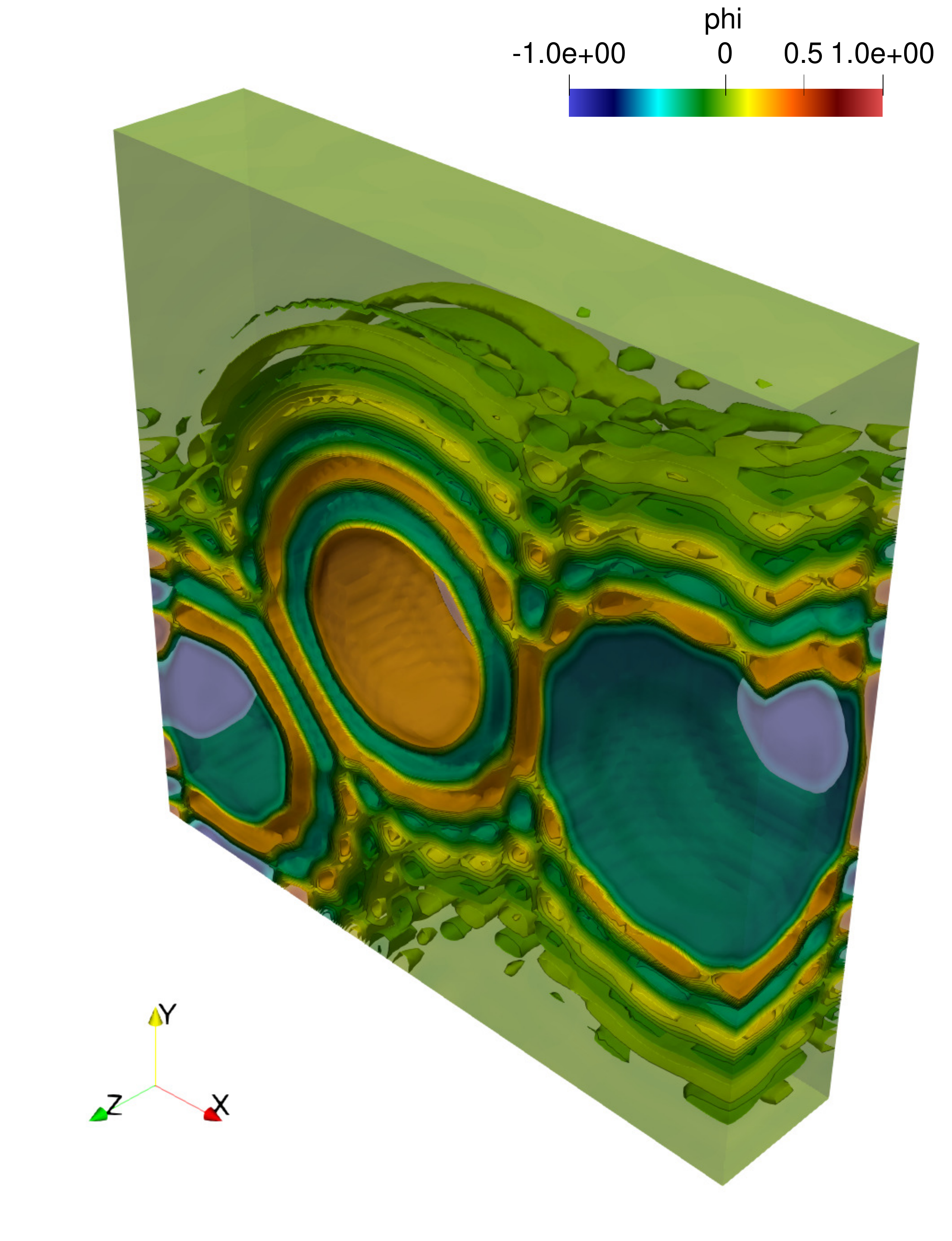}
\includegraphics[width=0.24\textwidth]{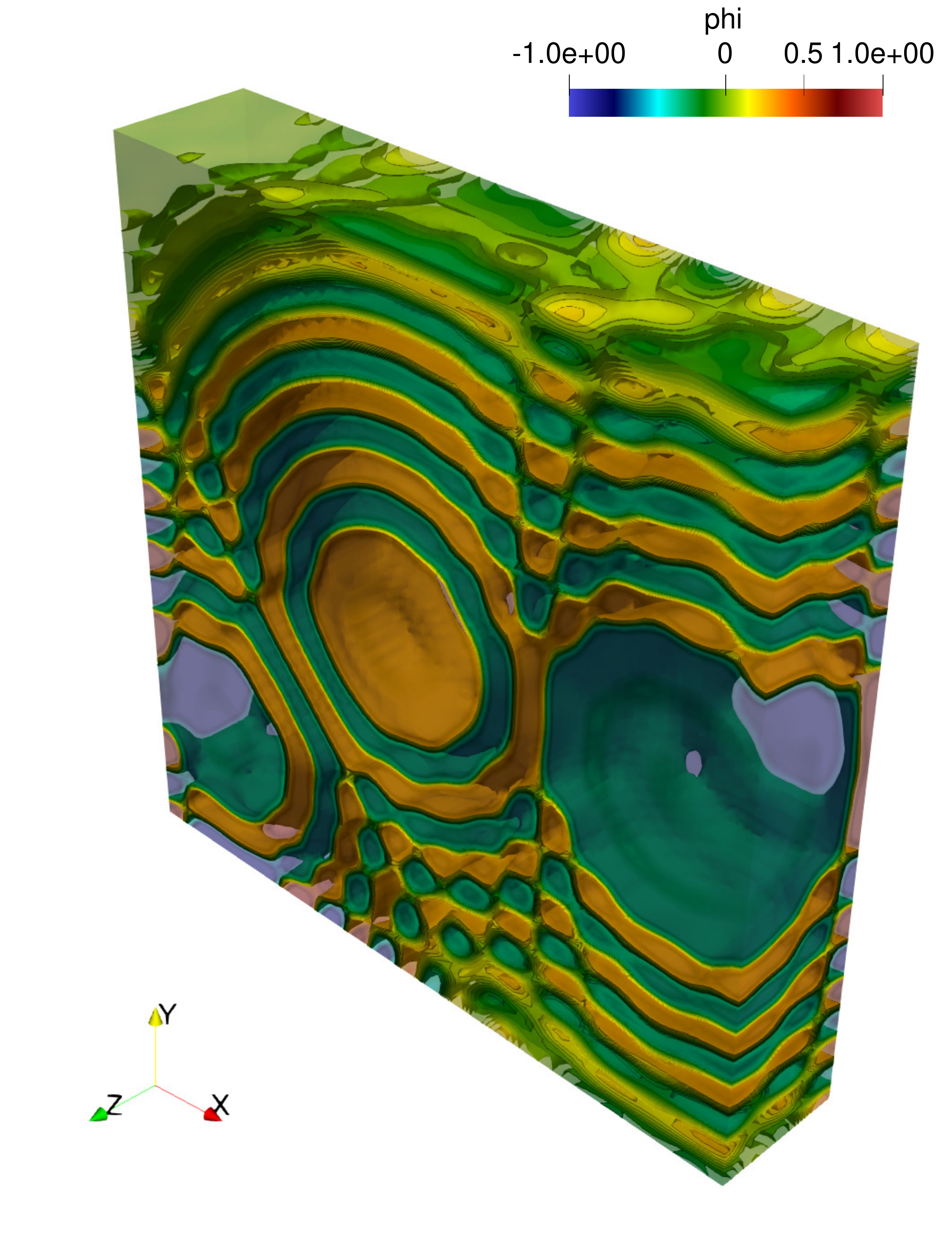}
\includegraphics[width=0.24\textwidth]{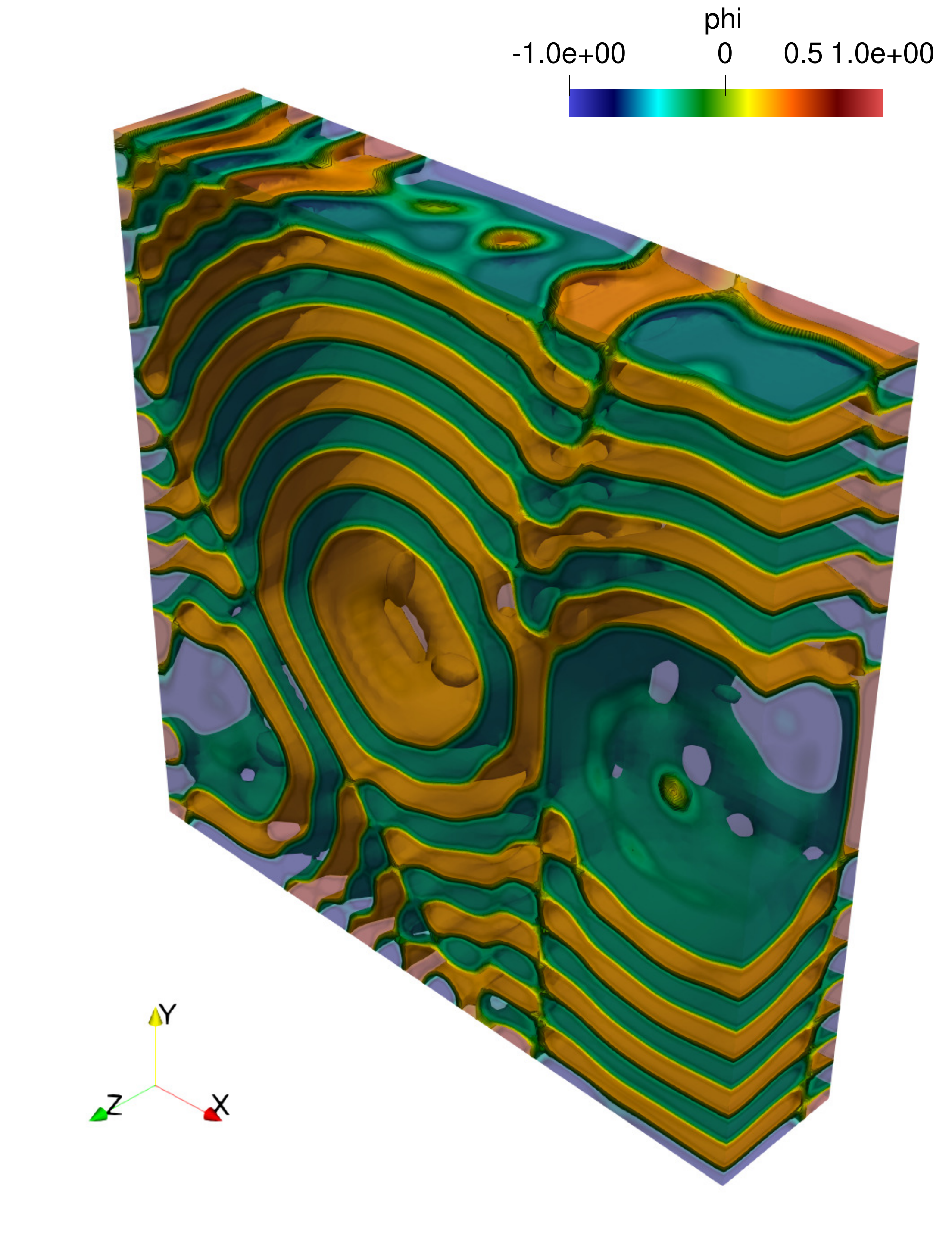}
\\
\includegraphics[width=0.24\textwidth]{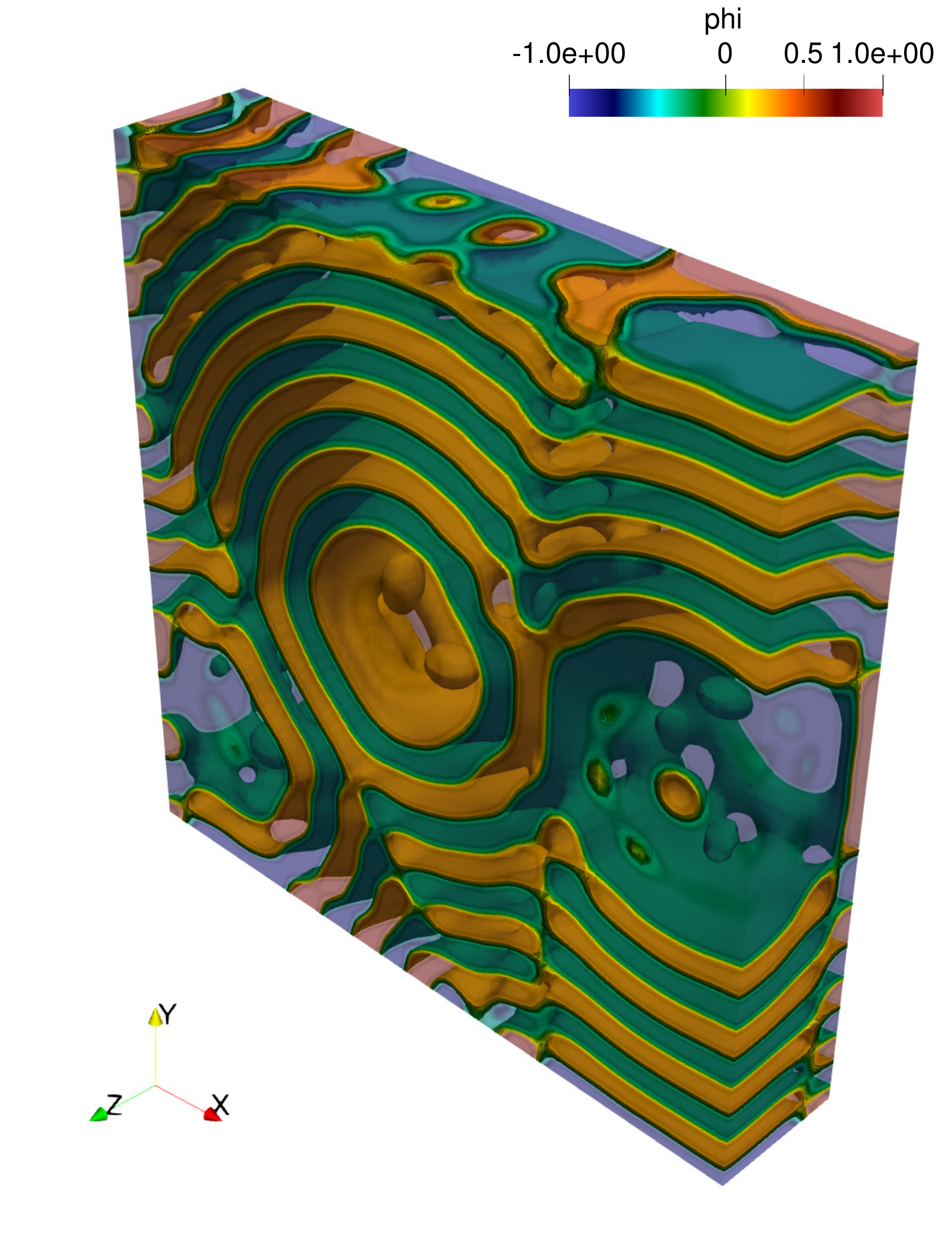}
\includegraphics[width=0.24\textwidth]{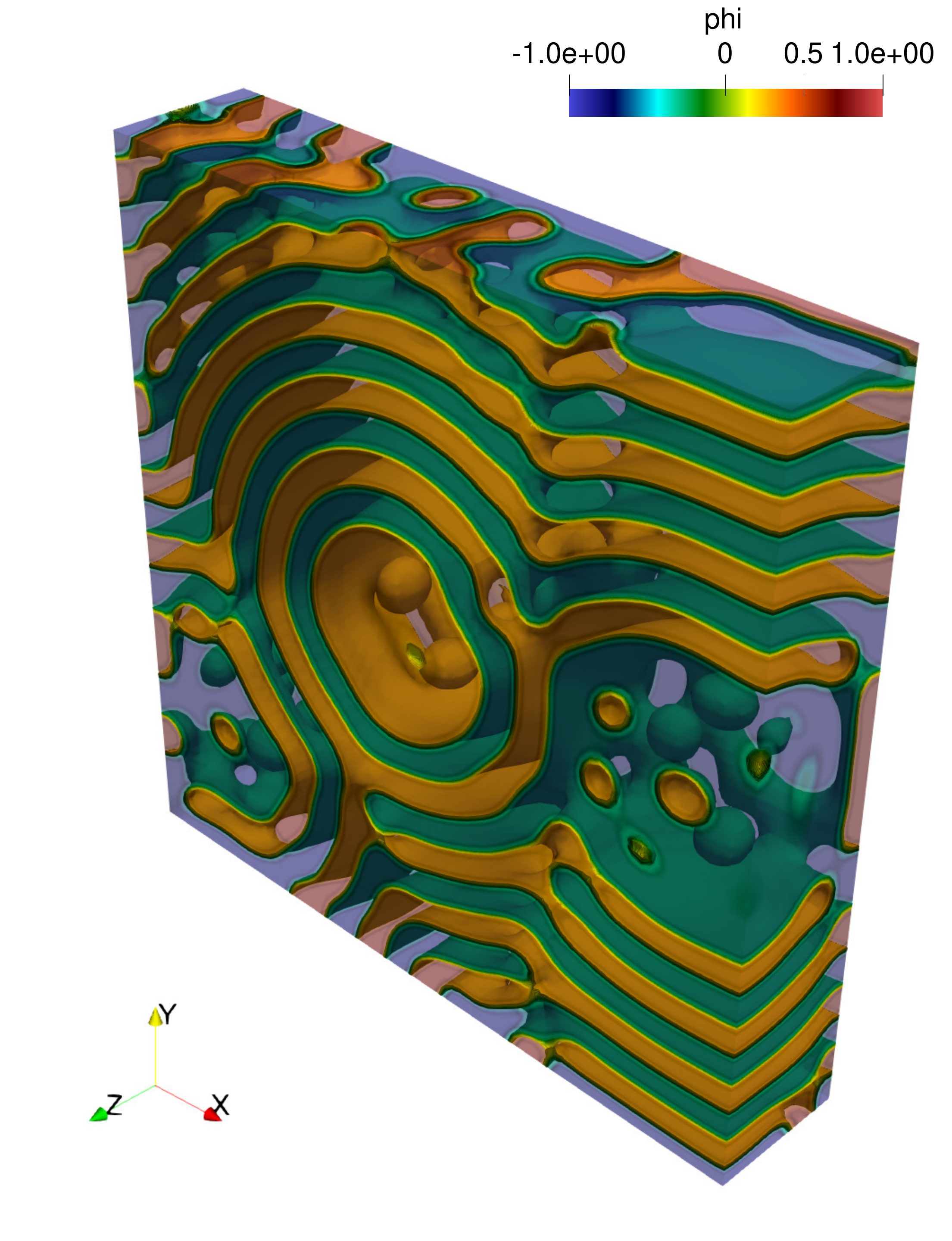}
\includegraphics[width=0.24\textwidth]{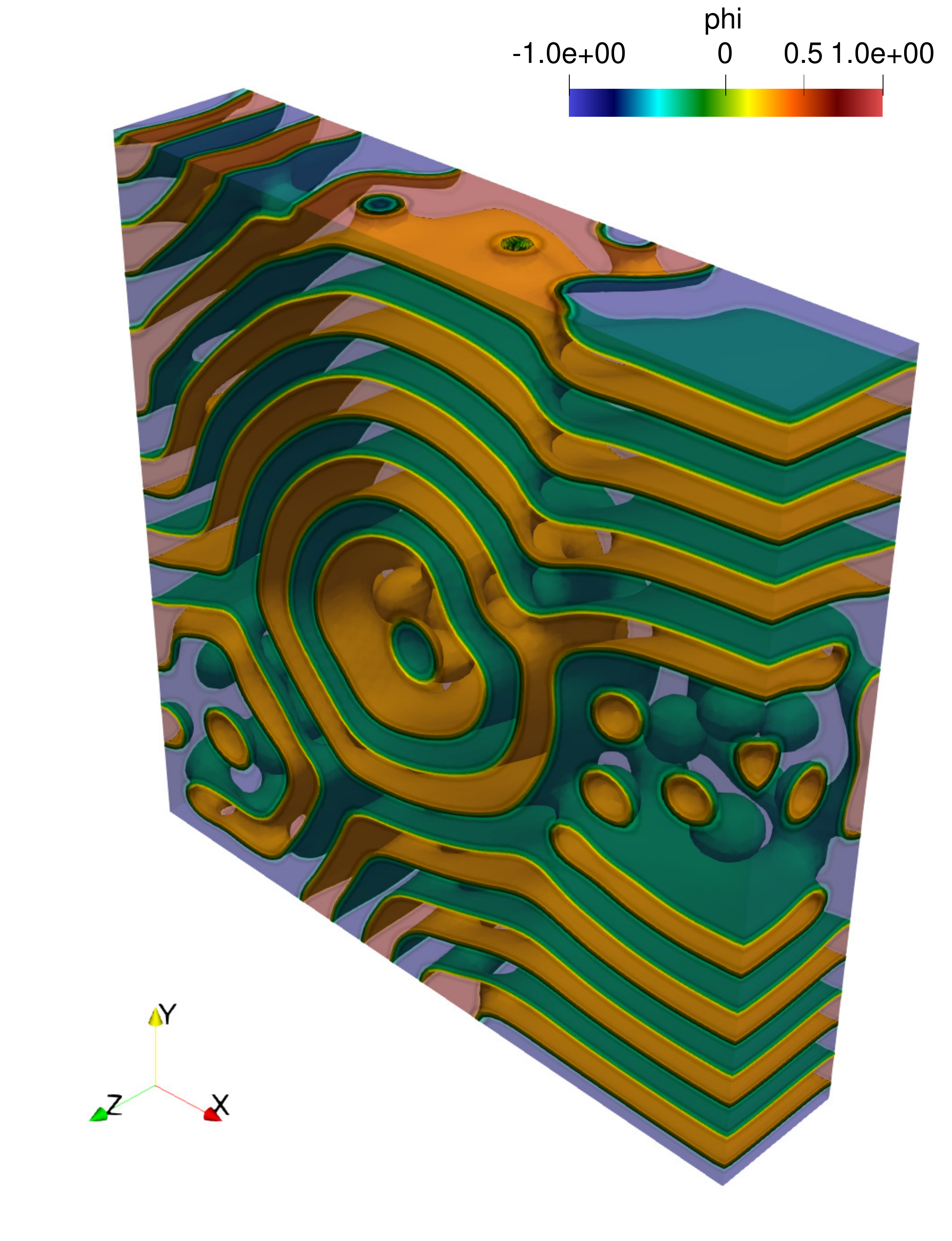}
\includegraphics[width=0.24\textwidth]{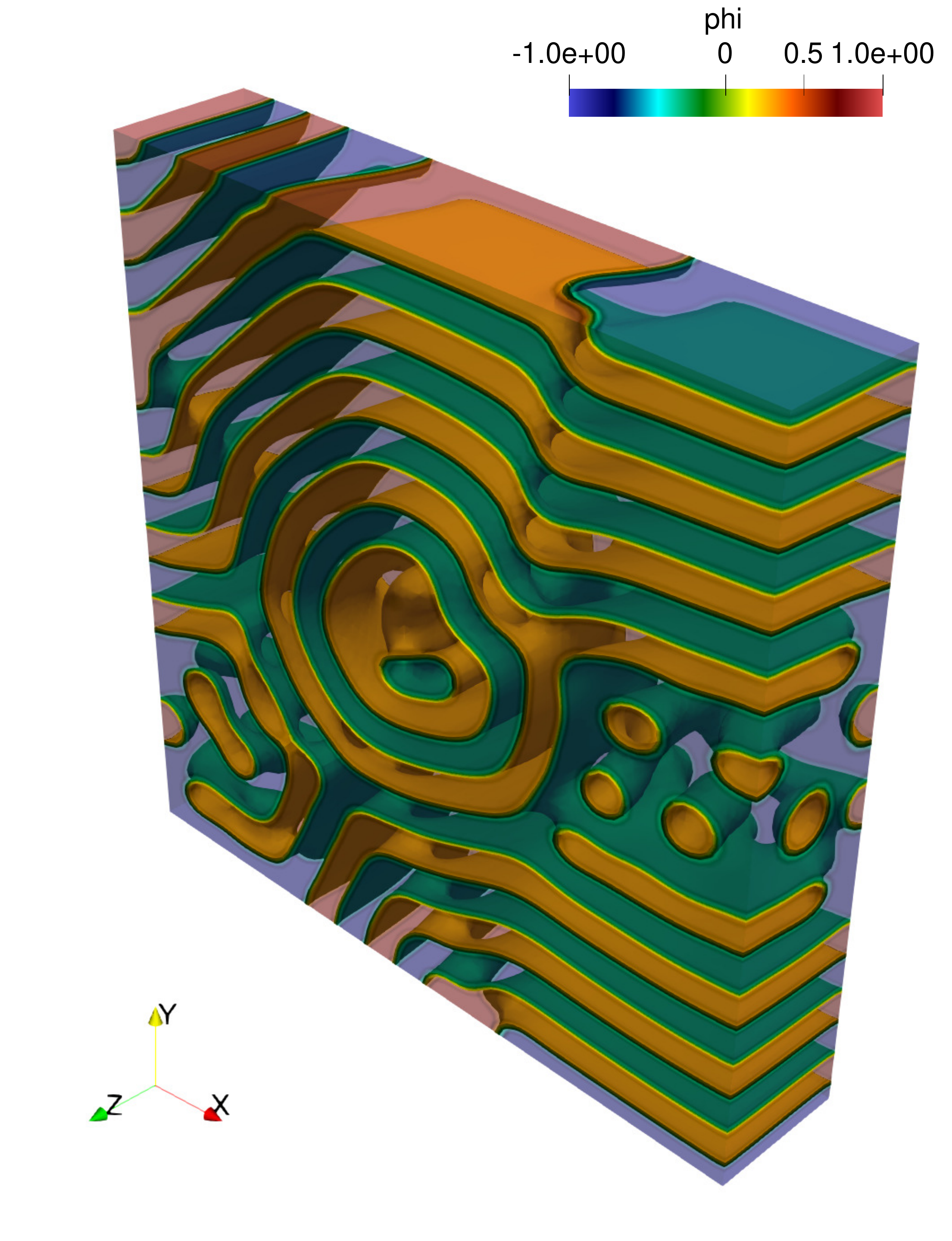}
\\
\includegraphics[width=0.24\textwidth]{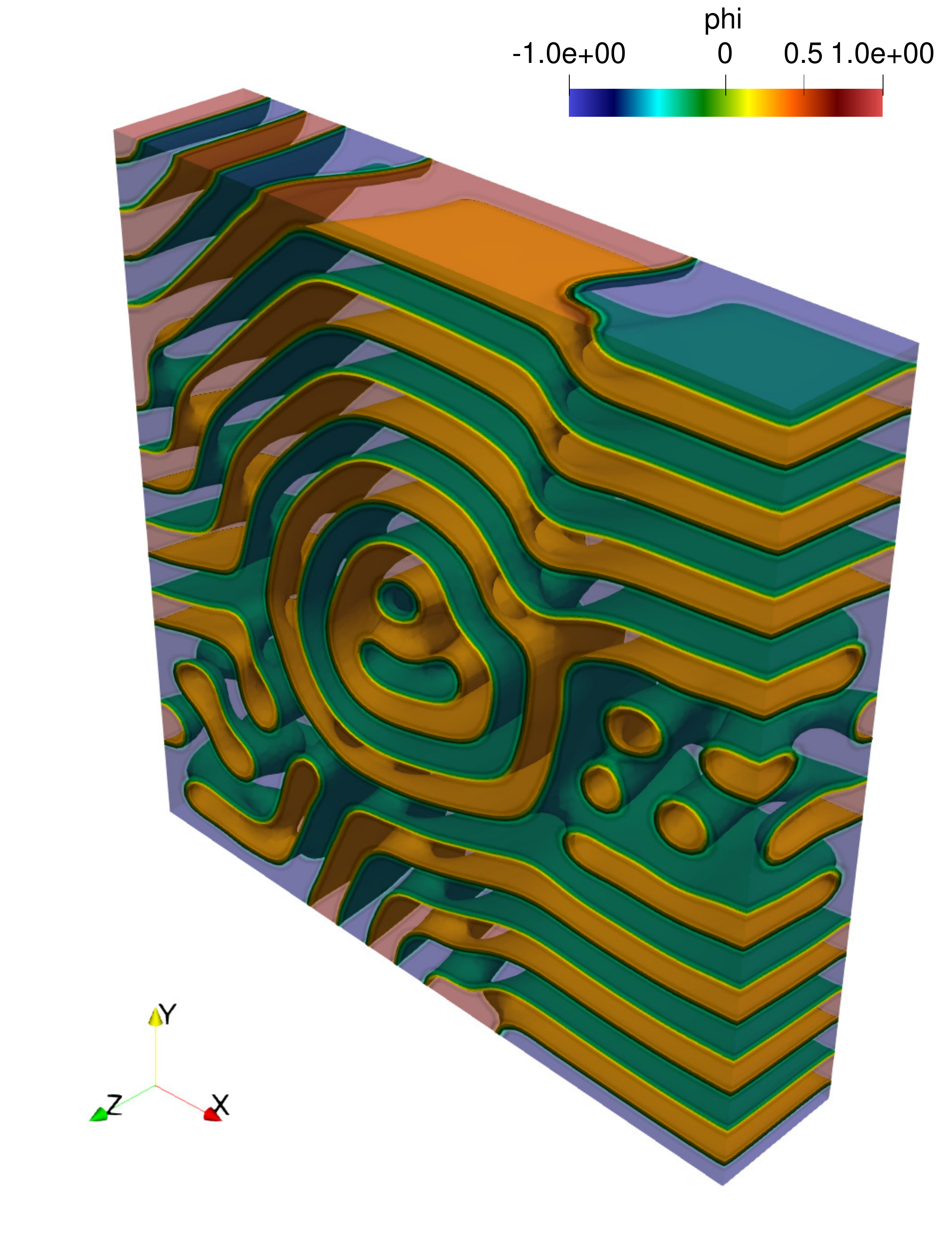}
\includegraphics[width=0.24\textwidth]{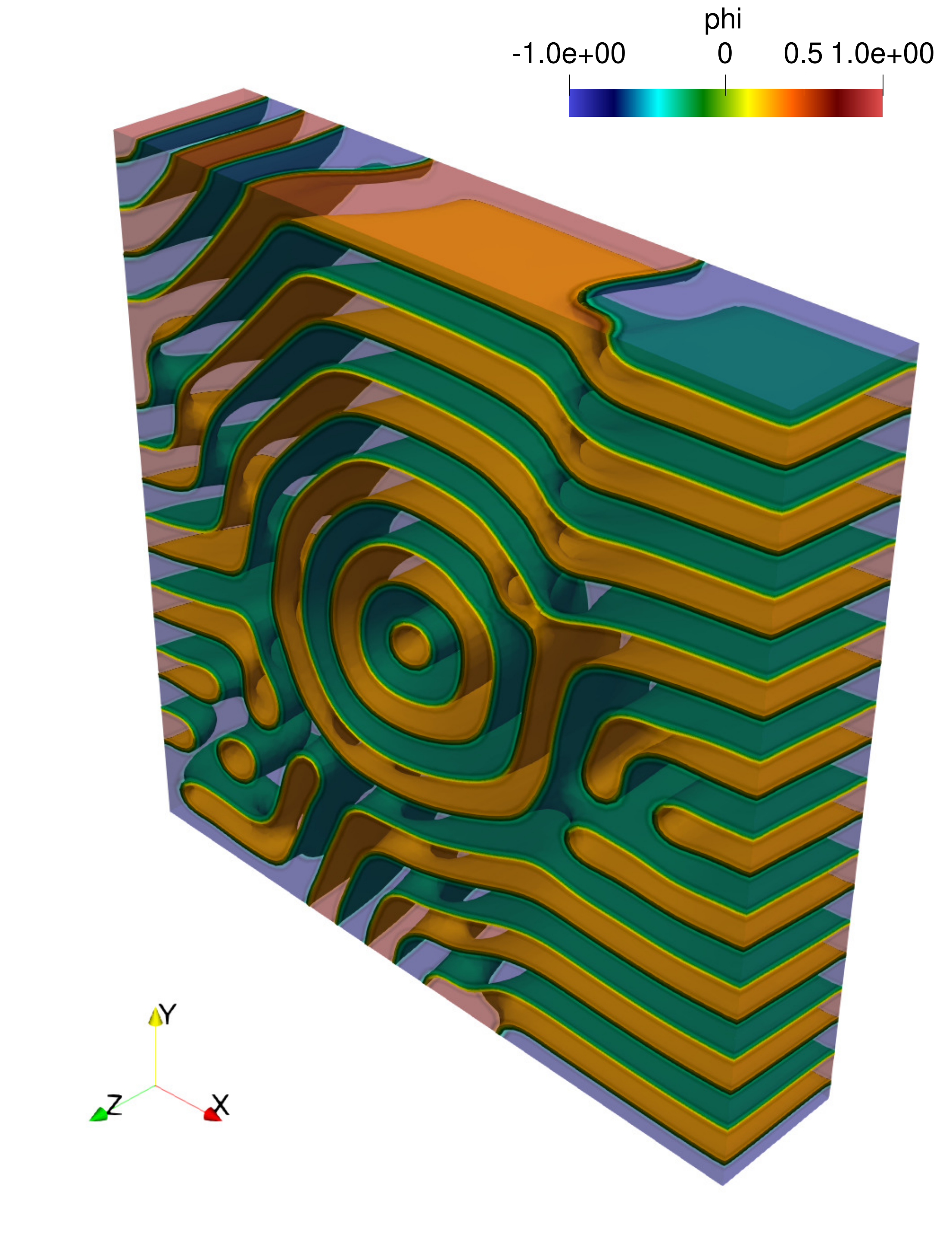}
\includegraphics[width=0.24\textwidth]{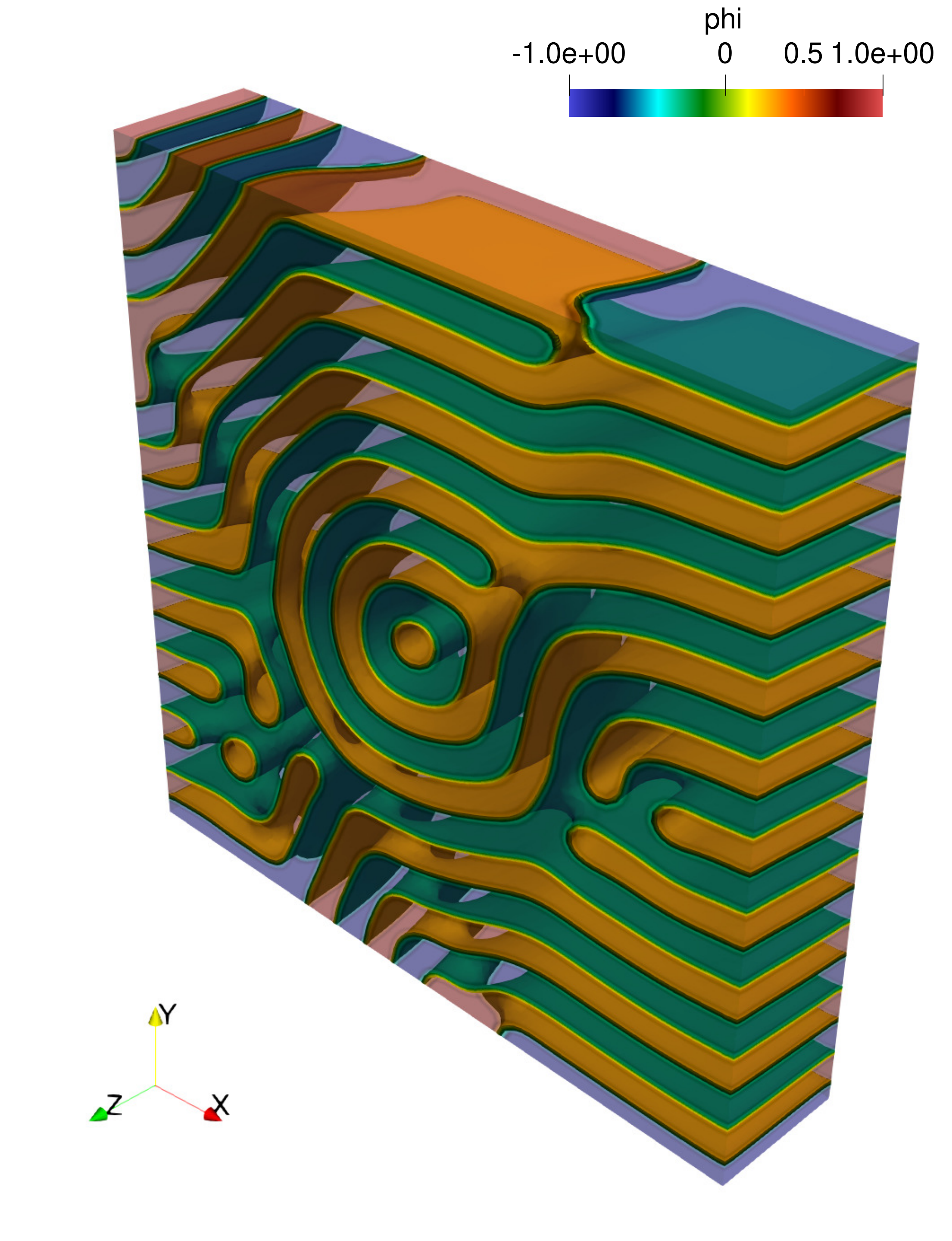}
\includegraphics[width=0.24\textwidth]{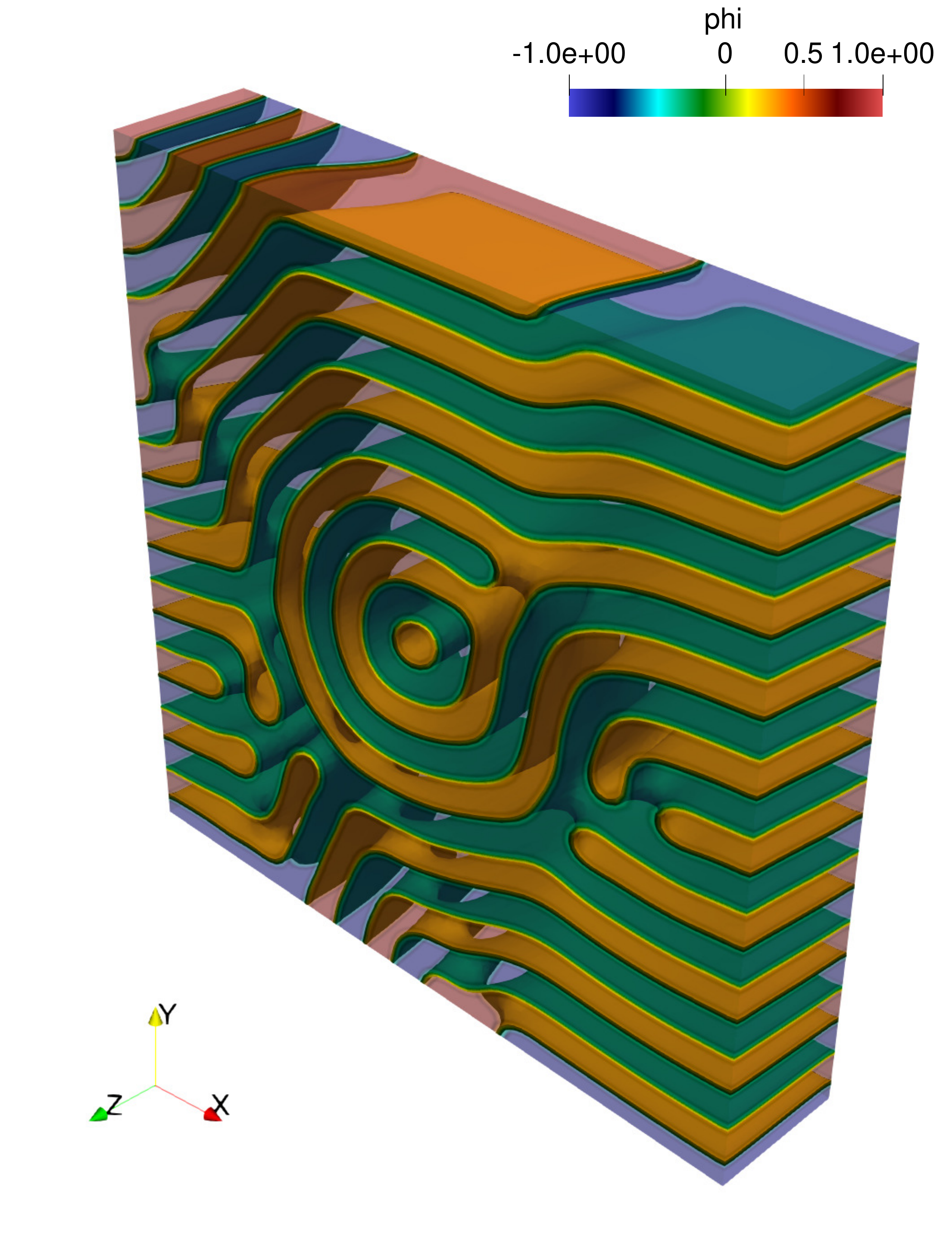}
\caption{Evolution of $\phi$ at times $t=0, 0.0002, 0.0004, 0.0007, 0.001, 0.0015, 0.005, 0.015, 0.025, 0.05, 0.075$ and $0.1$ (from left to right and top to bottom) taking $M=0.1$, $g_2=1$, $h_0=0.5$, $\lambda=0.01$ and $\beta=5$.} \label{fig:3Ddynamics}
\end{center}
\end{figure}

%$M=0.1$, $g_0 = -4$,  $g_2=1$, $h_0=0.5$, $\lambda=0.01$, $\beta=5$
\begin{figure}[h]
\begin{center}
\includegraphics[width=0.45\textwidth]{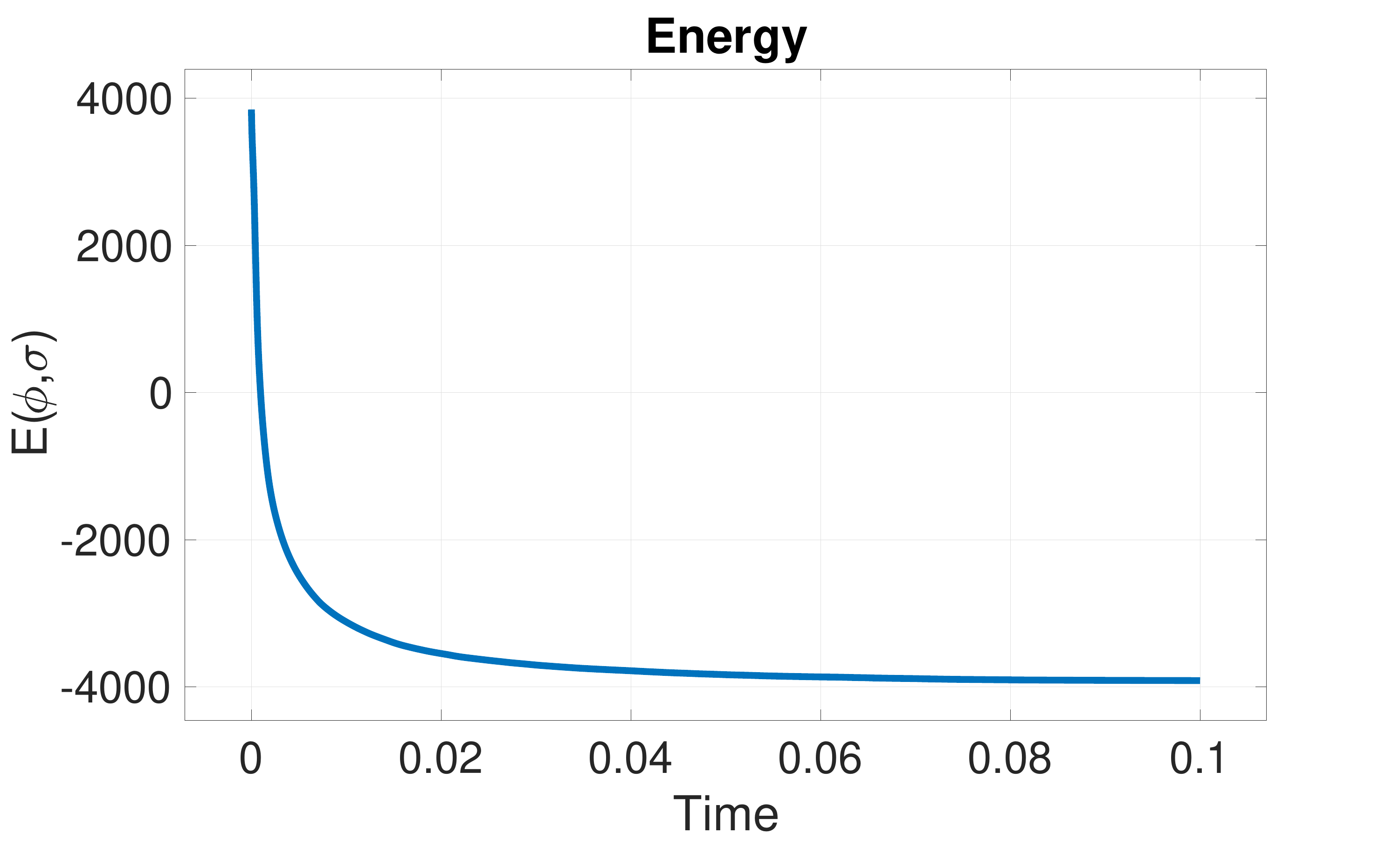}
\includegraphics[width=0.45\textwidth]{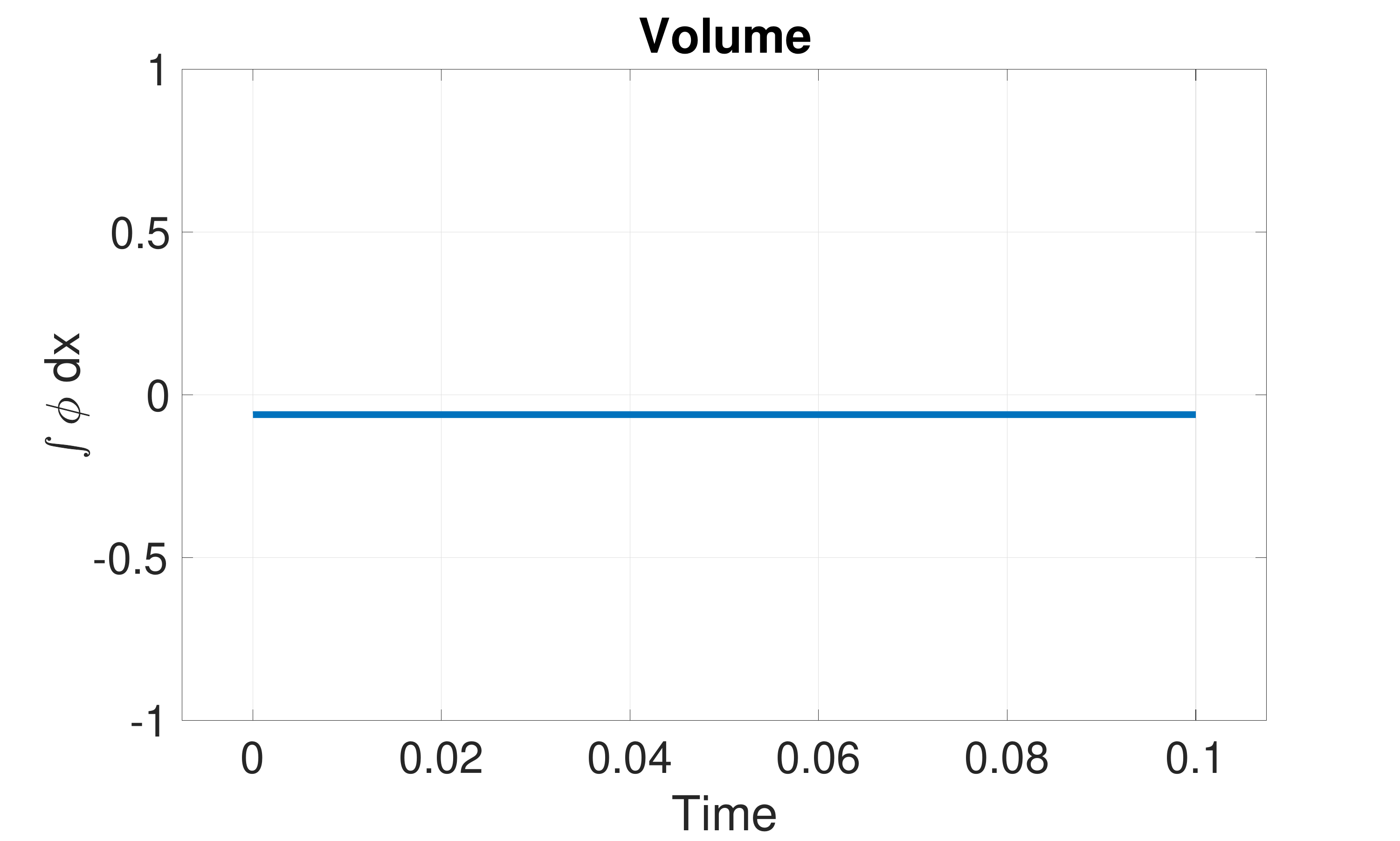}
\\
\includegraphics[width=0.45\textwidth]{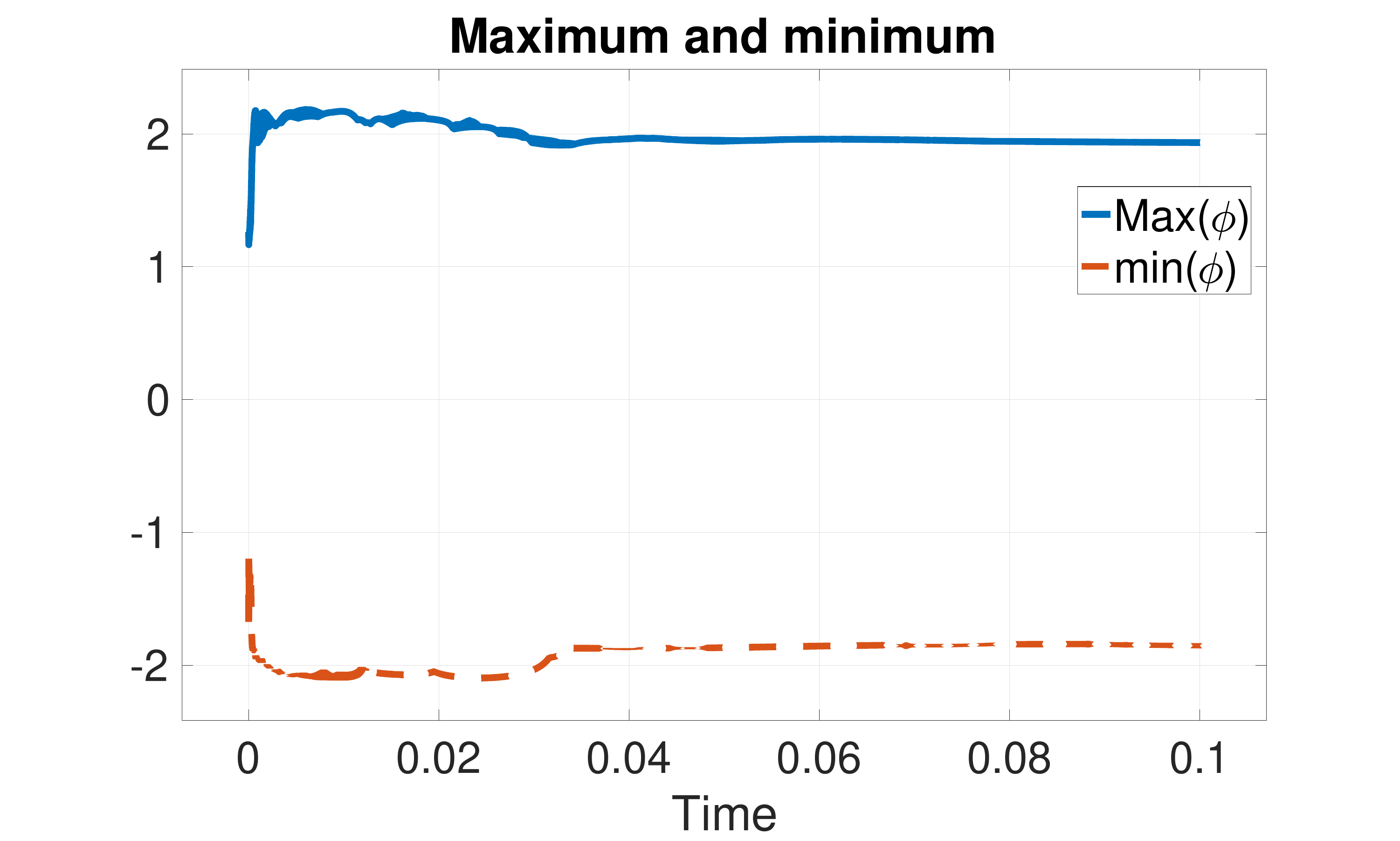}
\includegraphics[width=0.45\textwidth]{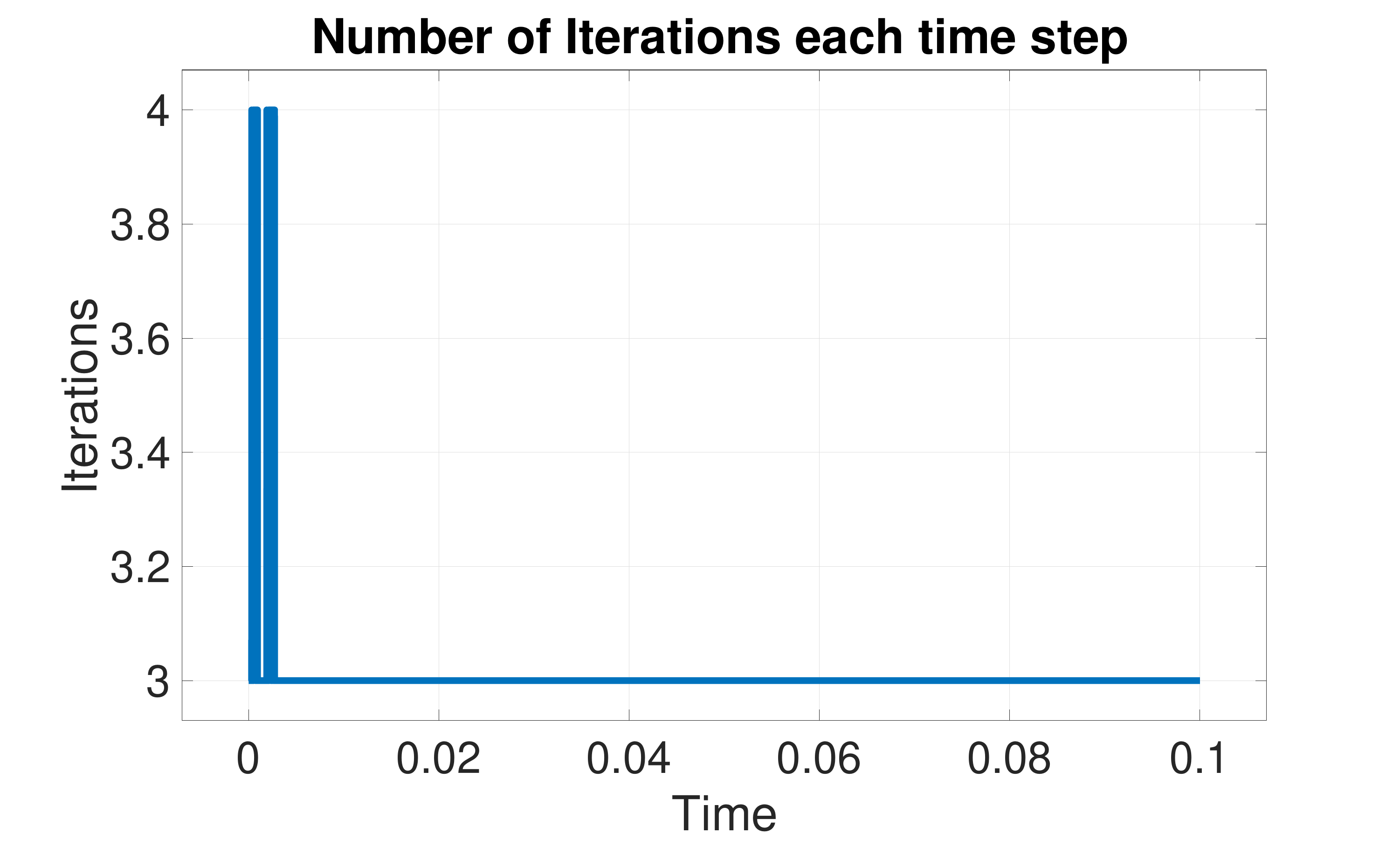}
\caption{Evolution in time of the energy (top left), volume of $\phi$ (top right), maximum and minimum of $\phi$ (bottom left) and the number of iterations to achieve tolerance $\texttt{TOL}=10^{-5}$ (bottom right)  taking $M=0.1$, $g_2=1$, $g_0=-4$, $h_0=0.5$, $\beta=5$ and $\lambda=0.01$.} \label{fig:3Dplot}
\end{center}
\end{figure}

\section{Conclusions}\label{sec:conclusions}
In this paper, we proposed a new numerical scheme that is second order and is energy stable with no  numerical dissipation added to the energy law. We note that unlike the existing numerical schemes for this model, our scheme has demonstrated its ability of capturing the dynamics for a wide range of physical parameters and even in three dimensions. We remark that the mathematical model does not satisfy a maximum principle, nor does it provide any information related to the surfactant concentration. It is a part of our future work to modify the existing model so that it overcomes these two shortcomings. Ideas to circumvent the first drawback include considering a degenerate mobility or a singular potential. Furthermore, a natural way to keep track of the surfactant concentration is to introduce an additional scalar order parameter, as proposed in Chapter 3 of~\cite{GompperSchick94}.

\section*{Acknowledgments}
The first author NSS would like to acknowledge the support of NSF Grant No. DMS-2110774.
\bibliographystyle{siamplain}
\bibliography{references-short}
\end{document}